%% file: manuscript20200604.tex
\documentclass[11pt]{article}

\usepackage{amsmath, amsfonts, amssymb, amsthm,  graphicx, mathtools, enumerate,amsbsy, dsfont}

\usepackage[blocks, affil-it]{authblk}

\usepackage{mathrsfs}

\usepackage[numbers, square, sort]{natbib}
\usepackage[CJKbookmarks=true,
            bookmarksnumbered=true,
			bookmarksopen=true,
			colorlinks=true,
			citecolor=red,
			linkcolor=blue,
			anchorcolor=red,
			urlcolor=blue]{hyperref}
\usepackage[usenames]{color}

\usepackage[letterpaper, left=1.2truein, right=1.2truein, top = 1.2truein, bottom = 1.2truein]{geometry}
\usepackage[ruled, vlined, lined, commentsnumbered]{algorithm2e}
\usepackage[font=small]{caption}
\usepackage{prettyref,soul}

\newtheorem{lemma}{Lemma}[section]
\newtheorem{proposition}{Proposition}[section]
\newtheorem{thm}{Theorem}[section]

\newtheorem{corollary}{Corollary}[section]

\newtheorem{example}{Example}[section]
\newrefformat{eq}{(\ref{#1})}
\newrefformat{chap}{Chapter~\ref{#1}}
\newrefformat{sec}{Section~\ref{#1}}
\newrefformat{algo}{Algorithm~\ref{#1}}
\newrefformat{fig}{Fig.~\ref{#1}}
\newrefformat{tab}{Table~\ref{#1}}
\newrefformat{rmk}{Remark~\ref{#1}}
\newrefformat{clm}{Claim~\ref{#1}}
\newrefformat{def}{Definition~\ref{#1}}
\newrefformat{cor}{Corollary~\ref{#1}}
\newrefformat{lmm}{Lemma~\ref{#1}}
\newrefformat{lemma}{Lemma~\ref{#1}}
\newrefformat{prop}{Proposition~\ref{#1}}
\newrefformat{app}{Appendix~\ref{#1}}
\newrefformat{ex}{Example~\ref{#1}}
\newrefformat{exer}{Exercise~\ref{#1}}
\newrefformat{soln}{Solution~\ref{#1}}
\newrefformat{cond}{Condition~\ref{#1}}

\def\Cov{\textsf{Cov}} % the symbol Cov for covariance used the sans serif letter
\def\Var{\textsf{Var}} % the symbol Var for covariance used the sans serif letter

\def\snr{\textsf{SNR}}
\def\h{\textsf{H}}

\def\text#1{\mbox{\rm #1}}

\def\ave{\textsf{ave}}
\def\sgn{\text{sign}}
\def\h{\textsf{H}}
\def\k{\textsf{K}}
\def\S{\mathfrak{S}}

\newcommand{\argmin}{\mathop{\rm argmin}}

\newcommand{\indc}[1]{{\mathbb{I}\left\{{#1}\right\}}}
\newcommand{\norm}[1]{\left\|{#1} \right\|}

\newcommand{\wh}{\widehat}

\newtheorem*{condb'}{Condition B'}

\newcommand{\br}[1]{\left( #1 \right)}

\newcommand{\cbr}[1]{\left\{ #1 \right\}}

\newcommand{\abs}[1]{\left| #1 \right|}

\newcommand{\p}{\mathbb{P}}
\newcommand{\E}{\mathbb{E}}

\title{Partial Recovery for Top-$k$ Ranking:\\
Optimality of MLE and Sub-Optimality of Spectral Method
}
\author[1]{Pinhan Chen}
\author[1]{Chao Gao}
\author[2]{Anderson Y. Zhang}
\affil[1]{
University of Chicago
}

\affil[2]{
University of Pennsylvania
}
\begin{document}
\maketitle

\begin{abstract}

Given partially observed pairwise comparison data generated by the Bradley-Terry-Luce (BTL) model, we study the problem of top-$k$ ranking. That is, to optimally identify the set of top-$k$ players. We derive the minimax rate with respect to a normalized Hamming loss. This provides the first result in the literature that characterizes the partial recovery error in terms of the proportion of mistakes for top-$k$ ranking. We also derive the optimal signal to noise ratio condition for the exact recovery of the top-$k$ set. The maximum likelihood estimator (MLE) is shown to achieve both optimal partial recovery and optimal exact recovery. On the other hand, we show another popular algorithm, the spectral method, is in general sub-optimal. Our results complement the recent work by \cite{chen2019spectral} that shows both the MLE and the spectral method achieve the optimal sample complexity for exact recovery. It turns out the leading constants of the sample complexity are different for the two algorithms. Another contribution that may be of independent interest is the analysis of the MLE without any penalty or regularization for the BTL model. This closes an important gap between theory and practice in the literature of ranking.

\smallskip

%\textbf{Keywords.} {$k$-means clustering, approximate ranking, high-dimensional statistics, Hamming distance, variable selection}
\end{abstract}
% \begin{keyword}[class=AMS]
% \kwd[Primary ]{62H12}
% \kwd[; secondary ]{62C20}
% \end{keyword}
% \begin{keyword}
% \kwd{Convex programming, group Lasso, Minimax rates, Rates of convergence, Sparse CCA (SCCA)}
% \end{keyword}

% \end{frontmatter}

\section{Introduction}\label{sec:intro}

Given partially observed pairwise comparison data from $n$ players, a central statistical question is how to optimally aggregate the comparison results and to find the leading top $k$ players. This problem is known as top-$k$ ranking, which has important applications in many areas such as web search \citep{dwork2001rank,cossock2006subset} and competitive sports \citep{motegi2012network,sha2016chalkboarding}. In this paper, our goal is to study the statistical limits of both \textit{partial} and \textit{exact} recovery of the top-$k$ ranking problem.

We will focus on the popular Bradley-Terry-Luce (BTL) pairwise comparison model \citep{bradley1952rank,luce2012individual}. That is, we observe $L$ games played between $i$ and $j$, and the outcome is modeled by
\begin{equation}
y_{ijl}\stackrel{ind}{\sim}\text{Bernoulli}\left(\frac{w_i^*}{w_i^*+w_j^*}\right),\quad l=1,\cdots, L. \label{eq:BTL-w}
\end{equation}
We only observe outcomes from a small subset of pairs. This subset $\mathcal{E}$ is modeled by edges generated by an Erd\H{o}s-R\'{e}nyi \citep{erdHos1960evolution} random graph with connection probability $p$ on the $n$ players. More details of the model will be given in Section \ref{sec:mm}. With the observations $\{y_{ijl}\}_{(i,j)\in\mathcal{E},l\in[L]}$, the goal is to reliably recover the set of top-$k$ players with the largest skill parameters $w_i^*$.

Theoretical properties of the top-$k$ ranking problem have been studied by \cite{chen2015spectral,jang2016top,shah2017simple,chen2017competitive,jang2017optimal,negahban2017rank,chen2019spectral} and references therein. The literature is mainly focused the problem of exact recovery. That is, to investigate the signal to noise ratio condition under which one can recovery the top-$k$ set without any error in probability. For this purpose, the state-of-the-art result is obtained by the recent work \cite{chen2019spectral}. It was shown by \cite{chen2019spectral} that both the MLE and the spectral method can perfectly identify the top-$k$ players under optimal sample complexity up to some constant factor. This discovery was also verified by a numerical experiment that shows almost identical performances of the two methods. The results of \cite{chen2019spectral} lead to the following intriguing research questions. What is the leading constant factor of the optimal sample complexity? Are the MLE and the spectral method still optimal if we take the leading constant into consideration?

In this paper, we give complete answers to the above questions. Our results show that while the MLE achieves a leading constant that is information-theoretically optimal, the spectral method only achieves a sub-optimal constant. In particular, the MLE achieves exact recovery when
\begin{equation}
npL\Delta^2 > 2.001 V(\kappa)\left(\sqrt{\log k} + \sqrt{\log(n-k)}\right)^2, \label{eq:exact-intro}
\end{equation}
and the spectral method requires
$$npL\Delta^2 > 2.001 \overline{V}(\kappa)\left(\sqrt{\log k} + \sqrt{\log(n-k)}\right)^2.$$
In the above two formulas, $\Delta$ is the logarithmic gap of the skill parameters between the top-$k$ group and the rest of the players. The parameter $\kappa$ is the dynamic range of the skill vector that will be defined in Section \ref{sec:mm}. The performances of the two methods are precisely characterized by the two functions $V(\kappa)$ and $\overline{V}(\kappa)$, which are understood to be the effective variances of the two algorithms. The two functions satisfy the strict inequality that $\overline{V}(\kappa)> V(\kappa)$ for all $\kappa> 0$, and the equality $\overline{V}(\kappa)= V(\kappa)$ only holds when $\kappa=0$. We also establish an information-theoretic lower bound that shows the MLE constant $V(\kappa)$ is optimal, and it characterizes the phase transition boundary of exact recovery for the top-$k$ ranking problem.

We would like to emphasize that our results do not contradict the conclusions of \cite{chen2019spectral}. On the contrary, the current paper complements and refines the results of \cite{chen2019spectral}. The optimality claim made by \cite{chen2019spectral} on both the MLE and the spectral method only refers to the order of the sample complexity. Our results show that the performances of the two algorithms can be drastically different when the dynamic range parameter $\kappa$ is strictly positive. We are also able to explain why the numerical experiment conducted in \cite{chen2019spectral} demonstrates nearly identical performances of the MLE and the spectral method. Note that the experiment in \cite{chen2019spectral} was conducted with the skill parameters $w_i^*$ only taking two possible values, $e^{\Delta}$ or $1$, depending on whether $i$ belongs to the top-$k$ group or not. We show in Section \ref{sec:compare} that this configuration of $w^*$ is asymptotically equivalent to $\kappa=0$,  which is the only case that makes $\overline{V}(\kappa)= V(\kappa)$, and thus the nearly identical performances of the two algorithms are actually well expected by our theory. As long as $w^*$ deviates from this simple two-piece structure, our extensive numerical experiments in this paper show that the MLE always dominates the spectral method, and the advantage of the MLE is usually quite significant.

In addition to the exact recovery results, we have also obtained a series of results for partial recovery. We observe that top-$k$ ranking can be viewed as a clustering problem. That is, one wants to cluster the players into two groups of sizes $k$ and $n-k$, respectively. Therefore, it is more natural to consider the problem of partial recovery by analyzing the proportion of players that are clustered into a wrong group. Clearly, this problem is more relevant in practice, since one rarely expects any real application where top-$k$ ranking can be done without any error. From a mathematical point of view, the partial recovery problem is more general and we will show in Section \ref{sec:MLE-result} that an optimal partial recovery error bound will lead to the optimal exact recovery condition (\ref{eq:exact-intro}). To the best of our knowledge, a systematic study of partial recovery for top-$k$ ranking has never been done in the literature. Our paper is perhaps the first work that formulates the top-$k$ ranking problem into a decision-theoretic framework and derives the minimax optimal partial recovery error rate.
Similar to the results of exact recovery, we show that the MLE is also optimal for partial recovery. It has an exponential error bound with respect to a normalized Hamming loss. The error exponent is shown to depend on the variance function $V(\kappa)$. In comparison, the spectral method still achieves a sub-optimal error rate for partial recovery, with the error exponent depending on $\overline{V}(\kappa)$.

Recently, a few papers provide sharp analysis of spectral methods on some high-dimensional estimation problems and show spectral methods can achieve optimal theoretical guarantees just as MLEs. For example, it was shown by \cite{abbe2017entrywise} that spectral clustering achieves optimal community detection for a special class of stochastic block models (SBMs). The paper \cite{loffler2019optimality} proved spectral clustering is also optimal under Gaussian mixture models. We emphasize that the results of both papers imply that not only the order of the sample complexity of spectral clustering is optimal, but even the leading constant is optimal, at least in the setting of SBMs and Gaussian mixture models. The results of the current paper, however, show that the optimality of spectral methods may not hold under more complicated settings such as the BTL model.

Finally, we discuss another contribution of the paper that may be of independent interest. That is, we are able to give a sharp analysis of the MLE under the BTL model. Previous analyses of the MLE in the literature \citep{chen2015spectral,negahban2017rank,chen2019spectral} all impose some additional regularization to address the challenge that the Hessian of the log-likelihood function is not well behaved. Whether the \textit{vanilla} MLE works theoretically without any penalty or regularization remains an open problem. Our analysis solves this open problem by relating a regularized MLE to an $\ell_{\infty}$-constrained MLE. This allows us to show that the solution to the $\ell_{\infty}$-constrained MLE lies in the interior of the constraint. Thus, we can conclude that the $\ell_{\infty}$-constrained MLE is equivalent to the vanilla MLE in its original form. This equivalence then leads to the desired control of the spectrum of the Hessian matrix, which is the most critical step of our analysis.

The rest of the paper is organized as follows. We introduce the setting of the problem in Section \ref{sec:mm}. The results of the MLE and the spectral method will be given in Section \ref{sec:MLE-result} and Section \ref{sec:spec-result}, respectively. We then comprehensively compare the two methods in Section \ref{sec:compare} by numerical experiments. Section \ref{sec:lower-partial} presents a minimax lower bound for partial recovery. In Section \ref{sec:local}, we analyze the error rates of the MLE and the spectral method for each individual parameter. The proofs of our main results are given in Sections \ref{sec:proof-MLE}-\ref{sec:pf_local}, with Section \ref{sec:proof-MLE} for the analysis of the MLE, Section \ref{sec:proof-spec} for the analysis of the spectral method, Section \ref{sec:pf-minmax-l} for the proofs of the lower bounds, and Section \ref{sec:pf_local} for the proof of local error rates. Finally, a few technical lemmas will be given and proved in Section \ref{sec:pf-tech}.

We close this section by introducing some notation that will be used in the paper. For an integer $d$, we use $[d]$ to denote the set $\{1,2,...,d\}$. Given two numbers $a,b\in\mathbb{R}$, we use $a\vee b=\max(a,b)$ and $a\wedge b=\min(a,b)$. We also write $a_+=\max(a,0)$. For two positive sequences $\{a_n\},\{b_n\}$, $a_n\lesssim b_n$ or $a_n=O(b_n)$ means $a_n\leq Cb_n$ for some constant $C>0$ independent of $n$, $a_n=\Omega(b_n)$ means $b_n=O(a_n)$, and $a_n\asymp b_n$ means $a_n\lesssim b_n$ and $b_n\lesssim a_n$. We also write $a_n=o(b_n)$ when $\limsup_n\frac{a_n}{b_n}=0$. For a set $S$, we use $\indc{S}$ to denote its indicator function and $|S|$ to denote its cardinality. For a vector $v\in\mathbb{R}^d$, its norms are defined by $\norm{v}_1=\sum_{i=1}^d|v_i|$, $\norm{v}^2=\sum_{i=1}^dv_i^2$ and $\norm{v}_{\infty}=\max_{1\leq i\leq d}|v_i|$. The notation $\mathds{1}_{d}$ means a $d$-dimensional column vector of all ones. For any $v\in\mathbb{R}^d$, we write $\ave(v)=d^{-1}\mathds{1}_{d}^Tv$. Given $p,q\in(0,1)$, the Kullback-Leibler divergence is defined by $D(p\|q)=p\log\frac{p}{q}+(1-p)\log\frac{1-p}{1-q}$. For a natural number $n$, $\S_n$ is the set of permutations on $[n]$. The notation $\mathbb{P}$ and $\mathbb{E}$ are used for generic probability and expectation whose distribution is determined from the context.

\section{Models and Methods}\label{sec:mm}

\paragraph{The BTL Model.}

We start by introducing the setting of our problem. Consider $n$ players, and each one is associated with a positive latent skill parameter $w_i^*$ for $i\in[n]$. The comparison scheme of the $n$ players is characterized by an Erd\H{o}s-R\'{e}nyi random graph $A\sim\mathcal{G}(n,p)$. That is, $A_{ij}\stackrel{iid}{\sim}\text{Bernoulli}(p)$ for all $1\leq i<j\leq n$. For a pair $(i,j)$ that is connected by the random graph and $A_{ij}=1$, we observe $L$ games played between $i$ and $j$. The outcome of the games is modeled by the Bradley-Terry-Luce (BTL) model (\ref{eq:BTL-w}).
Our goal is to identify the top-$k$ players whose skill parameters $w_i^*$'s have the largest values.

To formulate this problem from a decision-theoretic point of view, we reparametrize the BTL model (\ref{eq:BTL-w}) by a sorted vector $\theta^*$ and a rank vector $r^*$. A sorted vector $\theta^*$ satisfies $\theta_1^*\geq \theta_2^* \geq \cdots \geq \theta_n^*$, and a rank vector $r^*$ is an element of permutation $r^*\in\S_n$. Then, the BTL model (\ref{eq:BTL-w}) can be equivalently written as
\begin{equation}
y_{ijl}\stackrel{ind}{\sim}\text{Bernoulli}(\psi(\theta^*_{r^*_i}-\theta^*_{r^*_j})),\quad l=1,\cdots, L. \label{eq:BTL-theta}
\end{equation}
where $\psi(\cdot)$ is the sigmoid function $\psi(t)=\frac{1}{1+e^{-t}}$. In the original representation, we have $w_i^*=\exp(\theta_{r_i^*}^*)$ for all $i\in[n]$. With (\ref{eq:BTL-theta}), the top-$k$ ranking problem is to identify the subset $\{i\in[n]:r_i^*\leq k\}$ from the random comparison data. This is a typical semiparametric problem because of the presence of the nuisance parameter $\theta^*$.

\paragraph{Loss Function for Top-$k$ Ranking.}

Our goal is to study optimal top-$k$ ranking in terms of both \textit{partial} and \textit{exact} recovery. We thus introduce a loss function to quantify the error of top-$k$ ranking. Given any $\wh{r},r^*\in\S_k$, define the normalized Hamming distance by
\begin{equation}
\h_k(\wh{r},r^*)=\frac{1}{2k}\left(\sum_{i=1}^n\indc{\wh{r}_i>k, r^*_i\leq k}+\sum_{i=1}^n\indc{\wh{r}_i\leq k, r^*_i> k}\right). \label{eq:h-loss}
\end{equation}
The definition (\ref{eq:h-loss}) gives a natural loss function for top-$k$ ranking, since $\h_k(\wh{r},r^*)$ can be equivalently written as the cardinality of the symmetric difference of the sets $\{i\in[n]:\wh{r}_i\leq k\}$ and $\{i\in[n]:r_i^*\leq k\}$ normalized by $2k$. The value of $\h_k(\wh{r},r^*)$ is always within the unit interval $[0,1]$. Moreover, $\h_k(\wh{r},r^*)=0$ if and only if $\{i\in[n]:\wh{r}_i\leq k\}=\{i\in[n]:r_i^*\leq k\}$.

The loss function (\ref{eq:h-loss}) can be related to various quantities previously defined in the literature. One of the most popular distances to compare two rank vectors is the \textit{Kendall tau distance}, defined as
$$\k(\wh{r},r^*)=\frac{1}{n}\sum_{1\leq i<j\leq n}\indc{\sgn(\wh{r}_i-\wh{r}_j)\sgn(r_i^*-r_j^*)<0}.$$
Since $\k(\wh{r},r^*)$ counts all pairwise differences in the ranking relation, it is a stronger distance than (\ref{eq:h-loss}). While $\k(\wh{r},r^*)=0$ requires $\wh{r}=r^*$, $\h_k(\wh{r},r^*)=0$ only requires the two top-$k$ sets are identical regardless of the actual ranks of the members of the sets. In fact, the study of the BTL model under $\k(\wh{r},r^*)$, called full ranking, is also a very interesting problem, and will be considered in a different paper.

 As we have discussed in Section \ref{sec:intro}, the top-$k$ ranking problem can be thought of as a special variable selection problem. Variable selection under the normalized Hamming loss has recently been studied by \cite{butucea2018variable,ndaoud2020optimal}. Consider either a Gaussian sequence model or a regression model with coefficient vector $\beta^*\in\mathbb{R}^p$ that satisfies either $\beta_j^*=0$ or $|\beta_j^*|>a$. The papers \cite{butucea2018variable,ndaoud2020optimal} consider estimating $\beta^*$ under the loss
$$\overline{\h}_s(\wh{\beta},\beta^*)=\frac{1}{2s}\left(\sum_{j=1}^p\indc{|\wh{\beta}_j|>a,\beta_j^*=0}+\sum_{j=1}^p\indc{\wh{\beta}_j=0,|\beta_j^*|>a}\right),$$
where $s$ is the number of $\beta_j^*$'s that are not zero. One can clearly see the similarity between the two loss functions $\h_k(\wh{r},r^*)$ and $\overline{\h}_s(\wh{\beta},\beta^*)$. Similarly, the loss $\overline{\h}_s(\wh{\beta},\beta^*)$ only characterizes the estimation error of the set $\{j\in[p]:|\beta_j^*| >a\}$, and $\overline{\h}_s(\wh{\beta},\beta^*)=0$ if and only if $\{j\in[p]:|\wh{\beta}_j| >a\}=\{j\in[p]:|\beta_j^*| >a\}$.

\paragraph{Parameter Space.}

For the nuisance parameter $\theta^*$ of the model (\ref{eq:BTL-theta}), it is necessary that there exists a positive gap between $\theta_k^*$ and $\theta_{k+1}^*$ for the top-$k$ set $\{i\in[n]:r_i^*\leq k\}$ to be identifiable. We introduce a parameter space for this purpose. For any $0\leq \Delta\leq \kappa$, define
$$\Theta(k,\Delta,\kappa)=\left\{\theta\in\mathbb{R}^n: \theta_1\geq\cdots\geq\theta_n, \theta_k-\theta_{k+1}\geq\Delta, \theta_1-\theta_n\leq\kappa\right\}.$$
For any $\theta^*\in\Theta(k,\Delta,\kappa)$, a positive $\Delta$ guarantees that there is a separation between the group of top-$k$ players and the rest. The number $\kappa$ is called \textit{dynamic range} of the problem.\footnote{For readers who are familiar with \cite{chen2019spectral}, we note that our definitions of $\Delta$ and $\kappa$ are slightly different from those in \cite{chen2019spectral}.} This is a very important quantity, since it is closely related to the effective variance of the problem. Our results will give the exact dependence of the top-$k$ ranking error on both $\Delta$ and $\kappa$.

\paragraph{MLE and Spectral Method.}

We study and compare the performances of two algorithms in the paper. The first algorithm is based on the maximum likelihood estimator (MLE). For each $(i,j)$, we use the notation $\bar{y}_{ij}=\frac{1}{L}\sum_{l=1}^Ly_{ijl}$. Throughout the paper, we adopt the convention of notation that $A_{ij}=A_{ji}$ and $\bar{y}_{ij}=1-\bar{y}_{ji}$. Then, the negative log-likelihood function is given by
\begin{equation}
\ell_n(\theta)=\sum_{1\leq i<j\leq n}A_{ij}\left[\bar{y}_{ij}\log\frac{1}{\psi(\theta_i-\theta_j)}+(1-\bar{y}_{ij})\log\frac{1}{1-\psi(\theta_i-\theta_j)}\right]. \label{eq:likelihood-BTL}
\end{equation}
Define the MLE,
\begin{equation}
\wh{\theta} \in \argmin_{\theta:\mathds{1}_{n}^T\theta=0}\ell_n(\theta). \label{eq:pure-MLE}
\end{equation}
It can be shown that $\wh{\theta}$ is unique as long as the comparison graph is connected.
Then, set $\wh{r}$ to be the rank of players based on $\wh{\theta}$. In other words, find any $\wh{r}\in\S_n$ such that $\wh{\theta}_{\wh{\sigma}_1}\geq\cdots\geq\wh{\theta}_{\wh{\sigma}_n}$ is satisfied, where $\wh{\sigma}$ is the inverse of $\wh{r}$. We emphasize that the MLE (\ref{eq:pure-MLE}) is written in its vanilla version, without any constraint or penalty. To the best of our knowledge, (\ref{eq:pure-MLE}) has not been previously analyzed in the literature.

Another popular algorithm for ranking is the spectral method, also known as Rank Centrality proposed by \cite{negahban2017rank}. Define a matrix $P\in\mathbb{R}^{n\times n}$ by
\begin{equation}
P_{ij}=\begin{cases}
\frac{1}{d}A_{ij}\bar{y}_{ji}, & i\neq j, \\
1 - \frac{1}{d}\sum_{l\in[n]\backslash\{i\}}A_{il}\bar{y}_{li}, & i=j,
\end{cases} \label{eq:spec-P}
\end{equation}
where $d$ needs to be at least the maximum degree of the random graph $A$. We just set $d=2np$ throughout the paper.
One can check that $P$ is a transition matrix of a Markov chain. To see why $P$ is useful, we can compute the conditional expectation of $P$ given the random graph $A$,
$$P_{ij}^*=\begin{cases}
\frac{1}{d}A_{ij}\psi(\theta^*_{r_j^*}-\theta^*_{r_i^*}), & i\neq j, \\
1 - \frac{1}{d}\sum_{l\in[n]\backslash\{i\}}A_{il}\psi(\theta^*_{r_l^*}-\theta^*_{r_i^*}), & i=j.
\end{cases}$$
The stationary distribution induced by the Markov chain $P^*$ is
$$(\pi^*)^T=\left(\frac{\exp(\theta_{r_1^*}^*)}{\sum_{i=1}^n\exp(\theta_{r_i^*}^*)},\cdots,\frac{\exp(\theta_{r_n^*}^*)}{\sum_{i=1}^n\exp(\theta_{r_i^*}^*)}\right).$$
One can easily check that $(\pi^*)^TP^*=(\pi^*)^T$. Since $\pi^*$ preserves the order of $\{\theta_{r_i^*}^*\}$, the set with the $k$ largest $\pi_i^*$'s is the top-$k$ group. With the sample version $P$, we can first compute its stationary distribution $\wh{\pi}$, and then find any $\wh{r}\in\S_n$ such that $\wh{\pi}_{\wh{\sigma}_1}\geq\cdots\geq \wh{\pi}_{\wh{\sigma}_n}$, with $\wh{\sigma}$ being the inverse of $\wh{r}$.

\section{Results for the MLE}\label{sec:MLE-result}

We study the property of MLE in this section. Our first result gives theoretical guarantees for (\ref{eq:pure-MLE}) under both $\ell_2$ and $\ell_{\infty}$ loss functions.

\begin{thm}\label{thm:MLE-estimation}
Assume $p\geq c_0\frac{\log n}{n}$ for some sufficiently large constant $c_0>0$ and $\kappa\leq c_1$ for some constant $c_1>0$. Then, for the estimator $\wh{\theta}$ defined by (\ref{eq:pure-MLE}), we have
\begin{eqnarray}
\label{eq:main-l2} \sum_{i=1}^n(\wh{\theta}_i-\theta_{r_i^*}^*)^2 &\leq& C\frac{1}{pL}, \\
\label{eq:main-linf} \max_{i\in[n]}|\wh{\theta}_i-\theta_{r_i^*}^*|^2 &\leq& C\frac{\log n}{npL},
\end{eqnarray}
for some constant $C>0$ only depending on $c_1$ with probability at least $1-O(n^{-7})$ uniformly over all $r^*\in\S_n$ and all $\theta^*\in\Theta(k,0,\kappa)$ such that $\mathds{1}_{n}^T\theta^*=0$.
\end{thm}

Let us give some comments on the assumptions and conclusions of Theorem \ref{thm:MLE-estimation}. We have established that the MLE achieves the error rates $O\left(\frac{1}{pL}\right)$ and $O\left(\frac{\log n}{npL}\right)$ for the squared $\ell_2$ loss and the squared $\ell_{\infty}$ loss, respectively. Both error rates are known to be optimal in the literature \citep{negahban2017rank,chen2019spectral}. Since the BTL model (\ref{eq:BTL-theta}) is defined through pairwise differences of $\theta_i^*$'s, the model parameter is only identifiable up to a constant shift. We therefore require both $\mathds{1}_{n}^T\wh{\theta}=0$ and $\mathds{1}_{n}^T\theta^*=0$ so that the two vectors are properly aligned. Note that the results for parameter estimation do not need a positive $\Delta$, and we only assume $\theta^*\in \Theta(k,0,\kappa)$. The condition $p\geq c_0\frac{\log n}{n}$ is imposed for the random graph $A$ to be well behaved in terms of both its degrees and the eigenvalues of the graph Laplacian. In fact, $p\gtrsim \frac{\log n}{n}$ is necessary to ensure the random graph is connected. Otherwise, ranking and parameter estimation would be impossible due to the identifiability issue caused by the lack of comparison between disconnected graph components. In the rest of the paper, some of the results will require a slightly stronger condition $\frac{np}{\log n}\rightarrow\infty$, but we will give very detailed remarks on when and why it will be needed. Last but not least, we require that the dynamic range $\kappa$ to be bounded by a constant. One can certainly allow $\kappa$ to tend to infinity, but the rates (\ref{eq:main-l2}) and (\ref{eq:main-linf}) would depend on $\kappa$ exponentially \citep{negahban2017rank,chen2019spectral}. This is because the eigenvalues of the Hessian of the objective function of (\ref{eq:pure-MLE}) will be exponentially small when $\kappa$ diverges. In fact, when $\kappa\rightarrow\infty$, it is not clear whether MLE still leads to optimal error rates for parameter estimation. In this paper, we will focus on the case $\kappa=O(1)$. We will see in later theorems that even with $\kappa=O(1)$, the exact value of $\kappa$ still plays a fundamental role in top-$k$ ranking.

To the best of our knowledge, Theorem \ref{thm:MLE-estimation} is the first result in the literature that gives optimal rates for parameter estimation by \textit{vanilla} MLE under the BTL model. Previous results in the literature including \cite{chen2015spectral,negahban2017rank,chen2019spectral} all work with regularized MLE
\begin{equation}
\wh{\theta}_{\lambda} = \argmin_{\theta:\mathds{1}_{n}^T\theta=0}\left[\ell_n(\theta) + \frac{\lambda}{2}\|\theta\|^2\right]. \label{eq:pen-MLE}
\end{equation}
In particular, the recent paper \cite{chen2019spectral} shows that $\wh{\theta}_{\lambda}$ also achieves the optimal rates (\ref{eq:main-l2}) and (\ref{eq:main-linf}) for a $\lambda$ that is chosen appropriately, though in practice it is known that the vanilla MLE performs very well. Theorem \ref{thm:MLE-estimation} shows that penalty is not needed for the MLE to be optimal, thus closing a gap between theory and practice.

The proof of Theorem \ref{thm:MLE-estimation} is built upon the elegant leave-one-out technique in \cite{chen2019spectral}. We first show that with a sufficiently small $\lambda$, a (sub-optimal) $\ell_{\infty}$ bound for $\wh{\theta}_{\lambda}$ can be transferred to $\wh{\theta}$. Then, we apply a leave-one-out argument to derive the optimal rates (\ref{eq:main-l2}) and (\ref{eq:main-linf}). We also note that our leave-one-out argument is actually different from the form used in \cite{chen2019spectral}. While the leave-one-out argument in \cite{chen2019spectral} is applied together with a gradient descent analysis, we do not need to follow this gradient descent analysis because of the $\ell_{\infty}$ bound that has already been obtained. As a result, we are able to remove the additional technical assumption $\log L=O(\log n)$ that is imposed in \cite{chen2019spectral}. A detailed analysis of the MLE will be given in Section \ref{sec:proof-MLE}.

Next, we study the theoretical property of $\wh{r}$, the rank induced by the MLE $\wh{\theta}$. Without loss of generality, let us assume $k\leq \frac{n}{2}$ throughout the paper. The case $k>\frac{n}{2}$ can be dealt with by a symmetric bottom-$k$ ranking problem. Before presenting the error bound for the loss function $\h_k(\wh{r},r^*)$, we need to introduce a few notation. We first define the effective variance of the MLE by
\begin{equation}
V(\kappa) = \max_{\substack{\kappa_1+\kappa_2\leq \kappa\\ \kappa_1,\kappa_2\geq 0}}\frac{n}{k\psi'(\kappa_1)+(n-k)\psi'(\kappa_2)}. \label{eq:var-function-V}
\end{equation}
Recall that $\psi(t)=\frac{1}{1+e^{-t}}$ is the sigmoid function so that $\psi'(t)=\psi(t)\psi(-t)$. Since $\kappa=O(1)$, we have $V(\kappa) \asymp 1$. Then, the signal to noise ratio is defined by
$$\snr=\frac{npL\Delta^2}{V(\kappa)}.$$
Note that $\snr$ is a function of $n,k,p,L,\Delta$, but we suppress the dependence for simplicity of notation. The following theorem shows that $\h_k(\wh{r},r^*)$ has an exponential rate with $\snr$ appearing in the exponent.

\begin{thm}\label{thm:MLE-ranking}
Assume $\frac{np}{\log n}\rightarrow\infty$ and $\kappa\leq c_1$ for some constant $c_1>0$. Then, for the rank vector $\wh{r}$ that is induced by the MLE (\ref{eq:pure-MLE}), there exists some $\delta=o(1)$, such that
\begin{equation}
\h_k(\wh{r},r^*)\leq C\exp\left(-\frac{1}{2}\left(\frac{\sqrt{(1-\delta)\snr}}{2}-\frac{1}{\sqrt{(1-\delta)\snr}}\log\frac{n-k}{k}\right)_+^2\right), \label{eq:upper-bound}
\end{equation}
for some constant $C>0$ only depending on $c_1$ with probability $1-o(1)$ uniformly over all $r^*\in\S_n$ and all $\theta^*\in\Theta(k,\Delta,\kappa)$.
\end{thm}

The error exponent of (\ref{eq:upper-bound}) is complicated. We present a special case of the bound when $k\asymp n$ to help understand the result.

\begin{corollary}\label{cor:MLE-ranking}
Assume $\frac{np}{\log n}\rightarrow\infty$, $\kappa=O(1)$ and $k\asymp n$. Then, as long as $\snr\rightarrow\infty$, the rank vector $\wh{r}$ induced by the MLE (\ref{eq:pure-MLE}) satisfies
\begin{equation}
\h_k(\wh{r},r^*)\leq \exp\left(-(1-o(1))\frac{\snr}{8}\right), \label{eq:upper-bound-special}
\end{equation}
with probability $1-o(1)$ uniformly over all $r^*\in\S_n$ and all $\theta^*\in\Theta(k,\Delta,\kappa)$.
\end{corollary}

Under the additional assumption $k\asymp n$, the top-$k$ ranking problem can be viewed as a clustering or community detection problem, with the goal to divide the $n$ players into two groups of sizes $k$ and $n-k$, respectively. The exponential convergence rate (\ref{eq:upper-bound-special}) is in a typical form of optimal clustering error \citep{zhang2016minimax,lu2016statistical}. It is intuitively clear that a larger $\snr$ leads to a faster convergence rate. When the sizes of the two clusters are of different orders, one can obtain a more general convergence rate in the form of (\ref{eq:upper-bound}). The extra term $\log\frac{n-k}{k}$ characterizes the unbalancedness of the two clusters. We note that for variable selection under Hamming loss \citep{butucea2018variable,ndaoud2020optimal}, the optimal rate is very similar to the form of (\ref{eq:upper-bound}). This is because variable selection can also be thought of as clustering with two clusters of sizes $s$ and $p-s$, whose orders can potentially be different.

Theorem \ref{thm:MLE-ranking} and Corollary \ref{cor:MLE-ranking} together reveal an interesting phenomenon for top-$k$ ranking. The result shows that the top-$k$ ranking problem can be very different for different orders of $k$. We note that in order to successfully identify the majority of the set $\{i\in[n]:r_i^*\leq k\}$, we need to have $\h_k(\wh{r},r^*)\rightarrow 0$. When $k=n/4$, Corollary \ref{cor:MLE-ranking} shows that $\h_k(\wh{r},r^*)\rightarrow 0$ is achieved when $\snr\rightarrow \infty$. In comparison, when $k=5$, Theorem \ref{thm:MLE-ranking} shows that $\h_k(\wh{r},r^*)\rightarrow 0$ when $\snr > (1+\epsilon)2\log n$ for some arbitrarily small constant $\epsilon>0$. In other words, in terms of partial recovery consistency, top-quarter ranking is an easier problem than top-$5$ ranking. In general, a larger $\snr$ is required for a smaller $k$ according to the formula (\ref{eq:upper-bound}).

Compared with Theorem \ref{thm:MLE-estimation}, we need a slightly stronger condition $\frac{np}{\log n}\rightarrow\infty$ for Theorem \ref{thm:MLE-ranking} and Corollary \ref{cor:MLE-ranking}. If we only assume $p\geq c_0\frac{\log n}{n}$, the $1-\delta$ factor in the exponent of (\ref{eq:upper-bound}) can be replaced by $1-\epsilon$ with some $\epsilon$ of constant order. The constant $\epsilon$ can be made arbitrarily small as long as $c_0$ is sufficiently large.

The proof of Theorem \ref{thm:MLE-ranking} relies on a very interesting lemma that is stated below.
\begin{lemma}\label{lem:anderson}
Suppose $\wh{r}$ is a rank vector induced by $\wh{\theta}$, we then have
$$\h_k(\wh{r},r^*) \leq \frac{1}{k}\min_{t\in\mathbb{R}}\left[\sum_{i:r^*_i\leq k}\indc{\wh{\theta}_i\leq t} + \sum_{i:r^*_i>k}\indc{\wh{\theta}_i\geq t}\right].$$
The inequality holds for any $r^*\in\S_n$.
\end{lemma}
We will prove Lemma \ref{lem:anderson} in Section \ref{sec:pf-tech}. This inequality shows that the error of ranking $\wh{\theta}$ is bounded by the error of any thresholding rule. Using this result, we immediately obtain that
$$\mathbb{E} \h_k(\wh{r},r^*) \leq \frac{1}{k}\min_{t\in\mathbb{R}}\left[\sum_{i:r^*_i\leq k}\mathbb{P}(\wh{\theta}_i\leq t) + \sum_{i:r^*_i>k}\mathbb{P}(\wh{\theta}_i\geq t)\right].$$
We then obtain the exponential error bound (\ref{eq:upper-bound}) by carefully analyzing the probability $\mathbb{P}(\wh{\theta}_i\leq t)$ (or $\mathbb{P}(\wh{\theta}_i\geq t)$) for each $i\in[n]$. The analysis of $\mathbb{P}(\wh{\theta}_i\leq t)$ is quite involved. We need to first obtain a local linear expansion of the MLE at each coordinate, and then apply the leave-one-out technique introduced by \cite{chen2019spectral} to decouple the dependence between the data and the coefficients of the local linear expansion. The details will be given in Section \ref{sec:proof-MLE}.

The result of Theorem \ref{thm:MLE-ranking} immediately implies a condition for exact recovery of the top-$k$ set. By the definition of $\h_k(\wh{r},r^*)$, it is easy to see that
\begin{equation}
\h_k(\wh{r},r^*)\in\{0,(2k)^{-1},2(2k)^{-1},3(2k)^{-1},\cdots, 1\}.\label{eq:grid}
\end{equation}
Then as long as $\h_k(\wh{r},r^*)<(2k)^{-1}$, we must have $\h_k(\wh{r},r^*)=0$. Under the condition that the right hand side of (\ref{eq:upper-bound}) is smaller than $(2k)^{-1}$, we obtain exact recovery of the top-$k$ set.
This result is stated as follows.
\begin{thm}\label{thm:MLE-exact}
Assume $\frac{np}{\log n}\rightarrow\infty$, $\kappa=O(1)$, and
\begin{equation}
\frac{npL\Delta^2}{V(\kappa)} > (1+\epsilon)2\left(\sqrt{\log k} + \sqrt{\log(n-k)}\right)^2, \label{eq:exact-threshold-upper}
\end{equation}
for some arbitrarily small constant $\epsilon>0$. Then, for the rank vector $\wh{r}$ that is induced by the MLE (\ref{eq:pure-MLE}), we have $\h_k(\wh{r},r^*)=0$ with probability $1-o(1)$ uniformly over all $r^*\in\S_n$ and all $\theta^*\in\Theta(k,\Delta,\kappa)$.
\end{thm}

We remark that the condition $\frac{np}{\log n}\rightarrow\infty$ can be relaxed to $p\geq c_0\frac{\log n}{n}$ for a sufficiently large constant $c_0$ without affecting the conclusion of Theorem \ref{thm:MLE-exact}. The result of Theorem \ref{thm:MLE-exact} improves the exact recovery threshold obtained in the literature. The paper \cite{chen2019spectral} proves that the MLE exactly recovers the top-$k$ set when $npL\Delta^2 > C\left(\sqrt{\log k}+\sqrt{\log(n-k)}\right)^2$ for some sufficiently large constant $C>0$. We complement the result of \cite{chen2019spectral} by showing that the leading constant should be $2V(\kappa)$, an increasing function of the dynamic range $\kappa$. Moreover, the symmetry of $k$ and $n-k$ in (\ref{eq:exact-threshold-upper}) agrees with the understanding that top-$k$ ranking and bottom-$k$ ranking are mathematically equivalent.

The next theorem shows that the exact recovery threshold (\ref{eq:exact-threshold-upper}) is optimal, and cannot be further improved.
\begin{thm}\label{thm:exact-lower}
Assume $\frac{np}{\log n}\rightarrow\infty$, $\kappa=O(1)$, $(\log n)^8=O(L)$, and
\begin{equation}
\frac{npL\Delta^2}{V(\kappa)} < (1-\epsilon)2\left(\sqrt{\log k} + \sqrt{\log(n-k)}\right)^2, \label{eq:exact-threshold-lower}
\end{equation}
for some arbitrarily small constant $\epsilon>0$. Then, we have
$$\liminf_{n\rightarrow\infty}\inf_{\wh{r}}\sup_{\substack{r^*\in\S_n\\\theta^*\in\Theta(k,\Delta,\kappa)}}\mathbb{P}_{(\theta^*,r^*)}\left(\h_k(\wh{r},r^*)>0\right)\geq 0.95,$$
where we use the notation $\mathbb{P}_{(\theta^*,r^*)}$ for the data generating process (\ref{eq:BTL-theta}).
\end{thm}

The proof of Theorem \ref{thm:exact-lower} relies on a precise lower bound characterization of the maximum of dependent binomial random variables. The extra assumption $L\gtrsim (\log n)^8$ allows us to apply a high-dimensional central limit theorem \citep{chernozhukov2013gaussian} for this purpose. Without this additional technical condition, we are not aware of any probabilistic tool to deal with maximum of dependent binomial random variables.

Theorem \ref{thm:MLE-exact} and Theorem \ref{thm:exact-lower} together nail down the phase transition boundary of exact recovery, which is $\frac{npL\Delta^2}{V(\kappa)} = 2\left(\sqrt{\log k} + \sqrt{\log(n-k)}\right)^2$. Thus, the MLE is an optimal procedure that achieves this boundary. The lower bound result of Theorem \ref{thm:exact-lower} also suggests that the partial recovery error rate obtained in Theorem \ref{thm:MLE-ranking} cannot be improved, since otherwise one would obtain a better $\snr$ condition for exact recovery in Theorem \ref{thm:MLE-exact}. A rigorous minimax lower bound for partial recovery will be given in Section \ref{sec:lower-partial}.

\section{Results for the Spectral Method}\label{sec:spec-result}

In this section, we study the theoretical property of the spectral method, also known as rank centrality \citep{negahban2017rank}. Let $\wh{\pi}$ be the stationary distribution of the Markov chain with transition probability (\ref{eq:spec-P}). The estimation error of $\wh{\pi}$ has already been investigated by \cite{negahban2017rank,chen2019spectral}. For both $\ell_2$ and $\ell_{\infty}$ loss functions, it has been shown by \cite{chen2019spectral} that $\wh{\pi}$ achieves the optimal rates (\ref{eq:main-l2}) and (\ref{eq:main-linf}) after an appropriate scaling. We therefore directly study the accuracy of the rank vector $\wh{r}$ induced by $\wh{\pi}$. This is where we can see the difference between the MLE and the spectral method.

We first define the effective variance of the spectral method,
\begin{equation}
\overline{V}(\kappa) = \max_{\substack{\kappa_1+\kappa_2\leq \kappa\\ \kappa_1,\kappa_2\geq 0}}\frac{k\psi'(\kappa_1)(1+e^{\kappa_1})^2+(n-k)\psi'(\kappa_2)(1+e^{-\kappa_2})^2}{(k\psi(\kappa_1)+(n-k)\psi(-\kappa_2))^2/n}. \label{eq:var-function-V-bar}
\end{equation}
Note that $\overline{V}(\kappa)\asymp 1$ when $\kappa=O(1)$. The signal to noise ratio is defined by
$$\overline{\snr}=\frac{npL\Delta^2}{\overline{V}(\kappa)}.$$
The error rate of the spectral method with respect to $\h_k(\wh{r},r^*)$ is stated as follows.

\begin{thm}\label{thm:spectral-ranking}
Assume $\frac{np}{\log n}\rightarrow\infty$ and $\kappa\leq c_1$ for some constant $c_1>0$. Then, for the rank vector $\wh{r}$ that is induced by the stationary distribution of the Markov chain (\ref{eq:spec-P}), there exists some $\delta=o(1)$, such that
\begin{equation}
\h_k(\wh{r},r^*)\leq C\exp\left(-\frac{1}{2}\left(\frac{\sqrt{(1-\delta)\overline{\snr}}}{2}-\frac{1}{\sqrt{(1-\delta)\overline{\snr}}}\log\frac{n-k}{k}\right)_+^2\right), \label{eq:upper-bound-spec}
\end{equation}
for some constant $C>0$ only depending on $c_1$ with probability $1-o(1)$ uniformly over all $r^*\in\S_n$ and all $\theta^*\in\Theta(k,\Delta,\kappa)$.
\end{thm}

The formula (\ref{eq:upper-bound-spec}) characterizes the convergence rate of partial recovery of the top-$k$ set by the spectral method. It can be compared with the MLE error bound (\ref{eq:upper-bound}). The only difference lies in the effective variance of the two methods. We will show in Lemma \ref{lem:2V} that $\overline{V}(\kappa)\geq V(\kappa)$ and the equality only holds when $\kappa=0$. Therefore, the spectral method is not optimal in general. Detailed comparisons of the two algorithms will be given in Section \ref{sec:compare}.

By the property (\ref{eq:grid}), we immediately obtain an exact recovery result from Theorem \ref{thm:spectral-ranking}.

\begin{thm}\label{thm:spectral-exact}
Assume $\frac{np}{\log n}\rightarrow\infty$, $\kappa=O(1)$, and
\begin{equation}
\frac{npL\Delta^2}{\overline{V}(\kappa)} > (1+\epsilon)2\left(\sqrt{\log k} + \sqrt{\log(n-k)}\right)^2, \label{eq:exact-threshold-upper-spec}
\end{equation}
for some arbitrarily small constant $\epsilon>0$. Then, for the rank vector $\wh{r}$ that is induced by the stationary distribution of the Markov chain (\ref{eq:spec-P}), we have $\h_k(\wh{r},r^*)=0$ with probability $1-o(1)$ uniformly over all $r^*\in\S_n$ and all $\theta^*\in\Theta(k,\Delta,\kappa)$.
\end{thm}

It has been shown in \cite{chen2019spectral} that the spectral method exactly recovers the top-$k$ set when $npL\Delta^2 > C\left(\sqrt{\log k}+\sqrt{\log(n-k)}\right)^2$ for some sufficiently large constant $C>0$. Without specifying the constant $C$, one cannot tell the difference between the MLE and the spectral method. In view of the lower bound result given by Theorem \ref{thm:exact-lower}, the exact recovery threshold (\ref{eq:exact-threshold-upper-spec}) of the spectral method does not achieve the phase transition boundary for a general $\kappa$. A careful reader may wonder whether this is resulted from a loose analysis in the proof. Our next result shows that the sub-optimality of the spectral method is intrinsic.

\begin{thm}\label{thm:spectral_lower}
Assume $\frac{np}{\log n}\rightarrow\infty$, $\kappa\leq c_1$ for some constant $c_1>0$, $k\rightarrow\infty$ and
\begin{equation}
\frac{npL\Delta^2}{\overline{V}(\kappa)} < (1-\epsilon)2\left(\sqrt{\log k} + \sqrt{\log(n-k)}\right)^2, \label{eq:spec-lower-condition}
\end{equation}
for some arbitrarily small constant $\epsilon>0$. Then, for the rank vector $\wh{r}$ that is induced by the stationary distribution of the Markov chain (\ref{eq:spec-P}), we have
$$\liminf_{n\rightarrow\infty}\sup_{\substack{r^*\in\S_n\\\theta^*\in\Theta(k,\Delta,\kappa)}}\mathbb{P}_{(\theta^*,r^*)}\left(\h_k(\wh{r},r^*)>0\right)\geq 0.95.$$
Moreover, there exists some $\delta=o(1)$, such that
\begin{equation}
\sup_{\substack{r^*\in\S_n\\\theta^*\in\Theta(k,\Delta,\kappa)}}\mathbb{E}_{(\theta^*,r^*)}\h_k(\wh{r},r^*)\geq C\exp\left(-\frac{1}{2}\left(\frac{\sqrt{(1+\delta)\overline{\snr}}}{2}-\frac{1}{\sqrt{(1+\delta)\overline{\snr}}}\log\frac{n-k}{k}\right)_+^2\right), \label{eq:lower-bound-spec}
\end{equation}
for some constant $C>0$ only depending on $c_1$ and $\epsilon$.
\end{thm}

Theorem \ref{thm:spectral_lower} shows that the results of Theorem \ref{thm:spectral-ranking} and Theorem \ref{thm:spectral-exact} on the performance of spectral method are sharp, under the additional condition that $k\rightarrow\infty$. The conclusion of Theorem \ref{thm:spectral_lower} can also be extended to the case of $k=O(1)$ via a similar argument that is used in the proof of Theorem \ref{thm:exact-lower}, as long as the technical condition $(\log n)^8=O(np)$ is further imposed.

To close this section, we remark that all the theorems we have obtained for the spectral method can be stated under the weaker assumption $p\geq c_0\frac{\log n}{n}$ for some sufficiently large constant $c_0>0$, as long as the $\delta$ in (\ref{eq:upper-bound-spec}) and (\ref{eq:lower-bound-spec}) are replaced by some sufficiently small constant.

\section{Comparison of the Two Methods}\label{sec:compare}

In this section, we compare the MLE and the spectral method based on the results obtained in Section \ref{sec:MLE-result} and Section \ref{sec:spec-result}. The statistical properties of the two methods in terms of partial and exact recovery are characterized by the two variance functions $V(\kappa)$ and $\overline{V}(\kappa)$, respectively. We first give a direct comparison of the two functions by plotting them together with different values of $k/n$. We observe in Figure \ref{fig:V} that $\overline{V}(\kappa)\geq V(\kappa)$ for all $\kappa\geq 0$.
\begin{figure}[h]
	\centering
	\includegraphics[width=0.9\textwidth]{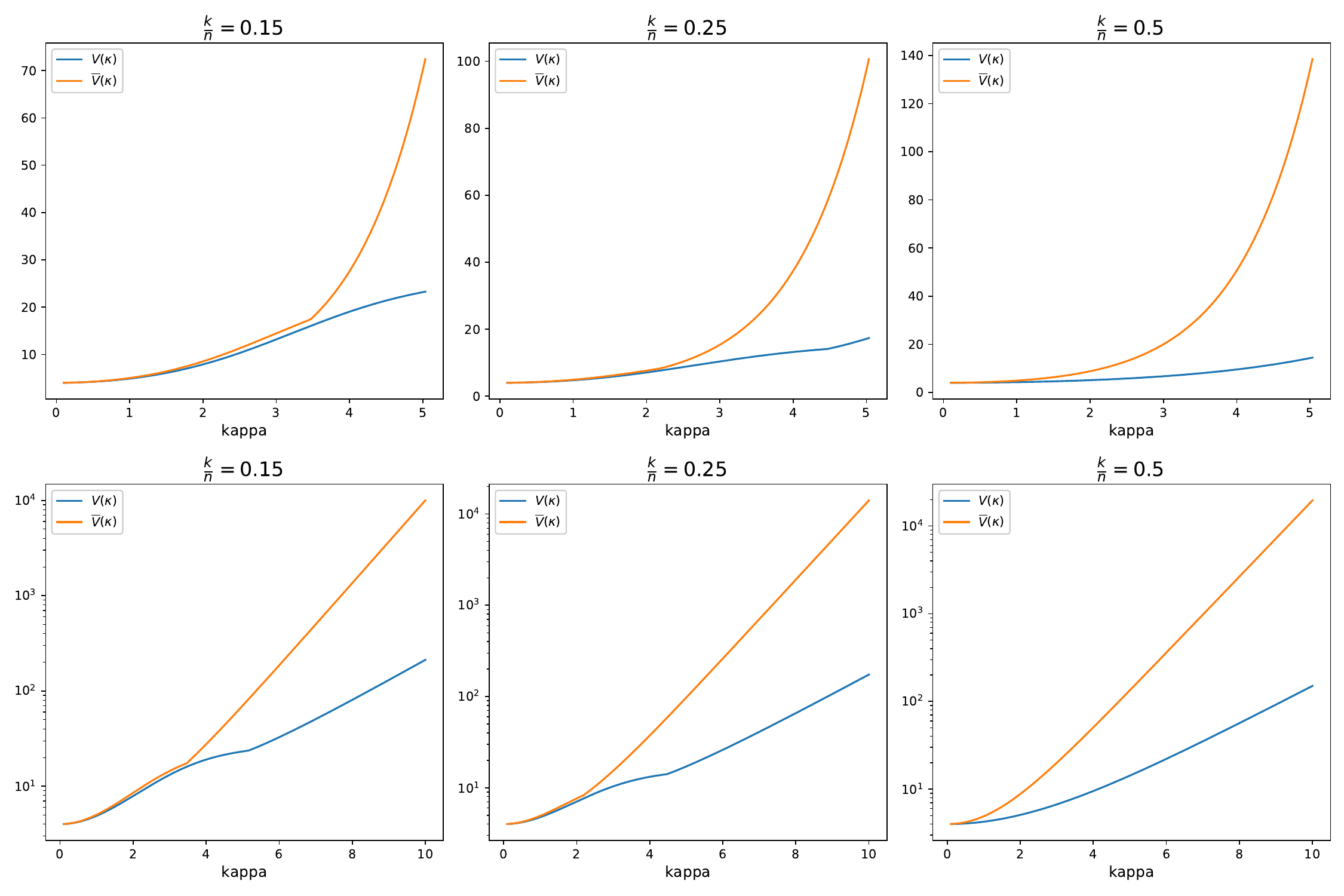}
	\caption{\textsl{The functions $V(\kappa)$ and $\overline{V}(\kappa)$ with $k/n\in\{0.15,0.25,0.5\}$. In the first row, we plot the functions for $\kappa\in[0,5]$. The second row plots the same functions for $\kappa\in[0,10]$ in a logarithmic scale to better illustrate the global structure. It is very interesting that both $V(\kappa)$ and $\overline{V}(\kappa)$ have a point at which the derivative is not continuous. Before this critical point, the optimization of $V(\kappa)$ is achieved by $(\kappa_1^*,\kappa_2^*)=(0,\kappa)$. Right after the critical point, $\kappa_1^*$ is immediately bounded away from $0$ and $\kappa_2^*$ is immediately bounded away from $\kappa$. The same property also holds for $\overline{V}(\kappa)$. Moreover, the critical point occurs earlier as $k/n$ becomes larger (when $k/n\leq 1/2$).}}
	\label{fig:V}
\end{figure}
This inequality is rigorously established by the following lemma.
\begin{lemma}\label{lem:2V}
For $V(\kappa)$ and $\overline{V}(\kappa)$ defined in (\ref{eq:var-function-V}) and (\ref{eq:var-function-V-bar}), respectively, we have
$$\overline{V}(\kappa)\geq V(\kappa),$$
for all $\kappa\geq 0$. Moreover, the equality holds if and only if $\kappa=0$.
\end{lemma}
\begin{proof}
By Jensen's inequality, we have
\begin{equation}
\frac{k\frac{e^{\kappa_1}}{(1+e^{\kappa_1})^2}+(n-k)\frac{e^{-\kappa_2}}{(1+e^{-\kappa_2})^2}}{ke^{\kappa_1}+(n-k)e^{-\kappa_2}} \geq \left(\frac{k\frac{e^{\kappa_1}}{1+e^{\kappa_1}}+(n-k)\frac{e^{-\kappa_2}}{1+e^{-\kappa_2}}}{ke^{\kappa_1}+(n-k)e^{-\kappa_2}}\right)^2. \label{eq:apply-Jensen}
\end{equation}
Another way to see the above inequality is to construct a random variable $X$ such that $\mathbb{P}\left(X=\frac{1}{1+e^{\kappa_1}}\right)=\frac{ke^{\kappa_1}}{ke^{\kappa_1}+(n-k)e^{-\kappa_2}}$ and $\mathbb{P}\left(X=\frac{1}{1+e^{-\kappa_2}}\right)=\frac{(n-k)e^{-\kappa_2}}{ke^{\kappa_1}+(n-k)e^{-\kappa_2}}$. Then, (\ref{eq:apply-Jensen}) is equivalent to $\mathbb{E}X^2\geq (\mathbb{E}X)^2$. The inequality (\ref{eq:apply-Jensen}) can be rearranged into
\begin{equation}
\frac{k\psi'(\kappa_1)(1+e^{\kappa_1})^2+(n-k)\psi'(\kappa_2)(1+e^{-\kappa_2})^2}{(k\psi(\kappa_1)+(n-k)\psi(-\kappa_2))^2/n} \geq \frac{n}{k\psi'(\kappa_1)+(n-k)\psi'(\kappa_2)}. \label{eq:Jensen-applied}
\end{equation}
Taking maximum over $\kappa_1$ and $\kappa_2$ on both sides, we obtain the inequality $\overline{V}(\kappa)\geq V(\kappa)$. When $\kappa=0$, we obviously have $V(\kappa)= \overline{V}(\kappa)$. When $\kappa>0$, we need to show $V(\kappa)\neq \overline{V}(\kappa)$. The optimization of $V(\kappa)$ must be achieved by some $(\kappa_1^*,\kappa_2^*)\neq (0,0)$. For such $(\kappa_1^*,\kappa_2^*)$, the constructed random variable $X$ has a positive variance, and thus both inequalities (\ref{eq:apply-Jensen}) and (\ref{eq:Jensen-applied}) are strict. We then have
\begin{eqnarray*}
\overline{V}(\kappa) &\geq& \frac{k\psi'(\kappa_1^*)(1+e^{\kappa_1^*})^2+(n-k)\psi'(\kappa_2^*)(1+e^{-\kappa_2^*})^2}{(k\psi(\kappa_1^*)+(n-k)\psi(-\kappa_2^*))^2/n} \\
&>& \frac{n}{k\psi'(\kappa_1^*)+(n-k)\psi'(\kappa_2^*)} \\
&=& V(\kappa).
\end{eqnarray*}
The proof is complete.
\end{proof}

The comparison between $V(\kappa)$ and $\overline{V}(\kappa)$ shows that the spectral method is not optimal in general. It has a worse error exponent for partial recovery and requires a larger signal to noise ratio threshold for exact recovery. In fact, the difference $\overline{V}(\kappa)-V(\kappa)$ eventually grows exponentially fast as a function of $\kappa$. See Figure \ref{fig:V}.
 
Note that both $V(\kappa)$ and $\overline{V}(\kappa)$ are the worst-case effective variances with respect to the parameter space $\Theta(k,\Delta,\kappa)$ for the two algorithms. In Section \ref{sec:local}, we will further show that the MLE outperforms the spectral method for each $\theta^*\in \Theta(k,\Delta,\kappa)$. This conclusion is supported by extensive numerical experiments. We set $n=200$, $p=0.25$, $L=20$ and $k=50$ throughout the experiments.

In our first experiment, we consider $\theta^*\in\mathbb{R}^n$ that has four pieces, with the three change-points located at $\{25,50,200\}$. The values of the four pieces are set as $10$, $10-\tau$, $10-\tau-\Delta$ and $0$, respectively, where $\tau=\theta_1^*-\theta_k^*\in\{1,4\}$ and $\Delta$ is varied from $0.01$ to $5$.
We apply both the MLE and the spectral method to the data.
\begin{figure}[h]
	\centering
	\includegraphics[width=0.8\textwidth]{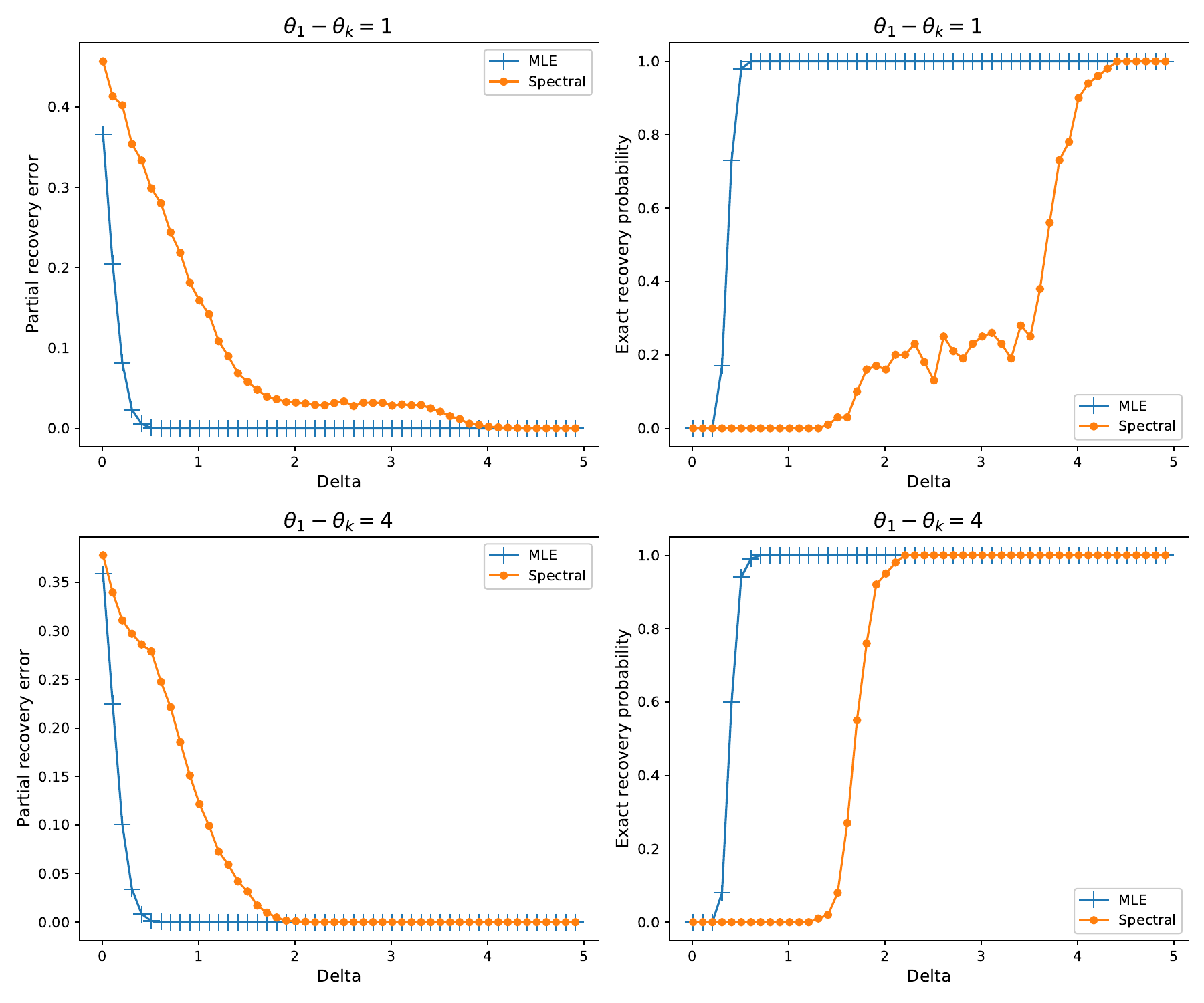}
	\caption{\textsl{The partial recovery error (left) and the exact recovery probability (right) for the MLE and the spectral method. The parameter $\theta^*$ is chosen to be a piecewise constant vector of four pieces of sizes $25,25,75,75$. The plots are obtained by averaging $100$ independent experiments.}} 
	\label{fig:rho-half}
\end{figure}
Figure \ref{fig:rho-half} shows the results for both partial and exact recovery. We observe that the MLE consistently outperforms the spectral method.

In the second experiment, we consider $\theta^*\in\mathbb{R}^n$ that has four pieces, with the three change-points located at $\{50(1-\rho),50,50+150\rho\}$. The values of the four pieces are set as $10,6,6-\Delta$ and $0$, respectively. The parameter $\rho$ is chosen in $\{0.1,0.5,0.9\}$ and $\Delta$ is varied from $0.01$ to $3$.
\begin{figure}[h]
	\centering
	\includegraphics[width=1.0\textwidth]{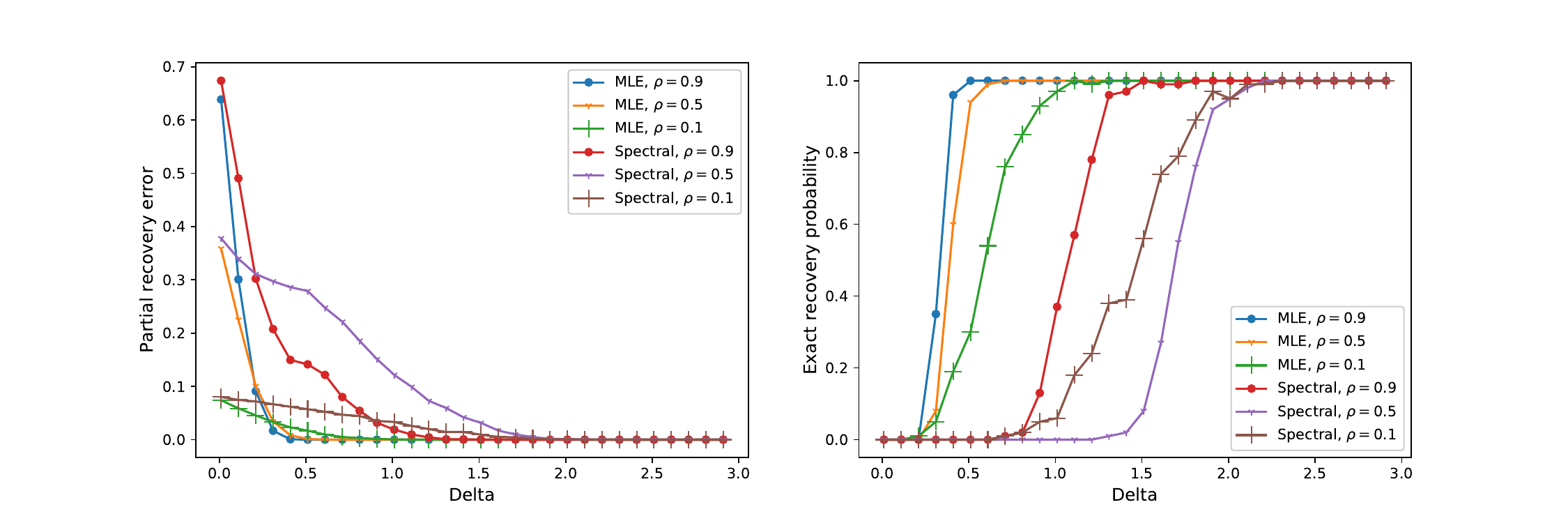}
	\caption{\textsl{The partial recovery error (left) and the exact recovery probability (right) for the MLE and the spectral method. The parameter $\theta^*$ is chosen to be a piecewise constant vector of four pieces of sizes $50(1-\rho), 50\rho, 150(1-\rho), 150\rho$.  The plots are obtained by averaging $100$ independent experiments.}} 
	\label{fig:varyingrho}
\end{figure}
The performance of the two methods for partial and exact recovery are plotted in Figure \ref{fig:varyingrho}. Again, the MLE always outperforms the spectral method.

Next, we consider a $\theta^*\in\mathbb{R}^n$ that has a more complicated structure. We fix $\theta_1^*=10, \theta_{200}^*=0$, generate $\theta_2^*,\cdots,\theta_{50}^*$ from $\text{Uniform}[6,10]$ and generate $\theta_{51}^*,\cdots,\theta_{199}^*$ from $\text{Uniform}[0,6-\Delta]$, and we vary $\Delta$ from $0.01$ to $2$.
\begin{figure}[h]
	\centering
	\includegraphics[width=1.0\textwidth]{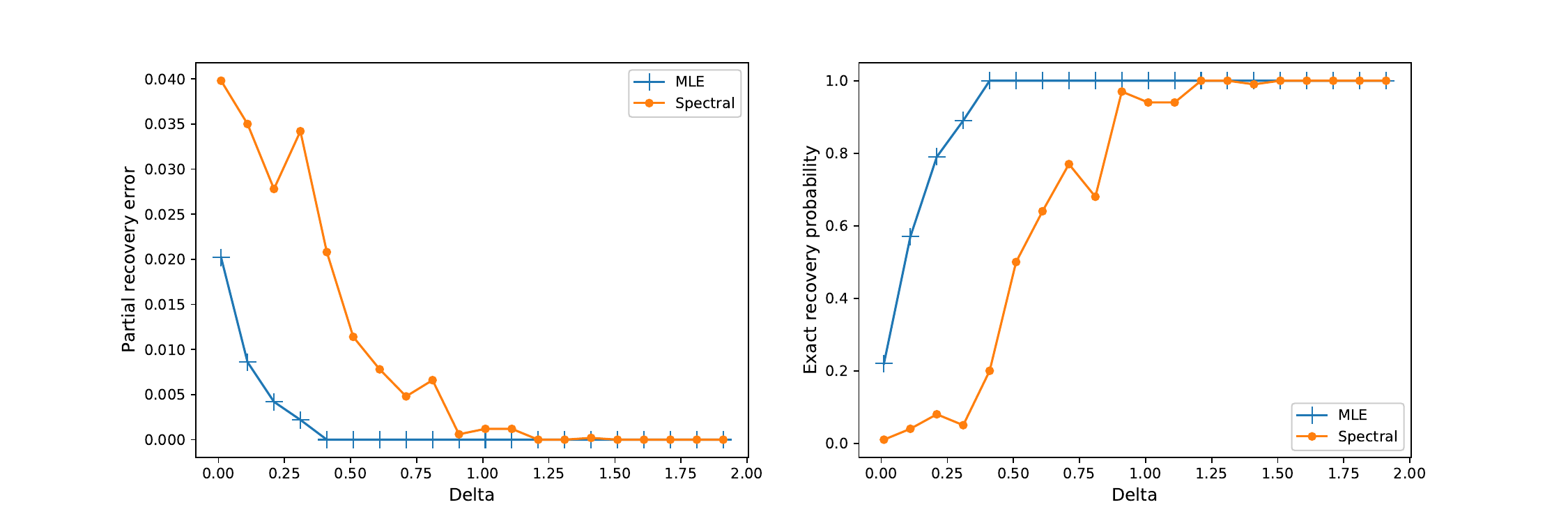}
	\caption{\textsl{The partial recovery error (left) and the exact recovery probability (right) for the MLE and the spectral method. The parameter $\theta^*$ is randomly generated from some distribution. The plots are obtained by averaging $100$ independent experiments.}}
	\label{fig:random-design} 
\end{figure}
We find even for such randomly generated $\theta^*$'s, the MLE always outperforms the spectral method. The results are summarized in Figure \ref{fig:random-design} for both partial and exact recovery.

In summary, we are able to confirm that the MLE is a much better algorithm than the spectral method under various scenarios. Our results complement the analysis in \cite{chen2019spectral}. It is claimed in \cite{chen2019spectral} that both the MLE and the spectral method are optimal in terms of the order of the exact recovery threshold. In addition, the paper conducts a very curious numerical experiment that shows the performances of the MLE and the spectral method are nearly identical. We note that the $\theta^*$ chosen in the numerical experiment of \cite{chen2019spectral} is a piecewise constant vector with only two pieces. We will explain why this choice leads to nearly identical performances of the two algorithms. Let us first conduct a similar experiment to replicate this conclusion. We continue to use the setting $n=200$, $p=0.25$, $L=20$ and $k=50$. Then, choose $\theta^*$ such that $\theta_1^*=\cdots=\theta_{50}^*=\Delta$ and $\theta_{51}^*=\cdots=\theta_{200}^*=0$. Figure \ref{fig:suboptimal} plots the results of partial and exact recovery with $\Delta$ varied from $0.01$ to $0.55$.
\begin{figure}[h]
	\centering
	\includegraphics[width=1.0\textwidth]{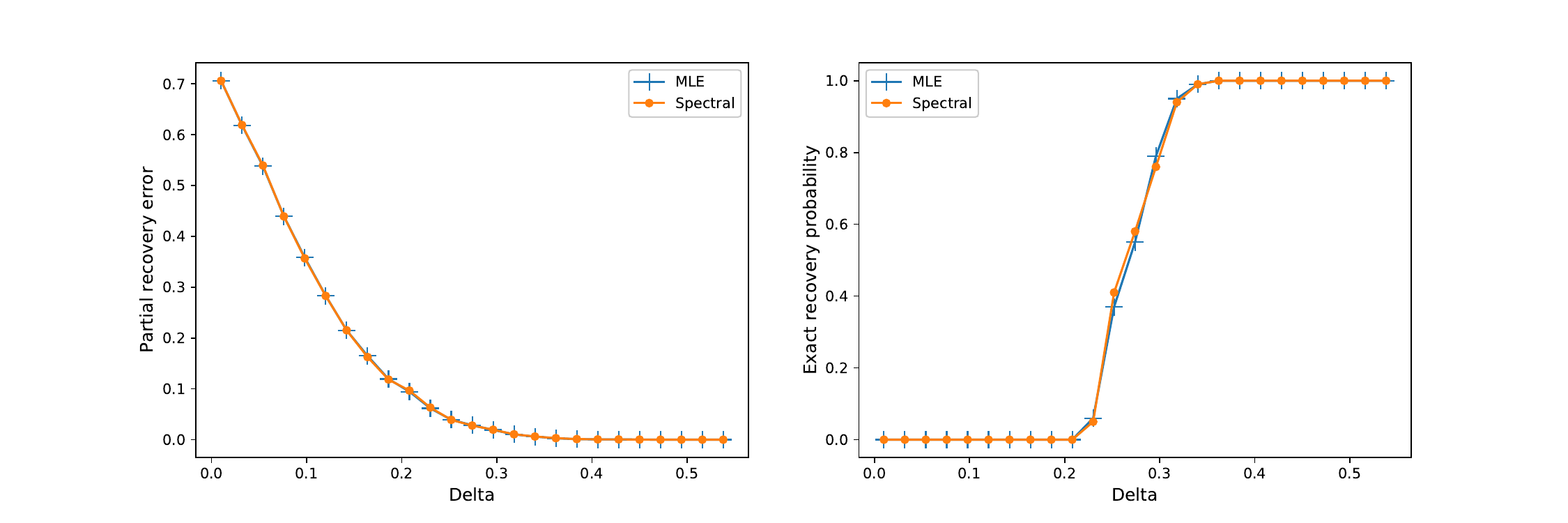}
	\caption{The partial recovery error (left) and the exact recovery probability (right) for the MLE and the spectral method. The parameter $\theta^*$ is chosen to be a piecewise constant vector of two pieces of sizes $50$ and $150$.  The plots are obtained by averaging $100$ independent experiments.} 
	\label{fig:suboptimal}
\end{figure}
For both partial recovery and exact recovery, the results are indeed nearly identical for the two algorithms. This phenomenon can be easily explained by our theory. For $\theta^*\in\Theta(k,\Delta,\kappa)$ with only two pieces, we must have $\kappa=\Delta$. When $\Delta=o(1)$, we have $V(\kappa)=(1+o(1))V(0)$ and $\overline{V}(\kappa)=(1+o(1))\overline{V}(0)$. This leads to the relation $\overline{V}(\kappa)=(1+o(1))V(\kappa)$, and thus the spectral method has the same asymptotic error exponent for partial recovery and achieves the optimal phase transition boundary for exact recovery. When $\Delta$ does not tend to zero but of a constant order, we have $\overline{\snr}\gtrsim npL\gg \log n$, and the error bound (\ref{eq:upper-bound-spec}) already leads to exact recovery because of the large value of $\overline{\snr}$. In either case, the spectral method is optimal. Let us summarize the optimality of the spectral method under this special situation by the following corollary.
\begin{corollary}\label{cor:opt-spec-ranking}
Assume $\frac{np}{\log n}\rightarrow\infty$ and $\kappa=\Delta\leq c_1$ for some constant $c_1>0$. Then, for the rank vector $\wh{r}$ that is induced by the stationary distribution of the Markov chain (\ref{eq:spec-P}), there exists some $\delta=o(1)$, such that
$$
\h_k(\wh{r},r^*)\leq C\exp\left(-\frac{1}{2}\left(\frac{\sqrt{(1-\delta){\snr}}}{2}-\frac{1}{\sqrt{(1-\delta){\snr}}}\log\frac{n-k}{k}\right)_+^2\right),
$$
for some constant $C>0$ only depending on $c_1$ with probability $1-o(1)$ uniformly over all $r^*\in\S_n$ and all $\theta^*\in\Theta(k,\Delta,\Delta)$. Moreover, as long as
$$
\frac{npL\Delta^2}{{V}(\kappa)} > (1+\epsilon)2\left(\sqrt{\log k} + \sqrt{\log(n-k)}\right)^2,
$$
for some arbitrarily small constant $\epsilon>0$. Then, $\h_k(\wh{r},r^*)=0$ with probability $1-o(1)$ uniformly over all $r^*\in\S_n$ and all $\theta^*\in\Theta(k,\Delta,\Delta)$.
\end{corollary}

To close this section, we remark that according to the equality condition of Lemma \ref{lem:2V}, the two-piece $\theta^*$, or equivalently $\kappa=\Delta$, is essentially the only situation where the spectral method is optimal and performs as well as the MLE. Moreover, since both functions $V(\kappa)$ are $\overline{V}(\kappa)$ are increasing, the setting with $\kappa=\Delta$ leads to the smallest effective variance and thus provides the two algorithms with the most favorable scenario.

\section{Minimax Lower Bound of Partial Recovery}\label{sec:lower-partial}

The purpose of this section is to show that the partial recovery error rate (\ref{eq:upper-bound}) achieved by the MLE cannot be improved from a minimax perspective. We are able to establish a matching lower bound for Theorem \ref{thm:MLE-ranking} using a slightly more general parameter space. Define
\begin{equation}
\Theta'(k,\Delta,\kappa)=\left\{\theta\in\mathbb{R}^n: \theta_1\geq\cdots\geq\theta_n, \theta_k-\theta_{k+2}\geq\Delta, \theta_1-\theta_n\leq\kappa\right\}. \label{eq:enlarged-space}
\end{equation}
Compared with $\Theta(k,\Delta,\kappa)$, the new definition (\ref{eq:enlarged-space}) imposes a gap between $\theta_k$ and $\theta_{k+2}$. It is clear that $\Theta(k,\Delta,\kappa)\subset\Theta'(k,\Delta,\kappa)$, and the only difference of $\Theta'(k,\Delta,\kappa)$ is the ambiguity of $\theta_{k+1}$. The player ranked at the $(k+1)$th position does not necessarily has a gap from either the top group or the bottom group. Though this additional uncertainty clearly better models scenarios in many real applications of top-$k$ ranking, the main reason we adopt the slightly larger parameter space is to have a clean lower bound analysis. Directly establishing a lower bound for $\Theta(k,\Delta,\kappa)$ is still possible, but it requires some additional technical assumptions that make the problem unnecessarily involved.

Throughout this section, we assume that (\ref{eq:exact-threshold-lower}) holds. This is the regime of partial recovery, since exact recovery is impossible by Theorem \ref{thm:exact-lower}. We first remark that with a slight modification of the proof of Theorem \ref{thm:MLE-ranking}, the MLE can be shown to achieve the same error rate (\ref{eq:upper-bound}) over the parameter space $\Theta'(k,\Delta,\kappa)$ as well. Thus, the space $\Theta'(k,\Delta,\kappa)$ does not increase the statistical complexity of the problem.

Our lower bound analysis is based on the two least favorable vectors $\theta',\theta''\in\Theta'(k,\Delta,\kappa)$. They are constructed as follows. Let $\rho=o(1)$ be a vanishing sequence that tends to zero with a sufficiently slow rate. We define $\kappa_1^*$ and $\kappa_2^*$ such that the optimization (\ref{eq:var-function-V}) is achieved at $(\kappa_1,\kappa_2)=(\kappa_1^*,\kappa_2^*)$. Then, define $\theta_i'=\kappa_1^*$ for  $1\leq i\leq k-\rho k$, $\theta_i'=0$ for $k-\rho k<i\leq k$, $\theta_i'=-\Delta$ for $k<i\leq k+\rho(n-k)$ and $\theta_i'=-\kappa_2^*$ for $k+\rho(n-k)< i\leq n$. For $\theta''$, we let $\theta_i''=\theta_i'$ for all $i\in[n]\backslash\{k+1\}$ and set $\theta_{k+1}''=0$. We will show that there exist $r',r''\in\S_n$, so that the hardness of top-$k$ ranking is characterized by an optimal testing problem,
\begin{equation}
\inf_{0\leq \phi\leq 1}\left[\mathbb{E}_{(\theta'',r'')}\phi+\frac{n-k-1}{k}\mathbb{E}_{(\theta',r')}(1-\phi)\right]. \label{eq:hard-test-lower}
\end{equation}
Moreover, there exists some $i\in[n]$, such that the two rank vectors $r',r''$ satisfy $\theta_{r_j'}'=\theta_{r_j''}''$ for all $j\in[n]\backslash\{i\}$. For the $i$th entry, we have $\theta_{r_i'}'=-\Delta$ and $\theta_{r_i''}''=0$. The reduction of the top-$k$ ranking problem to the testing problem (\ref{eq:hard-test-lower}) is the most important step in our lower bound analysis. A rigorous argument will be given in Section \ref{sec:pf-minmax-l}.

The testing problem (\ref{eq:hard-test-lower}) can be roughly understood as to test whether the $i$th player belongs to the top-$k$ set or not. The two hypotheses receive different weights $1$ and $\frac{n-k-1}{k}$ because of the definition of the loss function $\h_k(\wh{r},r^*)$. The optimal procedure to (\ref{eq:hard-test-lower}) is given by the likelihood ratio test
$$\phi=\mathbb{I}\left\{\frac{d\mathbb{P}_{(\theta',r')}}{d\mathbb{P}_{(\theta'',r'')}} \geq \frac{k}{n-k-1}\right\},$$
according to Neyman-Pearson lemma. Since the vectors $\{\theta_{r_i'}'\}_{i\in[n]}$ and $\{\theta_{r_i''}''\}_{i\in[n]}$ only differ at the $i$th entry, the likelihood ratio statistic only depends on $\{\bar{y}_{ij}\}_{j\in[n]\backslash\{i\}}$ and $\{{A}_{ij}\}_{j\in[n]\backslash\{i\}}$. Therefore, the testing error (\ref{eq:hard-test-lower}) is relatively easy to quantify. A sharp lower bound can be obtained by a large deviation analysis.

\begin{thm}\label{thm:partial-lower}
Assume $\frac{np}{\log n}\rightarrow\infty$, $\kappa\leq c_1$ for some constant $c_1>0$, and (\ref{eq:exact-threshold-lower}) holds for some arbitrarily small constant $\epsilon>0$. Then, there exists some $\delta=o(1)$, such that
$$
\inf_{\wh{r}}\sup_{\substack{r^*\in\S_n\\\theta^*\in\Theta'(k,\Delta,\kappa)}}\mathbb{E}_{(\theta^*,r^*)}\h_k(\wh{r},r^*)\geq C\exp\left(-\frac{1}{2}\left(\frac{\sqrt{(1+\delta){\snr}}}{2}-\frac{1}{\sqrt{(1+\delta){\snr}}}\log\frac{n-k}{k}\right)_+^2\right),
$$
for some constant $C>0$ only depending on $c_1$ and $\epsilon$.
\end{thm}

\section{Local Error Rates}\label{sec:local}

So far, our study of the top-$k$ ranking problem has been conducted under the minimax decision-theoretic framework laid out in Section \ref{sec:mm}. The upper and lower bounds for the MLE and the spectral method are established uniformly over the parameter space $\Theta(k,\Delta,\kappa)$. To complement the minimax results, in this section, we present local error rates for the MLE and the spectral method, which leads to a refined comparison between the two popular methods.

\paragraph{Local Error Rate for the MLE.} 

To analyze the statistical property of the MLE for each individual $\theta$, we first need to generalize the effective variance (\ref{eq:var-function-V}). For any $\theta\in\mathbb{R}^n$ and any $i\in[n]$, define
\begin{equation}
V_i(\theta)=\frac{n}{\sum_{j=1}^n\psi'(\theta_i-\theta_j)}. \label{eq:var-function-V-theta}
\end{equation}
With the help of $V_i(\theta)$, for any subset $S\subset[n]$, we also define
\begin{eqnarray}
\label{eq:MLE-R1} R_1(S,\theta, t, \delta) &=& \sum_{i\in S}\exp\left(-\frac{(1-\delta)(\theta_i-t)_{+}^2npL}{2V_i(\theta)}\right), \\
\label{eq:MLE-R2} R_2(S,\theta, t, \delta) &=& \sum_{i\in S}\exp\left(-\frac{(1-\delta)(t-\theta_i)_+^2npL}{2V_i(\theta)}\right).
\end{eqnarray}

\begin{thm}\label{thm:MLE-ranking-theta}
Assume $\frac{np}{\log n}\rightarrow\infty$, $\kappa\leq c_1$ for some constant $c_1>0$. Then, for the rank vector $\wh{r}$ that is induced by the MLE (\ref{eq:pure-MLE}), any small constant $0<\delta<0.1$, any $r^*\in\S_n$ and any $\theta^*\in\Theta(k,0,\kappa)$, we have
\begin{align}
\E_{(\theta^*, r^*)}\h_k(\wh{r},r^*)\leq C_1\left(\inf_t\frac{R_1([k],\theta^*, t, \delta) + R_2([n]\backslash[k],\theta^*, t, \delta)}{k}+n^{-3}\right), \label{eq:MLE-ranking-theta-u}
\end{align}
where $C_1>0$ is a constant only depending on $c_1$ and $\delta$. Moreover, we also have
\begin{align}
\E_{(\theta^*, r^*)}\h_k(\wh{r},r^*)\geq C_2\left(\inf_t\frac{R_1([k],\theta^*, t, -\delta) + R_2([n]\backslash[k],\theta^*, t, -\delta)}{k}\right), \label{eq:MLE-ranking-theta-l}
\end{align}
for some constant $C_2>0$ only depending on $c_1$ and $\delta$, if we additionally assume that $\inf_t({R_1([k],\theta^*, t, -\delta) + R_2([n]\backslash[k],\theta^*, t, -\delta)})\rightarrow\infty$.
%the right hand side of ?? .
\end{thm}

Theorem \ref{thm:MLE-ranking-theta} gives matching upper and lower bounds for the error of the MLE for each individual $\theta^*\in\Theta(k,0,\kappa)$ and $r^*\in\S_n$, except for the additional $n^{-3}$ term and an arbitrarily small $\delta$. We remark that the $n^{-3}$ term in the upper bound can be replaced by $n^{-C}$ for an arbitrarily large constant $C$. The upper bound (\ref{eq:MLE-ranking-theta-u}) can be viewed as an extension of Theorem \ref{thm:MLE-ranking}, though the $\delta$ in Theorem \ref{thm:MLE-ranking} is allowed to vanish because of the less general setting. The lower bound (\ref{eq:MLE-ranking-theta-l}) requires an extra condition $\inf_t({R_1([k],\theta^*, t, -\delta) + R_2([n]\backslash[k],\theta^*, t, -\delta)})\rightarrow\infty$, which implies the error rate is of higher order than $O(k^{-1})$. It plays the same role as the condition (\ref{eq:spec-lower-condition}) in Theorem \ref{thm:spectral_lower}. This assumption covers most interesting partial recovery cases, since $O(k^{-1})$ is already the error rate of exact recovery.

Let $t^*$ be a minimizer of the right hand side of (\ref{eq:MLE-ranking-theta-u}) or (\ref{eq:MLE-ranking-theta-l}). Then, we can interpret $\sum_{i=1}^k\exp\left(-\frac{(\theta_i^*-t^*)_{+}^2npL}{2V_i(\theta^*)}\right)$ as the order of the number of top $k$ players that are ranked among the bottom group, and $\sum_{i=k+1}^n\exp\left(-\frac{(t^*-\theta_i^*)_+^2npL}{2V_i(\theta^*)}\right)$ as the order of the number of bottom $n-k$ players that are ranked in the top group.

A careful reader may notice that the error rate in Theorem \ref{thm:MLE-ranking-theta} does not have a clear dependence on the signal gap $\theta_k^*-\theta_{k+1}^*$. This is because the current error rate depends on $\theta^*$ more explicitly rather than just the difference between $\theta_k^*$ and $\theta_{k+1}^*$. Even when $\theta_k^*=\theta_{k+1}^*$, it is still possible that the right hand side of (\ref{eq:MLE-ranking-theta-u}) converges to zero as long as the majority of $\{\theta_i^*\}_{1\leq i\leq k}$ are separated from most of $\{\theta_{i}^*\}_{k+1\leq i\leq n}$.

\paragraph{Local Error Rate for the Spectral Method.}

To present a similar local error rate for the spectral method, we also need to generalize the effective variance (\ref{eq:var-function-V-bar}). For any $\theta\in\mathbb{R}^n$ and any $i\in[n]$, define
\begin{equation}
\overline{V}_i(\theta)=\frac{n\sum_{j=1}^n\psi^\prime(\theta_i-\theta_j)(1+e^{\theta_j-\theta_i})^2}{\left(\sum_{j=1}^n\psi(\theta_j-\theta_i)\right)^2}. \label{eq:var-function-V-bar-theta}
\end{equation}
We also introduce two quantities similar to (\ref{eq:MLE-R1}) and (\ref{eq:MLE-R2}),
\begin{eqnarray*}
\overline{R}_1(S,\theta, t, \delta) &=& \sum_{i\in S}\exp\left(-\frac{(1-\delta)(\theta_i-t)_{+}^2npL}{2\overline{V}_i(\theta)}\right), \\
\overline{R}_2(S,\theta, t, \delta) &=& \sum_{i\in S}\exp\left(-\frac{(1-\delta)(t-\theta_i)_+^2npL}{2\overline{V}_i(\theta)}\right).
\end{eqnarray*}

\begin{thm}\label{thm:spectral-ranking-theta}
Assume $\frac{np}{\log n}\rightarrow\infty$, $\kappa\leq c_1$ for some constant $c_1>0$. Then, for the rank vector $\wh{r}$ that is induced by the stationary distribution of the Markov chain (\ref{eq:spec-P}), any small constant $0<\delta<0.1$, any $r^*\in\S_n$ and any $\theta^*\in\Theta(k,0,\kappa)$, we have
\begin{align}\label{eq:spectral-ranking-theta-u}
\E_{(\theta^*, r^*)}\h_k(\wh{r},r^*)\leq C_1\left(\inf_t\frac{\overline{R}_1([k],\theta^*, t, \delta) + \overline{R}_2([n]\backslash[k],\theta^*, t, \delta)}{k}+n^{-3}\right),
\end{align}
where $C_1>0$ is a constant only depending on $c_1$ and $\delta$. Moreover, we also have
\begin{align}\label{eq:spectral-ranking-theta-l}
\E_{(\theta^*, r^*)}\h_k(\wh{r},r^*)\geq C_2\left(\inf_t\frac{\overline{R}_1([k],\theta^*, t, -\delta) + \overline{R}_2([n]\backslash[k],\theta^*, t, -\delta)}{k}\right),
\end{align}
for some constant $C_2>0$ only depending on $c_1$ and $\delta$, if we additionally assume that $\inf_t(\overline{R}_1([k],\theta^*, t, -\delta) + \overline{R}_2([n]\backslash[k],\theta^*, t, -\delta))\rightarrow\infty$.
%the right hand side of ?? .
\end{thm}

Similar to Theorem \ref{thm:MLE-ranking-theta}, Theorem \ref{thm:spectral-ranking-theta} also gives matching upper and lower bounds for the error of the spectral method for each individual $\theta^*\in\Theta(k,0,\kappa)$ and $r^*\in\S_n$.

Let us remark that the results of Theorem \ref{thm:MLE-ranking-theta} and Theorem \ref{thm:spectral-ranking-theta} can be further extended beyond the setting of Erd\H{o}s-R\'{e}nyi graph and exactly $L$ comparisons on each edge. To be specific, we can consider a random graph $A_{ij}\sim\text{Bernoulli}(p_{ij})$ independently for all $1\leq i<j\leq n$. For each edge, we observe $L_{ij}$ independent games. Then, as long as $\max_{ij}p_{ij}\leq C\min_{ij}p_{ij}$ and $\max_{ij}L_{ij}\leq C\min_{ij}L_{ij}$ hold for some constant $C>0$, the results of Theorem \ref{thm:MLE-ranking-theta} and Theorem \ref{thm:spectral-ranking-theta} continue to hold with $\frac{V_i(\theta)}{npL}$ and $\frac{\overline{V}_i(\theta)}{npL}$ replaced by
$
\frac{\sum_{j\in[n]\backslash\{i\}}\frac{p_{ij}}{L_{ij}}\psi(\theta_i-\theta_j)\psi(\theta_j-\theta_i)}{\left(\sum_{j\in[n]\backslash\{i\}}p_{ij}\psi(\theta_i-\theta_j)\psi(\theta_j-\theta_i)\right)^2}
$
and
$
\frac{\sum_{j\in[n]\backslash\{i\}}\frac{p_{ij}}{L_{ij}}\psi(\theta_i-\theta_j)\psi(\theta_j-\theta_i)\left(1+e^{\theta_j-\theta_i}\right)^2}{\left(\sum_{j\in[n]\backslash\{i\}}p_{ij}\psi(\theta_j-\theta_i)\right)^2},
$
respectively.

\paragraph{Comparison of the Two Methods for each $\theta^*$.}

Theorem \ref{thm:MLE-ranking-theta} and Theorem \ref{thm:spectral-ranking-theta} allow us to give a refined comparison between the MLE and the spectral method. By ignoring the $n^{-3}$ term in the upper bounds and the $\delta$ in each exponent, we can write the error rates of the MLE and the spectral method as
\begin{equation}
\inf_t\frac{1}{k}\left[\sum_{i=1}^k\exp\left(-\frac{(\theta_i^*-t)_{+}^2npL}{2V_i(\theta^*)}\right)+\sum_{i=k+1}^n\exp\left(-\frac{(t-\theta_i^*)_+^2npL}{2V_i(\theta^*)}\right)\right], \label{eq:MLE-case-simp}
\end{equation}
and
\begin{equation}
\inf_t\frac{1}{k}\left[\sum_{i=1}^k\exp\left(-\frac{(\theta_i^*-t)_{+}^2npL}{2\overline{V}_i(\theta^*)}\right)+\sum_{i=k+1}^n\exp\left(-\frac{(t-\theta_i^*)_+^2npL}{2\overline{V}_i(\theta^*)}\right)\right]. \label{eq:spectral-case-simp}
\end{equation}
It is clear that the only difference between (\ref{eq:MLE-case-simp}) and (\ref{eq:spectral-case-simp}) lies in the difference of the variance functions (\ref{eq:var-function-V-theta}) and (\ref{eq:var-function-V-bar-theta}), whose comparison is given by the following lemma.

\begin{lemma}\label{lem:Vi-smaller-than-Vibar}
For any $\theta^*\in\mathbb{R}$ and any $i\in[n]$, we have $V_i(\theta^*)\leq\overline{V}_i(\theta^*)$. The equality holds if and only if $\theta_1^*=...=\theta_n^*$.
\end{lemma}
\begin{proof}
Notice the following chain of equalities and inequality,
\begin{align}
\nonumber V_i(\theta^*)&=\frac{n}{\sum_{j\in[n]}\psi^\prime(\theta_j^*-\theta_i^*)}\\
\nonumber & =\frac{n\left(\sum_{j\in[n]}e^{\theta_j^*-\theta_i^*}\right)}{\left(\sum_{j\in[n]}\psi^\prime(\theta_j^*-\theta_i^*)\right)\left(\sum_{j\in[n]}e^{\theta_j^*-\theta_i^*}\right)}\\
&\leq\frac{n\left(\sum_{j\in[n]}\psi^\prime(\theta_j^*-\theta_i^*)(1+e^{\theta_j^*-\theta_i^*})^2\right)}{\left(\sum_{j\in[n]}\psi(\theta_j^*-\theta_i^*)\right)^2}\label{eq:cs}\\
\nonumber&=\overline{V}_i(\theta^*),
\end{align}
where (\ref{eq:cs}) is by Cauchy-Schwarz inequality on the denominator. According to the equality condition of the Cauchy-Schwarz inequality, we know that $V_i(\theta^*)=\overline{V}_i(\theta^*)$ only when $\theta_1^*=...=\theta_n^*$.
\end{proof}

To close this section, we discuss two special cases of $\theta^*$, under which the error rates recover the results of Theorem \ref{thm:MLE-ranking}, Theorem \ref{thm:spectral-ranking}, and Corollary \ref{cor:opt-spec-ranking}.

\begin{example}
According to the proof of Theorem \ref{thm:spectral_lower} and the construction discussed in Section \ref{sec:lower-partial}, the least favorable $\theta^*\in\Theta(k,\Delta,\kappa)$ takes the following form: $\theta^*_i=\kappa_1$ for all $1\leq i\leq k-\rho k$, $\theta^*_i=0$ for $k-\rho k<i\leq k$, $\theta^*_i=-\Delta$ for $k<i\leq k+\rho(n-k)$ and $\theta^*_i=-\kappa_2$ for $k+\rho(n-k)< i\leq n$. Here, $\kappa_1$ and $\kappa_2$ are maximizers of either (\ref{eq:var-function-V}) for the MLE or (\ref{eq:var-function-V-bar}) for the spectral method, and $\rho$ is a sufficiently small constant. For this $\theta^*$, the formulas (\ref{eq:MLE-case-simp}) and (\ref{eq:spectral-case-simp}) recover the minimax rates obtained in Theorem \ref{thm:MLE-ranking} and Theorem \ref{thm:spectral-ranking}.
\end{example}

\begin{example}
Another interesting $\theta^*$ is the two-piece model $\theta^*\in\Theta(k,\Delta,\Delta)$. By the translational invariance of the variance functions, we can consider $\theta_i^*=\Delta$ for all $1\leq i\leq k$ and $\theta_i^*=0$ for all $k<i\leq n$. We discuss the consequence of this choice of $\theta^*$ under two situations. First, consider $\Delta=o(1)$, and one can check that $V_i(\theta^*)=(1+o(1))4$ and $\overline{V}_i(\theta^*)=(1+o(1))4$ for all $i\in[n]$, which implies the equivalence of error rates of the MLE and the spectral method. Second, consider $\Delta$ lower bounded by some constant. In this case, both the formulas (\ref{eq:MLE-case-simp}) and (\ref{eq:spectral-case-simp}) are $o(k^{-1})$, which implies both the MLE and the spectral method achieve exact recovery with high probability. As shown in Corollary \ref{cor:opt-spec-ranking}, the spectral method is actually optimal for $\theta^*\in\Theta(k,\Delta,\Delta)$. We are therefore able to give a theoretical justification of the numerical experiment of \cite{chen2019spectral}.
\end{example}

\section{Analysis of the MLE}\label{sec:proof-MLE}

In this section, we analyze the MLE (\ref{eq:pure-MLE}), and prove Theorem \ref{thm:MLE-estimation}, Theorem \ref{thm:MLE-ranking}, and Theorem \ref{thm:MLE-exact}. Since the BTL model (\ref{eq:BTL-theta}) is invariant to a shift of the model parameter, we can assume $\mathds{1}_{n}^T\theta^*=0$ without loss of generality. For simplicity of notation, we also assume $r_i^*=i$ for each $i\in[n]$, and thus we have $\theta_{r_i^*}^*=\theta_i^*$. Recall the convention of notation that $A_{ij}=A_{ji}$ and $\bar{y}_{ij}=1-\bar{y}_{ji}$ for any $i<j$. We also set $A_{ii}=0$ for all $i\in[n]$. Throughout the analysis, we will repeatedly use the properties that both $\psi(t)$ and $\psi'(t)$ are bounded continuous functions with bounded Lipschitz constants.

The section is organized as follows. We will first give a brief overview of the techniques and the main steps of the analysis in Section \ref{sec:MLE-over}. We then present a few technical lemmas in Section \ref{sec:MLE-tech}. In Section \ref{sec:MLE-pf-prop}, we establish an important result on the $\ell_{\infty}$ bound of the MLE. Theorem \ref{thm:MLE-estimation} will be proved in Section \ref{sec:MLE-pf-thm-est}. Finally, we prove Theorem \ref{thm:MLE-ranking} and Theorem \ref{thm:MLE-exact} in Section \ref{sec:MLE-pf-thm-rank}.

\subsection{Overview of the Techniques}\label{sec:MLE-over}

A major difficulty of analyzing the MLE is to control the spectrum of the Hessian matrix of the negative log-likelihood function. Recall the definition of $\ell_n(\theta)$ in (\ref{eq:likelihood-BTL}). Its Hessian $\nabla^2\ell_n(\theta)=H(\theta)\in\mathbb{R}^{n\times n}$ is given by the formula
$$H_{ij}(\theta)=\begin{cases}
\sum_{l\in[n]\backslash\{i\}}A_{il}\psi'(\theta_i-\theta_l), & i=j, \\
-A_{ij}\psi'(\theta_i-\theta_j), & i\neq j.
\end{cases}$$
It can be viewed as the Laplacian of the weighted random graph $\{\psi'(\theta_i-\theta_j)A_{ij}\}$. For $\theta$ that satisfies $\max_{i<j}|\theta_i-\theta_j|=O(1)$, the spectrum of $H(\theta)$ can be well controlled via some standard random matrix tool \citep{tropp2015introduction}. The property $\max_{i<j}|\theta_i-\theta_j|=O(1)$ certainly holds for $\theta^*\in\Theta(k,\Delta,\kappa)$. However, when analyzing the Taylor expansion of $\ell_n(\theta)$, we actually need to understand $H(\theta)$ for $\theta$ that is a convex combination between $\wh{\theta}$ and $\theta^*$. Since the MLE is defined without any constraint or regularization, there is no such control for $\wh{\theta}$. Our first step is to establish the following proposition that shows $\|\wh{\theta}-\theta^*\|_{\infty}$ is bounded with high probability even though the MLE has no constraint or regularization.

\begin{proposition}\label{prop:entrywise-MLE}
Under the setting of Theorem \ref{thm:MLE-estimation}, we have
\begin{equation}
\|\wh{\theta}-\theta^*\|_{\infty}\leq 5, \label{eq:niubility}
\end{equation}
with probability at least $1-O(n^{-7})$.
\end{proposition}

The proof of Proposition \ref{prop:entrywise-MLE} borrows strength from the property of a regularized MLE. Recall the definition of $\wh{\theta}_{\lambda}$ in (\ref{eq:pen-MLE}). This is the version of MLE that has been analyzed by \cite{chen2019spectral}. We will choose $\lambda=n^{-1}$ in order that $\wh{\theta}_{\lambda}$ is close to $\wh{\theta}$. Following the techniques in \cite{chen2019spectral}, we can first show $\|\wh{\theta}_{\lambda}-\theta^*\|_{\infty}\leq 4$ with high probability. The presence of the penalty in (\ref{eq:pen-MLE}) is crucial for the result $\|\wh{\theta}_{\lambda}-\theta^*\|_{\infty}\leq 4$ to be established. Next, we have an argument to show that the two estimators $\wh{\theta}_{\lambda}$ and $\wh{\theta}$ are sufficiently close. This leads to the bound (\ref{eq:niubility}). A detailed proof of Proposition \ref{prop:entrywise-MLE}  will be given in Section \ref{sec:MLE-pf-prop}.

The result of Proposition \ref{prop:entrywise-MLE} is arguably the most important step in the analysis of the MLE. It directly leads to the control of the spectrum of $H(\theta)$. Then, the first bound (\ref{eq:main-l2}) of Theorem \ref{thm:MLE-estimation} can be obtained by a Taylor expansion of the objective function $\ell_n(\theta)$. The second bound (\ref{eq:main-linf}) of Theorem \ref{thm:MLE-estimation} and Theorem \ref{thm:MLE-ranking} requires an entrywise analysis of $\wh{\theta}$, and is therefore more complicated. We need to take advantage of the powerful leave-one-out argument in \cite{chen2019spectral}. The intuition of the leave-one-out technique has been thoroughly discussed in \cite{chen2019spectral}, and we do not repeat it here. We would like to emphasize that our version of the leave-one-out argument is in fact different from the form introduced in \cite{chen2019spectral}. We do not need to combine the leave-one-out argument with a gradient descent analysis as in \cite{chen2019spectral}. This helps us to avoid the extra technical condition $\log L =O(\log n)$ in \cite{chen2019spectral} when proving the theorems.

\subsection{Some Technical Lemmas}\label{sec:MLE-tech}

Let us present a few technical lemmas that facilitate our analysis of the MLE. The first two lemmas are concentration properties of the random graph $A\sim\mathcal{G}(n,p)$. We define $\mathcal{L}_A=D-A$ to be the graph Laplacian of $A$, where $D$ is a diagonal matrix whose entries are given by $D_{ii}=\sum_{j\in[n]\backslash\{i\}}A_{ij}$.
\begin{lemma}\label{lem:A-basic}
Assume $p\geq\frac{c_0\log n}{n}$ for some sufficiently large $c_0>0$. We then have
$$\frac{1}{2}np\leq \min_{i\in[n]}\sum_{j\in[n]\backslash\{i\}}A_{ij} \leq \max_{i\in[n]}\sum_{j\in[n]\backslash\{i\}}A_{ij} \leq 2np,$$
and
$$\lambda_{\min,\perp}(\mathcal{L}_A)=\min_{u\neq 0: \mathds{1}_{n}^Tu=0}\frac{u^T\mathcal{L}_Au}{\|u\|^2} \geq \frac{np}{2},$$
$$\lambda_{\max}(\mathcal{L}_A)=\max_{u\neq 0}\frac{u^T\mathcal{L}_Au}{\|u\|^2} \leq 2np$$
with probability at least $1-O(n^{-10})$.
\end{lemma}

\begin{lemma}\label{lem:A-bern}
Assume $p\geq\frac{c_0\log n}{n}$ for some sufficiently large $c_0>0$. For any fixed $\{w_{ij}\}$, we have
$$\max_{i\in[n]}\sum_{j\in[n]\backslash\{i\}}w_{ij}^2(A_{ij}-p)^2\leq Cnp\max_{i,j\in[n]}|w_{ij}|^2,$$
and
$$\max_{i\in[n]}\left(\sum_{j\in[n]\backslash\{i\}}w_{ij}(A_{ij}-p)\right)^2\leq C(\log n)^2\max_{i,j\in[n]}|w_{ij}|^2 + Cp\log n\max_{i\in[n]}\sum_{j\in[n]}w_{ij}^2,$$
for some constant $C>0$ with probability at least $1-O(n^{-10})$.
\end{lemma}

With $\lambda_{\min,\perp}(\mathcal{L}_A)$ shown to be well behaved, the next lemma establishes a similar control for $\lambda_{\min,\perp}(H(\theta))$.

\begin{lemma}\label{lem:hessian-spec}
Assume $p\geq\frac{c_0\log n}{n}$ for some sufficiently large $c_0>0$. For any $\theta\in\mathbb{R}^{n}$ that satisfies $\max_{i\in[n]}\theta_i-\min_{i\in[n]}\theta_i\leq M$, we have
$$\lambda_{\min,\perp}(H(\theta)) \geq \frac{1}{8}e^{-M}np,$$
with probability at least $1-O(n^{-10})$.
\end{lemma}

Finally, we need a few concentration inequalities.

\begin{lemma}\label{lem:concentration}
Assume $\kappa=O(1)$ and $p\geq\frac{c_0\log n}{n}$ for some sufficiently large $c_0>0$.
Then, we have
$$\sum_{i=1}^n\left(\sum_{j\in[n]\backslash\{i\}}A_{ij}(\bar{y}_{ij}-\psi(\theta_i^*-\theta_j^*))\right)^2\leq C\frac{n^2p}{L},$$
$$\max_{i\in[n]}\left(\sum_{j\in[n]\backslash\{i\}}A_{ij}(\bar{y}_{ij}-\psi(\theta_i^*-\theta_j^*))\right)^2\leq C\frac{np\log n}{L},$$
$$\max_{i\in[n]}\sum_{j\in[n]\backslash\{i\}}A_{ij}(\bar{y}_{ij}-\psi(\theta_i^*-\theta_j^*))^2\leq C\frac{np}{L},$$
for some constant $C>0$ with probability at least $1-O(n^{-10})$ uniformly over all $\theta^*\in\Theta(k,0,\kappa)$.
\end{lemma}

The proofs of the four lemmas above will be given in Section \ref{sec:pf-tech}.

\subsection{Proof of Proposition \ref{prop:entrywise-MLE}}\label{sec:MLE-pf-prop}

As we have outlined in Section \ref{sec:MLE-over}, the main argument to bound $\|\wh{\theta}-\theta^*\|_{\infty}$ is to first derive a bound for $\|\wh{\theta}_{\lambda}-\theta^*\|_{\infty}$, where $\wh{\theta}_{\lambda}$ is the penalized MLE defined in (\ref{eq:pen-MLE}). Then, we only need to show $\wh{\theta}_{\lambda}$ and $\wh{\theta}$ are close with $\lambda$ as small as $\lambda=n^{-1}$. We first state a lemma that bounds $\|\wh{\theta}_{\lambda}-\theta^*\|_{\infty}$.

\begin{lemma}\label{lem:entrywise-pen-MLE}
Under the setting of Theorem \ref{thm:MLE-estimation}, for the estimator $\wh{\theta}_{\lambda}$ with $\lambda=n^{-1}$, we have
$$\|\wh{\theta}_{\lambda}-\theta^*\|_{\infty}\leq 4,$$
with probability at least $1-O(n^{-7})$.
\end{lemma}

We first prove Proposition \ref{prop:entrywise-MLE} with the help of Lemma \ref{lem:entrywise-pen-MLE}. We then prove Lemma \ref{lem:entrywise-pen-MLE} at the end of this section.

\begin{proof}[Proof of Proposition \ref{prop:entrywise-MLE}]
Define a constraint MLE as
\begin{equation}
\wh{\theta}^{\sf con}=\argmin_{\mathds{1}_{n}^T\theta=0: \|\theta-\theta^*\|_{\infty}\leq 5}\ell_n(\theta).\label{eq:con-MLE}
\end{equation}
By Lemma \ref{lem:entrywise-pen-MLE}, $\wh{\theta}_{\lambda}$ is feasible for the constraint of (\ref{eq:con-MLE}). We then have
\begin{equation}
\ell_n(\wh{\theta}_{\lambda}) \geq \ell_n(\wh{\theta}^{\sf con}). \label{eq:basic:con-pen}
\end{equation}
We apply Taylor expansion, and obtain
$$
\ell_n(\wh{\theta}^{\sf con}) = \ell_n(\wh{\theta}_{\lambda})  + (\wh{\theta}^{\sf con}-\wh{\theta}_{\lambda})^T\nabla\ell_n(\wh{\theta}_{\lambda}) + \frac{1}{2}(\wh{\theta}_{\lambda}-\wh{\theta}^{\sf con})^TH(\xi)(\wh{\theta}_{\lambda}-\wh{\theta}^{\sf con}),
$$
where $\xi$ is a convex combination of $\wh{\theta}^{\sf con}$ and $\wh{\theta}_{\lambda}$. By Lemma \ref{lem:entrywise-pen-MLE}, we know that $\|\wh{\theta}_{\lambda}-\theta^*\|_{\infty}\leq 4$. We also have $\|\wh{\theta}^{\sf con}-\theta^*\|_{\infty}\leq 5$ by the definition of $\wh{\theta}^{\sf con}$. Thus, $\|\xi-\theta^*\|_{\infty}\leq 5$. By Lemma \ref{lem:hessian-spec}, we get the lower bound
\begin{equation}
\ell_n(\wh{\theta}^{\sf con}) \geq \ell_n(\wh{\theta}_{\lambda})  + (\wh{\theta}^{\sf con}-\wh{\theta}_{\lambda})^T\nabla\ell_n(\wh{\theta}_{\lambda}) + c_1np\|\wh{\theta}^{\sf con}-\wh{\theta}_{\lambda}\|^2, \label{eq:taylor:con-pen}
\end{equation}
for some constant $c_1>0$. By (\ref{eq:basic:con-pen}) and (\ref{eq:taylor:con-pen}), we have
$$
\|\wh{\theta}^{\sf con}-\wh{\theta}_{\lambda}\|^2 \leq \frac{|(\wh{\theta}^{\sf con}-\wh{\theta}_{\lambda})^T\nabla\ell_n(\wh{\theta}_{\lambda})|}{c_1np}.
$$
By Cauchy-Schwarz inequality and the fact that $\nabla\ell_n(\wh{\theta}_{\lambda})+\lambda\wh{\theta}_{\lambda}=0$, we have
$$
\|\wh{\theta}^{\sf con}-\wh{\theta}_{\lambda}\|^2 \leq \frac{\|\nabla\ell_n(\wh{\theta}_{\lambda})\|^2}{(c_1np)^2} = \frac{\lambda^2\|\wh{\theta}_{\lambda}\|^2}{(c_1np)^2} \lesssim \frac{n\lambda^2}{(c_1np)^2} \lesssim n^{-1}.
$$
Finally, since
$$\|\wh{\theta}^{\sf con}-\theta^*\|_{\infty} \leq \|\wh{\theta}_{\lambda}-\theta^*\|_{\infty}+\|\wh{\theta}^{\sf con}-\wh{\theta}_{\lambda}\|\leq 4 + \frac{c_2}{\sqrt{n}} \leq \frac{9}{2},$$
the minimizer of (\ref{eq:con-MLE}) is in the interior of the constraint. By the convexity of (\ref{eq:con-MLE}), we have $\wh{\theta}^{\sf con}=\wh{\theta}$, and thus the desired conclusion $\|\wh{\theta}-\theta^*\|_{\infty}\leq 5$ is obtained.
\end{proof}

\begin{proof}[Proof of Lemma \ref{lem:entrywise-pen-MLE}]
Our proof largely follows the arguments in \cite{chen2019spectral} that analyze the regularized MLE. Since we only need to show $\|\wh{\theta}_{\lambda}-\theta^*\|_{\infty}\leq 4$ rather than the optimal rate, the condition on $L$ imposed by \cite{chen2019spectral} is not needed anymore. This requires a few minor changes in the proof of \cite{chen2019spectral}. We still write down every step of the proof for the result to be self-contained.

Define a gradient descent sequence
\begin{equation}
\theta^{(t+1)}=\theta^{(t)}-\eta\left(\nabla\ell_n(\theta^{(t)})+\lambda\theta^{(t)}\right). \label{eq:gd-pen}
\end{equation}
We also need to introduce a leave-one-out gradient descent sequence. Define
\begin{eqnarray*}
\ell_n^{(m)}(\theta) &=& \sum_{1\leq i<j\leq n: i,j\neq m}A_{ij}\left[\bar{y}_{ij}\log\frac{1}{\psi(\theta_i-\theta_j)}+(1-\bar{y}_{ij})\log\frac{1}{1-\psi(\theta_i-\theta_j)}\right] \\
&& + \sum_{i\in[n]\backslash\{m\}}p\left[\psi(\theta_i^*-\theta_m^*)\log\frac{1}{\psi(\theta_i-\theta_m)}+\psi(\theta_m^*-\theta_i^*)\log\frac{1}{\psi(\theta_m-\theta_i)}\right].
\end{eqnarray*}
With the objective $\ell_n^{(m)}(\theta)$, we define
\begin{equation}
\theta^{(t+1,m)}=\theta^{(t,m)}-\eta\left(\nabla\ell_n^{(m)}(\theta^{(t,m)})+\lambda\theta^{(t,m)}\right). \label{eq:gd-pen-m}
\end{equation}
We initialize both (\ref{eq:gd-pen}) and (\ref{eq:gd-pen-m}) by $\theta^{(0)}=\theta^{(0,m)}=\theta^*$ and use the same step size $\eta=\frac{1}{\lambda + np}$. Note that $\mathds{1}_{n}^T\theta^*=0$ implies $\mathds{1}_{n}^T\theta^{(t)}=\mathds{1}_{n}^T\theta^{(t,m)}=0$ for all $t$. See Section 4.3 of \cite{chen2019spectralsupp}. We will establish the following bounds,
\begin{eqnarray}
\label{eq:induction-gd-1} && \max_{m\in[n]}\|\theta^{(t,m)}-\theta^{(t)}\| \leq 1, \\
\label{eq:induction-gd-2} && \|\theta^{(t)}-\theta^*\| \leq \sqrt{\frac{n}{\log n}}, \\
\label{eq:induction-gd-3} && \max_{m\in[n]}|\theta_m^{(t,m)}-\theta_m^*| \leq 1.
\end{eqnarray}
It is obvious that (\ref{eq:induction-gd-1}), (\ref{eq:induction-gd-2}) and (\ref{eq:induction-gd-3}) hold for $t=0$. We use a mathematical induction argument to show (\ref{eq:induction-gd-1}), (\ref{eq:induction-gd-2}) and (\ref{eq:induction-gd-3}) for a general $t$. Let us suppose (\ref{eq:induction-gd-1}), (\ref{eq:induction-gd-2}) and (\ref{eq:induction-gd-3}) are true, and we need to show the same conclusions continue to hold for $t+1$.

First, we have
\begin{eqnarray*}
\theta^{(t+1)} - \theta^{(t+1,m)} &=& (1-\eta\lambda)(\theta^{(t)} - \theta^{(t,m)}) - \eta(\nabla\ell_n(\theta^{(t)})-\nabla\ell_n^{(m)}(\theta^{(t,m)})) \\
&=& \left((1-\eta\lambda)I_n-\eta H(\xi)\right)(\theta^{(t)} - \theta^{(t,m)}) -\eta\left(\nabla\ell_n(\theta^{(t,m)})-\nabla\ell_n^{(m)}(\theta^{(t,m)})\right),
\end{eqnarray*}
where $\xi$ is a convex combination of $\theta^{(t)}$ and $\theta^{(t,m)}$. By (\ref{eq:induction-gd-1}) and (\ref{eq:induction-gd-3}), we have
\begin{equation}
\|\theta^{(t)}-\theta^*\|_{\infty} \leq \max_{m\in[n]}\|\theta^{(t,m)}-\theta^{(t)}\| + \max_{m\in[n]}|\theta_m^{(t,m)}-\theta_m^*| \leq 2, \label{eq:induction-t-max}
\end{equation}
and
\begin{equation}
\|\theta^{(t,m)}-\theta^*\|_{\infty} \leq \|\theta^{(t)}-\theta^*\|_{\infty} + \|\theta^{(t,m)}-\theta^{(t)}\| \leq 3. \label{eq:induction-t-max-m}
\end{equation}
We thus have $\|\xi-\theta^*\|_{\infty}\leq 3$, and we can apply Lemma \ref{lem:hessian-spec} to obtain the bound
\begin{equation}
\|\left((1-\eta\lambda)I_n-\eta H(\xi)\right)(\theta^{(t)} - \theta^{(t,m)})\| \leq (1-\eta\lambda - c_1\eta np)\|\theta^{(t)} - \theta^{(t,m)}\|, \label{eq:ind-1-1}
\end{equation}
for some constant $c_1>0$. We also note that
\begin{eqnarray}
\nonumber && \|\nabla\ell_n(\theta^{(t,m)})-\nabla\ell_n^{(m)}(\theta^{(t,m)})\|^2 \\
\nonumber &=&  \left(\sum_{j\in[n]\backslash\{m\}}A_{jm}(\bar{y}_{jm}-\psi(\theta_j^*-\theta_m^*))-\sum_{j\in[n]\backslash\{m\}}(A_{jm}-p)(\psi(\theta_j^{(t,m)}-\theta_{m}^{(t,m)})-\psi(\theta_j^*-\theta_m^*))\right)^2 \\
\nonumber && + \sum_{j\in[n]\backslash\{m\}}\left(A_{jm}(\bar{y}_{jm}-\psi(\theta_j^*-\theta_m^*))-(A_{jm}-p)(\psi(\theta_j^{(t,m)}-\theta_{m}^{(t,m)})-\psi(\theta_j^*-\theta_m^*))\right)^2 \\
\label{eq:ind-1-2} &\leq& C_1\frac{np\log n}{L} + C_1np\log n\|\theta^{(t,m)}-\theta^*\|_{\infty}^2,
\end{eqnarray}
for some constant $C_1>0$ by Lemma \ref{lem:A-bern} and Lemma \ref{lem:concentration}. We combine the two bounds (\ref{eq:ind-1-1}) and (\ref{eq:ind-1-2}), and obtain
\begin{eqnarray}
\nonumber \|\theta^{(t+1)} - \theta^{(t+1,m)}\| &\leq& (1-\eta\lambda - c_1\eta np)\|\theta^{(t)} - \theta^{(t,m)}\| + \eta\sqrt{C_1np\log n\left(L^{-1}+\|\theta^{(t,m)}-\theta^*\|_{\infty}^2\right)} \\
\label{eq:anderson} &\leq& (1-c_1\eta np) + \eta\sqrt{C_1np\log n\left(L^{-1}+9\right)} \\
\label{eq:pinhan} &\leq& 1
\end{eqnarray}
where the inequality (\ref{eq:anderson}) is by (\ref{eq:induction-gd-1}) and (\ref{eq:induction-t-max-m}). The inequality (\ref{eq:pinhan}) requires that $\sqrt{C_1np\log n\left(L^{-1}+9\right)}\leq c_1np$, which is implied by the condition that $p\geq\frac{c_0\log n}{n}$ for some sufficiently large $c_0>0$. We thus have proved (\ref{eq:induction-gd-1}) for $t+1$.

Next, we have
\begin{eqnarray*}
\theta^{(t+1)}-\theta^* &=& \theta^{(t)}-\theta^*-\eta\left(\nabla\ell_n(\theta^{(t)})+\lambda\theta^{(t)}\right) \\
&=& (1-\eta\lambda)(\theta^{(t)}-\theta^*) -\eta\left(\nabla\ell_n(\theta^{(t)})-\nabla\ell_n(\theta^*)\right) -\eta\lambda\theta^* -\eta\nabla\ell_n(\theta^*) \\
&=& \left((1-\eta\lambda)I_n-\eta H(\xi)\right)(\theta^{(t)}-\theta^*)  -\eta\lambda\theta^* -\eta\nabla\ell_n(\theta^*),
\end{eqnarray*}
where $\xi$ is abused for a vector that is a  convex combination of $\theta^{(t)}$ and $\theta^*$. Since by (\ref{eq:induction-t-max}) we get $\|\xi-\theta^*\|_{\infty}\leq \|\theta^{(t)}-\theta^*\|_{\infty}\leq 2$,
we can use Lemma \ref{lem:hessian-spec} to obtain the bound
\begin{equation}
\left((1-\eta\lambda)I_n-\eta H(\xi)\right)(\theta^{(t)}-\theta^*) \leq (1-\eta\lambda-c_2\eta np)\|\theta^{(t)}-\theta^*\|, \label{eq:ind-2-1}
\end{equation}
for some constant $c_2>0$. We also note that
\begin{equation}
\|\nabla\ell_n(\theta^*)\|^2 = \sum_{i=1}^n\left(\sum_{j\in[n]\backslash\{i\}}A_{ij}(\bar{y}_{ij}-\psi(\theta_i^*-\theta_j^*))\right)^2 \leq C_2\frac{n^2p}{L}, \label{eq:ind-2-2}
\end{equation}
for some constant $C_2>0$ with high probability by Lemma \ref{lem:concentration}. Combine the bounds (\ref{eq:ind-2-1}) and (\ref{eq:ind-2-2}), and we obtain
\begin{eqnarray*}
\|\theta^{(t+1)}-\theta^*\| &\leq& (1-\eta\lambda-c_2\eta np)\|\theta^{(t)}-\theta^*\| + \eta\sqrt{C_2\frac{n^2p}{L}} + \eta\lambda\|\theta^*\| \\
&\leq& (1-c_2\eta np)\sqrt{\frac{n}{\log n}} + \eta\sqrt{C_2\frac{n^2p}{L}} + \eta\lambda\|\theta^*\| \\
&\leq& \sqrt{\frac{n}{\log n}},
\end{eqnarray*}
where the last inequality is due to $\eta\sqrt{C_2\frac{n^2p}{L}} + \eta\lambda\|\theta^*\|\lesssim \frac{1}{\sqrt{Lp}}+\frac{1}{n^{3/2}p}=o\left(\eta np\sqrt{\frac{n}{\log n}}\right)$ by the choice of $\eta$ and $\lambda$. Hence, (\ref{eq:induction-gd-2}) holds for $t+1$.

Finally, we have
\begin{eqnarray*}
\theta_m^{(t+1,m)} -\theta_m^* &=& \theta_m^{(t,m)}-\theta_m^* + \eta p\sum_{j\in[n]\backslash\{m\}}\left(\psi(\theta_m^*-\theta_j^*)-\psi(\theta_m^{(t,m)}-\theta_j^{(t,m)})\right) -\lambda\eta\theta_{m}^{(t,m)} \\
&=& \theta_m^{(t,m)}-\theta_m^* + \eta p\sum_{j\in[n]\backslash\{m\}}\psi'(\xi_j)(\theta_m^*-\theta_j^*-\theta_m^{(t,m)}+\theta_j^{(t,m)}) -\lambda\eta\theta_{m}^{(t,m)} \\
&=& \left(1-\eta\lambda - \eta p\sum_{j\in[n]\backslash\{m\}}\psi'(\xi_j)\right)(\theta_m^{(t,m)}-\theta_m^*) - \lambda\eta\theta_m^* \\
&& + \eta p\sum_{j\in[n]\backslash\{m\}}\psi'(\xi_j)(\theta_j^{(t,m)}-\theta_j^*),
\end{eqnarray*}
where $\xi_j$ is a scalar between $\theta_m^*-\theta_j^*$ and $\theta_m^{(t,m)}-\theta_j^{(t,m)}$. By (\ref{eq:induction-t-max-m}), we have $|\xi_j-\theta_m^*+\theta_j^*|\leq |\theta_m^*-\theta_j^*-\theta_m^{(t,m)}+\theta_j^{(t,m)}|\leq 6$, which implies $\|\xi\|_{\infty}$ is bounded. We then have $\sum_{j\in[n]\backslash\{m\}}\psi'(\xi_j)\geq c_3n$ for some constant $c_3>0$, and thus
\begin{equation}
\left|\left(1-\eta\lambda - \eta p\sum_{j\in[n]\backslash\{m\}}\psi'(\xi_j)\right)(\theta_m^{(t,m)}-\theta_m^*)\right| \leq (1-\eta\lambda-c_3\eta np)|\theta_m^{(t,m)}-\theta_m^*|. \label{eq:ind-3-1}
\end{equation}
We also have
\begin{equation}
\left|\sum_{j\in[n]\backslash\{m\}}\psi'(\xi_j)(\theta_j^{(t,m)}-\theta_j^*)\right| \leq \|\theta^{(t,m)}-\theta^*\|_1 \leq \sqrt{n}\|\theta^{(t,m)}-\theta^*\| \leq \sqrt{n}\left(1+\sqrt{\frac{n}{\log n}}\right), \label{eq:ind-3-2}
\end{equation}
where the last inequality is by (\ref{eq:induction-gd-1}) and (\ref{eq:induction-gd-2}). Combine the bounds (\ref{eq:ind-3-1}) and (\ref{eq:ind-3-2}), and we get
\begin{eqnarray*}
|\theta_m^{(t+1,m)} -\theta_m^*| &\leq& (1-\eta\lambda-c_3\eta np)|\theta_m^{(t,m)}-\theta_m^*| + \eta p\sqrt{n}\left(1+\sqrt{\frac{n}{\log n}}\right) + \lambda\eta|\theta_m^*| \\
&\leq& \left(1-c_3\eta np\right) + \eta p\sqrt{n} + \eta p\frac{n}{\sqrt{\log n}} + \lambda\eta|\theta_m^*| \\
&\leq& 1,
\end{eqnarray*}
where the last inequality is because of $\eta p\sqrt{n} + \eta p\frac{n}{\sqrt{\log n}} + \lambda\eta|\theta_m^*| = o\left(\eta np\right)$ by the choice of $\eta$ and $\lambda$. Hence, (\ref{eq:induction-gd-3}) holds for $t+1$.

To summarize, we have shown that (\ref{eq:induction-gd-1}), (\ref{eq:induction-gd-2}) and (\ref{eq:induction-gd-3}) hold for all $t\leq t^*$ with probability at least $1-O(t^*n^{-10})$. The reason why we have the probability $1-O(t^*n^{-10})$ is because we need to apply Lemma \ref{lem:A-bern} with a different weight at each iteration to show (\ref{eq:ind-1-2}). Note that the bound (\ref{eq:induction-t-max}) holds for all $t\leq t^*$ as well and we thus have $\|\theta^{(t^*)}-\theta^*\|_{\infty}\leq 2$. With a standard optimization result for a strongly convex objective function, we have
$$\|\theta^{(t^*)}-\wh{\theta}_{\lambda}\|\leq\left(1-\frac{\lambda}{\lambda+np}\right)^{t^*}\|\wh{\theta}_{\lambda}-\theta^*\|.$$
See Lemma 6.7 of \cite{chen2019spectral}. By triangle inequality, we have
$$\|\wh{\theta}_{\lambda}-\theta^*\|_{\infty}\leq \|\theta^{(t^*)}-\wh{\theta}_{\lambda}\| + \|\theta^{(t^*)}-\theta^*\|_{\infty}\leq \left(1-\frac{\lambda}{\lambda+np}\right)^{t^*}\sqrt{n}\|\wh{\theta}_{\lambda}-\theta^*\|_{\infty} + 2.$$
Since $\left(1-\frac{\lambda}{\lambda+np}\right)\leq 1-\frac{1}{1+n^2}$, we can take $t^*=n^3$ in order that $\left(1-\frac{\lambda}{\lambda+np}\right)^{t^*}\sqrt{n}\leq \frac{1}{2}$. This implies $\|\wh{\theta}_{\lambda}-\theta^*\|_{\infty}\leq 4$ with probability at least $1-O(n^{-7})$ as desired.
\end{proof}

\subsection{Proof of Theorem \ref{thm:MLE-estimation}}\label{sec:MLE-pf-thm-est}

We give separate proofs for the conclusions (\ref{eq:main-l2}) and (\ref{eq:main-linf}) in this section.

\begin{proof}[Proof of (\ref{eq:main-l2}) of Theorem \ref{thm:MLE-estimation}]
By the definition of $\wh{\theta}$, we have $\ell_n(\theta^*)\geq\ell_n(\wh{\theta})$. We then apply Taylor expansion and obtain
$$\ell_n(\wh{\theta})=\ell_n(\theta^*) + (\wh{\theta}-\theta^*)^T\nabla\ell_n(\theta^*) + \frac{1}{2}(\wh{\theta}-\theta^*)^TH(\xi)(\wh{\theta}-\theta^*),$$
where $\xi$ is a convex combination of $\wh{\theta}$ and $\theta^*$. By Proposition \ref{prop:entrywise-MLE}, we have $\|\wh{\theta}-\theta^*\|_{\infty}\leq 5$, which implies $\|\xi-\theta^*\|_{\infty}\leq 5$. Thus, we can apply Lemma \ref{lem:hessian-spec} and get $\frac{1}{2}(\wh{\theta}-\theta^*)^TH(\xi)(\wh{\theta}-\theta^*)\geq c_1np\|\wh{\theta}-\theta^*\|^2$ for some constant $c_1>0$. Together with $\ell_n(\theta^*)\geq\ell_n(\wh{\theta})$ and a Cauchy-Schwarz inequality, we have $\|\wh{\theta}-\theta^*\|^2\leq \frac{\|\nabla\ell_n(\theta^*)\|^2}{(c_1np)^2}$. Use (\ref{eq:ind-2-2}) and Lemma \ref{lem:concentration}, we obtain the desired conclusion that $\|\wh{\theta}-\theta^*\|^2\lesssim \frac{1}{Lp}$.
\end{proof}

The proof of (\ref{eq:main-linf}) is more involved. It is based on a leave-one-out argument that is very different from the one used in \cite{chen2019spectral}. Let us decompose the objective function $\ell_n(\theta)$ as
\begin{equation}
\ell_n(\theta) = {\ell}_n^{(-m)}(\theta_{-m}) + {\ell}_n^{(m)}(\theta_m|\theta_{-m}), \label{eq:l-1-out-decomp}
\end{equation}
where we use $\theta_m\in\mathbb{R}$ for the $m$th entry of $\theta$ and $\theta_{-m}\in\mathbb{R}^{n-1}$ for the remaining entries. The two functions in (\ref{eq:l-1-out-decomp}) are defined as
\begin{eqnarray*}
{\ell}_n^{(-m)}(\theta_{-m}) &=& \sum_{1\leq i<j\leq n: i,j\neq m}A_{ij}\left[\bar{y}_{ij}\log\frac{1}{\psi(\theta_i-\theta_j)}+(1-\bar{y}_{ij})\log\frac{1}{1-\psi(\theta_i-\theta_j)}\right], \\
{\ell}_n^{(m)}(\theta_m|\theta_{-m}) &=& \sum_{j\in[n]\backslash\{m\}}A_{mj}\left[\bar{y}_{mj}\log\frac{1}{\psi(\theta_m-\theta_j)}+(1-\bar{y}_{mj})\log\frac{1}{1-\psi(\theta_m-\theta_j)}\right].
\end{eqnarray*}
Define
\begin{equation}
{\theta}_{-m}^{(m)} = \argmin_{\theta_{-m}:\|\theta_{-m}-\theta_{-m}^*\|_{\infty}\leq 5}{\ell}_n^{(-m)}(\theta_{-m}). \label{eq:con-MLE-l-m}
\end{equation}
We first present an $\ell_2$ norm bound for ${\theta}_{-m}^{(m)}$. We also use $H^{(-m)}(\theta_{-m})$ for the Hessian matrix $\nabla^2{\ell}_n^{(-m)}(\theta_{-m})$.
\begin{lemma}\label{lem:global-rate-l-1-o}
Under the setting of Theorem \ref{thm:MLE-estimation}, there exists some constant $C>0$ such that
$$\max_{m\in[n]}\|\theta_{-m}^{(m)}-\theta_{-m}^*-a_m\mathds{1}_{n-1}\|^2\leq C\frac{1}{pL},$$
with probability at least $1-O(n^{-9})$, where $a_m=\ave(\theta_{-m}^{(m)}-\theta_m^*)$.
\end{lemma}
\begin{proof}
The proof is very similar to that of (\ref{eq:main-l2}), since ${\theta}_{-m}^{(m)}$ can be thought of as a constrained MLE on a subset of the data. By the definition of ${\theta}_{-m}^{(m)}$, we have
\begin{eqnarray*}
{\ell}_n^{(-m)}(\theta_{-m}^*) &\geq& {\ell}_n^{(-m)}(\theta_{-m}^{(m)}) \\
&=& {\ell}_n^{(-m)}(\theta_{-m}^*) + (\theta_{-m}^{(m)}-\theta_{-m}^*-a_m\mathds{1}_{n-1})^T\nabla{\ell}_n^{(-m)}(\theta_{-m}^*) \\
&& + \frac{1}{2}(\theta_{-m}^{(m)}-\theta_{-m}^*-a_m\mathds{1}_{n-1})^TH^{(-m)}(\xi)(\theta_{-m}^{(m)}-\theta_{-m}^*-a_m\mathds{1}_{n-1}),
\end{eqnarray*}
where $\xi$ is a convex combination of $\theta_{-m}^{(m)}$ and $\theta_{-m}^*$. In the above Taylor expansion, we have also used the property that ${\ell}_n^{(-m)}(\theta_{-m})={\ell}_n^{(-m)}(\theta_{-m}+c\mathds{1}_{n-1})$, $\nabla{\ell}_n^{(-m)}(\theta_{-m})=\nabla{\ell}_n^{(-m)}(\theta_{-m}+c\mathds{1}_{n-1})$ and $H^{(-m)}(\theta_{-m})=H^{(-m)}(\theta_{-m}+c\mathds{1}_{n-1})$ for any $c\in\mathbb{R}$. Since $\|\xi-\theta_{-m}^*\|_{\infty}\leq \|\theta_{-m}^{(m)}-\theta_{-m}^*\|_{\infty}\leq 5$, we can apply Lemma \ref{lem:hessian-spec} to the subset of the data, and obtain
$$\frac{1}{2}(\theta_{-m}^{(m)}-\theta_{-m}^*-a_m\mathds{1}_{n-1})^TH^{(-m)}(\xi)(\theta_{-m}^{(m)}-\theta_{-m}^*-a_m\mathds{1}_{n-1})\geq c_1np\|\theta_{-m}^{(m)}-\theta_{-m}^*-a_m\mathds{1}_{n-1}\|^2,$$
with probability at least $1-O(n^{-10})$ for some constant $c_1>0$. By Cauchy-Schwarz inequality, we have
$$\|\theta_{-m}^{(m)}-\theta_{-m}^*-a_m\mathds{1}_{n-1}\|^2\leq \frac{\|\nabla{\ell}_n^{(-m)}(\theta_{-m}^*)\|^2}{(c_1np)^2}.$$
Apply (\ref{eq:ind-2-2}) and Lemma \ref{lem:concentration} to the subset of the data, and we obtain that $\|\theta_{-m}^{(m)}-\theta_{-m}^*-a_m\mathds{1}_{n-1}\|^2\leq C\frac{1}{pL}$ with probability at least $1-O(n^{-10})$. Finally, a union bound argument leads to the desired result.
\end{proof}

With the help of Lemma \ref{lem:global-rate-l-1-o}, we are ready to prove (\ref{eq:main-linf}).

\begin{proof}[Proof of (\ref{eq:main-linf}) of Theorem \ref{thm:MLE-estimation}]
By Proposition \ref{prop:entrywise-MLE}, we have $\|\wh{\theta}_{-m}-\theta^*_{-m}\|_{\infty}\leq \|\wh{\theta}-\theta^*\|_{\infty}\leq 5$, and thus $\wh{\theta}_{-m}$ is feasible for the constraint of (\ref{eq:con-MLE-l-m}).
By the definition of ${\theta}_{-m}^{(m)}$, we have
\begin{eqnarray*}
\ell_n^{(-m)}(\wh{\theta}_{-m}) &\geq& {\ell}_n^{(-m)}(\theta_{-m}^{(m)}) \\
&=& \ell_n^{(-m)}(\wh{\theta}_{-m}) + (\theta^{(m)}_{-m}-\wh{\theta}_{-m}-\bar{a}_m\mathds{1}_{n-1})^T\nabla\ell_n^{(-m)}(\wh{\theta}_{-m}) \\
&& + \frac{1}{2}(\theta^{(m)}_{-m}-\wh{\theta}_{-m}-\bar{a}_m\mathds{1}_{n-1})^TH^{(-m)}(\xi)(\theta^{(m)}_{-m}-\wh{\theta}_{-m}-\bar{a}_m\mathds{1}_{n-1}),
\end{eqnarray*}
where $\bar{a}_m=\ave(\theta^{(m)}_{-m}-\wh{\theta}_{-m})$ and $\xi$ is a convex combination of $\theta^{(m)}_{-m}$ and $\wh{\theta}_{-m}$. Since both $\theta^{(m)}_{-m}$ and $\wh{\theta}_{-m}$ satisfy the constraint of (\ref{eq:con-MLE-l-m}), we must have $\|\xi-\theta_{-m}^*\|_{\infty}\leq 5$. Then, we can apply Lemma \ref{lem:hessian-spec} to the subset of the data, and obtain
$$\frac{1}{2}(\theta^{(m)}_{-m}-\wh{\theta}_{-m}-\bar{a}_m\mathds{1}_{n-1})^TH^{(-m)}(\xi)(\theta^{(m)}_{-m}-\wh{\theta}_{-m}-\bar{a}_m\mathds{1}_{n-1})\geq c_1np\|\theta^{(m)}_{-m}-\wh{\theta}_{-m}-\bar{a}_m\mathds{1}_{n-1}\|^2,$$
for some constant $c_1>0$. By Cauchy-Schwarz inequality, we have
$$\|\theta^{(m)}_{-m}-\wh{\theta}_{-m}-\bar{a}_m\mathds{1}_{n-1}\|^2\leq \frac{\|\nabla\ell_n^{(-m)}(\wh{\theta}_{-m})\|^2}{(c_1np)^2}.$$
For each $i\in[n]\backslash\{m\}$, by the decomposition (\ref{eq:l-1-out-decomp}), we have
$$\frac{\partial}{\partial\theta_i}{\ell}_n^{(-m)}(\theta_{-m}) = \frac{\partial}{\partial\theta_i}\ell_n(\theta) - \frac{\partial}{\partial\theta_i}{\ell}_n^{(m)}(\theta_m|\theta_{-m}).$$
Since $\nabla\ell_n(\wh{\theta})=0$, we have
$$\frac{\partial}{\partial\theta_i}{\ell}_n^{(-m)}(\theta_{-m})_{|\theta=\wh{\theta}}=- \frac{\partial}{\partial\theta_i}{\ell}_n^{(m)}(\theta_m|\theta_{-m})_{|\theta=\wh{\theta}}=-A_{mi}(\bar{y}_{mi}-\psi(\wh{\theta}_m-\wh{\theta}_i)).$$
We therefore have the bound
\begin{eqnarray*}
\|\nabla\ell_n^{(-m)}(\wh{\theta}_{-m})\|^2 &=& \sum_{i\in[n]\backslash\{m\}}A_{mi}(\bar{y}_{mi}-\psi(\wh{\theta}_m-\wh{\theta}_i))^2 \\
&\leq& 2\sum_{i\in[n]\backslash\{m\}}A_{mi}(\bar{y}_{mi}-\psi(\theta_m^*-\theta_i^*))^2 \\
&& + 2\sum_{i\in[n]\backslash\{m\}}A_{mi}(\psi(\theta_m^*-\theta_i^*)-\psi(\wh{\theta}_m-\wh{\theta}_i))^2 \\
&\leq& 2\sum_{i\in[n]\backslash\{m\}}A_{mi}(\bar{y}_{mi}-\psi(\theta_m^*-\theta_i^*))^2 + 2\|\wh{\theta}-\theta^*\|_{\infty}^2\sum_{i\in[n]\backslash\{m\}}A_{mi} \\
&\leq& 2\sum_{i\in[n]\backslash\{m\}}A_{mi}(\bar{y}_{mi}-\psi(\theta_m^*-\theta_i^*))^2 + 4np\|\wh{\theta}-\theta^*\|_{\infty}^2,
\end{eqnarray*}
where the last inequality is by Lemma \ref{lem:A-basic}. This implies
\begin{eqnarray*}
\max_{m\in[n]}\|\theta^{(m)}_{-m}-\wh{\theta}_{-m}-\bar{a}_m\mathds{1}_{n-1}\|^2 &\leq& \frac{\max_{m\in[n]}\sum_{i\in[n]\backslash\{m\}}A_{mi}(\bar{y}_{mi}-\psi(\theta_m^*-\theta_i^*))^2}{(c_1np)^2/2} \\
&& + \frac{\|\wh{\theta}-\theta^*\|_{\infty}^2}{c_1^2np/4}.
\end{eqnarray*}
Since we need a bound for $\max_{m\in[n]}\|\theta^{(m)}_{-m}-\wh{\theta}_{-m}-{a}_m\mathds{1}_{n-1}\|^2$, we need to quantify the difference between $a_m$ and $\bar{a}_m$. Recall that $a_m=\ave(\theta_{-m}^{(m)}-\theta_m^*)$. Since $\mathds{1}_n^T\wh{\theta}=\mathds{1}_n^T\theta^*=0$, we have
$$\|{a}_m\mathds{1}_{n-1}-\bar{a}_m\mathds{1}_{n-1}\|^2=(n-1)(\ave(\wh{\theta}_{-m}-\theta^*_{-m}))^2=\frac{(\wh{\theta}_m-\theta_m^*)^2}{n-1}\leq \frac{\|\wh{\theta}-\theta^*\|_{\infty}^2}{n-1}.$$
We then have
\begin{eqnarray}
\nonumber \max_{m\in[n]}\|\theta^{(m)}_{-m}-\wh{\theta}_{-m}-a_m\mathds{1}_{n-1}\|^2 &\leq& C_1\frac{\max_{m\in[n]}\sum_{i\in[n]\backslash\{m\}}A_{mi}(\bar{y}_{mi}-\psi(\theta_m^*-\theta_i^*))^2}{n^2p^2} \\
\label{eq:interesting} && + C_1\frac{\|\wh{\theta}-\theta^*\|_{\infty}^2}{np},
\end{eqnarray}
for some constant $C_1>0$.

Next, let us derive a bound for $\|\wh{\theta}-\theta^*\|_{\infty}^2$ in terms of $\max_{m\in[n]}\|\theta^{(m)}_{-m}-\wh{\theta}_{-m}-a_m\mathds{1}_{n-1}\|^2$. We introduce the notation
\begin{eqnarray*}
f^{(m)}(\theta_m|\theta_{-m}) &=& \frac{\partial}{\partial\theta_m}{\ell}_n^{(m)}(\theta_m|\theta_{-m}) = -\sum_{i\in[n]\backslash\{m\}}A_{mi}(\bar{y}_{mi}-\psi(\theta_m-\theta_i)), \\
g^{(m)}(\theta_m|\theta_{-m}) &=& \frac{\partial^2}{\partial\theta_m^2}{\ell}_n^{(m)}(\theta_m|\theta_{-m}) = \sum_{i\in[n]\backslash\{m\}}A_{mi}\psi(\theta_m-\theta_i)\psi(\theta_i-\theta_m).
\end{eqnarray*}
By the definition of $\wh{\theta}$, we know that $\ell_n(\wh{\theta})=\min_{\theta:\mathds{1}_{n}^T\theta=0}\ell_n(\theta)$. Since $\ell_n(\theta)=\ell_n(\theta+c\mathds{1}_{n})$ for any $c\in\mathbb{R}$, we also have $\ell_n(\wh{\theta})=\min_{\theta}\ell_n(\theta)$. This allows us to compare the value of the objective $\ell_n(\theta)$ at $\wh{\theta}$ with any vector that is not necessarily centered. We then have
$$\ell_n^{(m)}(\theta_m^*|\wh{\theta}_{-m}) + \ell_n^{(-m)}(\wh{\theta}_{-m})\geq\ell_n(\wh{\theta}),$$
which implies
\begin{eqnarray*}
\ell_n^{(m)}(\theta_m^*|\wh{\theta}_{-m}) &\geq& \ell_n^{(m)}(\wh{\theta}_m|\wh{\theta}_{-m}) \\
&=& \ell_n^{(m)}(\theta_m^*|\wh{\theta}_{-m}) + (\wh{\theta}_m-\theta_m^*)f^{(m)}(\theta_m^*|\wh{\theta}_{-m}) + \frac{1}{2}(\wh{\theta}_m-\theta_m^*)^2g^{(m)}(\xi|\wh{\theta}_{-m}),
\end{eqnarray*}
where $\xi$ is a scalar between $\theta_m^*$ and $\wh{\theta}_m$. By Proposition \ref{prop:entrywise-MLE}, $|\xi-\theta_m^*|\leq |\wh{\theta}_m-\theta_m^*|\leq \|\wh{\theta}-\theta^*\|_{\infty}\leq 5$. Therefore, for any $i\in[n]\backslash\{m\}$, $|\xi-\wh{\theta}_i|\leq |\xi-\theta_m^*|+|\theta_m^*-\theta_i^*|+|\wh{\theta}_i-\theta_i^*|\leq 10 + \kappa$. This implies $\frac{1}{2}g^{(m)}(\xi|\wh{\theta}_{-m})\geq c_2np$ for some constant $c_2>0$ with the help of Lemma \ref{lem:A-basic}. We then have the bound
\begin{equation}
(\wh{\theta}_m-\theta_m^*)^2\leq \frac{|f^{(m)}(\theta_m^*|\wh{\theta}_{-m})|^2}{(c_2np)^2}. \label{eq:fixec}
\end{equation}
We bound $|f^{(m)}(\theta_m^*|\wh{\theta}_{-m})|$ by
\begin{eqnarray}
\nonumber |f^{(m)}(\theta_m^*|\wh{\theta}_{-m})| &=& \left|\sum_{i\in[n]\backslash\{m\}}A_{mi}(\bar{y}_{mi}-\psi(\theta_m^*-\wh{\theta}_i))\right| \\
\label{eq:anderson-trick0} &\leq& \left|\sum_{i\in[n]\backslash\{m\}}A_{mi}(\bar{y}_{mi}-\psi(\theta_m^*-\theta_i^*))\right| \\
\label{eq:anderson-trick1} && + \left|\sum_{i\in[n]\backslash\{m\}}A_{mi}(\psi(\theta_m^*-\theta_i^*)-\psi(\theta_m^*-\theta_i^{(m)}+a_m))\right| \\
\label{eq:anderson-trick2} && + \left|\sum_{i\in[n]\backslash\{m\}}A_{mi}(\psi(\theta_m^*-\theta_i^{(m)}+a_m)-\psi(\theta_m^*-\wh{\theta}_i))\right|.
\end{eqnarray}
We use Lemma \ref{lem:A-bern} to bound (\ref{eq:anderson-trick1}). We have
\begin{eqnarray}
\nonumber && \left|\sum_{i\in[n]\backslash\{m\}}A_{mi}(\psi(\theta_m^*-\theta_i^*)-\psi(\theta_m^*-\theta_i^{(m)}+a_m))\right| \\
\label{eq:anderson-trick1-1} &\leq& p\left|\sum_{i\in[n]\backslash\{m\}}(\psi(\theta_m^*-\theta_i^*)-\psi(\theta_m^*-\theta_i^{(m)}+a_m))\right| \\
\label{eq:anderson-trick1-2} && + \left|\sum_{i\in[n]\backslash\{m\}}(A_{mi}-p)(\psi(\theta_m^*-\theta_i^*)-\psi(\theta_m^*-\theta_i^{(m)}+a_m))\right| \\
\nonumber &\leq& p\sqrt{n}\|\theta_{-m}^{(m)}-\theta_{-m}^*-a_m\mathds{1}_{n-1}\| + C_2\log n \|\theta_{-m}^{(m)}-\theta_{-m}^*-a_m\mathds{1}_{n-1}\|_{\infty} \\
\nonumber && + C_2\sqrt{p\log n}\|\theta_{-m}^{(m)}-\theta_{-m}^*-a_m\mathds{1}_{n-1}\| \\
\nonumber &\leq& (p\sqrt{n}+C_2\sqrt{p\log n})\|\theta_{-m}^{(m)}-\theta_{-m}^*-a_m\mathds{1}_{n-1}\| + C_2\log n\|\wh{\theta}-\theta^*\|_{\infty}  \\
\nonumber && + C_2\log n\|\theta_{-m}^{(m)}-\wh{\theta}_{-m}-a_m\mathds{1}_{n-1}\|.
\end{eqnarray}
With the help of \ref{lem:A-basic}, we can also bound (\ref{eq:anderson-trick2}), and we get
\begin{eqnarray}
\nonumber && \left|\sum_{i\in[n]\backslash\{m\}}A_{mi}(\psi(\theta_m^*-\theta_i^{(m)}+a_m)-\psi(\theta_m^*-\wh{\theta}_i))\right| \\
\nonumber &\leq& \sqrt{\sum_{i\in[n]\backslash\{m\}}A_{mi}}\|\theta_{-m}^{(m)}-\wh{\theta}_{-m}-a_m\mathds{1}_{n-1}\| \\
\label{eq:anderson-trick2-bound} &\leq& C_3\sqrt{np}\|\theta_{-m}^{(m)}-\wh{\theta}_{-m}-a_m\mathds{1}_{n-1}\|.
\end{eqnarray}
Plug the bounds into (\ref{eq:fixec}), and we have
\begin{eqnarray*}
\|\wh{\theta}-\theta^*\|_{\infty} &\leq& \frac{\max_{m\in[n]}\left|\sum_{i\in[n]\backslash\{m\}}A_{mi}(\bar{y}_{mi}-\psi(\theta_m^*-\theta_i^*))\right| }{c_2np} \\
&& + \frac{(p\sqrt{n}+C_2\sqrt{p\log n})\max_{m\in[n]}\|\theta_{-m}^{(m)}-\theta_{-m}^*-a_m\mathds{1}_{n-1}\|}{c_2np} \\
&& + \frac{(C_2\log n+C_3\sqrt{np})\|\theta_{-m}^{(m)}-\wh{\theta}_{-m}-a_m\mathds{1}_{n-1}\|}{c_2np} + \frac{C_2\log n\|\wh{\theta}-\theta^*\|_{\infty} }{c_2np}.
\end{eqnarray*}
Since $np\geq c_0\log n$ for some sufficiently large $c_0$, we obtain the bound
\begin{eqnarray}
\nonumber \|\wh{\theta}-\theta^*\|_{\infty} &\leq& C_4\frac{\max_{m\in[n]}\left|\sum_{i\in[n]\backslash\{m\}}A_{mi}(\bar{y}_{mi}-\psi(\theta_m^*-\theta_i^*))\right| }{np} \\
\nonumber && + C_4\frac{p\sqrt{n}\max_{m\in[n]}\|\theta_{-m}^{(m)}-\theta_{-m}^*-a_m\mathds{1}_{n-1}\|}{np} \\
\label{eq:dangerous} && + C_4\frac{(\log n+\sqrt{np})\|\theta_{-m}^{(m)}-\wh{\theta}_{-m}-a_m\mathds{1}_{n-1}\|}{np}.
\end{eqnarray}
Let us plug the above bound into (\ref{eq:interesting}). Then, after some rearrangement, we obtain
\begin{eqnarray*}
\max_{m\in[n]}\|\theta_{-m}^{(m)}-\wh{\theta}_{-m}-a_m\mathds{1}_{n-1}\| &\leq& C_5\frac{\max_{m\in[n]}\sqrt{\sum_{i\in[n]\backslash\{m\}}A_{mi}(\bar{y}_{mi}-\psi(\theta_m^*-\theta_i^*))^2}}{np} \\
&& + C_5\frac{\max_{m\in[n]}\left|\sum_{i\in[n]\backslash\{m\}}A_{mi}(\bar{y}_{mi}-\psi(\theta_m^*-\theta_i^*))\right| }{np\sqrt{np}} \\
&& + C_5\frac{\max_{m\in[n]}\|\theta_{-m}^{(m)}-\theta_{-m}^*-a_m\mathds{1}_{n-1}\|}{n\sqrt{p}}.
\end{eqnarray*}
By Lemma \ref{lem:concentration} and Lemma \ref{lem:global-rate-l-1-o}, we have
\begin{equation}
\max_{m\in[n]}\|\theta_{-m}^{(m)}-\wh{\theta}_{-m}-a_m\mathds{1}_{n-1}\| \leq C_7\sqrt{\frac{1}{npL}}. \label{eq:important-split}
\end{equation}
Now we can plug the bound (\ref{eq:important-split}) back into (\ref{eq:dangerous}), and together with Lemma \ref{lem:concentration} and Lemma \ref{lem:global-rate-l-1-o}, we have
\begin{equation}
\|\wh{\theta}-\theta^*\|_{\infty} \leq C_8\sqrt{\frac{\log n}{npL}}, \label{eq:finally-done}
\end{equation}
which is the desired conclusion. Tracking all the probabilistic events that we have used in the proof, we can conclude that both (\ref{eq:important-split}) and (\ref{eq:finally-done}) hold with probability at least $1-O(n^{-7})$.
\end{proof}

\subsection{Proofs of Theorem \ref{thm:MLE-ranking} and Theorem \ref{thm:MLE-exact}}\label{sec:MLE-pf-thm-rank}

In the proof of (\ref{eq:main-linf}), we have established the byproduct (\ref{eq:important-split}). This bound turns out to be extremely important for us to establish the result of Theorem \ref{thm:MLE-ranking}. We therefore list it, together with its consequence, as a lemma.
\begin{lemma}\label{lem:MLE-split-sharp}
Under the setting of Theorem \ref{thm:MLE-estimation}, there exists some constant $C>0$ such that
$$\max_{m\in[n]}\|\theta_{-m}^{(m)}-\wh{\theta}_{-m}-a_m\mathds{1}_{n-1}\|^2\leq C\frac{1}{npL},$$
$$\max_{m\in[n]}\|\theta_{-m}^{(m)}-\theta^*_{-m}-a_m\mathds{1}_{n-1}\|_{\infty}^2\leq C\frac{\log n}{npL},$$
with probability at least $1-O(n^{-7})$, where $a_m=\ave(\theta_{-m}^{(m)}-\theta_m^*)$ and $\theta_{-m}^{(m)}$ is defined by (\ref{eq:con-MLE-l-m}).
\end{lemma}
\begin{proof}
The first conclusion has been established in (\ref{eq:important-split}). The second conclusion is a consequence of the inequality
$$\max_{m\in[n]}\|\theta_{-m}^{(m)}-\theta^*_{-m}-a_m\mathds{1}_{n-1}\|_{\infty}^2 \leq 2\max_{m\in[n]}\|\theta_{-m}^{(m)}-\wh{\theta}_{-m}-a_m\mathds{1}_{n-1}\|^2 + 2\|\wh{\theta}-\theta^*\|_{\infty}^2,$$
and (\ref{eq:finally-done}).
\end{proof}

Now we are ready to prove Theorem \ref{thm:MLE-ranking}.
\begin{proof}[Proof of Theorem \ref{thm:MLE-ranking}]
When the error exponent is of constant order, the bound is also a constant, and the result already holds since $\h_k(\wh{r},r^*)\leq 1$. Therefore, we only need to consider the case when the error exponent tends to infinity.
We first introduce some notation. Define
\begin{equation}
\eta = \frac{1}{2} - \frac{V(\kappa)}{(1-\bar{\delta})\Delta^2npL}\log\frac{n-k}{k}, \label{eq:eta}
\end{equation}
where $\bar{\delta}=o(1)$ is chosen such that $\eta>0$. The specific choice of $\bar{\delta}$ will be specified in the proof.
Then let
\begin{equation}
\bar{\Delta}_i = \begin{cases}
\min\left(\eta(\theta_k^*-\theta_{k+1}^*)+\theta_i^*-\theta_k^*,\left(\frac{\log n}{np}\right)^{1/4}\right), & 1\leq i\leq k, \\
\min\left((1-\eta)(\theta_k^*-\theta_{k+1}^*)+\theta_{k+1}^*-\theta_i^*,\left(\frac{\log n}{np}\right)^{1/4}\right), & k+1 \leq i\leq n.
\end{cases} \label{eq:delta-bar-i}
\end{equation}
Since the diverging error exponent implies $\snr\rightarrow\infty$, we have $\min_{i\in[n]}\bar{\Delta}_i^2Lnp\rightarrow\infty$ and $\max_{i\in[n]}\bar{\Delta}_i\rightarrow 0$.

The proof involves several steps. In the first step, we need to derive a sharp probabilistic bound for $|f^{(m)}(\theta_m^*|\wh{\theta}_{-m})|$. In the proof of (\ref{eq:main-linf}) of Theorem \ref{thm:MLE-estimation}, we have shown that $|f^{(m)}(\theta_m^*|\wh{\theta}_{-m})|$ can be bounded by the sum of (\ref{eq:anderson-trick0}), (\ref{eq:anderson-trick2}), (\ref{eq:anderson-trick1-1}) and (\ref{eq:anderson-trick1-2}). For (\ref{eq:anderson-trick0}), we can use Hoeffding's inequality and Lemma \ref{lem:A-basic} and obtain the bound
\begin{align}
\left|\sum_{i\in[n]\backslash\{m\}}A_{mi}(\bar{y}_{mi}-\psi(\theta_m^*-\theta_i^*))\right|\leq C_1\sqrt{\frac{x\sum_{i\in[n]\backslash\{m\}}A_{mi}}{L}}\leq C_2\sqrt{\frac{xnp}{L}}, \label{eq:bound-f-0}
\end{align}
with probability at least $1-O(n^{-10})-e^{-x}$. Take $x=\bar{\Delta}_m^{3/2}Lnp$, and we have
\begin{equation}
\left|\sum_{i\in[n]\backslash\{m\}}A_{mi}(\bar{y}_{mi}-\psi(\theta_m^*-\theta_i^*))\right| \leq C_2\sqrt{\bar{\Delta}_m^{3/2}(np)^2}, \label{eq:bound-f-1}
\end{equation}
with probability at least $1-O(n^{-10})-e^{-\bar{\Delta}_m^{3/2}Lnp}$. Since we have already shown (\ref{eq:anderson-trick2}) can be bounded by (\ref{eq:anderson-trick2-bound}) with probability at least $1-O(n^{-10})$, an application of Lemma \ref{lem:MLE-split-sharp} implies that
\begin{equation}
\left|\sum_{i\in[n]\backslash\{m\}}A_{mi}(\psi(\theta_m^*-\theta_i^{(m)}+a_m)-\psi(\theta_m^*-\wh{\theta}_i))\right| \leq C_3\sqrt{\frac{1}{L}}, \label{eq:bound-f-2}
\end{equation}
with probability at least $1-O(n^{-7})$. By Cauchy-Schwarz inequality, we can bound (\ref{eq:anderson-trick1-1}) by $p\sqrt{n}\|\theta^{(m)}_{-m}-\theta^*_{-m}-a_m\mathds{1}_{n-1}\|$. With the help of Lemma \ref{lem:global-rate-l-1-o}, we have
\begin{equation}
p\left|\sum_{i\in[n]\backslash\{m\}}(\psi(\theta_m^*-\theta_i^*)-\psi(\theta_m^*-\theta_i^{(m)}+a_m))\right|\leq C_4\sqrt{\frac{np}{L}}, \label{eq:bound-f-3}
\end{equation}
with probability at least $1-O(n^{-9})$. For (\ref{eq:anderson-trick1-2}), we use Bernstein's inequality, and we have
\begin{align}
& \left|\sum_{i\in[n]\backslash\{m\}}(A_{mi}-p)(\psi(\theta_m^*-\theta_i^*)-\psi(\theta_m^*-\theta_i^{(m)}+a_m))\right| \nonumber \\
&\leq C_5\sqrt{px}\|\theta^{(m)}_{-m}-\theta^*_{-m}-a_m\mathds{1}_{n-1}\| + C_5x\|\theta^{(m)}_{-m}-\theta^*_{-m}-a_m\mathds{1}_{n-1}\|_{\infty}, \label{eq:bound-f-4}
\end{align}
with probability at least $1-e^{-x}$. We choose $x=\min\left(\bar{\Delta}_m^2 Lnp\frac{np}{\log n}, 7\log n\right)$. Then, with the help of Lemma \ref{lem:global-rate-l-1-o} and Lemma \ref{lem:MLE-split-sharp}, we have
\begin{eqnarray}
\nonumber && \left|\sum_{i\in[n]\backslash\{m\}}(A_{mi}-p)(\psi(\theta_m^*-\theta_i^*)-\psi(\theta_m^*-\theta_i^{(m)}+a_m))\right| \\
\label{eq:bound-f-4} &\leq& C_6\frac{1}{\sqrt{L}}\sqrt{\min\left(\bar{\Delta}_m^2 Lnp\frac{np}{\log n}, 7\log n\right)} + C_6\sqrt{\frac{\log n}{npL}}\min\left(\bar{\Delta}_m^2 Lnp\frac{np}{\log n}, 7\log n\right),
\end{eqnarray}
with probability at least $1-O(n^{-7})-\exp\left(-\bar{\Delta}_m^2npL\frac{np}{\log n}\right)$.
Combining the bounds (\ref{eq:bound-f-1})-(\ref{eq:bound-f-4}), we obtain a bound for $|f^{(m)}(\theta_m^*|\wh{\theta}_{-m})|$. This also implies a bound for $|\wh{\theta}_m-\theta_m^*|$ because of the inequality (\ref{eq:fixec}).

In the second step, we define
\begin{align}
\bar{\theta}_m = \theta_m^* - \frac{f^{(m)}(\theta_m^*|\wh{\theta}_{-m})}{g^{(m)}(\theta_m^*|\wh{\theta}_{-m})}.\label{eqn:bar_theta_def}
\end{align}
We need to show $\bar{\theta}_m$ and $\wh{\theta}_m$ are close.
By Proposition \ref{prop:entrywise-MLE}, $\|\wh{\theta}-\theta^*\|_{\infty}\leq 5$, and thus $g^{(m)}(\theta_m^*|\wh{\theta}_{-m})\geq c_1np$ for some constant $c_1>0$, so that we have the bound $|\bar{\theta}_m-\theta_m^*|\leq \frac{|f^{(m)}(\theta_m^*|\wh{\theta}_{-m})|}{c_1np}$. In fact, given the inequality (\ref{eq:fixec}), we can choose $c_1$ to be sufficiently small so that $|\wh{\theta}_m-\theta_m^*|\leq \frac{|f^{(m)}(\theta_m^*|\wh{\theta}_{-m})|}{c_1np}$ is also true. Therefore, we can express $\bar{\theta}_m$ and $\wh{\theta}_m$ as
\begin{eqnarray*}
\bar{\theta}_m &=& \argmin_{|\theta_m-\theta_m^*|\leq \frac{|f^{(m)}(\theta_m^*|\wh{\theta}_{-m})|}{c_1np}}\bar{\ell}_n^{(m)}(\theta_m|\wh{\theta}_{-m}), \\
\wh{\theta}_m &=& \argmin_{|\theta_m-\theta_m^*|\leq \frac{|f^{(m)}(\theta_m^*|\wh{\theta}_{-m})|}{c_1np}}\ell_n^{(m)}(\theta_m|\wh{\theta}_{-m}),
\end{eqnarray*}
where
$$\bar{\ell}_n^{(m)}(\theta_m|\wh{\theta}_{-m}) = \ell_n^{(m)}(\theta_m^*|\wh{\theta}_{-m})+(\theta_m-\theta_m^*)f^{(m)}(\theta_m^*|\wh{\theta}_{-m}) + \frac{1}{2}(\theta_m-\theta_m^*)^2g^{(m)}(\theta_m^*|\wh{\theta}_{-m}).$$
Recall the definition of $\ell_n^{(m)}(\theta_m|{\theta}_{-m})$ in (\ref{eq:l-1-out-decomp}) and the display afterwards. We will show $\bar{\theta}_m$ and $\wh{\theta}_m$ are close by bounding the difference between the two objective functions. By Taylor expansion, we have
$$
\left|\ell_n^{(m)}(\theta_m|\wh{\theta}_{-m})-\bar{\ell}_n^{(m)}(\theta_m|\wh{\theta}_{-m})\right| = \frac{1}{2}(\theta_m-\theta_m^*)^2\left|g^{(m)}(\xi|\wh{\theta}_{-m})-g^{(m)}(\theta_m^*|\wh{\theta}_{-m})\right|,
$$
where $\xi$ is a scalar between $\theta_m$ and $\theta_m^*$. We then have
\begin{eqnarray*}
&& \left|g^{(m)}(\xi|\wh{\theta}_{-m})-g^{(m)}(\theta_m^*|\wh{\theta}_{-m})\right| \\
&=& \left|\sum_{i\in[n]\backslash\{m\}}A_{mi}\psi(\xi-\wh{\theta}_i)\psi(\wh{\theta}_i-\xi)-\sum_{i\in[n]\backslash\{m\}}A_{mi}\psi(\theta_m^*-\wh{\theta}_i)\psi(\wh{\theta}_i-\theta_m^*)\right| \\
&\leq& |\xi-\theta_m^*|\sum_{i\in[n]\backslash\{m\}}A_{mi} \\
&\leq& C_7|\theta_m-\theta_m^*|np,
\end{eqnarray*}
where the last inequality uses Lemma \ref{lem:A-basic}. Therefore, for any $\theta_m$ that satisfies $|\theta_m-\theta_m^*|\leq \frac{|f^{(m)}(\theta_m^*|\wh{\theta}_{-m})|}{c_1np}$, the difference between the two objective functions can be bounded by
$$\left|\ell_n^{(m)}(\theta_m|\wh{\theta}_{-m})-\bar{\ell}_n^{(m)}(\theta_m|\wh{\theta}_{-m})\right| \leq \frac{C_7np}{2}|\theta_m-\theta_m^*|^3\leq \frac{C_7np}{2}\left(\frac{|f^{(m)}(\theta_m^*|\wh{\theta}_{-m})|}{c_1np}\right)^3.$$
By Pythagorean identity, $\bar{\ell}_n^{(m)}(\wh{\theta}_m|\wh{\theta}_{-m})=\bar{\ell}_n^{(m)}(\bar{\theta}_m|\wh{\theta}_{-m})+\frac{1}{2}g^{(m)}(\theta_m^*|\wh{\theta}_{-m})(\wh{\theta}_m-\bar{\theta}_m)^2$. Then,
\begin{eqnarray*}
&& \frac{1}{2}g^{(m)}(\theta_m^*|\wh{\theta}_{-m})(\wh{\theta}_m-\bar{\theta}_m)^2 \\
&=& \bar{\ell}_n^{(m)}(\wh{\theta}_m|\wh{\theta}_{-m}) - \bar{\ell}_n^{(m)}(\bar{\theta}_m|\wh{\theta}_{-m}) \\
&\leq& {\ell}_n^{(m)}(\wh{\theta}_m|\wh{\theta}_{-m}) - \bar{\ell}_n^{(m)}(\bar{\theta}_m|\wh{\theta}_{-m}) +  \frac{C_7np}{2}\left(\frac{|f^{(m)}(\theta_m^*|\wh{\theta}_{-m})|}{c_1np}\right)^3 \\
&\leq& {\ell}_n^{(m)}(\bar{\theta}_m|\wh{\theta}_{-m}) - \bar{\ell}_n^{(m)}(\bar{\theta}_m|\wh{\theta}_{-m}) +  \frac{C_7np}{2}\left(\frac{|f^{(m)}(\theta_m^*|\wh{\theta}_{-m})|}{c_1np}\right)^3 \\
&\leq& 2\frac{C_7np}{2}\left(\frac{|f^{(m)}(\theta_m^*|\wh{\theta}_{-m})|}{c_1np}\right)^3.
\end{eqnarray*}
Since $g^{(m)}(\theta_m^*|\wh{\theta}_{-m})\geq c_1np$, we obtain the bound
\begin{align*}%\label{eqn:MLE_hat_theta_bar_theta_diff}
(\wh{\theta}_m-\bar{\theta}_m)^2 \leq \frac{2C_7}{c_1^4}\left(\frac{|f^{(m)}(\theta_m^*|\wh{\theta}_{-m})|}{np}\right)^3.
\end{align*}
Since $|f^{(m)}(\theta_m^*|\wh{\theta}_{-m})|$ has been shown to be bounded by the sum of (\ref{eq:bound-f-1})-(\ref{eq:bound-f-4}), we have
\begin{equation}
|\wh{\theta}_m-\bar{\theta}_m| \leq \delta\bar{\Delta}_m, \label{eq:MLE-local-approx-1}
\end{equation}
for some $\delta=o(1)$ with probability at least $1-O(n^{-7})-\exp(-\bar{\Delta}_m^{3/2}Lnp)-\exp\left(-\bar{\Delta}_m^2npL\frac{np}{\log n}\right)$ under the condition that $\bar{\Delta}_m=o(1)$ and $\frac{np}{\log n}\rightarrow\infty$.

In the third step, we need to show that $\frac{f^{(m)}(\theta_m^*|\wh{\theta}_{-m})}{g^{(m)}(\theta_m^*|\wh{\theta}_{-m})}$ in the definition of $\bar{\theta}_m$ can be replaced by $\frac{f^{(m)}(\theta_m^*|\theta^*_{-m})}{g^{(m)}(\theta_m^*|\theta^*_{-m})}$ with a negligible error. By triangle inequality, we can bound $|f^{(m)}(\theta_m^*|\wh{\theta}_{-m})-f^{(m)}(\theta_m^*|\theta^*_{-m})|$ by the sum of (\ref{eq:bound-f-2}), (\ref{eq:bound-f-3}) and (\ref{eq:bound-f-4}). Given that $g^{(m)}(\theta_m^*|\theta_{-m}^*)\gtrsim np$, we have
\begin{equation}
\frac{|f^{(m)}(\theta_m^*|\wh{\theta}_{-m})-f^{(m)}(\theta_m^*|\theta^*_{-m})|}{g^{(m)}(\theta_m^*|\theta_{-m}^*)} \leq \delta\bar{\Delta}_m, \label{eq:MLE-local-approx-2}
\end{equation}
for some $\delta=o(1)$ with probability at least $1-O(n^{-7})-\exp\left(-\bar{\Delta}_m^2npL\frac{np}{\log n}\right)$ under the assumption that $npL\bar{\Delta}_m^2\rightarrow\infty$ and $\frac{np}{\log n}\rightarrow\infty$. Note that we can choose the same $\delta$ to accommodate the two bounds (\ref{eq:MLE-local-approx-1}) and (\ref{eq:MLE-local-approx-2}).
We also need to give a sharp approximation to $g^{(m)}(\theta_m^*|\wh{\theta}_{-m})$. We have
\begin{eqnarray*}
&& \left|g^{(m)}(\theta_m^*|\wh{\theta}_{-m})-g^{(m)}(\theta_m^*|\theta_{-m}^*)\right| \\
&\leq& \left|g^{(m)}(\theta_m^*|\wh{\theta}_{-m})-g^{(m)}(\theta_m^*|\theta_{-m}^{(m)}-a_m\mathds{1}_{n-1})\right| + \left|g^{(m)}(\theta_m^*|\theta_{-m}^{(m)}-a_m\mathds{1}_{n-1})-g^{(m)}(\theta_m^*|\theta_{-m}^*)\right| \\
&\leq& \sqrt{\sum_{i\in[n]\backslash\{m\}}A_{mi}}\|\theta_{-m}^{(m)}-a_m\mathds{1}_{n-1}-\wh{\theta}_{-m}\| + p\sqrt{n}\|\theta_{-m}^{(m)}-a_m\mathds{1}_{n-1}-\theta^*\| \\
&& + \sum_{i\in[n]\backslash\{m\}}(A_{mi}-p)\left|\theta_i^{(m)}-a_m-\theta_i^*\right|.
\end{eqnarray*}
By Lemma \ref{lem:A-basic}, Lemma \ref{lem:global-rate-l-1-o} and Lemma \ref{lem:MLE-split-sharp}, the first two terms can be bounded by $C_8\sqrt{\frac{np}{L}}$ with probability at least $1-O(n^{-7})$. To bound the third term, we can use Lemma \ref{lem:A-bern}, and then $\sum_{i\in[n]\backslash\{m\}}(A_{mi}-p)\left|\theta_i^{(m)}-a_m-\theta_i^*\right|$ can be bounded by
$$C_8\sqrt{p\log n}\|\theta_{-m}^{(m)}-a_m\mathds{1}_{n-1}-\theta^*\| + C_8\log n\|\theta_{-m}^{(m)}-a_m\mathds{1}_{n-1}-\theta^*\|_{\infty},$$
with probability at least $1-O(n^{-10})$. By Lemma \ref{lem:global-rate-l-1-o} and Lemma \ref{lem:MLE-split-sharp}, the above display is at most $C_9\sqrt{\frac{\log n}{L}}+C_9\frac{(\log n)^{3/2}}{\sqrt{npL}}$ with probability at least $1-O(n^{-7})$. Combining our bounds, we obtain
\begin{align}\label{eqn:MLE_gm_gm_diff}
\left|g^{(m)}(\theta_m^*|\wh{\theta}_{-m})-g^{(m)}(\theta_m^*|\theta_{-m}^*)\right|\lesssim \sqrt{\frac{np}{L}}+\frac{(\log n)^{3/2}}{\sqrt{npL}}.
\end{align}
Since $g^{(m)}(\theta_m^*|\theta_{-m}^*)\gtrsim np$, we have
\begin{equation}
\frac{\left|g^{(m)}(\theta_m^*|\wh{\theta}_{-m})-g^{(m)}(\theta_m^*|\theta_{-m}^*)\right|}{g^{(m)}(\theta_m^*|\theta_{-m}^*)} \leq \delta, \label{eq:MLE-local-approx-3}
\end{equation}
for some $\delta=o(1)$ with probability at least $1-O(n^{-7})$. Note that we can choose the same $\delta$ to accommodate the three bounds (\ref{eq:MLE-local-approx-1}), (\ref{eq:MLE-local-approx-2}) and (\ref{eq:MLE-local-approx-3}).

In the last step, we will apply Lemma \ref{lem:anderson} with $t=(1-\eta)\theta_k^*+\eta\theta_{k+1}^*$ to finish the proof. Recall the definition of $\eta$ in (\ref{eq:eta}). For any $i\leq k$, we have
\begin{eqnarray}
\nonumber && \mathbb{P}\left(\wh{\theta}_i \leq (1-\eta)\theta_k^*+\eta\theta_{k+1}^*)\right) \\
\nonumber &\leq& \mathbb{P}\left(\wh{\theta}_i-\theta_i^* \leq -\eta(\theta_k^*-\theta_{k+1}^*) - (\theta_i^*-\theta_k^*)\right) \\
\nonumber &\leq& \mathbb{P}\left(\bar{\theta}_i-\theta_i^*\leq -(1-\delta)\bar{\Delta}_i\right) + \mathbb{P}\left(|\bar{\theta}_i-\wh{\theta}_i|>\delta\bar{\Delta}_i\right) \\
\nonumber &\leq& \mathbb{P}\left(-\frac{f^{(i)}(\theta_i^*|{\theta}^*_{-i})}{g^{(i)}(\theta_i^*|{\theta}^*_{-i})} \leq -(1+\delta^2-3\delta)\bar{\Delta}_i\right) + \mathbb{P}\left(|\bar{\theta}_i-\wh{\theta}_i|>\delta\bar{\Delta}_i\right) \\
\nonumber && + \mathbb{P}\left(\frac{\left|g^{(i)}(\theta_i^*|\wh{\theta}_{-i})-g^{(i)}(\theta_i^*|\theta_{-i}^*)\right|}{g^{(i)}(\theta_i^*|\theta_{-i}^*)} > \delta\right) + \mathbb{P}\left(\frac{|f^{(i)}(\theta_i^*|\wh{\theta}_{-i})-f^{(i)}(\theta_i^*|\theta^*_{-i})|}{g^{(i)}(\theta_i^*|\theta_{-i}^*)} > \delta\bar{\Delta}_i\right) \\
\label{eq:left} &\leq& \mathbb{P}\left(-\frac{f^{(i)}(\theta_i^*|{\theta}^*_{-i})}{g^{(i)}(\theta_i^*|{\theta}^*_{-i})} \leq -(1-3\delta)\bar{\Delta}_i\right) + O(n^{-7}) \\
\nonumber && +\exp(-\bar{\Delta}_i^{3/2}Lnp)+\exp\left(-\bar{\Delta}_i^2npL\frac{np}{\log n}\right),
\end{eqnarray}
where the last inequality is due to (\ref{eq:MLE-local-approx-1}), (\ref{eq:MLE-local-approx-2}) and (\ref{eq:MLE-local-approx-3}). Define the event
$$\mathcal{A}_i=\left\{A: \left|\frac{\sum_{j\in[n]\backslash\{i\}}A_{ij}\psi(\theta_i^*-\theta_j^*)\psi(\theta_j^*-\theta_i^*)}{p\sum_{j\in[n]\backslash\{i\}}\psi(\theta_i^*-\theta_j^*)\psi(\theta_j^*-\theta_i^*)}-1\right|\leq\delta\right\}.$$
By Bernstein's inequality, we have $\mathbb{P}(A\in\mathcal{A}_i^c)\leq O(n^{-7})$ for some $\delta=o(1)$. Again, we shall adjust the value of $\delta$ so that (\ref{eq:MLE-local-approx-1}), (\ref{eq:MLE-local-approx-2}) and (\ref{eq:MLE-local-approx-3}) are still true. We then have
\begin{eqnarray}
\nonumber && \mathbb{P}\left(-\frac{f^{(i)}(\theta_i^*|{\theta}^*_{-i})}{g^{(i)}(\theta_i^*|{\theta}^*_{-i})} \leq -(1-3\delta)\bar{\Delta}_i\right) \\
\nonumber &\leq& \sup_{A\in\mathcal{A}_i}\mathbb{P}\left(\frac{\sum_{j\in[n]\backslash\{i\}}A_{ij}(\bar{y}_{ij}-\psi(\theta_i^*-\theta_j^*))}{\sum_{j\in[n]\backslash\{i\}}A_{ij}\psi(\theta_i^*-\theta_j^*)\psi(\theta_j^*-\theta_i^*)} \leq -(1-3\delta)\bar{\Delta}_i\Big|A\right) + \mathbb{P}(A\in\mathcal{A}_i^c) \\
\label{eq:apply-bern} &\leq& \sup_{A\in\mathcal{A}_i}\exp\left(-\frac{\frac{1}{2}(1-3\delta)^2\bar{\Delta}_i^2\left(L\sum_{j\in[n]\backslash\{i\}}A_{ij}\psi'(\theta_i^*-\theta_j^*)\right)^2}{L\sum_{j\in[n]\backslash\{i\}}A_{ij}\psi'(\theta_i^*-\theta_j^*)+\frac{1-3\delta}{3}\bar{\Delta}_iL\sum_{j\in[n]\backslash\{i\}}A_{ij}\psi'(\theta_i^*-\theta_j^*)}\right) \\
\nonumber &&  + O(n^{-7}) \\
\label{eq:apply-bern2} &=& \exp\left(-\frac{1+o(1)}{2}\bar{\Delta}_i^2Lp\sum_{j\in[n]\backslash\{i\}}\psi^\prime(\theta_i^*-\theta_j^*)\right) + O(n^{-7}) \\
\label{eq:free} &\leq& \exp\left(-\frac{1+o(1)}{2}(\eta(\theta_k^*-\theta_{k+1}^*)+(\theta_i^*-\theta_k^*))^2Lp\sum_{j\in[n]\backslash\{i\}}\psi^\prime(\theta_i^*-\theta_j^*)\right) \\
\nonumber && + O(n^{-7}) \\
\label{eq:extra-step} &\leq& \exp\left(-\frac{1+o(1)}{2}(\bar{\Delta}+(\theta_i^*-\theta_k^*))^2Lp\sum_{j\in[n]\backslash\{i\}}\psi^\prime(\theta_i^*-\theta_j^*)\right) + O(n^{-7}).
\end{eqnarray}
The bound (\ref{eq:apply-bern}) is by Bernstein's inequality. We then use the definition of $\mathcal{A}_i$ to obtain the expression (\ref{eq:apply-bern2}). To see why (\ref{eq:free}) is true, note that when $\bar{\Delta}_i^2=\sqrt{\frac{\log n}{np}}$, the first term of (\ref{eq:apply-bern2}) can be absorbed into $O(n^{-7})$. Finally, in (\ref{eq:extra-step}), we have used the notation $\bar{\Delta}=\min\left(\eta(\theta_k^*-\theta_{k+1}^*),\left(\frac{\log n}{np}\right)^{1/4}\right)$. For each $j\in[n]$, define
$$h_j(t)=\left(\bar{\Delta}+t\right)^2\psi'(t+\theta_k^*-\theta_j^*),\quad\text{for all }t\geq 0.$$
The derivative of this function is
$$h_j'(t)=\left(\bar{\Delta}+t\right)\psi'(t+\theta_k^*-\theta_j^*)\left[2+(\bar{\Delta}+t)(1-2\psi(t+\theta_k^*-\theta_j^*))\right].$$
Since $\max_{j,k}|\theta_k^*-\theta_j^*|=O(1)$, we can find a sufficiently small constant $c_2>0$, such that $h_j(t)$ is increasing on $[0,c_2]$. Moreover, there exists another small constant $c_3>0$ such that $\min_{t\in(c_2,\kappa]}h_j(t)\geq c_3$. With this fact, we can bound the exponent of (\ref{eq:extra-step}) as
\begin{eqnarray}
\nonumber && (\bar{\Delta}+(\theta_i^*-\theta_k^*))^2Lp\sum_{j\in[n]\backslash\{i\}}\psi^\prime(\theta_i^*-\theta_j^*)\\
\nonumber &\geq& Lp\sum_{j\in[n]\backslash\{i\}}\min\left(\bar{\Delta}^2\psi'(\theta_k^*-\theta_j^*),c_3\right) \\
\nonumber &\geq& Lp(k-1)\min\left(\bar{\Delta}^2\psi'(\theta_1^*-\theta_{k}^*),c_3\right) + Lp(n-k)\min\left(\bar{\Delta}^2\psi'(\theta_{k}^*-\theta_n^*),c_3\right) \\
\label{eq:drop-min} &=& Lp\bar{\Delta}^2\left((k-1)\psi'(\theta_1^*-\theta_{k}^*)+(n-k)\psi'(\theta_{k}^*-\theta_n^*)\right) \\
\nonumber &\geq& (1+o(1))Lp\min\left(\eta^2\Delta^2,\sqrt{\frac{\log n}{np}}\right)\frac{n}{V(\kappa)}
\end{eqnarray}
where the equality (\ref{eq:drop-min}) uses the fact that $\bar{\Delta}\rightarrow 0$. Therefore, we can further bound (\ref{eq:extra-step}) as
\begin{eqnarray*}
&& \exp\left(-\frac{1+o(1)}{2}Lp\min\left(\eta^2\Delta^2,\sqrt{\frac{\log n}{np}}\right)\frac{n}{V(\kappa)}\right) + O(n^{-7}) \\
&\leq& \exp\left(-\frac{(1+o(1))\eta^2\Delta^2npL}{2V(\kappa)}\right) + O(n^{-7}).
\end{eqnarray*}
The last inequality holds because when $\min\left(\eta^2\Delta^2,\sqrt{\frac{\log n}{np}}\right)=\sqrt{\frac{\log n}{np}}$, the first term becomes $\exp\left(-\frac{(1+o(1))L\sqrt{np\log n}}{2V(\kappa)}\right)$, which can be absorbed by $O(n^{-7})$. Since $\exp(-\bar{\Delta}_i^{3/2}Lnp)+\exp\left(-\bar{\Delta}_i^2npL\frac{np}{\log n}\right)\leq \exp\left(-\frac{(1+o(1))\eta^2\Delta^2npL}{2V(\kappa)}\right) + O(n^{-7})$, we have
\begin{equation}
\mathbb{P}\left(\wh{\theta}_i \leq (1-\eta)\theta_k^*+\eta\theta_{k+1}^*)\right)\leq \exp\left(-\frac{(1-\delta')\eta^2\Delta^2npL}{2V(\kappa)}\right) + O(n^{-7}), \label{eq:prob-bound-1}
\end{equation}
with some $\delta'=o(1)$ for all $i\leq k$.
With a similar argument, we also have
\begin{equation}
\mathbb{P}\left(\wh{\theta}_i \geq (1-\eta)\theta_k^*+\eta\theta_{k+1}^*)\right)\leq \exp\left(-\frac{(1-\delta')(1-\eta)^2\Delta^2npL}{2V(\kappa)}\right) + O(n^{-7}), \label{eq:prob-bound-2}
\end{equation}
for all all $i\geq k+1$. It can be checked that the $\delta'$ above is independent of the $\bar{\delta}$ in the definition of $\eta$. Now we can choose $\eta$ as in (\ref{eq:eta}) with $\bar{\delta}=\delta'$. By Lemma \ref{lem:anderson}, we have
\begin{eqnarray*}
\mathbb{E}\h_k(\wh{r},r^*) &\leq& \exp\left(-\frac{(1-\bar{\delta})\eta^2\Delta^2npL}{2V(\kappa)}\right) + \frac{n-k}{k}\exp\left(-\frac{(1-\bar{\delta})(1-\eta)^2\Delta^2npL}{2V(\kappa)}\right) + O(n^{-7}) \\
&\leq& 2\exp\left(-\frac{1}{2}\left(\frac{\sqrt{(1-\bar{\delta})\snr}}{2}-\frac{1}{\sqrt{(1-\bar{\delta})\snr}}\log\frac{n-k}{k}\right)^2\right) + O(n^{-7}).
\end{eqnarray*}
By Markov's inequality, the above bound implies
$$\h_k(\wh{r},r^*)\leq \exp\left(-\frac{1}{2}\left(\frac{\sqrt{(1-\delta_1)\snr}}{2}-\frac{1}{\sqrt{(1-\delta_1)\snr}}\log\frac{n-k}{k}\right)^2\right) + O(n^{-6}),$$
for some $\delta_1=o(1)$ with high probability. One can take, for example,
$$\delta_1=\bar{\delta}+\frac{1}{\frac{\sqrt{(1-\bar{\delta})\snr}}{2}-\frac{1}{\sqrt{(1-\bar{\delta})\snr}}\log\frac{n-k}{k}}.$$
When $O(n^{-6})$ dominates the bound, we have $\h_k(\wh{r},r^*)=O(n^{-6})$, which implies $\h_k(\wh{r},r^*)=0$ since $\h_k(\wh{r},r^*)\in\{0,(2k)^{-1}, 2(2k)^{-1}, 3(2k)^{-1},\cdots, 1\}$. Therefore, we always have
$$\h_k(\wh{r},r^*)\leq 2\exp\left(-\frac{1}{2}\left(\frac{\sqrt{(1-\delta_1)\snr}}{2}-\frac{1}{\sqrt{(1-\delta_1)\snr}}\log\frac{n-k}{k}\right)^2\right),$$
with high probability for some $\delta_1=o(1)$. The proof is complete.
\end{proof}

\begin{proof}[Proof of Theorem \ref{thm:MLE-exact}]
With some rearrangements, the condition is equivalent to
$$\frac{npL\Delta^2}{2(1+\epsilon)V(\kappa)}\left(\frac{1}{2}-\frac{(1+\epsilon)V(\kappa)}{npL\Delta^2}\log\frac{n-k}{k}\right)^2 > \log k.$$
Since $\epsilon$ is a constant, it implies
$$\frac{npL\Delta^2}{2V(\kappa)}\left(\frac{1}{2}-\frac{V(\kappa)}{(1-\delta)npL\Delta^2}\log\frac{n-k}{k}\right)^2 > (1+\epsilon)\log k,$$
for any ${\delta}=o(1)$. Therefore, $\h_k(\wh{r},r^*)=o(k^{-1})$ when $k\rightarrow\infty$. Given the fact that $\h_k(\wh{r},r^*)\in\{0,(2k)^{-1}, 2(2k)^{-1}, 3(2k)^{-1},\cdots, 1\}$, we must have $\h_k(\wh{r},r^*)=0$. When $k=O(1)$, the condition implies $\frac{npL\Delta^2}{2V(\kappa)}>(1+\epsilon')\log n$ for some constant $\epsilon'>0$. This leads to the fact that $\left(\frac{1}{2}-\frac{V(\kappa)}{(1-\delta)npL\Delta^2}\log\frac{n-k}{k}\right)^2>c_1$ for some constant $c_1>0$. Therefore, $\h_k(\wh{r},r^*)=o(1)=o(k^{-1})$, which implies $\h_k(\wh{r},r^*)=0$.
\end{proof}

\section{Analysis of the Spectral Method}\label{sec:proof-spec}

We prove results for the spectral method in this section. This includes Theorem \ref{thm:spectral-ranking}, Theorem \ref{thm:spectral-exact} and Theorem \ref{thm:spectral_lower}. The proofs of Theorem \ref{thm:spectral-ranking} and Theorem \ref{thm:spectral-exact} are given in Section \ref{sec:pf-spec1}, and then we prove Theorem \ref{thm:spectral_lower} in Section \ref{sec:pf-spec2}.

\subsection{Proofs of Theorem \ref{thm:spectral-ranking} and Theorem \ref{thm:spectral-exact}} \label{sec:pf-spec1}

The proof of Theorem \ref{thm:spectral-ranking} relies on a leave-one-out argument introduced by \cite{chen2019spectral}. Without loss of generality, we consider $r_i^*=i$ so that $\theta_{r_i^*}^*=\theta_i^*$. Following \cite{chen2019spectral}, we define a transition matrix $P^{(m)}$ for each $m\in[n]$. For any $i\neq j$, $P_{ij}^{(m)}=P_{ij}$ if $i\neq m$ and $j\neq m$ and otherwise $P_{ij}^{(m)}=\frac{p}{d}\psi(\theta_i^*-\theta_j^*)$. For any $i\in[n]$, $P_{ii}^{(m)}=\sum_{j\in[n]\backslash\{i\}}P_{ij}^{(m)}$. Let $\pi^{(m)}$ be the stationary distribution of $P^{(m)}$. The following $\ell_{2}$ norm bound has essentially been proved in \cite{chen2019spectral}.

\begin{lemma}\label{lem:yuxin-chen}
Under the setting of Theorem \ref{thm:spectral-ranking}, there exists a constant $C>0$ such that
$$\max_{m\in[n]}\|\pi^{(m)}-\wh{\pi}\|\leq C\frac{1}{n}\sqrt{\frac{\log n}{npL}},$$
$$\max_{m\in[n]}\|\pi^{(m)}-\pi^*\|_{\infty}\leq C\frac{1}{n}\sqrt{\frac{\log n}{npL}},$$
$$\max_{m\in[n]}\|\pi^{(m)}-\pi^*\|\leq C\frac{1}{n}\sqrt{\frac{1}{pL}},$$
with probability at least $1-O(n^{-4})$.
\end{lemma}
\begin{proof}
By Lemma 5.6 and Lemma 5.7 of \cite{chen2019spectral}, one can obtain $\|\pi^{(m)}-\wh{\pi}\|\leq C_1\sqrt{\frac{\log n}{npL}}\|\pi^*\|_{\infty}+\|\wh{\pi}-\pi^*\|_{\infty}$ for some constant $C_1>0$ with probability at least $1-O(n^{-5})$. Theorem 2.6 of \cite{chen2019spectral} gives the bound $\|\wh{\pi}-\pi^*\|_{\infty}\leq C_2\sqrt{\frac{\log n}{npL}}\|\pi^*\|_{\infty}$ with probability at least $1-O(n^{-5})$. A union bound argument together with the fact that $\|\pi^*\|_{\infty}\asymp n^{-1}$ leads to the first conclusion. The second conclusion is a consequence of triangle inequality. By Theorem 5.2 of \cite{chen2019spectral}, we have $\|\wh{\pi}-\pi^*\|\leq C_3\frac{1}{n}\sqrt{\frac{1}{pL}}$ with probability at least $1-O(n^{-1})$. Thus, we obtain the last conclusion by applying triangle inequality again.
\end{proof}

We also need a lemma that relates the asymptotic variance of $\wh{\pi}_i$ to the function $\overline{V}(\kappa)$.
\begin{lemma}\label{lem:minimizer}
For any positive $\kappa_1,\kappa_2=O(1)$, we have
\begin{eqnarray*}
&& \min_{\substack{x_1,...,x_k\in[0,\kappa_1]\\x_{k+1},...,x_{n}\in[0,\kappa_2]}}\frac{(\sum_{i=1}^k\psi(x_i)+\sum_{i=k+1}^n\psi(-x_i))^2}{\sum_{i=1}^k\psi^\prime(x_i)(1+e^{x_i})^2+\sum_{i=k+1}^n\psi^\prime(x_i)(1+e^{-x_i})^2}\\
&=& \frac{(k\psi(\kappa_1)+(n-k)\psi(-\kappa_2))^2}{k\psi^\prime(\kappa_1)(1+e^{\kappa_1})^2+(n-k)\psi^\prime(\kappa_2)(1+e^{-\kappa_2})^2},
\end{eqnarray*}
for $n$ that is sufficiently large.
\end{lemma}
\begin{proof}
The problem is equivalent to the solution of the following: the optimum of the problem
$$\min_{\substack{x_1,...,x_k\in[1,M_1]\\x_{k+1},...,x_{n}\in[1,M_2]}}\frac{(\sum_{i=1}^k\frac{2x_i}{1+x_i}+\sum_{i=k+1}^n\frac{2}{1+x_i})^2}{\sum_{i=1}^kx_i+\sum_{i=k+1}^n\frac{1}{x_i}} = \min_{\substack{x_1,...,x_k\in[1,M_1]\\x_{k+1},...,x_{n}\in[1,M_2]}}f(x_1,\cdots,x_n)$$
is obtained at $x_1=...=x_k=M_1,x_{k+1}=...=x_n=M_2$. We will show that for any given $x_{k+1},...,x_n\in[1,M_2]$, the function is minimized at $x_1=...=x_k=M_1$. Moreover, for any given $x_1,...,x_k$, the function is minimized at $x_{k+1}=...=x_n=M_2$. We only need to prove the former claim and the latter one can be proved similarly.
Define
$$g(x_1,\cdots,x_k)=\frac{\left(\sum_{i=1}^k\frac{2x_i}{1+x_i}+\alpha\right)^2}{\sum_{i=1}^kx_i+\beta},$$
where $\alpha=\sum_{i=k+1}^n\frac{2}{1+x_i}, \beta=\sum_{i=k+1}^n\frac{1}{x_i}$. We first analyze the behavior of $g(x_1,\cdots,x_k)$ at each coordinate. By direct calculation, we have
\begin{eqnarray*}
\frac{\partial \log g(x_1,\cdots,x_k)}{\partial x_1} &=& \frac{4}{(1+x_1)^2(\sum_{i=1}^k\frac{2x_i}{1+x_i}+\alpha)}-\frac{1}{\sum_{i=1}^kx_i+\beta} \\
&=& \frac{4(\sum_{i=1}^kx_i+\beta)-(1+x_1)^2(\sum_{i=1}^k\frac{2x_i}{1+x_i}+\alpha)}{(1+x_1)^2(\sum_{i=1}^k\frac{2x_i}{1+x_i}+\alpha)(\sum_{i=1}^kx_i+\beta)}.
\end{eqnarray*}
The sign of the partial derivative is determined by its numerator
\begin{eqnarray*}
&& 4(\sum_{i=1}^kx_i+\beta)-(1+x_1)^2(\sum_{i=1}^k\frac{2x_i}{1+x_i}+\alpha) \\
&=& -\left(\sum_{i=2}^k\frac{2x_i}{1+x_i}+\alpha+2\right)x_1^2-\left(\sum_{i=2}^k\frac{4x_i}{1+x_i}+2\alpha-2\right)x_1 \\
&& +4(\sum_{i=2}^kx_i+\beta)-\left(\sum_{i=2}^k\frac{2x_i}{1+x_i}+\alpha\right),
\end{eqnarray*}
which is a quadratic decreasing function of $x_1\in[1,M_1]$. Therefore, $g(x_1,\cdots,x_k)$ is either monotone of $x_1\in[1,M_1]$, or it is first increasing then decreasing. This implies that the optimum is achieved either at $x_1=1$ or $x_1=M_1$. Since $g(x_1,\cdots,x_k)$ is symmetric, we therefore know that the optimizer must satisfy $(x_1,\cdots,x_k)\in \{1,M_1\}^k$. Using symmetry again, we can conclude that the value of $\min_{x_1,\cdots,x_k\in[1,M_1]}g(x_1,\cdots,x_k)$ is determined by the number of coordinates that take $M_1$. For $i\in[k]$, we define $g_i$ to be the value of $g(x_1,\cdots,x_k)$ with $x_1=\cdots=x_i=M_1$ and $x_{i+1}=\cdots=x_k=1$. We now need to show $g_i$ is nonincreasing in $i\in[k]$.
Note that
\begin{eqnarray}
\nonumber g_i\geq g_{i+1} &\Longleftrightarrow& \frac{(i\frac{2M_1}{M_1+1}+k-i+\alpha)^2}{iM_1+k-i+\beta}\geq\frac{(\frac{M_1-1}{M_1+1}+i\frac{2M_1}{M_1+1}+k-i+\alpha)^2}{M_1-1+iM_1+k-i+\beta}\\
\nonumber &\Longleftrightarrow& \frac{M_1-1}{iM_1+k-i+\beta}\geq\frac{(\frac{M_1-1}{M_1+1})^2}{(i\frac{2M_1}{M_1+1}+k-i+\alpha)^2}+\frac{2(\frac{M_1-1}{M_1+1})}{i\frac{2M_1}{M_1+1}+k-i+\alpha}\\
\nonumber &\Longleftrightarrow& \frac{(M_1+1)^2}{iM_1+k-i+\beta}-\frac{M_1-1}{(i\frac{2M_1}{M_1+1}+k-i+\alpha)^2}-\frac{2(M_1+1)}{i\frac{2M_1}{M_1+1}+k-i+\alpha}\geq0\\
\nonumber &\Longleftrightarrow& \frac{(i\frac{M_1-1}{M_1+1}+k+\alpha)(M_1+1)^2}{i(M_1-1)+k+\beta}-\frac{M_1-1}{i\frac{M_1-1}{M_1+1}+k+\alpha}-2(M_1+1)\geq0\\
\nonumber &\Longleftrightarrow& \frac{i(M_1-1)+(k+\alpha)(M_1+1)}{i(M_1-1)+k+\beta}-\frac{M_1-1}{i(M_1-1)+(k+\alpha)(M_1+1)}-2\geq0\\
\label{eq:pinhan-chen} &\Longleftarrow& \frac{i(M_1-1)+(k+\beta)(M_1+1)}{i(M_1-1)+k+\beta}-\frac{M_1-1}{i(M_1-1)+(k+\beta)(M_1+1)}-2\geq0\\
\nonumber &\Longleftrightarrow& \frac{-i(M_1-1)+(k+\beta)(M_1-1)}{i(M_1-1)+k+\beta}-\frac{M_1-1}{i(M_1-1)+(k+\beta)(M_1+1)}\geq0\\
\nonumber &\Longleftarrow& \frac{-i+(k+\beta)}{i(M_1-1)+k+\beta}-\frac{1}{i(M_1-1)+(k+\beta)(M_1+1)}\geq0\\
\nonumber &\Longleftrightarrow& (k+\beta)^2(M_1+1)\geq i(M_1-1)+i^2(M_1-1)+(2i+1)(k+\beta)\\
\nonumber &\Longleftarrow& (k+\beta)^2(M_1+1)\geq(k-1)^2(M_1-1)+(k-1)(M_1-1)+(2k-1)(k+\beta)\\
\nonumber &\Longleftrightarrow&  k^2(M_1+1)+2\beta(M_1+1)k+\beta^2(M_1+1)\geq k^2(M_1+1)+(-M_1+2\beta)k-\beta\\
\nonumber &\Longleftrightarrow&  (2\beta+1)M_1k+\beta^2(M_1+1)+\beta\geq0
\end{eqnarray}
where the last display is trivially true. We have used $\alpha\geq\beta$ for the step (\ref{eq:pinhan-chen}). Therefore, $\min_{x_1,\cdots,x_k\in[1,M_1]}g(x_1,\cdots,x_k)=g_k$, and the proof is complete.
\end{proof}

Now we are ready to prove Theorem \ref{thm:spectral-ranking}.

\begin{proof}[Proof of Theorem \ref{thm:spectral-ranking}]
When the error exponent is of constant order, the bound is also a constant, and the result already holds since $\h_k(\wh{r},r^*)\leq 1$. Therefore, we only need to consider the case when the error exponent tends to infinity. We first introduce some notation. 
Define
\begin{equation}
\eta = \frac{1}{2} - \frac{\overline{V}(\kappa)}{(1-\bar{\delta})\Delta^2npL}\log\frac{n-k}{k}, \label{eq:eta-spec}
\end{equation}
where $\bar{\delta}=o(1)$ is chosen so that $\eta>0$ is satisfied. The specific choice of $\bar{\delta}$ will be determined later in the proof.
We will continue to use the notation $\bar{\Delta}_i$ that is defined in (\ref{eq:delta-bar-i}).
Since the diverging exponent implies $\overline{\snr}\rightarrow\infty$, we have $\min_{i\in[n]}\bar{\Delta}_i^2Lnp\rightarrow\infty$ and $\max_{i\in[n]}\bar{\Delta}_i\rightarrow 0$.

Since $\wh{\pi}$ is the stationary distribution of $P$, we have $\hat{\pi}^TP=\hat{\pi}^T$. This implies that for any $m\in[n]$, we have $\sum_{j=1}^nP_{jm}\wh{\pi}_j = \wh{\pi}_m$. We can equivalently write this identity as
$$\wh{\pi}_m=\frac{\sum_{j\in[n]\backslash\{m\}}P_{jm}\wh{\pi}_j}{1-P_{mm}}=\frac{\sum_{j\in[n]\backslash\{m\}}A_{jm}\bar{y}_{mj}\wh{\pi}_j}{\sum_{j\in[n]\backslash\{m\}}A_{jm}\bar{y}_{jm}}.$$
We approximate $\wh{\pi}_m$ by
\begin{align}
\bar{\pi}_m=\frac{\sum_{j\in[n]\backslash\{m\}}A_{jm}\bar{y}_{mj}{\pi}^*_j}{\sum_{j\in[n]\backslash\{m\}}A_{jm}\bar{y}_{jm}}.\label{eqn:pi_bar_def}
\end{align}
%$$\bar{\pi}_m=\frac{\sum_{j\in[n]\backslash\{m\}}A_{jm}\bar{y}_{mj}{\pi}^*_j}{\sum_{j\in[n]\backslash\{m\}}A_{jm}\bar{y}_{jm}}.$$
The approximation error can be bounded by
\begin{eqnarray}
\label{eq:spec-l-a} \left|\wh{\pi}_m-\bar{\pi}_m\right| &\leq& \left|\frac{\sum_{j\in[n]\backslash\{m\}}A_{jm}\bar{y}_{mj}(\wh{\pi}_j-\pi_j^{(m)})}{\sum_{j\in[n]\backslash\{m\}}A_{jm}\bar{y}_{jm}}\right| \\
\label{eq:spec-l-e} && + \left|\frac{\sum_{j\in[n]\backslash\{m\}}A_{jm}\bar{y}_{mj}(\pi_j^{(m)}-\pi_j^*)}{\sum_{j\in[n]\backslash\{m\}}A_{jm}\bar{y}_{jm}}\right|.
\end{eqnarray}
The two terms (\ref{eq:spec-l-a}) and (\ref{eq:spec-l-e}) share a common denominator, which can be lower bounded by
\begin{equation}
\sum_{j\in[n]\backslash\{m\}}A_{jm}\bar{y}_{jm}\geq \sum_{j\in[n]\backslash\{m\}}A_{jm}\psi(\theta_j^*-\theta_m^*)-\left|\sum_{j\in[n]\backslash\{m\}}A_{jm}(\bar{y}_{jm}-\psi(\theta_j^*-\theta_m^*))\right|. \label{eq:lower-denom-pi}
\end{equation}
By Lemma \ref{lem:A-basic} and Lemma \ref{lem:concentration}, we have $\sum_{j\in[n]\backslash\{m\}}A_{jm}\bar{y}_{jm}\geq c_1np$ for some constant $c_1>0$ with probability at least $1-O(n^{-10})$. With this lower bound, we then bound (\ref{eq:spec-l-a}) as
\begin{eqnarray*}
\left|\frac{\sum_{j\in[n]\backslash\{m\}}A_{jm}\bar{y}_{mj}(\wh{\pi}_j-\pi_j^{(m)})}{\sum_{j\in[n]\backslash\{m\}}A_{jm}\bar{y}_{jm}}\right| &\leq& \frac{\sqrt{\sum_{j\in[n]\backslash\{m\}}A_{1j}\bar{y}_{mj}^2}\|\wh{\pi}-\pi^{(m)}\|}{c_1np} \\
&\leq& \frac{\sqrt{\sum_{j\in[n]\backslash\{m\}}A_{1j}}\|\wh{\pi}-\pi^{(m)}\|}{c_1np} \\
&\leq& C_1\frac{1}{n}\sqrt{\frac{\log n}{(np)^2L}},
\end{eqnarray*}
with probability at least $1-O(n^{-4})$. In the last inequality, we have used Lemma \ref{lem:A-basic} and Lemma \ref{lem:yuxin-chen}. For (\ref{eq:spec-l-e}), we can bound it as
\begin{eqnarray*}
&& \left|\frac{\sum_{j\in[n]\backslash\{m\}}A_{jm}\bar{y}_{mj}(\pi_j^{(m)}-\pi_j^*)}{\sum_{j\in[n]\backslash\{m\}}A_{jm}\bar{y}_{jm}}\right| \\
&\leq& \frac{\left|\sum_{j\in[n]\backslash\{m\}}A_{jm}(\bar{y}_{mj}-\psi(\theta_m^*-\theta_j^*))(\pi_j^{(m)}-\pi_j^*)\right|}{c_1np} + \frac{p\left|\sum_{j\in[n]\backslash\{m\}}\psi(\theta_m^*-\theta_j^*)(\pi_j^{(m)}-\pi_j^*)\right|}{c_1np} \\
&& + \frac{\left|\sum_{j\in[n]\backslash\{m\}}(A_{jm}-p)\psi(\theta_m^*-\theta_j^*)(\pi_j^{(m)}-\pi_j^*)\right|}{c_1np}.
\end{eqnarray*}
We bound the three terms above separately. For the first term, we use Hoeffding's inequality (Lemma \ref{lem:hoeffding}), and get
\begin{align}\label{eq:spec-add-x-1}
\frac{\left|\sum_{j\in[n]\backslash\{m\}}A_{jm}(\bar{y}_{mj}-\psi(\theta_m^*-\theta_j^*))(\pi_j^{(m)}-\pi_j^*)\right|}{c_1np}\leq C_2\frac{\sqrt{\frac{x}{L}\sum_{j\in[n]\backslash\{m\}}A_{jm}(\pi_j^{(m)}-\pi_j^*)^2}}{np},
\end{align}
with probability at least $1-e^{-x}$. By Lemma \ref{lem:A-basic} and Lemma \ref{lem:yuxin-chen}, we have
$$\sqrt{\sum_{j\in[n]\backslash\{m\}}A_{jm}(\pi_j^{(m)}-\pi_j^*)^2}\leq \|\pi^{(m)}-\pi^*\|_{\infty}\sqrt{\sum_{j\in[n]\backslash\{m\}}A_{jm}}\leq C_3\frac{1}{n}\sqrt{\frac{\log n}{L}},$$
with probability at least $1-O(n^{-4})$. Taking $x=\bar{\Delta}_m^2npL\sqrt{\frac{npL}{\log n}}$, we have
$$\frac{\left|\sum_{j\in[n]\backslash\{m\}}A_{jm}(\bar{y}_{mj}-\psi(\theta_m^*-\theta_j^*))(\pi_j^{(m)}-\pi_j^*)\right|}{c_1np}\leq C_4\frac{1}{n}\bar{\Delta}_m\left(\frac{\log n}{Lnp}\right)^{1/4},$$
with probability at least $1-O(n^{-4})-\exp\left(-\bar{\Delta}_m^2npL\sqrt{\frac{npL}{\log n}}\right)$. Next, for the second term, we apply Lemma \ref{lem:yuxin-chen} and get
\begin{align*}
\frac{p\left|\sum_{j\in[n]\backslash\{m\}}\psi(\theta_m^*-\theta_j^*)(\pi_j^{(m)}-\pi_j^*)\right|}{c_1np} \leq \frac{\|\pi^{(m)}-\pi^*\|}{c_1\sqrt{n}} \leq C_5\frac{1}{n}\sqrt{\frac{1}{npL}},
\end{align*}
with probability at least $1-O(n^{-4})$. For the third term, we use Bernstein's inequality (Lemma \ref{lem:bernstein}), and get
\begin{align}\label{eq:spec-add-x-2}
\frac{\left|\sum_{j\in[n]\backslash\{m\}}(A_{jm}-p)\psi(\theta_m^*-\theta_j^*)(\pi_j^{(m)}-\pi_j^*)\right|}{c_1np} \leq C_6\frac{\sqrt{px}\|\pi^{(m)}-\pi^*\|}{np} + C_6\frac{x\|\pi^{(m)}-\pi^*\|_{\infty}}{np},
\end{align}
with probability at least $1-e^{-x}$. We choose $x=\min\left(\bar{\Delta}_m^2 Lnp\frac{np}{\log n}, 4\log n\right)$. Then, with the help of Lemma \ref{lem:yuxin-chen}, we have
\begin{align}
& \frac{\left|\sum_{j\in[n]\backslash\{m\}}(A_{jm}-p)\psi(\theta_m^*-\theta_j^*)(\pi_j^{(m)}-\pi_j^*)\right|}{c_1np} \nonumber \\
&\leq C_7\frac{1}{n}\frac{1}{np\sqrt{L}}\sqrt{\min\left(\bar{\Delta}_m^2 Lnp\frac{np}{\log n}, \log n\right)} + C_7\frac{1}{n}\frac{1}{np}\sqrt{\frac{\log n}{npL}}\min\left(\bar{\Delta}_m^2 Lnp\frac{np}{\log n}, \log n\right),\label{eqn:spectral_equationone}
\end{align}
with probability at least $1-O(n^{-4})-\exp\left(-\bar{\Delta}_m^2npL\frac{np}{\log n}\right)$.

To summarize, we have proved that
\begin{equation}
\frac{|\wh{\pi}_m-\bar{\pi}_m|}{\pi_m^*} \leq \delta(1-e^{-\bar{\Delta}_m}), \label{eq:pi-hat-bar-approx}
\end{equation}
for some $\delta=o(1)$ with probability at least $1-O(n^{-4})-\exp\left(-\bar{\Delta}_m^2npL\frac{np}{\log n}\right)-\exp\left(-\bar{\Delta}_m^2npL\sqrt{\frac{npL}{\log n}}\right)$ under the assumption that $\bar{\Delta}_m=o(1)$, $npL\bar{\Delta}_m^2\rightarrow\infty$ and $\frac{np}{\log n}\rightarrow\infty$.

Next, we note that by the definition of $\bar{\pi}_m$, we have
\begin{equation}
\bar{\pi}_m - \pi_m^* = \frac{\sum_{j\in[n]\backslash\{m\}}A_{jm}(\bar{y}_{mj}-\psi(\theta_m^*-\theta_j^*))({\pi}^*_j+\pi_m^*)}{\sum_{j\in[n]\backslash\{m\}}A_{jm}\bar{y}_{jm}}. \label{eq:pi-bar-m}
\end{equation}
By Lemma \ref{lem:concentration} and the inequality (\ref{eq:lower-denom-pi}), the denominator of (\ref{eq:pi-bar-m}) satisfies
\begin{equation}
\left|\frac{\sum_{j\in[n]\backslash\{m\}}A_{jm}\bar{y}_{jm}}{\sum_{j\in[n]\backslash\{m\}}A_{jm}\psi(\theta_j^*-\theta_m^*)}-1\right| \leq \delta, \label{eq:easy-de-pi}
\end{equation}
for some $\delta=o(1)$ with probability at least $1-O(n^{-10})$. Note that we can choose the same $\delta$ to accommodate both bounds (\ref{eq:pi-hat-bar-approx}) and (\ref{eq:easy-de-pi}).

We will apply Lemma \ref{lem:anderson} with 
\begin{align}
t=\frac{e^{(1-\eta)\theta_k^*+\eta\theta_{k+1}^*}}{\sum_{j=1}^ne^{\theta_j^*}} \label{eqn:spectral_t}
\end{align}
%$t=\frac{e^{(1-\eta)\theta_k^*+\eta\theta_{k+1}^*}}{\sum_{j=1}^ne^{\theta_j^*}}$ 
to finish the proof. Recall the definition of $\eta$ in (\ref{eq:eta-spec}).
For $i\leq k$, we have
\begin{eqnarray}
\nonumber && \mathbb{P}\left(\wh{\pi}_i \leq \frac{e^{(1-\eta)\theta_k^*+\eta\theta_{k+1}^*}}{\sum_{j=1}^ne^{\theta_j^*}}\right) \\
\nonumber &=& \mathbb{P}\left(\frac{\wh{\pi}_i-\pi_i^*}{\pi_i^*} \leq e^{(1-\eta)\theta_k^*+\eta\theta_{k+1}^*-\theta_i^*}-1\right) \\
\nonumber &\leq& \mathbb{P}\left(\frac{\wh{\pi}_i-\pi_i^*}{\pi_i^*} \leq e^{-\bar{\Delta}_i}-1\right) \\
\nonumber &\leq& \mathbb{P}\left(\frac{\bar{\pi}_i-\pi_i^*}{\pi_i^*}\leq -(1-\delta)(1-e^{-\bar{\Delta}_i})\right) + \mathbb{P}\left(\frac{|\bar{\pi}_i-\wh{\pi}_i|}{\pi_i^*}>\delta(1-e^{-\bar{\Delta}_i})\right) \\
\nonumber &\leq& \mathbb{P}\left(\frac{\sum_{j\in[n]\backslash\{i\}}A_{ji}(\bar{y}_{ij}-\psi(\theta_i^*-\theta_j^*))(1+e^{\theta_j^*-\theta_i^*})}{\sum_{j\in[n]\backslash\{i\}}A_{ji}\psi(\theta_j^*-\theta_i^*)} \leq -(1-\delta)^2(1-e^{-\bar{\Delta}_i})\right) \\
\nonumber &&  + \mathbb{P}\left(\frac{|\bar{\pi}_i-\wh{\pi}_i|}{\pi_i^*}>\delta(1-e^{-\bar{\Delta}_i})\right) + \mathbb{P}\left(\left|\frac{\sum_{j\in[n]\backslash\{i\}}A_{ji}\bar{y}_{ji}}{\sum_{j\in[n]\backslash\{i\}}A_{ji}\psi(\theta_j^*-\theta_i^*)}-1\right| > \delta\right) \\
\nonumber &\leq& \mathbb{P}\left(\frac{\sum_{j\in[n]\backslash\{i\}}A_{ji}(\bar{y}_{ij}-\psi(\theta_i^*-\theta_j^*))(1+e^{\theta_j^*-\theta_i^*})}{\sum_{j\in[n]\backslash\{i\}}A_{ji}\psi(\theta_j^*-\theta_i^*)} \leq -(1-\delta)^2(1-e^{-\bar{\Delta}_i})\right) \\
&& + O(n^{-4})+\exp\left(-\bar{\Delta}_i^2npL\frac{np}{\log n}\right)+\exp\left(-\bar{\Delta}_i^2npL\sqrt{\frac{npL}{\log n}}\right), \label{eqn:spectral_equationtwo}
\end{eqnarray}
where the last inequality is by (\ref{eq:pi-hat-bar-approx}) and (\ref{eq:easy-de-pi}). Define the event
\begin{align}
\mathcal{A}_i=\left\{A: \left|\frac{\sum_{j\in[n]\backslash\{i\}}A_{ij}\psi'(\theta_i^*-\theta_j^*)\left(1+e^{\theta_j^*-\theta_i^*}\right)^2}{p\sum_{j\in[n]\backslash\{i\}}\psi'(\theta_i^*-\theta_j^*)\left(1+e^{\theta_j^*-\theta_i^*}\right)^2}-1\right|\leq\delta, \left|\frac{\sum_{j\in[n]\backslash\{i\}}A_{ji}\psi(\theta_j^*-\theta_i^*)}{p\sum_{j\in[n]\backslash\{i\}}\psi(\theta_j^*-\theta_i^*)}-1\right|\leq \delta\right\}.\label{eqn:A_i}
\end{align}
%$$\mathcal{A}_i=\left\{A: \left|\frac{\sum_{j\in[n]\backslash\{i\}}A_{ij}\psi'(\theta_i^*-\theta_j^*)\left(1+e^{\theta_j^*-\theta_i^*}\right)^2}{p\sum_{j\in[n]\backslash\{i\}}\psi'(\theta_i^*-\theta_j^*)\left(1+e^{\theta_j^*-\theta_i^*}\right)^2}-1\right|\leq\delta, \left|\frac{\sum_{j\in[n]\backslash\{i\}}A_{ji}\psi(\theta_j^*-\theta_i^*)}{p\sum_{j\in[n]\backslash\{i\}}\psi(\theta_j^*-\theta_i^*)}-1\right|\delta\right\}.$$
Then, by Bernstein's inequality, we have
\begin{eqnarray}
\nonumber && \mathbb{P}\left(\frac{\sum_{j\in[n]\backslash\{i\}}A_{ji}(\bar{y}_{ij}-\psi(\theta_i^*-\theta_j^*))(1+e^{\theta_j^*-\theta_i^*})}{\sum_{j\in[n]\backslash\{i\}}A_{ji}\psi(\theta_j^*-\theta_i^*)} \leq -(1-\delta)^2(1-e^{-\bar{\Delta}_i})\right) \\
\nonumber &\leq& \sup_{A\in\mathcal{A}_i}\mathbb{P}\left(\frac{\sum_{j\in[n]\backslash\{i\}}A_{ji}(\bar{y}_{ij}-\psi(\theta_i^*-\theta_j^*))(1+e^{\theta_j^*-\theta_i^*})}{\sum_{j\in[n]\backslash\{i\}}A_{ji}\psi(\theta_j^*-\theta_i^*)} \leq -(1-\delta)^2(1-e^{-\bar{\Delta}_i})\Big|A\right) \\
\nonumber && + \mathbb{P}(A\in\mathcal{A}_i^c) \\
\label{eqn:spectral_equationthree} &\leq& \exp\left(-\frac{(1-o(1))Lp\bar{\Delta}_i^2\left(\sum_{j\in[n]\backslash\{i\}}\psi(\theta_j^*-\theta_i^*)\right)^2}{2\sum_{j\in[n]\backslash\{i\}}\psi'(\theta_i^*-\theta_j^*)\left(1+e^{\theta_j^*-\theta_i^*}\right)^2}\right) + O(n^{-4}) \\
\label{eq:use-before-MLE} &\leq& \exp\left(-\frac{(1-o(1))Lp(\bar{\Delta}+\theta_i^*-\theta_k^*)^2\left(\sum_{j\in[n]\backslash\{i\}}\psi(\theta_j^*-\theta_i^*)\right)^2}{2\sum_{j\in[n]\backslash\{i\}}\psi'(\theta_i^*-\theta_j^*)\left(1+e^{\theta_j^*-\theta_i^*}\right)^2}\right) + O(n^{-4}).
\end{eqnarray}
The inequality (\ref{eq:use-before-MLE}) is by the same argument that leads to (\ref{eq:free}) and (\ref{eq:extra-step}). We use the notation $\bar{\Delta}=\min\left(\eta(\theta_k^*-\theta_{k+1}^*),\left(\frac{\log n}{np}\right)^{1/4}\right)$ in (\ref{eq:use-before-MLE}).
Define
$$h_i(t)=\frac{(\bar{\Delta}+t)^2\left(\sum_{j\in[n]\backslash\{i\}}\psi(\theta_j^*-\theta_k^*-t)\right)^2}{\sum_{j\in[n]\backslash\{i\}}\psi'(t+\theta_k^*-\theta_j^*)\left(1+e^{\theta_j^*-\theta_k^*-t}\right)^2},\quad\text{ for all }t\geq 0.$$
Though $h_i(t)$ is a complicated function, by the fact that $\bar{\Delta}=o(1)$ and $\max_{j,k}|\theta_j^*-\theta_k^*|\leq \kappa=O(1)$, one can directly analyze the derivative of $h_i(t)$ to conclude that there exists some small constant $c_2>0$ such that $h_i(t)$ is increasing on $[0,c_2]$. Moreover, there also exists a small constant $c_3>0$ such that $\min_{t\in[c_2,\kappa]}h_i(t)\geq c_3n$. This implies
\begin{eqnarray*}
&& \frac{Lp(\bar{\Delta}+\theta_i^*-\theta_k^*)^2\left(\sum_{j\in[n]\backslash\{i\}}\psi(\theta_j^*-\theta_i^*)\right)^2}{2\sum_{j\in[n]\backslash\{i\}}\psi'(\theta_i^*-\theta_j^*)\left(1+e^{\theta_j^*-\theta_i^*}\right)^2} \\
&\geq& \frac{Lp\bar{\Delta}^2\left(\sum_{j\in[n]\backslash\{i\}}\psi(\theta_j^*-\theta_k^*)\right)^2}{2\sum_{j\in[n]\backslash\{i\}}\psi'(\theta_k^*-\theta_j^*)\left(1+e^{\theta_j^*-\theta_k^*}\right)^2}\wedge \frac{c_3npL}{2} \\
&=& \frac{Lp\bar{\Delta}^2\left(\sum_{j\in[n]\backslash\{i\}}\psi(\theta_j^*-\theta_k^*)\right)^2}{2\sum_{j\in[n]\backslash\{i\}}\psi'(\theta_k^*-\theta_j^*)\left(1+e^{\theta_j^*-\theta_k^*}\right)^2},
\end{eqnarray*}
where the last inequality is due to the fact that $\bar{\Delta}=o(1)$. We further bound the above exponent by
\begin{eqnarray}
\nonumber && \frac{Lp\bar{\Delta}^2\left(\sum_{j\in[n]\backslash\{i\}}\psi(\theta_j^*-\theta_k^*)\right)^2}{2\sum_{j\in[n]\backslash\{i\}}\psi'(\theta_k^*-\theta_j^*)\left(1+e^{\theta_j^*-\theta_k^*}\right)^2} \\
\nonumber &=& (1-o(1))\frac{Lp\bar{\Delta}^2\left(\sum_{j=1}^n\psi(\theta_j^*-\theta_k^*)\right)^2}{2\sum_{j=1}^n\psi'(\theta_k^*-\theta_j^*)\left(1+e^{\theta_j^*-\theta_k^*}\right)^2} \\
\nonumber &\geq& (1-o(1))\frac{Lp\bar{\Delta}^2}{2} \min_{\substack{\kappa_1+\kappa_2\leq \kappa\\ \kappa_1,\kappa_2\geq 0}}\min_{\substack{x_1,\cdots,x_k\in[0,\kappa_1]\\x_{k+1},\cdots,x_n\in[0,\kappa_2]}}\frac{\left(\sum_{j=1}^k\psi(x_j) + \sum_{j=k+1}^n\psi(-x_j)\right)^2}{\sum_{j=1}^k\psi'(x_j)(1+e^{x_j})^2 + \sum_{j=k+1}^n\psi'(x_j)(1+e^{-x_j})^2} \\
\label{eq:piece-constant} &=& (1-o(1))\frac{Lp\bar{\Delta}^2}{2} \min_{\substack{\kappa_1+\kappa_2\leq \kappa\\ \kappa_1,\kappa_2\geq 0}}\frac{\left(k\psi(\kappa_1)+(n-k)\psi(-\kappa_2)\right)^2}{k\psi'(\kappa_1)(1+e^{\kappa_1})^2+(n-k)\psi'(\kappa_2)(1+e^{-\kappa_2})^2} \\
\nonumber &=& (1-o(1))\frac{Lpn\bar{\Delta}^2}{2\overline{V}(\kappa)}.
\end{eqnarray}
The equality (\ref{eq:piece-constant}) is due to Lemma \ref{lem:minimizer}.
With the above analysis of the error exponent, we can further bound (\ref{eq:use-before-MLE}) as
\begin{eqnarray*}
&& \exp\left(-\frac{1-o(1)}{2}Lp\min\left(\eta^2\Delta^2,\sqrt{\frac{\log n}{np}}\right)\frac{n}{\overline{V}(\kappa)}\right) + O(n^{-4}) \\
&\leq& \exp\left(-\frac{(1-o(1))\eta^2\Delta^2npL}{2\overline{V}(\kappa)}\right) + O(n^{-4}).
\end{eqnarray*}
The last inequality holds because when $\min\left(\eta^2\Delta^2,\sqrt{\frac{\log n}{np}}\right)=\sqrt{\frac{\log n}{np}}$, the first term becomes $\exp\left(-\frac{(1-o(1))L\sqrt{np\log n}}{2\overline{V}(\kappa)}\right)$, which can be absorbed by $O(n^{-4})$. Since $\exp\left(-\bar{\Delta}_i^2npL\frac{np}{\log n}\right)+\exp\left(-\bar{\Delta}_i^2npL\sqrt{\frac{npL}{\log n}}\right)\leq \exp\left(-\frac{(1-o(1))\eta^2\Delta^2npL}{2\overline{V}(\kappa)}\right) + O(n^{-4})$, we have
\begin{equation}
\mathbb{P}\left(\wh{\pi}_i \leq \frac{e^{(1-\eta)\theta_k^*+\eta\theta_{k+1}^*}}{\sum_{j=1}^ne^{\theta_j^*}}\right)\leq \exp\left(-\frac{(1-\delta_1)\eta^2\Delta^2npL}{2\overline{V}(\kappa)}\right) + O(n^{-4}), \label{eq:prob-bound-1-spec}
\end{equation}
with some $\delta_1=o(1)$ for all $i\leq k$.
With a similar argument, we also have
\begin{equation}
\mathbb{P}\left(\wh{\pi}_i \geq \frac{e^{(1-\eta)\theta_k^*+\eta\theta_{k+1}^*}}{\sum_{j=1}^ne^{\theta_j^*}}\right) \leq \exp\left(-\frac{(1-\delta_1)(1-\eta)^2\Delta^2npL}{2\overline{V}(\kappa)}\right) + O(n^{-4}), \label{eq:prob-bound-2-spec}
\end{equation}
for all all $i\geq k+1$. It can be checked that the $\delta_1$ above can be set independent of the $\bar{\delta}$ in the definition of $\eta$. Now we choose $\eta$ as in (\ref{eq:eta-spec}) with $\bar{\delta}=\delta_1$. By Lemma \ref{lem:anderson}, we have
\begin{eqnarray*}
\mathbb{E}\h_k(\wh{r},r^*) &\leq& \exp\left(-\frac{(1-\bar{\delta})\eta^2\Delta^2npL}{2\overline{V}(\kappa)}\right) + \frac{n-k}{k}\exp\left(-\frac{(1-\bar{\delta})(1-\eta)^2\Delta^2npL}{2\overline{V}(\kappa)}\right) + O(n^{-4}) \\
&\leq& 2\exp\left(-\frac{1}{2}\left(\frac{\sqrt{(1-\bar{\delta})\overline{\snr}}}{2}-\frac{1}{\sqrt{(1-\bar{\delta})\overline{\snr}}}\log\frac{n-k}{k}\right)^2\right) + O(n^{-4}).
\end{eqnarray*}
By Markov's inequality, the above bound implies
$$\h_k(\wh{r},r^*)\leq \exp\left(-\frac{1}{2}\left(\frac{\sqrt{(1-\delta')\overline{\snr}}}{2}-\frac{1}{\sqrt{(1-\delta')\overline{\snr}}}\log\frac{n-k}{k}\right)^2\right) + O(n^{-3}),$$
for some $\delta'=o(1)$ with high probability. One can take, for example,
$$\delta'=\bar{\delta}+\frac{1}{\frac{\sqrt{(1-\bar{\delta})\overline{\snr}}}{2}-\frac{1}{\sqrt{(1-\bar{\delta})\overline{\snr}}}\log\frac{n-k}{k}}.$$
When $O(n^{-3})$ dominates the bound, we have $\h_k(\wh{r},r^*)=O(n^{-3})$, which implies $\h_k(\wh{r},r^*)=0$ since $\h_k(\wh{r},r^*)\in\{0,(2k)^{-1}, 2(2k)^{-1}, 3(2k)^{-1},\cdots, 1\}$. Therefore, we always have
$$\h_k(\wh{r},r^*)\leq 2\exp\left(-\frac{1}{2}\left(\frac{\sqrt{(1-\delta')\overline{\snr}}}{2}-\frac{1}{\sqrt{(1-\delta')\overline{\snr}}}\log\frac{n-k}{k}\right)^2\right),$$
with high probability with some $\delta'=o(1)$. The proof is complete.
\end{proof}

\begin{proof}[Proof of Theorem \ref{thm:spectral-exact}]
The proof is the same as that of Theorem \ref{thm:MLE-exact}.
\end{proof}

\input{anderson}

\section{Proofs of Lower Bounds}\label{sec:pf-minmax-l}

This section collects the proofs of lower bound results of the paper. The lower bound for exact recovery is proved in Section \ref{sec:pf-exact-lower}, and the partial recovery lower bound is proved in Section \ref{sec:partial-lower-pf}.

\subsection{Proof of Theorem \ref{thm:exact-lower}} \label{sec:pf-exact-lower}

The key mathematical argument in the proof of Theorem \ref{thm:exact-lower} is to characterize the maximum of dependent binomial random variables. For this purpose, we need a high-dimensional central limit theorem result by \cite{chernozhukov2013gaussian}. The following lemma is adapted from \cite{chernozhukov2013gaussian} for our purpose.

\begin{lemma} \label{lem:CCK}
Consider independent random vectors $X_1,\cdots,X_n\in\mathbb{R}^d$ with mean zero. Assume there exist constants $c_1,c_2,C_1,C_2>0$ such that $\min_{i,j}\mathbb{E}X_{ij}^2\geq c_1$, $\max_{i,j}\mathbb{E}\exp(|X_{ij}|/C_1)\leq 2$ and $(\log(nd))^7\leq C_2n^{-(1+c_2)}$. Then, there exist independent Gaussian vectors $Z_1,\cdots, Z_n$ satisfying $\mathbb{E}Z_i=0$ and $\Cov(Z_i)=\Cov(X_i)$, such that
$$\sup_{t\in\mathbb{R}}\left|\mathbb{P}\left(\max_{j\in[d]}\sum_{i=1}^nX_{ij}\leq t\right)-\mathbb{P}\left(\max_{j\in[d]}\sum_{i=1}^nZ_{ij}\leq t\right)\right|\leq Cn^{-c},$$
for some constants $c,C>0$ only depending on $c_1,c_2,C_1,C_2$.
\end{lemma}

With the above Gaussian approximation, we only need to analyze the maximum of dependent Gaussian random variables. The following lemma can be found in \cite{hartigan2014bounding}.

\begin{lemma} \label{lem:hartigan}
Consider $Z=(Z_1,\cdots,Z_n)^T\sim N(0,\Sigma)$. Then, for any $\alpha\in(0,1)$, there exists some constant $C_{\alpha}>0$ such that for all $n\geq\sqrt{2\pi}e^3\log1/\alpha$,
$$\mathbb{P}\left(\max_{i\in[n]}Z_i > \lambda^{1/2}\sqrt{2\log n-\log\log n-C_{\alpha}}-\Lambda^{1/2}\Phi^{-1}(1-\alpha)\right) \geq 1-2\alpha,$$
where $\lambda=\min_{i\in[n]}\Sigma_{ii}-\frac{\max_{i\in[n]}\sum_{j\in[n]\backslash\{i\}}\Sigma_{ij}^2}{\lambda_{\min}(\Sigma)}$ and $\Lambda=\max_{i\in[n]}\Sigma_{ii}$.
\end{lemma}

Now we are ready to prove Theorem \ref{thm:exact-lower}.

\begin{proof}[Proof of Theorem \ref{thm:exact-lower}]
We first note that the condition (\ref{eq:exact-threshold-lower}) implies that $\Delta=o(1)$. Choose $\kappa_1,\kappa_2\geq 0$ such that we have both $\kappa_1+\kappa_2\leq \kappa$ and
$$\frac{n}{k\psi'(\kappa_1)+(n-k)\psi'(\kappa_2)}=V(\kappa).$$
We first consider the case $k\rightarrow\infty$ and $\kappa=\Omega(1)$. In this case, one can easily check that $\kappa_2=\Omega(1)$.
Our least favorable $\theta^*\in\Theta(k,\Delta,\kappa)$ is constructed as follows. Let $\rho=o(1)$ be a vanishing number that will be specified later. Define $\theta_i^*=\kappa_1$ for all $1\leq i\leq k-\rho k$, $\theta_i^*=0$ for $k-\rho k<i\leq k$, $\theta_i^*=-\Delta$ for $k<i\leq k+\rho(n-k)$ and $\theta_i^*=-\kappa_2$ for $k+\rho(n-k)< i\leq n$. For the simplicity of proof, we choose $\rho$ so that both $\rho k$ and $\rho(n-k)$ are integers. Consider a subset $\mathcal{R}_{k,\rho}\subset\S_n$ that is defined by
\begin{equation}
\mathcal{R}_{k,\rho}=\left\{r\in\S_n: r_i=i\text{ for all }i\leq k-\rho k\text{ or }i > k+\rho(n-k)\right\}.\label{eq:subset-R-in-pf}
\end{equation}
We then have the lower bound
$$
\inf_{\wh{r}}\sup_{\substack{r^*\in\S_n\\\theta^*\in\Theta(k,\Delta,\kappa)}}\mathbb{P}_{(\theta^*,r^*)}\left(\h_k(\wh{r},r^*)>0\right) \geq \inf_{\wh{r}}\sup_{r^*\in\mathcal{R}_{k,\rho}}\mathbb{P}_{(\theta^*,r^*)}\left(\h_k(\wh{r},r^*)>0\right).
$$
For each $z=\{z_i\}_{k-\rho k<i\leq k+\rho(n-k)}\in\{0,1\}^{\rho n}$, we define $\mathbb{Q}_{z}$ as a joint probability of the observations $\{A_{ij}\}$ and $\{y_{ijl}\}$. To sample data from $\mathbb{Q}_{z}$, we first sample $A\sim\mathcal{G}(n,p)$, and then for any $(i,j)$ such that $A_{ij}=1$, sample $y_{ijl}\sim \text{Bernoulli}(\psi(\mu_i(z)-\mu_j(z)))$ independently for $l\in[L]$. The vector $\mu(z)$ is defined by $\mu_i(z)=\theta_i^*$ for all $i\leq k-\rho k$ or $i> \rho(n-k)$ and $\mu_i(z)=\Delta\indc{z_i=1}$ for all $k-\rho k<i\leq k+\rho(n-k)$. Then, we have
\begin{eqnarray*}
\inf_{\wh{r}}\sup_{r^*\in\mathcal{R}_{k,\rho}}\mathbb{P}_{(\theta^*,r^*)}\left(\h_k(\wh{r},r^*)>0\right) &\geq& \inf_{\wh{z}}\sup_{z^*\in\mathcal{Z}_k}\mathbb{Q}_{z^*}\left(\wh{z}\neq z^*\right)\\
&\geq& \inf_{\wh{z}}\frac{1}{|\mathcal{Z}_k|}\sum_{z^*\in\mathcal{Z}_k}\mathbb{Q}_{z^*}(\wh{z}\neq z^*),
\end{eqnarray*}
where
$$\mathcal{Z}_k=\left\{z=\{z_i\}_{k-\rho k<i\leq k+\rho(n-k)}\in\{0,1\}^{\rho n}: \sum_i z_i=\rho k\right\}.$$
The Bayes risk $\frac{1}{|\mathcal{Z}_k|}\sum_{z^*\in\mathcal{Z}_k}\mathbb{Q}_{z^*}(\wh{z}\neq z^*)$ is minimized by
\begin{equation}
\wh{z}=\argmin_{z\in\mathcal{Z}_k}\ell_n(\mu(z)), \label{eq:MLE-oracle}
\end{equation}
where
$$\ell_n(\mu(z))=\sum_{1\leq i<j\leq n}A_{ij}\left[\bar{y}_{ij}\log\frac{1}{\psi(\mu_i(z)-\mu_j(z))}+(1-\bar{y}_{ij})\log\frac{1}{1-\psi(\mu_i(z)-\mu_j(z))}\right].$$
It suffices to lower bound the probability $\mathbb{Q}_{z^*}(\wh{z}\neq z^*)$ for the estimator (\ref{eq:MLE-oracle}) and for each $z^*\in\mathcal{Z}_k$. By symmetry, the value of $\mathbb{Q}_{z^*}(\wh{z}\neq z^*)$ is the same for any $z^*\in\mathcal{Z}_k$. We therefore can set $z_i^*=\indc{i\leq k}$ without loss of generality. Define
$$\mathcal{N}(z^*)=\left\{z\in\mathcal{Z}_k: \sum_i\indc{z_i\neq z_i^*}=2\right\}.$$
Then, we have
$$\mathbb{Q}_{z^*}(\wh{z}\neq z^*) \geq \mathbb{Q}_{z^*}\left(\min_{z\in\mathcal{N}(z^*)}\ell_n(\mu(z))<\ell_n(\mu(z^*))\right).$$
By direct calculation, we have
\begin{eqnarray*}
&& \ell_n(\mu(z)) - \ell_n(\mu(z^*)) \\
&=& \sum_{1\leq i<j\leq n}A_{ij}(\bar{y}_{ij}-\psi(\mu_i(z^*)-\mu_j(z^*)))(\mu_i(z^*)-\mu_j(z^*)-\mu_i(z)+\mu_j(z)) \\
&& + \sum_{1\leq i<j\leq n}A_{ij}D\left(\psi(\mu_i(z^*)-\mu_j(z^*))\|\psi(\mu_i(z)-\mu_j(z))\right).
\end{eqnarray*}
For any $z\in\mathcal{N}(z^*)$, there exists some $k-\rho k<a\leq k$ and some $k<b\leq k+\rho(n-k)$ such that $z_a=0$, $z_b=1$ and $z_i=z_i^*$ for all other $i$'s. Then,
\begin{eqnarray}
\nonumber && \sum_{1\leq i<j\leq n}A_{ij}D\left(\psi(\mu_i(z^*)-\mu_j(z^*))\|\psi(\mu_i(z)-\mu_j(z))\right) \\
\nonumber &\leq& \sum_{i=1}^{k-\rho k}A_{ia}D(\psi(\kappa_1)\|\psi(\kappa_1+\Delta)) + \sum_{i=k+\rho(n-k)+1}^nA_{ia}D(\psi(-\kappa_2)\|\psi(-\kappa_2+\Delta)) \\
\nonumber&& +\sum_{i=1}^{k-\rho k}A_{ib}D(\psi(\kappa_1+\Delta)\|\psi(\kappa_1))  + \sum_{i=k+\rho(n-k)+1}^nA_{ib}D(\psi(-\kappa_2+\Delta)\|\psi(-\kappa_2)) \\
\nonumber && + \sum_{i=k-\rho k+1}^kA_{ia}D(\psi(0)\|\psi(\Delta)) + \sum_{i=k+1}^{k+\rho(n-k)}A_{ia}D(\psi(-\Delta)\|\psi(0))\\
\nonumber && + \sum_{i=k-\rho k+1}^kA_{ib}D(\psi(\Delta)\|\psi(0)) + \sum_{i=k+1}^{k+\rho(n-k)}A_{ib}D(\psi(0)\|\psi(-\Delta)) + A_{ab}D(\psi(\Delta)\|\psi(-\Delta))\\
\label{eq:bern-lower-exact} &\leq& (1+\delta)(1-\rho)p\left[kD(\psi(\kappa_1)\|\psi(\kappa_1+\Delta)) + (n-k)D(\psi(-\kappa_2)\|\psi(-\kappa_2+\Delta))\right] \\
\nonumber && + (1+\delta)(1-\rho)p\left[kD(\psi(\kappa_1+\Delta)\|\psi(\kappa_1)) + (n-k)D(\psi(-\kappa_2+\Delta)\|\psi(-\kappa_2))\right] \\
\nonumber && +(1+\delta)\rho p\left[kD(\psi(0)\|\psi(\Delta)) + (n-k)D(\psi(-\Delta)\|\psi(0))\right] \\
\nonumber && +(1+\delta)\rho p\left[kD(\psi(\Delta)\|\psi(0)) + (n-k)D(\psi(0)\|\psi(-\Delta))\right] + (1+\delta)pD(\psi(\Delta)\|\psi(-\Delta))\\
\label{eq:lower-exact-taylor} &\leq& (1+\delta)^2(1-\rho)p\Delta^2\left[k\psi'(\kappa_1) + (n-k)\psi'(\kappa_2)\right] + (1+\delta)^2\rho p\Delta^2\frac{n}{4}\\
\label{eq:rho-absorb} &\leq& (1+\delta)^3p\Delta^2\frac{n}{V(\kappa)}.
\end{eqnarray}
The inequality (\ref{eq:bern-lower-exact}) holds with probability at least $1-O(n^{-10})$ by Bernstein's inequality. The inequality (\ref{eq:lower-exact-taylor}) is a Taylor expansion argument with the help of $\Delta=o(1)$. We obtain (\ref{eq:rho-absorb}) by the choice that $\rho=o(1)$. Note that we can choose some $\delta=o(1)$ to make all of (\ref{eq:bern-lower-exact}), (\ref{eq:lower-exact-taylor}) and (\ref{eq:rho-absorb}) hold. We also have
\begin{eqnarray*}
&& \sum_{1\leq i<j\leq n}A_{ij}(\bar{y}_{ij}-\psi(\mu_i(z^*)-\mu_j(z^*)))(\mu_i(z^*)-\mu_j(z^*)-\mu_i(z)+\mu_j(z)) \\
&=& -\Delta\sum_{i\in[n]\backslash\{a\}}A_{ia}(\bar{y}_{ia}-\mathbb{E}\bar{y}_{ia}) + \Delta\sum_{i\in[n]\backslash\{b\}}A_{ib}(\bar{y}_{ib}-\mathbb{E}y_{ib}).
\end{eqnarray*}
Therefore,
\begin{eqnarray*}
&& \min_{z\in\mathcal{N}(z^*)}\ell_n(\mu(z)) - \ell_n(\mu(z^*)) \\
&\leq& -\max_{(1-\rho)k<a\leq k}\Delta\sum_{i\in[n]\backslash\{a\}}A_{ia}(\bar{y}_{ia}-\mathbb{E}\bar{y}_{ia}) + \Delta\min_{k<b\leq k+\rho(n-k)}\sum_{i\in[n]\backslash\{b\}}A_{ib}(\bar{y}_{ib}-\mathbb{E}y_{ib}) \\
&& + (1+\delta)^3p\Delta^2\frac{n}{V(\kappa)},
\end{eqnarray*}
with probability at least $1-O(n^{-10})$.
This leads to the bound
\begin{eqnarray}
\nonumber && \mathbb{Q}_{z^*}\left(\min_{z\in\mathcal{N}(z^*)}\ell_n(\mu(z))<\ell_n(\mu(z^*))\right) \\
\nonumber &\geq& \mathbb{Q}_{z^*}\Bigg(\max_{(1-\rho)k<a\leq k}\sum_{i\in[n]\backslash\{a\}}A_{ia}(\bar{y}_{ia}-\mathbb{E}\bar{y}_{ia}) \\
\nonumber && - \min_{k<b\leq k+\rho(n-k)}\sum_{i\in[n]\backslash\{b\}}A_{ib}(\bar{y}_{ib}-\mathbb{E}y_{ib}) > (1+\delta)^3p\Delta\frac{n}{V(\kappa)}\Bigg) - O(n^{-10}) \\
\label{eq:nightwish} &\geq& \mathbb{Q}_{z^*}\Bigg(\max_{(1-\rho)k<a\leq k}\sum_{i\in[n]\backslash\{a\}}A_{ia}(\bar{y}_{ia}-\mathbb{E}\bar{y}_{ia}) - \min_{k<b\leq k+\rho(n-k)}\sum_{i\in[n]\backslash\{b\}}A_{ib}(\bar{y}_{ib}-\mathbb{E}y_{ib}) \\
\nonumber && > \sqrt{2(1-\epsilon/2)}\sqrt{\frac{np}{LV(\kappa)}}\left(\sqrt{\log k}+\sqrt{\log (n-k)}\right)\Bigg) - O(n^{-10}) \\
\label{eq:endless} &\geq& \mathbb{Q}_{z^*}\left(\max_{(1-\rho)k<a\leq k}\sum_{i\in[n]\backslash\{a\}}A_{ia}(\bar{y}_{ia}-\mathbb{E}\bar{y}_{ia}) > \sqrt{2(1-\epsilon/2)}\sqrt{\frac{np}{LV(\kappa)}}\sqrt{\log k}\right) \\
\nonumber && + \mathbb{Q}_{z^*}\left(- \min_{k<b\leq k+\rho(n-k)}\sum_{i\in[n]\backslash\{b\}}A_{ib}(\bar{y}_{ib}-\mathbb{E}y_{ib}) > \sqrt{2(1-\epsilon/2)}\sqrt{\frac{np}{LV(\kappa)}}\sqrt{\log (n-k)}\right) \\
\nonumber && - 1 - O(n^{-10}),
\end{eqnarray}
where we have used the condition of the theorem to derive (\ref{eq:nightwish}). The last inequality (\ref{eq:endless}) is by union bound $\mathbb{P}(A\cap B)\geq \mathbb{P}(A)+\mathbb{P}(B)-1$. To lower bound (\ref{eq:endless}), we introduce the notation
$$T_a=\sum_{i\in[n]\backslash\{a\}}A_{ia}(\bar{y}_{ia}-\mathbb{E}\bar{y}_{ia}),\quad (1-\rho)k<a\leq k.$$
The covariance structure of $\{T_a\}_{(1-\rho)k<a\leq k}$ can be quantified by the matrix $\Sigma\in\mathbb{R}^{(\rho k)\times (\rho k)}$, which is defined by $\Sigma_{ab}=\Cov(T_a,T_b|A)$. We then construct a vector $S=\{S_a\}_{(1-\rho)k<a\leq k}$ that is jointly Gaussian conditioning on $A$. The conditional covariance of $S$ is also $\Sigma$. By Lemma \ref{lem:CCK}, we have
\begin{eqnarray}
\label{eq:Q-prob} && \mathbb{Q}_{z^*}\left(\max_{(1-\rho)k<a\leq k}\sum_{i\in[n]\backslash\{a\}}A_{ia}(\bar{y}_{ia}-\mathbb{E}\bar{y}_{ia}) > \sqrt{2(1-\epsilon/2)}\sqrt{\frac{np}{LV(\kappa)}}\sqrt{\log k}\right) \\
\label{eq:apply-CCK} &\geq& \mathbb{P}\left(\max_{(1-\rho)k<a\leq k}S_a > \sqrt{2(1-\epsilon/2)}\sqrt{\frac{np}{LV(\kappa)}}\sqrt{\log k}\right) - O\left(\frac{1}{(\log n)^c}\right).
\end{eqnarray}
To see how Lemma \ref{lem:CCK} implies (\ref{eq:apply-CCK}), we can take $X_{la}=\frac{1}{\sqrt{np}}\sum_{i\in[n]\backslash\{a\}}A_{ia}(y_{ial}-\mathbb{E}y_{ial})$. Conditioning on $A$, we observe that $\{X_{la}\}$ is independent across $l\in[L]$. The conditional variance of $X_{la}$ given $A$ is bounded away from zero with high probability by Lemma \ref{lem:A-basic}. Moreover, one can find a constant $C>0$, such that $\mathbb{E}\left[\exp(|X_{la}|/C)\big|A\right]\leq 2$ by Hoeffding's inequality. Then, we can apply Lemma \ref{lem:CCK} for a given $A$ and obtain (\ref{eq:apply-CCK}) under the condition $L>(\log n)^8$.
We need Lemma \ref{lem:hartigan} to lower bound the probability in (\ref{eq:apply-CCK}). For each $a$,
\begin{eqnarray*}
\Sigma_{aa} &=& \Var(T_a|A) \\
&=& \frac{1}{L}\sum_{i\in[n]\backslash\{a\}}A_{ia}\psi'(\mu_i(z^*)-\mu_a(z^*)) \\
&=& \frac{\psi'(\kappa_1)}{L}\sum_{i=1}^{k-\rho k}A_{ia} + \frac{1}{4L}\sum_{i=k-\rho k+1}^kA_{ia} + \frac{\psi'(\kappa_2)}{L}\sum_{i=k+1}^{k+\rho(n-k)}A_{ia} + \frac{\psi'(\Delta)}{L}\sum_{i=k+\rho(n-k)+1}^nA_{ia}.
\end{eqnarray*}
By Lemma \ref{lem:A-basic}, we have
\begin{equation}
\max_{(1-\rho)k<a\leq k}\Sigma_{aa} \leq \frac{1}{4L}\sum_{i\in[n]\backslash\{a\}}A_{ia} \leq \frac{np}{2L}, \label{eq:har1}
\end{equation}
with probability at least $1-O(n^{-10})$. Similar to the proof of Lemma \ref{lem:A-basic}, we can use Bernstein's inequality and a union bound argument to obtain that
\begin{eqnarray}
\nonumber \min_{(1-\rho)k<a\leq k}\Sigma_{aa} &\geq& \min_{(1-\rho)k<a\leq k}\left[\frac{\psi'(\kappa_1)}{L}\sum_{i=1}^{k-\rho k}A_{ia} + \frac{\psi'(\kappa_2)}{L}\sum_{i=k+1}^{k+\rho(n-k)}A_{ia}\right] \\
\nonumber &\geq& \frac{(1-\delta)(1-\rho)p}{L}\left(k\psi'(\kappa_1)+(n-k)\psi'(\kappa_2)\right) \\
&=& \frac{(1-\delta)(1-\rho)pn}{LV(\kappa)}, \label{eq:har2}
\end{eqnarray}
for some $\delta=o(1)$ with probability at least $1-O(n^{-10})$.
For each $a\neq b$,
$$\Sigma_{ab}=\Cov(T_a,T_b|A)=A_{ab}\frac{\psi'(\mu_a(z^*)-\mu_b(z^*))}{L}.$$
Then, Bernstein's inequality and a union bound argument, we have
\begin{equation}
\max_a\sum_{b:b\neq a}\Sigma_{ab}^2\leq \frac{1}{16L^2}\max_{(1-\rho)k<a\leq k}\sum_{b:b\neq a}A_{ab} \leq C_1\frac{\rho kp + \log n}{L^2}, \label{eq:har3}
\end{equation}
with probability at least $1-O(n^{-10})$. We can also obtain a similar bound for $\max_a\sum_{b:b\neq a}\Sigma_{ab}$. This allows us to give a lower bound on $\lambda_{\min}(\Sigma)$:
\begin{equation}
\lambda_{\min}(\Sigma) \geq \min_{(1-\rho)k<a\leq k}\Sigma_{aa} - \max_a\sum_{b:b\neq a}\Sigma_{ab} \geq \frac{(1-\delta)(1-\rho)pn}{LV(\kappa)} - C_2\frac{\rho kp + \log n}{L} \geq c_1\frac{pn}{L}. \label{eq:har4}
\end{equation}
To apply Lemma \ref{lem:hartigan}, we shall choose $\rho$ that satisfies both $\log(\rho k)=(1+o(1))\log k$ and $\rho=o(1)$. The existence of such $\rho$ is guaranteed by $k\rightarrow\infty$. With the bounds (\ref{eq:har1})-(\ref{eq:har4}), we can apply Lemma \ref{lem:hartigan}, and obtain
$$\mathbb{P}\left(\max_{(1-\rho)k<a\leq k}S_a > \sqrt{2(1-\epsilon/2)}\sqrt{\frac{np}{LV(\kappa)}}\sqrt{\log k}\right) \geq 0.98 - O(n^{-1}).$$
We then obtain the desired lower bound for (\ref{eq:Q-prob}). A similar argument also leads to
\begin{eqnarray*}
&& \mathbb{Q}_{z^*}\left(- \min_{k<b\leq k+\rho(n-k)}\sum_{i\in[n]\backslash\{b\}}A_{ib}(\bar{y}_{ib}-\mathbb{E}y_{ib}) > \sqrt{2(1-\epsilon/2)}\sqrt{\frac{np}{LV(\kappa)}}\sqrt{\log (n-k)}\right) \\
&\geq& 0.99-O\left(\frac{1}{(\log n)^c}\right).
\end{eqnarray*}
Therefore, $\mathbb{Q}_{z^*}(\wh{z}\neq z^*)\geq 0.95$ and we obtain the desired conclusion.

The above proof assumes that $k\rightarrow\infty$ and $\kappa=\Omega(1)$. When these two conditions do not hold, we need to slightly modify the argument. Let us briefly discuss two cases. In the first case, $k=O(1)$ and $\kappa=\Omega(1)$. In this case, we can construct $\theta^*$ by $\theta_i^*=0$ for $1\leq i\leq k$, $\theta_i^*=-\Delta$ for $k<i\leq k+\rho(n-k)$ and $\theta_i^*=-\kappa$ for $k+\rho(n-k)< i\leq n$. In the second case, $\kappa=o(1)$, and then we can take $\theta^*$ with $\theta_i^*=0$ for $1\leq i\leq k$ and $\theta_i^*=-\Delta$ for $k<i\leq n$. The remaining part of the proof will go through with similar arguments, and we will omit the details.
\end{proof}

\subsection{Proof of Theorem \ref{thm:partial-lower}} \label{sec:partial-lower-pf}

We first establish a lemma that lower bounds the error of a critical testing problem.
\begin{lemma}\label{lem:two_point_testing_lower}
Assume $\frac{np}{\log n}\rightarrow\infty$, $\kappa=O(1)$, $\rho=o(1)$, $k \rightarrow\infty$ and (\ref{eq:exact-threshold-lower}) holds for some arbitrarily small constant $\epsilon>0$.
Choose $\kappa_1,\kappa_2\geq 0$ such that we have both $\kappa_1+\kappa_2\leq \kappa$ and
$$\frac{n}{k\psi'(\kappa_1)+(n-k)\psi'(\kappa_2)}=V(\kappa).$$
Define $\theta_i=\kappa_1$ for $1\leq i\leq k-\rho k$, $\theta_i=0$ for $k-\rho k<i\leq k$, $\theta_i=-\Delta$ for $k+2\leq i\leq k+\rho(n-k)$ and $\theta_i=-\kappa_2$ for $k+\rho(n-k)< i\leq n$. Suppose we have independent $A_i\sim \text{Bernoulli}(p)$ and $z_{il}\sim\text{Bernoulli}(\psi(\theta_i))$ for all $i\in[n]\backslash\{k+1\}$ and $l\in[L]$. Then, there exists some $\delta=o(1)$ such that
\begin{eqnarray*}
&& \mathbb{P}\left(\sum_{l=1}^L\sum_{i\in[n]\backslash\{k+1\}}A_i\left[z_{il}\log\frac{\psi(\theta_i+\Delta)}{\psi(\theta_i)}+(1-z_{il})\log\frac{1-\psi(\theta_i+\Delta)}{1-\psi(\theta_i)}\right]\geq\log\frac{k}{n-k-1}\right) \\
&\geq& C\exp\left(-\frac{1}{2}\left(\frac{\sqrt{(1+\delta){\snr}}}{2}-\frac{1}{\sqrt{(1+\delta){\snr}}}\log\frac{n-k}{k}\right)_+^2\right),
\end{eqnarray*}
for some constant $C >0$.
\end{lemma}
\begin{proof}
We first consider the case
\begin{equation}
\frac{\sqrt{(1+\delta){\snr}}}{2}-\frac{1}{\sqrt{(1+\delta){\snr}}}\log\frac{n-k}{k} \rightarrow \infty, \label{eq:optimal-exp-div}
\end{equation}
for some $\delta=o(1)$ to be specified later.
Throughout the proof, we use $\mathbb{P}_A$ for the conditional distribution $\mathbb{P}(\cdot|A)$. We use the notation
$$Z_l=\sum_{i\in[n]\backslash\{k+1\}}A_i\left[z_{il}\log\frac{\psi(\theta_i+\Delta)}{\psi(\theta_i)}+(1-z_{il})\log\frac{1-\psi(\theta_i+\Delta)}{1-\psi(\theta_i)}\right].$$
Its conditional cumulant generating function is
$$K(u)=\sum_{i\in[n]\backslash\{k+1\}}A_i\log\left(\psi(\theta_i)^{1-u}\psi(\theta_i+\Delta)^u + (1-\psi(\theta_i))^{1-u}(1-\psi(\theta_i+\Delta))^u\right).$$
Define
$$u^*=\argmin_{u\geq0}\left(LK(u)-u\log\frac{k}{n-k-1}\right).$$
By direct calculation, we have
\begin{eqnarray*}
K'(0) &=& -\sum_{i\in[n]\backslash\{k+1\}}A_iD(\psi(\theta_i)\|\psi(\theta_i+\Delta)). \\
K'(1) &=& \sum_{i\in[n]\backslash\{k+1\}}A_iD(\psi(\theta_i+\Delta)\|\psi(\theta_i)).
\end{eqnarray*}
By Bernstein's inequality,
\begin{eqnarray}
\label{eq:K-d-0} K'(0) &\leq& -(1-\delta_1)p\sum_{i\in[n]\backslash\{k+1\}}D(\psi(\theta_i)\|\psi(\theta_i+\Delta)), \\
\label{eq:K-d-1} K'(1) &\geq& (1-\delta_1)p\sum_{i\in[n]\backslash\{k+1\}}D(\psi(\theta_i+\Delta)\|\psi(\theta_i)),
\end{eqnarray}
with some $\delta_1=o(1)$ for probability at least $1-O(n^{-1})$. Given that $\Delta=o(1)$, which is implied by (\ref{eq:exact-threshold-lower}), and $\rho=o(1)$, we have $\sum_{i\in[n]\backslash\{k+1\}}D(\psi(\theta_i)\|\psi(\theta_i+\Delta))=(1+o(1))\frac{n\Delta^2}{2V(\kappa)}$ and $\sum_{i\in[n]\backslash\{k+1\}}D(\psi(\theta_i+\Delta)\|\psi(\theta_i))=(1+o(1))\frac{n\Delta^2}{2V(\kappa)}$. With the condition (\ref{eq:optimal-exp-div}), we know that $LK'(0)-\log\frac{k}{n-k-1}<0$ and $LK'(1)-\log\frac{k}{n-k-1}>0$. Thus, we must have $u^*\in(0,1)$. In fact, the range of $u^*$ can be further narrowed down. We apply a Taylor expansion of $K'(u)$ as a function of $\Delta$ near $0$, and we obtain
$$K'(u)=\sum_{i\in[n]\backslash\{k+1\}}A_i\left[-\frac{1}{2}\psi'(\theta_i)\Delta^2 + \psi'(\theta_i)u\Delta^2 + O(|\Delta|^3)\right].$$
Note that the remainder term $O(|\Delta|^3)$ can be bounded by $|\Delta|^3$ up to some constant uniformly for all $u\in(0,1)$. By Bernstein's inequality, we have
\begin{equation}
K'(u) \geq -(1+\delta_1)\left(\frac{1}{2}-u\right)\frac{np\Delta^2}{V(\kappa)}, \label{eq:K-derivative-approx}
\end{equation}
for all $u\in(0,1/2)$ with probability at least $1-O(n^{-1})$. By (\ref{eq:K-derivative-approx}), there exists $\delta'=o(1)$ such that
$$K'\left(\frac{1}{2}-\frac{1}{(1+\delta')\snr}\log\frac{n-k}{k}\right)>0,$$
and therefore, we must have
\begin{equation}
u^*\in\left(0,\frac{1}{2}-\frac{1}{(1+\delta')\snr}\log\frac{n-k}{k}\right). \label{eq:u-star-range}
\end{equation}
We also introduce a quadratic approximation for $K(u)$, which is
$$\overline{K}(u)=\frac{np\Delta^2}{2V(\kappa)}(u^2-u).$$
It can be shown that
\begin{equation}
1-\delta_2\leq \frac{K(u)}{\overline{K}(u)} \leq 1+\delta_2, \label{eq:chen-pinhan}
\end{equation}
uniformly over all $u\in(0,1)$ for some $\delta_2=o(1)$ with probability at least $1-O(n^{-1})$. The inequality (\ref{eq:chen-pinhan}) can be obtained by a Taylor expansion argument followed by Bernstein's inequality, similar to the approximation obtained in (\ref{eq:K-derivative-approx}).

Define a probability distribution $\mathbb{Q}_A$, under which $Z_1,\cdots, Z_L$ are i.i.d. given $A$ and follow
$$\mathbb{Q}_A(Z_l=s)=\mathbb{P}_A(Z_l=s)e^{u^*s-K(u^*)},$$
for any $s$. It fact, each $Z_l$, under the measure $\mathbb{Q}_A$ can be written as the sum of several independent random variables, i.e. $Z_l=\sum_{i\in[n]\backslash\{k+1\}}Z_{il}$ where 
$$\mathbb{Q}_A(Z_{il}=s)=e^{A_iu^*s-A_iK_i(u^*)}\mathbb{P}_A\left(A_i\left[z_{il}\log\frac{\psi(\theta_i+\Delta)}{\psi(\theta_i)}+(1-z_{il})\log\frac{1-\psi(\theta_i+\Delta)}{1-\psi(\theta_i)}\right]=s\right),$$
and $K_i(u)=\log\left(\psi(\theta_i)^{1-u}\psi(\theta_i+\Delta)^u + (1-\psi(\theta_i))^{1-u}(1-\psi(\theta_i+\Delta))^u\right)$. Then for each $Z_{il}$ such that $A_i=1$, we can compute its second and 4th moment as
\begin{equation}
\mathbb{Q}_A((Z_{il}-\mathbb{Q}_A(Z_{il}))^2)=K_i^{\prime\prime}(u^*)=\psi^\prime(\theta_i)\Delta^2\frac{e^{u^*\Delta}}{(1-\psi(\theta_i)+e^{u^*\Delta}\psi(\theta_i))^2}\in(C_1^\prime\Delta^2, C_2^\prime\Delta^2),\label{eq:Zil-second}\end{equation}
\begin{equation}
\mathbb{Q}_A((Z_{il}-\mathbb{Q}_A(Z_{il}))^4)=K_i^{\prime\prime\prime\prime}(u^*)+3K_i^{\prime\prime}(u^*)^2\leq\Delta^2K_i^{\prime\prime}(u^*)+3K_i^{\prime\prime}(u^*)^2,\label{eq:Zil-fourth}
\end{equation}
where $C_1^\prime, C_2^\prime>0$ in (\ref{eq:Zil-second}) are some constants and we have used
\begin{align*}
&K_i^{\prime\prime\prime\prime}(u^*)=\psi^\prime(\theta_i)\Delta^4e^{u^*\Delta}\frac{\psi(\theta_i)^3e^{3u^*\Delta}-3\psi(\theta_i)\psi^\prime(\theta_i)e^{2u^*\Delta}-3\psi^\prime(\theta_i)(1-\psi(\theta^*))e^{u^*\Delta}+(1-\psi(\theta_i))^3}{(1-\psi(\theta_i)+\psi(\theta_i)e^{u^*\Delta})^5}\\
&\leq\psi^\prime(\theta_i)\Delta^4e^{u^*\Delta}\frac{1}{(1-\psi(\theta_i)+\psi(\theta_i)e^{u^*\Delta})^2}=\Delta^2K_i^{\prime\prime}(u^*)
\end{align*}
in (\ref{eq:Zil-fourth}).

Define $\mathcal{A}$ to be the event of $A$ that (\ref{eq:K-d-0}), (\ref{eq:K-d-1}), (\ref{eq:K-derivative-approx}), (\ref{eq:chen-pinhan}) and
\begin{equation}
\frac{1}{2}np \leq \sum_{i\in[n]\backslash\{k+1\}}A_i \leq 2np, \label{eq:real-madrid}
\end{equation}
all hold. We know that $\mathbb{P}(A\in\mathcal{A})\geq 1-O(n^{-1})$.

With the above preparations, we can lower bound $\mathbb{P}\left(\sum_{l=1}^LZ_l \geq \log\frac{k}{n-k-1}\right)$ by
$$\inf_{A\in\mathcal{A}}\mathbb{P}_A\left(\sum_{l=1}^LZ_l \geq \log\frac{k}{n-k-1}\right)\mathbb{P}(A\in\mathcal{A}) \geq \frac{1}{2}\inf_{A\in\mathcal{A}}\mathbb{P}_A\left(\sum_{l=1}^LZ_l \geq \log\frac{k}{n-k-1}\right).$$
For any $A\in\mathcal{A}$, a change-of-measure argument leads to the lower bound
\begin{eqnarray*}
&& \mathbb{P}_A\left(\sum_{l=1}^LZ_l \geq \log\frac{k}{n-k-1}\right)\\
&=&\exp\left(LK(u^*)-u^*\frac{k}{n-k-1}\right)\\
&&\times\mathbb{Q}_A\left[\mathbb{I}\left\{\sum_{l=1}^LZ_l-\log\frac{k}{n-k-1}\geq0\right\}\exp\left(-u^*(\sum_{l=1}^LZ_l-\log\frac{k}{n-k-1})\right)\right]\\
&\geq& \exp\left(-u^*T+LK(u^*)-u^*\log\frac{k}{n-k-1}\right)\mathbb{Q}_A\left(0\leq \sum_{l=1}^LZ_l-\log\frac{k}{n-k-1}\leq T\right),
\end{eqnarray*}
for any $T>0$ to be specified. We first lower bound the exponent $LK(u^*)-u^*\log\frac{k}{n-k-1}$ by
\begin{eqnarray*}
LK(u^*)-u^*\log\frac{k}{n-k-1} &=& \min_{u\in(0,1)}\left(LK(u)-u\log\frac{k}{n-k-1}\right) \\
&\geq& \min_{u\in(0,1)}\left(L(1+\delta_2)\overline{K}(u)-u\log\frac{k}{n-k-1}\right) \\
&\geq& -\frac{1}{2}\left(\frac{\sqrt{(1+\delta_3){\snr}}}{2}-\frac{1}{\sqrt{(1+\delta_3){\snr}}}\log\frac{n-k}{k}\right)^2,
\end{eqnarray*}
for some $\delta_3=o(1)$. We then need to choose an appropriate $T$ so that the probability $\mathbb{Q}_A\left(0\leq \sum_{l=1}^LZ_l-\log\frac{k}{n-k-1}\leq T\right)$ can be bounded below by some constant. To  achieve this purpose, we note that
$$\Var_{\mathbb{Q}_A}\left(\sum_{l=1}^LZ_l\right)=L\sum_{i\in[n]\backslash\{k+1\}}A_iK_i^{\prime\prime}(u^*)\leq C_1\Delta^2L\sum_{i\in[n]\backslash\{k+1\}}A_i\leq 2C_1\Delta^2Lnp,$$
for some constant $C_1>0$ due to (\ref{eq:Zil-second}), where $\Var_{\mathbb{Q}_A}$ is the variance operator under the measure $\mathbb{Q}_A$.
Thus, we set $T=\sqrt{2C_1\Delta^2Lnp}$. With this choice, and by (\ref{eq:u-star-range}), we have
$$u^*T\leq \sqrt{2C_1\Delta^2Lnp}\left(\frac{1}{2}-\frac{1}{(1+\delta')\snr}\log\frac{n-k}{k}\right).$$
Therefore, $u^*T$ is at most the order of the square-root of the desired exponent, and thus it is negligible.

Finally, we need to show $\mathbb{Q}_A\left(0\leq \sum_{l=1}^LZ_l-\log\frac{k}{n-k-1}\leq T\right)$ is lower bounded by some constant. Note that the definition of $u^*$ implies that $\sum_{l=1}^LZ_l-\log\frac{k}{n-k-1}$ has mean zero under $\mathbb{Q}_A$. By the definition of $T$, we have
\begin{eqnarray*}
&& \mathbb{Q}_A\left(0\leq \sum_{l=1}^LZ_l-\log\frac{k}{n-k-1}\leq T\right) \\
&\geq& \mathbb{Q}_A\left(0\leq \sum_{l=1}^LZ_l-\log\frac{k}{n-k-1}\leq \sqrt{\Var\left(\sum_{l=1}^LZ_l\Bigg|A\right)}\right)\\
&=& \mathbb{Q}_A\left(0\leq \sum_{l=1}^L\sum_{i\in[n]\backslash\{k+1\}}Z_{il}-\log\frac{k}{n-k-1}\leq \sqrt{\Var\left(\sum_{l=1}^L\sum_{i\in[n]\backslash\{k+1\}}Z_{il}\Bigg|A\right)}\right).
\end{eqnarray*}
We apply the central limit theorem in Lemma \ref{lem:CLT-stein} to bound the above probability. The $4$th moment approximation bound in Lemma \ref{lem:CLT-stein} is 
\begin{eqnarray}
\nonumber&&\sqrt{L\sum_{i\in[n]\backslash\{k+1\}}A_i\left(\frac{K_i^{\prime\prime\prime\prime}(u^*)+3K_i^{\prime\prime}(u^*)^2}{(L\sum_{i\in[n]\backslash\{k+1\}}A_iK_i^{\prime\prime}(u^*))^2}\right)^{3/4}}\\
&&\leq\sqrt{L\sum_{i\in[n]\backslash\{k+1\}}A_i\left(\frac{\Delta^2K_i^{\prime\prime}(u^*)+3K_i^{\prime\prime}(u^*)^2}{(L\sum_{i\in[n]\backslash\{k+1\}}A_iK_i^{\prime\prime}(u^*))^2}\right)^{3/4}}\label{eq:use-fourth}\\
&&\leq\sqrt{L\sum_{i\in[n]\backslash\{k+1\}}A_i\left(\frac{C_2^\prime+3C_2^{\prime2}}{(L\sum_{i\in[n]\backslash\{k+1\}}A_iC_1^\prime)^2}\right)^{3/4}}\label{eq:use-second}\\
&&\leq C_2\left(L\sum_{i\in[n]\backslash\{k+1\}}A_i\right)^{-1/4}\label{vanishing-approx}
\end{eqnarray}
which tends to zero by (\ref{eq:real-madrid}). We have used (\ref{eq:Zil-fourth}) in (\ref{eq:use-fourth}),  (\ref{eq:Zil-second}) in (\ref{eq:use-second}). We thus have
$$\mathbb{Q}_A\left(0\leq \sum_{l=1}^LZ_l-\log\frac{k}{n-k-1}\leq T\right)\geq \mathbb{P}\left(0\leq N(0,1)\leq 1\right)-o(1),$$
which is bounded below by a constant. To summarize, we have shown that
$$\mathbb{P}\left(\sum_{l=1}^LZ_l \geq \log\frac{k}{n-k-1}\right)\geq C_3\exp\left(-\frac{1}{2}\left(\frac{\sqrt{(1+\delta_4){\snr}}}{2}-\frac{1}{\sqrt{(1+\delta_4){\snr}}}\log\frac{n-k}{k}\right)^2\right),$$
for some $\delta_4=o(1)$ and some constant $C_3>0$ when (\ref{eq:optimal-exp-div}) holds with $\delta=\delta_4$.

To close the proof, we need a different argument when
$$\frac{\sqrt{(1+\delta_4){\snr}}}{2}-\frac{1}{\sqrt{(1+\delta_4){\snr}}}\log\frac{n-k}{k} \leq C_4,$$
for some constant $C_4>0$. This condition, together with Bernstein's inequality, implies that
\begin{equation}
\sum_{l=1}^L\mathbb{E}(Z_l|A)-\log\frac{k}{n-k-1} \geq -C_5\sqrt{Lnp\Delta^2}, \label{eq:man-u}
\end{equation}
with probability at least $1-O(n^{-1})$.
Define $\overline{\mathcal{A}}$ to be an event of $A$ such that both (\ref{eq:real-madrid}) and (\ref{eq:man-u}) hold. It is clear that $\mathbb{P}(\overline{\mathcal{A}})\geq 1-O(n^{-1})$. We then have
\begin{eqnarray}
\nonumber \mathbb{P}\left(\sum_{l=1}^LZ_l \geq \log\frac{k}{n-k-1}\right) &\geq& \frac{1}{2}\inf_{A\in\overline{\mathcal{A}}}\mathbb{P}_A\left(\sum_{l=1}^LZ_l \geq \log\frac{k}{n-k-1}\right) \\
\label{eq:ex-l-m-dif} &\geq& \frac{1}{2}\inf_{A\in\overline{\mathcal{A}}}\mathbb{P}_A\left(\sum_{l=1}^L(Z_l-\mathbb{E}(Z_l|A)) \geq C_5\sqrt{Lnp\Delta^2}\right) \\
\label{eq:ex-l-clt} &\geq& c_1 - o(1),
\end{eqnarray}
for some constant $c_1>0$.
The inequality (\ref{eq:ex-l-m-dif}) is by (\ref{eq:man-u}). For (\ref{eq:ex-l-clt}), we use the Gaussian approximation in Lemma \ref{lem:CLT-stein}, and the $4$th moment approximation bound is of order $\left(L\sum_{i\in[n]\backslash\{k+1\}}A_i\right)^{-1/4}$ by similar calculation as in (\ref{vanishing-approx}) under measure $\p_A$, which tends to zero by (\ref{eq:real-madrid}). The proof is complete.
\end{proof}

\begin{proof}[Proof of Theorem \ref{thm:partial-lower}]
We first note that the condition (\ref{eq:exact-threshold-lower}) implies that $\Delta=o(1)$. Choose $\kappa_1,\kappa_2\geq 0$ such that we have both $\kappa_1+\kappa_2\leq \kappa$ and
$$\frac{n}{k\psi'(\kappa_1)+(n-k)\psi'(\kappa_2)}=V(\kappa).$$
We first consider the case $k\rightarrow\infty$ and $\kappa=\Omega(1)$. In this case, one can easily check that $\kappa_2=\Omega(1)$.
Our least favorable $\theta',\theta''\in\Theta'(k,\Delta,\kappa)$ is constructed as follows. Let $\rho=o(1)$ be a vanishing number that will be specified later. Define $\theta_i'=\kappa_1$ for all $1\leq i\leq k-\rho k$, $\theta_i'=0$ for $k-\rho k<i\leq k$, $\theta_i'=-\Delta$ for $k<i\leq k+\rho(n-k)$ and $\theta_i'=-\kappa_2$ for $k+\rho(n-k)< i\leq n$. For the simplicity of proof, we choose $\rho$ so that both $\rho k$ and $\rho(n-k)$ are integers. For $\theta''$, we set $\theta_i''=\theta_i'$ for all $i\in[n]\backslash\{k+1\}$ and $\theta_{k+1}''=0$.
Recall the definition of the subset $\mathcal{R}_{k,\rho}\subset\S_n$ in (\ref{eq:subset-R-in-pf}). We then have
\begin{align*}
\inf_{\wh{r}}\sup_{\substack{r^*\in\S_n\\\theta^*\in\Theta'(k,\Delta,\kappa)}}\mathbb{E}_{(\theta^*,r^*)}\h_k(\wh{r},r^*) &\geq \inf_{\wh{r}}\sup_{\substack{r^*\in\mathcal{R}_{k,\rho}\\ \theta^* \in\{\theta',\theta''\}}}\mathbb{E}_{(\theta^*,r^*)}\h_k(\wh{r},r^*)\\
&\geq \inf_{\wh{r}} \frac{1}{2}\sum_{\theta^* \in\{\theta',\theta''\}} \frac{1}{\abs{\mathcal{R}_{k,\rho}}} \sum_{r^*\in\mathcal{R}_{k,\rho} }\mathbb{E}_{(\theta^*,r^*)}\h_k(\wh{r},r^*).
\end{align*}
That is, we first lower bound the minimax risk by the Bayes risk.
Since
$$\h_k(\wh r,r^*) \geq \frac{1}{2k}\sum_{k-\rho k <i \leq k + \rho(n-k)}\left(\indc{\wh{r}_i>k, r^*_i\leq k}+\indc{\wh{r}_i\leq k, r^*_i> k}\right),$$
we have
\begin{align*}
& \inf_{\wh{r}}\sup_{\substack{r^*\in\S_n\\\theta^*\in\Theta'(k,\Delta,\kappa)}}\mathbb{E}_{(\theta^*,r^*)}\h_k(\wh{r},r^*)\\
& \geq  \inf_{\wh{r}} \frac{1}{2}\sum_{\theta^* \in\{\theta',\theta''\}} \frac{1}{\abs{\mathcal{R}_{k,\rho}}} \sum_{r^*\in\mathcal{R}_{k,\rho} }\mathbb{E}_{(\theta^*,r^*)}\frac{1}{2k}\sum_{k-\rho k <i \leq k + \rho(n-k)}\left(\indc{\wh{r}_i>k, r^*_i\leq k}+\indc{\wh{r}_i\leq k, r^*_i> k}\right) \\
& \geq    \frac{1}{4k\abs{\mathcal{R}_{k,\rho}}} \sum_{k-\rho k <i \leq k + \rho(n-k)}   \inf_{\wh{r}}\sum_{\theta^* \in\{\theta',\theta''\}}  \left( \sum_{\substack{r^*\in\mathcal{R}_{k,\rho}\\r^*_i \leq k}} \mathbb{P}_{(\theta^*,r^*)}(\wh r_i> k) + \sum_{\substack{r^*\in\mathcal{R}_{k,\rho}\\r^*_i \geq k+2}} \mathbb{P}_{(\theta^*,r^*)}(\wh r_i\leq k)\right)\\
&\geq  \frac{1}{4k\abs{\mathcal{R}_{k,\rho}}} \sum_{k-\rho k <i \leq k + \rho(n-k)}   \inf_{\wh{r}} \left( \sum_{\substack{r^*\in\mathcal{R}_{k,\rho}\\r^*_i \leq k}} \mathbb{P}_{(\theta'',r^*)}(\wh r_i> k) + \sum_{\substack{r^*\in\mathcal{R}_{k,\rho}\\r^*_i \geq k+2}} \mathbb{P}_{(\theta',r^*)}(\wh r_i\leq k)\right).
\end{align*}
At this point, we need to introduce some extra notation.
For any $r,r'\in \S_n$, we define the Hamming distance without normalization as $\mathcal{H}(r,r') = \sum_{i=1}^n \indc{r_i \neq r'_i}$. For each $k-\rho k <i \leq k + \rho(n-k)$, we can partition the set $\mathcal{R}_{k,\rho}$ into three disjoint subsets. Define
\begin{eqnarray*}
\mathcal{R}_{k,\rho}^{(1)} &=& \left\{r\in\mathcal{R}_{k,\rho} : r_i \leq k \right\}, \\
\mathcal{R}_{k,\rho}^{(2)} &=& \left\{r\in\mathcal{R}_{k,\rho} : r_i = k+1  \right\}, \\
\mathcal{R}_{k,\rho}^{(3)} &=& \left\{r\in\mathcal{R}_{k,\rho} : r_i \geq k+2 \right\}.
\end{eqnarray*}
It is easy to see that $\mathcal{R}_{k,\rho}=\cup_{j=1}^3\mathcal{R}_{k,\rho}^{(j)}$. We note that the three subsets all depend on the index $i$, but we shall suppress this dependence to avoid notational clutter. For any $r\in\mathcal{R}_{k,\rho}^{(2)}$, define
\begin{eqnarray*}
\mathcal{N}_{2\rightarrow 1} (r) &=& \left\{r''\in\mathcal{R}_{k,\rho}^{(1)}: \mathcal{H}(r,r'')=2 \right\},\\
\mathcal{N}_{2\rightarrow 3} (r) &=& \left\{r'\in\mathcal{R}_{k,\rho}^{(3)}: \mathcal{H}(r,r')=2 \right\}.
\end{eqnarray*}
Since for any different permutations, the smallest Hamming distance between them is $2$, $\mathcal{N}_{2\rightarrow 1} (r)$ and $\mathcal{N}_{2\rightarrow 3} (r)$ can be understood as neighborhoods $r$ within $\mathcal{R}_{k,\rho}^{(1)}$ and $\mathcal{R}_{k,\rho}^{(3)}$, respectively.  It is easy to check that $\{\mathcal{N}_{2\rightarrow 1} (r)\}_{r\in\mathcal{R}_{k,\rho}^{(2)}}$ are disjoint subsets, and they form a partition of $\mathcal{R}_{k,\rho}^{(1)}$. Similarly, $\{\mathcal{N}_{2\rightarrow 3} (r)\}_{r\in\mathcal{R}_{k,\rho}^{(2)}}$ are disjoint subsets, and form a partition of $\mathcal{R}_{k,\rho}^{(3)}$.
With these notation, we have
\begin{align*}
& \inf_{\wh{r}}\sup_{\substack{r^*\in\S_n\\\theta^*\in\Theta'(k,\Delta,\kappa)}}\mathbb{E}_{(\theta^*,r^*)}\h_k(\wh{r},r^*)\\
&\geq  \frac{1}{4k\abs{\mathcal{R}_{k,\rho}}} \sum_{k-\rho k <i \leq k + \rho(n-k)}   \inf_{\wh{r}}   \sum_{r\in\mathcal{R}_{k,\rho}^{(2)}}   \left( \sum_{r''\in \mathcal{N}_{2\rightarrow 1} (r)} \mathbb{P}_{(\theta'',r'')}(\wh r_i> k) + \sum_{r'\in \mathcal{N}_{2\rightarrow 3} (r)} \mathbb{P}_{(\theta',r')}(\wh r_i\leq k)\right)\\
& =   \frac{1}{4k\abs{\mathcal{R}_{k,\rho}}} \sum_{k-\rho k <i \leq k + \rho(n-k)}   \inf_{\wh{r}}   \sum_{r\in\mathcal{R}_{k,\rho}^{(2)}}    \sum_{\substack{r''\in \mathcal{N}_{2\rightarrow 1} (r) \\ r'\in \mathcal{N}_{2\rightarrow 3} (r)}} \left(  \frac{1}{n-k-1}\mathbb{P}_{(\theta'',r'')}(\wh r_i> k)  +  \frac{1}{k}\mathbb{P}_{(\theta',r')}(\wh r_i\leq k) \right) \\
& \geq \frac{1}{4k(n-k-1)\abs{\mathcal{R}_{k,\rho}}} \sum_{k-\rho k <i \leq k + \rho(n-k)}  \sum_{r\in\mathcal{R}_{k,\rho}^{(2)}}    \sum_{\substack{r''\in \mathcal{N}_{2\rightarrow 1} (r) \\ r'\in \mathcal{N}_{2\rightarrow 3} (r)}} \inf_{0\leq \phi\leq 1}\left[\mathbb{E}_{(\theta'',r'')}\phi+\frac{n-k-1}{k}\mathbb{E}_{(\theta',r')}(1-\phi)\right],
\end{align*}
where we have used the fact $| \mathcal{N}_{2\rightarrow 1} (r)|=k$ and $| \mathcal{N}_{2\rightarrow 3} (r)|=n-k-1$ to obtain the equality in the above display. To this end, it suffices to give a lower bound for the testing problem
\begin{equation}
\inf_{0\leq \phi\leq 1}\left[\mathbb{E}_{(\theta'',r'')}\phi+\frac{n-k-1}{k}\mathbb{E}_{(\theta',r')}(1-\phi)\right], \label{eq:coldplay}
\end{equation}
for any $r''\in \mathcal{N}_{2\rightarrow 1} (r)$ and any $r'\in \mathcal{N}_{2\rightarrow 3} (r)$ with any $r\in\mathcal{R}_{k,\rho}^{(2)}$ and any $k-\rho k <i \leq k + \rho(n-k)$.

For the two probability distributions in (\ref{eq:coldplay}), the probability $\mathbb{P}_{(\theta'',r'')}$ is the BTL model with parameter $\{\theta_{r''_i}''\}_{i\in[n]}$ and the probability $\mathbb{P}_{(\theta',r')}$ is the BTL model with parameter $\{\theta_{r'_i}'\}_{i\in[n]}$. It turns out the two vectors $\{\theta_{r''_i}''\}_{i\in[n]}$ and $\{\theta_{r'_i}'\}_{i\in[n]}$ only differ by one entry.
To see this, let $i$ and $j'$ be the two coordinates that $r$ and $r'$ differ and let $i$ and $j''$ be the two coordinates that $r$ and $r''$ differ. Then, $r'$ and $r''$ differ at the $i$th, the $j'$th and the $j''$th coordinates. 
This immediately implies $\theta'_{r'_l} = \theta''_{r''_l}$ for all $l\in[n]\backslash\{i,j',j''\}$.
By the definitions of $\mathcal{N}_{2\rightarrow 1}$ and $\mathcal{N}_{2\rightarrow 3}$, we have $r'_i = r_{j'},r'_{j'}=k+1,r'_{j''} = r_{j''}$ and $r''_{i}= r_{j''},r''_{j'} = r_{j'},r''_{j''}=k+1$. Moreover, we also have $r_{j'} \geq k+2$ and $r_{j''} \leq k$. We remind the readers that all the three coordinates are in the interval $[k-\rho k+1,k+\rho(n-k)]$. According to the definitions of $\theta'$ and $\theta''$, we then have $\theta'_{r'_{j''}} = \theta''_{r''_{j''}} = 0$ and $\theta'_{r'_{j'}} = \theta''_{r''_{j'}} = -\Delta$. For the only different coordinate, we have $\theta'_{r'_i} = -\Delta$ and $\theta''_{r''_i} = 0$.

Since $\{\theta_{r''_i}''\}_{i\in[n]}$ and $\{\theta_{r'_i}'\}_{i\in[n]}$ only differ by a single coordinate, the testing problem (\ref{eq:coldplay}) is equivalent to
\begin{equation}
\inf_{0\leq \phi\leq 1}\left[\mathbb{E}_{(\theta'',\bar{r})}\phi+\frac{n-k-1}{k}\mathbb{E}_{(\theta',\bar{r})}(1-\phi)\right], \label{eq:overwatch}
\end{equation}
where $\bar{r}_i=i$ for all $i\in[n]$. The equivalence between (\ref{eq:coldplay}) and (\ref{eq:overwatch}) can be obtained by the existence of a simultaneous permutation that maps the two vectors $\{\theta_{r''_i}''\}_{i\in[n]}$ and $\{\theta_{r'_i}'\}_{i\in[n]}$ to $\theta''$ and $\theta'$. By Neyman-Pearson lemma, we can lower bound (\ref{eq:overwatch}) by
\begin{equation}
\mathbb{P}_{(\theta'',\bar{r})}\left(\frac{d\mathbb{P}_{(\theta',\bar{r})}}{d\mathbb{P}_{(\theta'',\bar{r})}}\geq \frac{k}{n-k-1}\right). \label{eq:NP-in-pf}
\end{equation}
This probability can be lower bounded by
$$C\exp\left(-\frac{1}{2}\left(\frac{\sqrt{(1+\delta){\snr}}}{2}-\frac{1}{\sqrt{(1+\delta){\snr}}}\log\frac{n-k}{k}\right)_+^2\right),$$
with some constant $C>0$ and some $\delta=o(1)$ according to Lemma \ref{lem:two_point_testing_lower}. Since $|\mathcal{R}_{k,\rho}^{(2)}|/|\mathcal{R}_{(k,\rho)}|=(1-\rho)n$, we have
\begin{eqnarray*}
&&\inf_{\wh{r}}\sup_{\substack{r^*\in\S_n\\\theta^*\in\Theta'(k,\Delta,\kappa)}}\mathbb{E}_{(\theta^*,r^*)}\h_k(\wh{r},r^*) \\
&\geq& C_1\rho\exp\left(-\frac{1}{2}\left(\frac{\sqrt{(1+\delta){\snr}}}{2}-\frac{1}{\sqrt{(1+\delta){\snr}}}\log\frac{n-k}{k}\right)_+^2\right),
\end{eqnarray*}
for some constant $C_1>0$. When the exponent diverges, we can choose $\rho$ that tends to zero sufficiently slow so that it can be absorbed into the exponent. Otherwise, we can simply set $\rho$ to be a sufficiently small constant, and the above proof will still go through. One can use a similar argument as Lemma \ref{lem:two_point_testing_lower} to show (\ref{eq:NP-in-pf}) is bounded below by some constant. In this case, we have $\inf_{\wh{r}}\sup_{\substack{r^*\in\S_n\\\theta^*\in\Theta'(k,\Delta,\kappa)}}\mathbb{E}_{(\theta^*,r^*)}\h_k(\wh{r},r^*)$ bounded below by some constant as desired.

Finally, we briefly discuss how to modify the proof when either $k\rightarrow\infty$ or $\kappa=\Omega(1)$ does not hold. When $k\rightarrow\infty$ and $\kappa=o(1)$, we can take $\theta_i'=0$ for $1\leq i\leq k$ and $\theta_i'=-\Delta$ for $k<i\leq n$. The vector $\theta''$ is still defined according to $\theta_i''=\theta_i'$ for all $i\in[n]\backslash\{k+1\}$ and $\theta_{k+1}''=0$. The proof will go through with some slight modifcation. When $k=O(1)$, the condition (\ref{eq:exact-threshold-lower}) is equivalent to $\snr<(1-\epsilon)2\log n$ for some constant $\epsilon>0$, and we only need to prove a constant minimax lower bound. 
This is obviously true becasue
\begin{eqnarray*}
\inf_{\wh{r}}\sup_{\substack{r^*\in\S_n\\\theta^*\in\Theta'(k,\Delta,\kappa)}}\mathbb{E}_{(\theta^*,r^*)}\h_k(\wh{r},r^*) &\geq& \inf_{\wh{r}}\sup_{\substack{r^*\in\S_n\\\theta^*\in\Theta(k,\Delta,\kappa)}}\mathbb{E}_{(\theta^*,r^*)}\h_k(\wh{r},r^*) \\
&\geq& \inf_{\wh{r}}\sup_{\substack{r^*\in\S_n\\\theta^*\in\Theta(k,\Delta,\kappa)}}\frac{1}{2k}\mathbb{P}_{(\theta^*,r^*)}\left(\h_k(\wh{r},r^*) > 0\right),
\end{eqnarray*}
which is lower bounded by a constant by Theorem \ref{thm:exact-lower} and the condition that $k=O(1)$.
\end{proof}

\section{Proofs of Local Error Rates}\label{sec:pf_local}
In this section, we prove Theorem \ref{thm:MLE-ranking-theta} and Theorem \ref{thm:spectral-ranking-theta}.

\subsection{Proof of Theorem \ref{thm:MLE-ranking-theta}}
We first give Lemma \ref{lem:tail-case} to characterize entrywise tail behaviors of the MLE (\ref{eq:pure-MLE}) which is crucial to the upper bound in Theorem \ref{thm:MLE-ranking-theta}.
\begin{lemma}\label{lem:tail-case}
Assume $\frac{np}{\log n}\rightarrow\infty$ and $\kappa=O(1)$. Then, for the rank vector $\wh{r}$ that is induced by the MLE (\ref{eq:pure-MLE}), for any small constant $0.1>\delta>0$, there exists some constant $C>0$, such that for any $t\in\mathbb{R}$, any $\theta^*\in\Theta(k,0,\kappa)$, $r^*\in\S_n$, we have
\begin{equation}
\p_{(\theta^*, r^*)}\left(\wh{\theta}_i\leq t\right)\leq C\exp\left(-\frac{(1-\delta)(\theta_{r_i^*}^*-t)_{+}^2npL}{2V_{r_i^*}(\theta^*)}\right)+Cn^{-7}, r_i^*\leq k;\label{eq:tail-case-1}
\end{equation}
\begin{equation}
\p_{(\theta^*, r^*)}\left(\wh{\theta}_i\geq t\right)\leq C\exp\left(-\frac{(1-\delta)(t-\theta_{r_i^*}^*)_{+}^2npL}{2V_{r_i^*}(\theta^*)}\right)+Cn^{-7}, r_i^*\geq k+1\label{eq:tail-case-2}
\end{equation}
\end{lemma}
\begin{proof}
The proof follows the proof of Theorem \ref{thm:MLE-ranking} with slight modifications. Without loss of generality, we can assume $r_i^*=i$ for all $\in[n]$. Let
\begin{equation}
\bar{\Delta}_i = \begin{cases}
\min\left((\theta_{i}^*-t)_{+},\left(\frac{\log n}{np}\right)^{1/4}\right), & 1\leq i\leq k, \\
\min\left((t-\theta_{i}^*)_{+},\left(\frac{\log n}{np}\right)^{1/4}\right), & k+1 \leq i\leq n.
\end{cases} \label{eq:delta-bar-i-case}
\end{equation}
We only need to prove (\ref{eq:tail-case-1}) since (\ref{eq:tail-case-2}) can be proved similarly. 

Consider any $m\in[k]$. When $(\theta^*_m-t)_{+}^2npL\leq c^\prime$ for some large enough constant to be specified later, we can directly bound the probability using the trivial bound $1$. Thus, we only need to  consider the regime when $(\theta^*_m-t)_{+}^2npL>c^\prime$. 
%Then we have $\bar{\Delta}_m^2Lnp>c'$ which is slightly different from the one studied  in the proof of Theorem  \ref{thm:MLE-ranking} that has $\bar{\Delta}_m^2Lnp \rightarrow\infty$.

Following the proof of Theorem \ref{thm:MLE-ranking}, we have (\ref{eq:bound-f-1})-(\ref{eq:MLE-local-approx-1}) and (\ref{eqn:MLE_gm_gm_diff}) hold. Note that we now have $\bar{\Delta}_m^2Lnp>c'$ instead of $\bar{\Delta}_m^2Lnp \rightarrow\infty$ which is needed  in the proof of Theorem  \ref{thm:MLE-ranking}. As a consequence, we now have (\ref{eq:MLE-local-approx-2}) and (\ref{eq:MLE-local-approx-3}) hold with $\delta =4C_4e^\kappa/\sqrt{c'}$ instead of some $o(1)$ as in the proof of Theorem \ref{thm:MLE-ranking}. To sum up, with this $\delta$, we have
\begin{align}
& |\wh{\theta}_m-\bar{\theta}_m| \leq \delta\bar{\Delta}_m,  \label{eq:MLE-local-approx-1-case}\\
& \frac{|f^{(m)}(\theta_m^*|\wh{\theta}_{-m})-f^{(m)}(\theta_m^*|\theta^*_{-m})|}{g^{(m)}(\theta_m^*|\theta_{-m}^*)} \leq \delta\bar{\Delta}_m,  \\
& \frac{\left|g^{(m)}(\theta_m^*|\wh{\theta}_{-m})-g^{(m)}(\theta_m^*|\theta_{-m}^*)\right|}{g^{(m)}(\theta_m^*|\theta_{-m}^*)} \leq \delta,  \label{eq:left-case}
\end{align}
hold with probability at least  $1-O(n^{-7})-\exp(-\bar{\Delta}_m^{3/2}Lnp)-\exp\left(-\bar{\Delta}_m^2npL\frac{np}{\log n}\right)$. We can make $\delta$ to be an arbitrarily small constant by setting $c'$ large as $\kappa=O(1)$. 

Then for any $i\leq k$, by the same argument as in the proof of Theorem \ref{thm:MLE-ranking}, we have
\begin{eqnarray}
\nonumber && \mathbb{P}\left(\wh{\theta}_i \leq t\right) \\
\nonumber &\leq& \mathbb{P}\left(\wh{\theta}_i-\theta_i^* \leq -(\theta_i^*-t)\right) \\
\nonumber &\leq& \mathbb{P}\left(\bar{\theta}_i-\theta_i^*\leq -(1-\delta)\bar{\Delta}_i\right) + \mathbb{P}\left(|\bar{\theta}_i-\wh{\theta}_i|>\delta\bar{\Delta}_i\right) \\
%\nonumber &\leq& \mathbb{P}\left(\frac{f^{(i)}(\theta_i^*|{\theta}^*_{-i})}{g^{(i)}(\theta_i^*|{\theta}^*_{-i})} \leq -(1+\delta^2-3\delta)\bar{\Delta}_i\right) + \mathbb{P}\left(|\bar{\theta}_i-\wh{\theta}_i|>\delta\bar{\Delta}_i\right) \\
%\nonumber && + \mathbb{P}\left(\frac{\left|g^{(i)}(\theta_i^*|\wh{\theta}_{-i})-g^{(i)}(\theta_i^*|\theta_{-i}^*)\right|}{g^{(i)}(\theta_i^*|\theta_{-i}^*)} > \delta\right) + \mathbb{P}\left(\frac{|f^{(i)}(\theta_i^*|\wh{\theta}_{-i})-f^{(i)}(\theta_i^*|\theta^*_{-i})|}{g^{(i)}(\theta_i^*|\theta_{-i}^*)} > \delta\bar{\Delta}_i\right) \\
%\label{eq:left-case} 
&\leq& \mathbb{P}\left(-\frac{f^{(i)}(\theta_i^*|{\theta}^*_{-i})}{g^{(i)}(\theta_i^*|{\theta}^*_{-i})} \leq -(1-3\delta)\bar{\Delta}_i\right) + O(n^{-7}) \\
\nonumber && +\exp(-\bar{\Delta}_i^{3/2}Lnp)+\exp\left(-\bar{\Delta}_i^2npL\frac{np}{\log n}\right),
\end{eqnarray}
which has the same upper bound as in (\ref{eq:left}). 
We then have the same (\ref{eq:apply-bern}) and the event $\mathcal{A}_i$ as in the proof of Theorem \ref{thm:MLE-ranking}. As a result,  
\begin{eqnarray}
\nonumber && \mathbb{P}\left(-\frac{f^{(i)}(\theta_i^*|{\theta}^*_{-i})}{g^{(i)}(\theta_i^*|{\theta}^*_{-i})} \leq -(1-3\delta)\bar{\Delta}_i\right) \\
\nonumber  &\leq& \sup_{A\in\mathcal{A}_i}\exp\left(-\frac{\frac{1}{2}(1-3\delta)^2\bar{\Delta}_i^2\left(L\sum_{j\in[n]\backslash\{i\}}A_{ij}\psi'(\theta_i^*-\theta_j^*)\right)^2}{L\sum_{j\in[n]\backslash\{i\}}A_{ij}\psi'(\theta_i^*-\theta_j^*)+\frac{1-3\delta}{3}\bar{\Delta}_iL\sum_{j\in[n]\backslash\{i\}}A_{ij}\psi'(\theta_i^*-\theta_j^*)}\right) \\
\nonumber &&  + O(n^{-7}) \\
\label{eq:apply-bern2-case} &=& \exp\left(-\frac{1-\delta^\prime}{2}\bar{\Delta}_i^2Lp\sum_{j\in[n]\backslash\{i\}}\psi^\prime(\theta_i^*-\theta_j^*)\right) + O(n^{-7}) \\
\label{eq:free-case} &\leq& \exp\left(-\frac{1-\delta^\prime}{2}(\theta_i^*-t)^2Lp\sum_{j\in[n]\backslash\{i\}}\psi^\prime(\theta_i^*-\theta_j^*)\right) + O(n^{-7})\\
\label{eq:eq-case}&=&\exp\left(-\frac{1-\delta^{\prime\prime}}{2V_i(\theta^*)}(\theta_i^*-t)^2npL\right) + O(n^{-7})
\end{eqnarray}
where $\delta^\prime, \delta^{\prime\prime}$ are able to be any small constant (by adjusting $c^\prime$).  We  use the definition of $\mathcal{A}_i$ to obtain the expression (\ref{eq:apply-bern2-case}). To see why (\ref{eq:free-case}) is true, note that when $\bar{\Delta}_i^2=\sqrt{\frac{\log n}{np}}$, the first term of (\ref{eq:apply-bern2-case}) can be absorbed into $O(n^{-7})$. (\ref{eq:eq-case}) comes from $\frac{\sum_{j\in[n]\backslash\{i\}}\psi^\prime(\theta_i^*-\theta_j^*)}{\sum_{j\in[n]}\psi^\prime(\theta_i^*-\theta_j^*)}=1+o(1)$. 

Since $\exp(-\bar{\Delta}_i^{3/2}Lnp)+\exp\left(-\bar{\Delta}_i^2npL\frac{np}{\log n}\right)\leq \exp\left(-\frac{1+o(1)}{2V_i(\theta^*)}(\theta_i^*-t)^2npL\right) + O(n^{-7})$, we have for any small constant $\delta>0$, there exists some constant $C>0$, such that
\begin{equation}
\mathbb{P}\left(\wh{\theta}_i \leq t\right)\leq C\exp\left(-\frac{1-\delta}{2V_i(\theta^*)}(\theta_i^*-t)^2npL\right) + Cn^{-7}, \label{eq:prob-bound-1-case}
\end{equation}
for all $i\leq k$ which completes the proof. 
\end{proof}

\begin{proof}[Proof of (\ref{eq:MLE-ranking-theta-u}) of Theorem \ref{thm:MLE-ranking-theta}]
The upper bound (\ref{eq:MLE-ranking-theta-u}) is a straightforward consequence of Lemma \ref{lem:anderson} and Lemma \ref{lem:tail-case}. We have
\begin{align*}
&\E_{(\theta^*, r^*)}\h_k(\wh{r},r^*)\\
&\leq C\frac{1}{k}\left[\sum_{i=1}^k\exp\left(-\frac{(1-\delta)(\theta_i^*-t)_+^2npL}{2V_i(\theta^*)}\right)+\sum_{i=k+1}^n\exp\left(-\frac{(1-\delta)(t-\theta_i^*)_+^2npL}{2V_i(\theta^*)}\right)\right]+Cn^{-6}.
\end{align*}
\end{proof}

The rest of the section focuses on the lower bound (\ref{eq:MLE-ranking-theta-l}). The proof follows the proof of Theorem \ref{thm:exact-lower} with some modification. We include it below for completeness.
\begin{proof}[Proof of (\ref{eq:MLE-ranking-theta-l}) of Theorem \ref{thm:MLE-ranking-theta}]
We are going to prove
\begin{equation}
\E_{(\theta^*, r^*)}\h_k(\wh{r},r^*)\gtrsim\frac{R_1([k],\theta^*, t^*, -\delta)+R_2([n]\backslash[k],\theta^*, t^*, -\delta)}{k}\label{eq:lower-bound-mle-case}
\end{equation}
where $t^*$ is the unique solution such that $R_1([k],\theta^*, t^*, -\delta)=R_2([n]\backslash[k],\theta^*, t^*, -\delta)$. We first show the existence and uniqueness of $t^*$. Note that $R_1([k],\theta^*, t, -\delta)$ increases with $t$ while $R_2([n]\backslash[k],\theta^*, t, -\delta)$ decreases with $t$. Moreover, since $\lim_{t\to-\infty}R_1([k],\theta^*, t, -\delta)=\lim_{t\to+\infty}R_2([n]\backslash[k],\theta^*, t, -\delta)=0$, such $t^*$ must exist due to continuity. The uniqueness comes from $R_1([k], \theta^*, t, -\delta)$, as a function of $t$, is strictly increasing on $(-\infty, \theta_1^*]$ and $R_2([n]\backslash[k], \theta^*, t, -\delta)$, as a function of $t$, is strictly decreasing on $[\theta^*_n, +\infty)$ and $\theta_1^*\geq\theta_n^*$.

Define
\begin{align}
S_1(t)=\left\{i\in[n]:i\leq k, (\theta_i^* - t)_+\leq(\log n/np)^{1/4}\right\},\label{eqn:S_1_t}\\
S_2(t)=\left\{i\in[n]:i\geq k + 1, (t-\theta_i^*)_+\leq(\log n/np)^{1/4}\right\}.\nonumber
\end{align}
%$$$$
%$$S_2(t)=\left\{i\in[n]:i\geq k + 1, (t-\theta_i^*)_+\leq(\log n/np)^{1/4}\right\}.$$
Since we assume $\inf_t({R_1([k],\theta^*, t, -\delta) + R_2([n]\backslash[k],\theta^*, t, -\delta)})\rightarrow\infty$, we must have  
\begin{align}\label{eq:mle-lower-condition-case}
R_1([k],\theta^*, t^*, -\delta)\to\infty
\end{align}
 and hence,
\begin{equation}
\frac{R_1(S_1(t^*),\theta^*, t^*, -\delta)}{R_1([k],\theta^*, t^*, -\delta)}\geq\frac{1}{2}, \frac{R_1(S_2(t^*),\theta^*, t^*, -\delta)}{R_1([n]\backslash[k],\theta^*, t^*, -\delta)}\geq\frac{1}{2}.\label{eq:S-k-same-order}
\end{equation}
This is because $R_1([k],\theta^*, t^*, -\delta)-R_1(S_1(t^*),\theta^*, t^*, -\delta)\leq n^{-6}$ and $R_2([n]\backslash[k],\theta^*, t^*, -\delta)-R_2(S_2(t^*),\theta^*, t^*, -\delta)\leq n^{-6}$ by the definition of $S_1(t^*)$, $S_2(t^*)$ and $np/\log n\to\infty$.

Now by Lemma \ref{lem:anderson_lower_bound}, we have
\begin{align}
\nonumber%\sup_{\substack{r\in\S_n\\\theta\in\Theta(k,\Delta,\kappa)}}\mathbb{E}_{(\theta,r)}
\h_k(\wh{r},r^*) & \geq  \frac{1}{k}\min\left(\sum_{i=1}^k\indc{\wh{\theta}_i< t^*},\sum_{i=k+1}^n\indc{\wh{\theta}_i> t^*}\right) \\
& \geq  \frac{1}{k}\min\left(\sum_{i\in S_1(t^*)}\indc{\wh{\theta}_i< t^*},\sum_{i\in S_2(t^*)}\indc{\wh{\theta}_i> t^*}  \right)\label{eq:H-lower-reduce-mle}.
\end{align}
It suffices to show there exists some constant $C>0$ such that
\begin{align}
&\mathbb{P}_{(\theta^*,r^*)}\left(\frac{1}{k}\sum_{i\in S_1(t^*)}\indc{\wh{\theta}_i< t^*} \geq \frac{4C}{k}R_1(S_1(t^*),\theta^*, t^*, -\delta) \right) \geq 3/4\label{eqn:H_lower_1_case_mle}\\
\text{and }& \mathbb{P}_{(\theta^*,r^*)}\left(\frac{1}{k}\sum_{i\in S_2(t^*)}\indc{\wh{\theta}_i> t} \geq \frac{4C}{k} R_2(S_2(t^*),\theta^*, t^*, -\delta)\right) \geq 3/4.\label{eqn:H_lower_2_case_mle}
\end{align}
This is because
\begin{align}
 \nonumber&\mathbb{E}_{(\theta^*,r^*)}\h_k(\wh{r},r^*)\\
 \nonumber&\geq  C\frac{R_1([k], \theta^*, t^*, -\delta)+R_2([n]\backslash[k], \theta^*, t^*, -\delta)}{k}\\
 &\quad\quad\times\mathbb{P}_{(\theta^*,r^*)}\left(\h_k(\wh{r},r^*)  \geq C\frac{R_1([k], \theta^*, t^*, -\delta)+R_2([n]\backslash[k], \theta^*, t^*, -\delta)}{k}\right)\label{eq:markov-mle}\\
&\geq C\frac{R_1([k], \theta^*, t^*, -\delta)+R_2([n]\backslash[k], \theta^*, t^*, -\delta)}{k}\mathbb{P}_{(\theta^*,r^*)}\left(\substack{\frac{\sum_{i\in S_1(t^*)}\indc{\wh{\theta}_i< t^*}}{k} \geq \frac{2C}{k}R_1([k],\theta^*, t^*, -\delta)\text{ and }\label{eq:H-lower-R-equal}\\\frac{\sum_{i\in S_2(t^*)}\indc{\wh{\theta}_i> t}}{k} \geq \frac{2C}{k} R_2([n]\backslash[k],\theta^*, t^*, -\delta)}\right)\\
&\geq C\frac{R_1([k], \theta^*, t^*, -\delta)+R_2([n]\backslash[k], \theta^*, t^*, -\delta)}{k}\mathbb{P}_{(\theta^*,r^*)}\left(\substack{\frac{\sum_{i\in S_1(t^*)}\indc{\wh{\theta}_i< t^*}}{k} \geq \frac{4C}{k}R_1(S_1(t^*),\theta^*, t^*, -\delta)\text{ and }\\\frac{\sum_{i\in S_2(t^*)}\indc{\wh{\theta}_i> t}}{k} \geq \frac{4C}{k} R_2(S_2(t^*),\theta^*, t^*, -\delta)}\right)\label{eq:use-S-mle}\\
&\geq\frac{C}{2}\frac{R_1([k], \theta^*, t^*, -\delta)+R_2([n]\backslash[k], \theta^*, t^*, -\delta)}{k}\label{eq:final-lower-mle}.
\end{align}
Therefore, we obtain the desired conclusion. (\ref{eq:markov-mle}) is a consequence of Markov inequality; (\ref{eq:H-lower-R-equal}) comes from (\ref{eq:H-lower-reduce-mle}) and the choice of $t^*$; (\ref{eq:use-S-mle}) is due to (\ref{eq:S-k-same-order}); (\ref{eqn:H_lower_1_case_mle}) and (\ref{eqn:H_lower_2_case_mle}) lead to (\ref{eq:final-lower-mle}).

In the rest of the proof, we are going to establish (\ref{eqn:H_lower_1_case_mle}) and then (\ref{eqn:H_lower_2_case_mle}) can be proved similarly. Define 
\begin{align}
&S_1^\prime(\rho, t^*)= \left\{i\in S_1(t^*):  \text{$\rho\abs{S_1(t^*)}$ indices in $S_1(t^*)$  with the smallest }\frac{(\theta_i^*-t^*)_+^2}{V_i(\theta^*)}\right\}\label{eq:mle-def-S1-prime}
\end{align}
for some small enough constant $\rho>0$ to be specified later. That is, $S_1^\prime(\rho, t^*)$ is a subset of $S_1(t^*)$ of size $\rho\abs{S_1(t^*)}$ with the smallest $\frac{(\theta_i^*-t^*)_+^2}{V_i(\theta^*)}$ values.
We remark that condition (\ref{eq:mle-lower-condition-case}) and (\ref{eq:S-k-same-order}) necessarily imply $\abs{S_1^\prime(\rho, t^*)}\to\infty$ when $\rho$ is a constant. We shall also assume $\rho \abs{S_1^\prime(\rho, t^*)}$ is an integer. Furthermore, note that the definition of $S_1^\prime(\rho, t^*)$ implies:
\begin{equation}
R_1(S_1(t^*), \theta^*, t^*, -\delta)\geq R_1(S_1^\prime(\rho, t^*), \theta^*, t^*, -\delta)\geq\rho R_1(S_1(t^*), \theta^*, t^*, -\delta).\label{eq:lower-R-rho-mle}
\end{equation}
Therefore, to establish (\ref{eqn:H_lower_1_case_mle}), we only need to show
\begin{equation}
\mathbb{P}_{(\theta^*,r^*)}\left(\sum_{i\in S_1^\prime(\rho, t^*)}\indc{\wh{\theta}_i< t^*} \geq C^\prime R_1(S_1^\prime(\rho, t^*),\theta^*, t^*, -\delta) \right) \geq 3/4.\label{eq:H-lower-1-case-reduce-mle}
\end{equation}
for some constant $C^\prime>0$. The remaining proof is then devoted to proving (\ref{eq:H-lower-1-case-reduce-mle}).

Recall the definition of $\bar \theta$ in (\ref{eqn:bar_theta_def}). Define $\tilde \Delta_i = (\theta^*_i - t^*)_+ \vee \alpha \sqrt{\frac{1}{npL}}$ where $\alpha$  is some large enough constant to be determined later.
Define the event $\mathcal{F}_i$ as
\begin{align*}
\mathcal{F}_i =\cbr{  |\wh{\theta}_i-\bar{\theta}_i| \leq \frac{\delta_0}{3}\tilde{\Delta}_i,  \frac{|f^{(i)}(\theta_i^*|\wh{\theta}_{-i})-f^{(i)}(\theta_i^*|\theta^*_{-i})|}{g^{(i)}(\theta_i^*|\theta_{-i}^*)} \leq \frac{\delta_0}{3}\tilde{\Delta}_i,  \frac{\left|g^{(i)}(\theta_i^*|\wh{\theta}_{-i})-g^{(i)}(\theta_i^*|\theta_{-i}^*)\right|}{g^{(i)}(\theta_i^*|\theta_{-i}^*)} \leq \frac{\delta_0}{3}}.
\end{align*}
When $(\theta^*_i -t^*)^2_+ npL >\alpha$, using a similar argument that leads to (\ref{eq:MLE-local-approx-1-case})-(\ref{eq:left-case}), we can show that there exists some constant $\delta_0>0$, such that
\begin{equation}
\mathbb{P}_{(\theta^*,r^*)} (\mathcal{F}_i) \geq 1-\left(  O(n^{-7})+\exp\left(-\tilde \Delta_i^2npL\frac{np}{\log n}\right)+\exp\left(-\tilde \Delta_i^{3/2}npL\right)   \right). \label{eq:indi-F-i-mle}
\end{equation}
When $(\theta^*_i -t^*)^2_+ npL \leq  \alpha$, we can show
\begin{equation}
\mathbb{P}_{(\theta^*,r^*)} (\mathcal{F}_i) \geq 1-\left(O(n^{-7})+e^{- (npL)^{1/4}}+e^{-\sqrt{\log n}}\right). \label{eq:indi-F-i-2-mle}
\end{equation}
instead. To establish it, we can choose $x=(npL)^{1/4}$ in (\ref{eq:bound-f-0}) and $x=\sqrt{\log n}$ in (\ref{eq:bound-f-4}) and then follow the same proof of (\ref{eq:MLE-local-approx-1}), (\ref{eq:MLE-local-approx-2}), and (\ref{eq:MLE-local-approx-3}) as in the proof of Theorem \ref{thm:MLE-ranking}. In both cases, this $\delta_0$ can be made arbitrarily small by setting $\alpha$ large. 

Assuming $\mathcal{F}_i$ is true, we can use arguments similar to the establishment of  (\ref{eq:left}) to have
\begin{align*}
\indc{\wh \theta_i < t^*  } & \geq  \indc{\frac{\sum_{j\in[n]\backslash\{i\}}A_{ji}(\bar{y}_{ij}-\psi(\theta_i^*-\theta_j^*))}{\sum_{j\in[n]\backslash\{i\}}A_{ji}\psi^\prime(\theta_j^*-\theta_i^*)} \leq -(1+ \delta_0)(\theta_i^*-t^*)_+ }.\label{eqn:mle_lower_L} 
\end{align*}
Define the RHS of the above display as $L_i$.
Then we have shown that
\begin{eqnarray}
\sum_{i\in S_1^\prime(\rho, t^*)} \indc{\wh \theta_i < t^*  }\geq \sum_{i\in S_1^\prime(\rho, t^*)}L_i\mathbb{I}_{\mathcal{F}_i} \geq \sum_{i\in S_1^\prime(\rho, t^*)}L_i - \sum_{i\in S_1^\prime(\rho, t^*)}\mathbb{I}_{\mathcal{F}_i^c}.\label{eq:mle-lower-pi}
\end{eqnarray}
By (\ref{eq:indi-F-i-mle}) and (\ref{eq:indi-F-i-2-mle}), we have
\begin{align*}
&\mathbb{E}\left(\sum_{i\in S_1^\prime(\rho, t^*)}\mathbb{I}_{\mathcal{F}_i^c}\right)\\
&\leq O(n^{-6}) + \sum_{i:i\in S_1^\prime(\rho, t^*), (\theta_i^*-t^*)_+^2npL>\alpha} \br{\exp\left(-\tilde \Delta^2npL\frac{np}{\log n}\right)+\exp\left(-\tilde \Delta^{3/2}npL\right)}\\
&\quad + \sum_{i:i\in S_1^\prime(\rho, t^*), (\theta_i^*-t^*)_+^2npL\leq\alpha} \br{\exp\left(-(npL)^{1/4}\right)+\exp\left(-\sqrt{\log n}\right)}.
\end{align*}
Using $\theta_i^*-t^*\leq(\log n/np)^{1/4}$ for $i\in S_1(t^*)$ and $np/\log n\to\infty$, we see that the above bound is of smaller order than
$$n^{-5.9}+\sum_{i\in S_1^\prime(\rho, t^*)}\exp\left[-\frac{\tilde{\Delta}_i^2npL}{2\overline{V}_i(\theta^*)}\left(\left(\frac{np}{\log n}\right)^{1/9}\wedge(\log n)^{1/5}\right)\right],$$
and we can use Markov's inequality and obtain
\begin{equation}
\mathbb{P}_{(\theta^*,r^*)}\left(\sum_{i\in S_i^\prime(t^*)}\mathbb{I}_{\mathcal{F}_i^c} \leq n^{-5.9}+\sum_{i\in S_1^\prime(\rho, t^*)}\exp\left[-\frac{\tilde{\Delta}_i^2npL}{2V_i(\theta^*)}\left(\left(\frac{np}{\log n}\right)^{1/9}\wedge(\log n)^{1/5}\right)\right]\right) \geq 1-o(1). \label{eq:etos-mle}
\end{equation}

Now to lower bound $\sum_{i\in S_1^\prime(\rho, t^*)}L_i$, we define 
\begin{align}
\mathcal{A} = \Bigg\{A: \forall i\in S_1(t^*), &\left|\frac{\sum_{j\in[n]\backslash\{i\}}A_{ij}\psi'(\theta_i^*-\theta_j^*)}{p\sum_{j\in[n]\backslash\{i\}}\psi'(\theta_i^*-\theta_j^*)}-1\right|\leq \delta_0,\label{eqn:A_1_mle} \\
& \abs{\sum_{j\in S_1^\prime(\rho, t^*)}A_{ji}\psi'(\theta_i^*-\theta_j^*)}\leq 2\rho kp + 10\log n \Bigg\}.\label{eqn:A_2_mle}
\end{align}

By Bernstein's inequality and union bound, we have $\mathbb{P}(A\in \mathcal{A})\geq 1- O(n^{-10})$. From now on, we use the notation $\mathbb{P}_A$ for the conditional probability $\mathbb{P}_{(\theta^*,r^*)}(\cdot|A)$ given $A$.
For any $s>0$,
\begin{align}
\mathbb{P}_{(\theta^*,r^*)}&\left(\sum_{i\in S_1^\prime(\rho, t^*)}L_i \geq s \right) \geq  \mathbb{P}(A\in\mathcal{A}) \inf_{A\in \mathcal{A}}\mathbb{P}_{A}\left(\sum_{i\in S_1^\prime(\rho, t^*)}L_i \geq s \right).  \label{eqn:mle_final_1}
\end{align}
Now we study  $\mathbb{P}_{A}\left(\sum_{i\in S_1^\prime(\rho, t^*)}L_i \geq s \right)$. Define $S = [n]\backslash S_1^\prime(\rho, t^*)$. Note that for each $i\in S_1^\prime(\rho, t^*)$, we have $L_i \geq L_{i,1}- L_{i,2}-L_{i,3}$, where
\begin{align*}
L_{i,1} & =\indc{\frac{\sum_{j\in S}A_{ji}(\bar{y}_{ij}-\psi(\theta_i^*-\theta_j^*))}{\sum_{j\in[n]\backslash\{i\}}A_{ji}\psi^\prime(\theta_j^*-\theta_i^*)} \leq -(1+2\delta')(1+ \delta_0)\tilde{\Delta}_i},\\
L_{i,2} & = \indc{\frac{\sum_{j\in S_1^\prime(\rho, t^*):j <i}A_{ji}(\bar{y}_{ij}-\psi(\theta_i^*-\theta_j^*))}{\sum_{j\in[n]\backslash\{i\}}A_{ji}\psi^\prime(\theta_j^*-\theta_i^*)} \geq \delta'(1+\delta_0)\tilde{\Delta}_i},\\
L_{i,3} & = \indc{\frac{\sum_{j\in S_1^\prime(\rho, t^*):i <j}A_{ji}(\bar{y}_{ij}-\psi(\theta_i^*-\theta_j^*))}{\sum_{j\in[n]\backslash\{i\}}A_{ji}\psi^\prime(\theta_j^*-\theta_i^*)} \geq \delta'(1+\delta_0)\tilde{\Delta}_i}
\end{align*}
for some small constant $\delta'>0$ whose value will be determined later. We are going to control each term separately. 

\textbf{(1).}
Analysis of $L_{i,1}$. Note that conditional on $A$, $\{L_{i,1}\}_{i\in S_1^\prime(\rho, t^*)}$ are all independent Bernoulli random variables. We have $L_{i,1} \sim \text{Bernoulli}(p_i)$, where $p_i = \E_{(\theta^*,r^*)} (L_{i,1}| A)$. By Chebyshev's inequality, we have 
\begin{align*}
\mathbb{P}_{A}\left(\sum_{i\in S_1^\prime(\rho, t^*)}L_{i,1} \geq \frac{1}{2} \sum_{i\in S_1^\prime(\rho, t^*)}p_{i}\right) \geq 1-\frac{4}{\sum_{i\in S_1^\prime(\rho, t^*)}p_{i}}.
\end{align*}
By Lemma \ref{lem:mle_lower_tail_prop_case}, we can lower bound each $p_i$ by
\begin{align}
p_i &= \p_{A} \left(\frac{\sum_{j\in S}A_{ji}(\bar{y}_{ij}-\psi(\theta_i^*-\theta_j^*))(1+e^{\theta_j^*-\theta_i^*})}{\sum_{j\in[n]\backslash\{i\}}A_{ji}\psi(\theta_j^*-\theta_i^*)} \leq -(1+2\delta')(1+\delta_0)^2\tilde{\Delta}_i\right)   \nonumber\\
& \geq C_1\exp\left(-\frac{1+\delta_2}{2}\frac{\tilde{\Delta}_i^2npL}{V_i(\theta^*)}-C_1^\prime\sqrt{\frac{\tilde{\Delta}_i^2npL}{V_i(\theta^*)}}\right),\nonumber
\end{align}
for some constants $C_1,C_1'>0$ and some small constant $\delta_2>0$. Note that $\delta_2$ can be an arbitrarily small constant by making $\delta^\prime$ and $\rho$ small as well as making $\alpha$ large. Thus we can choose $\delta^\prime, \rho$ small  enough  and $\alpha$ large enough to let $\delta_2<\delta/2$. Then we have
\begin{align}
\nonumber&\sum_{i\in S_1^\prime(\rho, t^*)} p_i \geq C_1\sum_{i\in S_1^\prime(\rho, t^*)}\exp\left(-\frac{1+\delta_2}{2}\frac{\tilde{\Delta}_i^2npL}{V_i(\theta^*)}-C_1^\prime\sqrt{\frac{\tilde{\Delta}_i^2npL}{V_i(\theta^*)}}\right)\\ &\geq C_1R_1(S_1^\prime(\rho, t^*), \theta^*, t^*, -\delta)\label{eq:alpha-large-mle}\\
&\geq C_1\rho R_1(S_1(t^*), \theta^*, t^*, -\delta).\label{eqn:p_i_lower_bound_mle}
\end{align}
where (\ref{eq:alpha-large-mle}) can be achieved by setting $\alpha$ large and (\ref{eqn:p_i_lower_bound_mle}) comes from (\ref{eq:lower-R-rho-mle}). As a result, under the condition (\ref{eq:mle-lower-condition-case}), we have $\sum_{i\in S_1^\prime(\rho, t^*)}p_i\rightarrow\infty$.

Hence, we have proved
\begin{align*}
\inf_{A\in \mathcal{A}}\mathbb{P}_{A}\left(\sum_{i\in S_1^\prime(\rho, t^*)}L_{i,1} \geq  \frac{1}{2}C_1\sum_{i\in S_1^\prime(\rho, t^*)}\exp\left(-\frac{1+\delta_2}{2}\frac{\tilde{\Delta}_i^2npL}{V_i(\theta^*)}-C_1^\prime\sqrt{\frac{\tilde{\Delta}_i^2npL}{V_i(\theta^*)}}\right)\right) \geq  1- o(1).
\end{align*}

\textbf{(2).} Analysis of $L_{i,2}$. By (\ref{eqn:A_1_mle})-(\ref{eqn:A_2_mle}) and Bernstein's inequality, we can bound $\E (L_{i,2}|A)$ by
\begin{align*}
& \exp\left( - \frac{\left(\delta'(1+\delta_0)^2\tilde{\Delta}_iL \sum_{j\in[n]\backslash\{i\}}A_{ji}\psi^\prime(\theta_j^*-\theta_i^*)   \right)^2}{2 \left( L \sum_{j\in S_1^\prime(\rho, t^*):j<i}A_{ji}\psi'(\theta_i^*-\theta_j^*)  + \frac{1}{3}\delta'(1+\delta_0)^2\tilde{\Delta}_i L  \sum_{j\in[n]\backslash\{i\}}A_{ji}\psi^\prime(\theta_j^*-\theta_i^*)   \right)}\right)\\
&\leq \exp\left( - \frac{\left(\delta'(1+\delta_0)^2\tilde{\Delta}_iL \sum_{j\in[n]\backslash\{i\}}p\psi^\prime(\theta_j^*-\theta_i^*)   \right)^2}{4 \left( 2L \rho kp +10\log n  + \frac{1}{3}\delta'(1+\delta_0)^2\tilde{\Delta}_i L  \sum_{j\in[n]\backslash\{i\}}p\psi^\prime(\theta_j^*-\theta_i^*)   \right)}\right).
\end{align*}
Now we set $\delta' = \rho^{1/8}$, and make $\rho$ small enough to ensure (\ref{eqn:p_i_lower_bound_mle}). Then, there exists some constants $C_2,C_3>0$ such that
$$
\E (L_{i,2}|A) \leq \exp\left(-C_2\rho^{-\frac{1}{2}}npL \tilde{\Delta}_i^2\right) \leq \exp\left(-C_3\rho^{-1/2}\frac{\tilde{\Delta}_i^2npL}{2V_i(\theta^*)}\right).
$$
due to $\tilde{\Delta}_i=o(1)$ and $np/\log n\to\infty$. Then,
$$
\E \left(\sum_{i\in S_1^\prime(\rho, t^*)}L_{i,2}\Bigg|A\right) \leq \sum_{i\in S_1^\prime(\rho, t^*)}\exp\left(-C_3\rho^{-1/2}\frac{\tilde{\Delta}_i^2npL}{2V_i(\theta^*)}\right).
$$
By Markov inequality, we have
\begin{equation}
\inf_{A\in \mathcal{A}}\mathbb{P}_{A}\left(\sum_{i\in S_1^\prime(\rho, t^*)}L_{i,2} \geq \sum_{i\in S_1^\prime(\rho, t^*)}\exp\left(-\frac{1}{2}C_3\rho^{-1/2}\frac{\tilde{\Delta}_i^2npL}{2V_i(\theta^*)}\right) \right) \leq \frac{\sum_{i\in S_1^\prime(\rho, t^*)}\exp\left(-C_3\rho^{-1/2}\frac{\tilde{\Delta}_i^2npL}{2V_i(\theta^*)}\right)}{\sum_{i\in S_1^\prime(\rho, t^*)}\exp\left(-\frac{1}{2}C_3\rho^{-1/2}\frac{\tilde{\Delta}_i^2npL}{2V_i(\theta^*)}\right)}. \label{eqn:L_i2_upper_mle}
\end{equation}

\textbf{(3).} Analysis of $L_{i,3}$. By a similar argument, we also have
\begin{equation}
\inf_{A\in \mathcal{A}}\mathbb{P}_{A}\left(\sum_{i\in S_1^\prime(\rho, t^*)}L_{i,3} \geq \sum_{i\in S_1^\prime(\rho, t^*)}\exp\left(-\frac{1}{2}C_3\rho^{-1/2}\frac{\tilde{\Delta}_i^2npL}{2V_i(\theta^*)}\right) \right) \leq \frac{\sum_{i\in S_1^\prime(\rho, t^*)}\exp\left(-C_3\rho^{-1/2}\frac{\tilde{\Delta}_i^2npL}{2V_i(\theta^*)}\right)}{\sum_{i\in S_1^\prime(\rho, t^*)}\exp\left(-\frac{1}{2}C_3\rho^{-1/2}\frac{\tilde{\Delta}_i^2npL}{2V_i(\theta^*)}\right)}. \label{eqn:L_i3_upper_mle}
\end{equation}

~\\
\indent Now we can combine the above analyses of $L_{i,1}$, $L_{i,2}$ and $L_{i,3}$.
Since we are allowed to choose $\rho$ to be an arbitrarily small constant, we shall make 
$$\sum_{i\in S_1^\prime(\rho, t^*)}\exp\left(-\frac{1}{2}C_3\rho^{-1/2}\frac{\tilde{\Delta}_i^2npL}{2V_i(\theta^*)}\right)\leq\frac{1}{8}C_1\sum_{i\in S_1^\prime(\rho, t^*)}\exp\left(-\frac{1+\delta_2}{2}\frac{\tilde{\Delta}_i^2npL}{V_i(\theta^*)}-C_1^\prime\sqrt{\frac{\tilde{\Delta}_i^2npL}{V_i(\theta^*)}}\right)$$
and
$$\frac{\sum_{i\in S_1^\prime(\rho, t^*)}\exp\left(-C_3\rho^{-1/2}\frac{\tilde{\Delta}_i^2npL}{2V_i(\theta^*)}\right)}{\sum_{i\in S_1^\prime(\rho, t^*)}\exp\left(-\frac{1}{2}C_3\rho^{-1/2}\frac{\tilde{\Delta}_i^2npL}{2V_i(\theta^*)}\right)}\leq\frac{1}{16}.$$
Thus, we have
\begin{align}
\inf_{A\in \mathcal{A}}\mathbb{P}_{A}\left(\sum_{i\in S_1^\prime(\rho, t^*)}L_i \geq C_4\sum_{i\in S_1^\prime(\rho, t^*)}\exp\left(-\frac{1+\delta_2}{2}\frac{\tilde{\Delta}_i^2npL}{V_i(\theta^*)}-C_1^\prime\sqrt{\frac{\tilde{\Delta}_i^2npL}{V_i(\theta^*)}}\right) \right)\geq \frac{7}{8}-o(1),\label{eqn:k_1_mle}
\end{align}
for some constant $C_4 >0$. Then  (\ref{eq:mle-lower-pi}), (\ref{eq:etos-mle}), (\ref{eqn:mle_final_1}) together with (\ref{eq:mle-lower-condition-case}) lead to
\begin{align}
\mathbb{P}_{(\theta^*,r^*)}\left(\sum_{i\in S_1^\prime(\rho, t^*)}\indc{\wh{\theta}_i <t^*}\geq \frac{C_4}{2}\sum_{i\in S_1^\prime(\rho, t^*)}\exp\left(-\frac{1+\delta_2}{2}\frac{\tilde{\Delta}_i^2npL}{V_i(\theta^*)}-C_1^\prime\sqrt{\frac{\tilde{\Delta}_i^2npL}{V_i(\theta^*)}}\right)\right)\geq \frac{7}{8}-o(1).\label{eqn:k_1_1}
\end{align}
Finally, (\ref{eq:H-lower-1-case-reduce-mle}) follows from (\ref{eqn:p_i_lower_bound_mle}) which completes the proof.
\end{proof}

We state  Lemma \ref{lem:mle_lower_tail_prop_case} to close this section. Its proof is essentially the same as the proof of Lemma \ref{lem:spectral_lower_tail_prop} and hence is omitted here.
\begin{lemma}\label{lem:mle_lower_tail_prop_case}
Assume $\frac{np}{\log n}\rightarrow\infty$, $\kappa =O(1)$. Recall the definition of $S_1^\prime(\rho, t^*)$ in (\ref{eq:mle-def-S1-prime}), $S=[n]\backslash S_1^\prime(\rho,t^*)$ and $\tilde{\Delta}_i=(\theta_i^*-t^*)_+\vee\alpha\sqrt{\frac{1}{npL}}$. There exists some constants $C_1, C_2>0$ such that for any small constant $0.1>\tilde{\delta}>0$, there exists constant $\delta_1>0$ such that for any constant $\alpha>0$, $i\in S_1^\prime(\rho,t^*)$, any $A\in\mathcal{A}$ where $\mathcal{A}$ is defined in (\ref{eqn:A_1_mle})-(\ref{eqn:A_2_mle}), any $\theta^*\in\Theta(k,0,\kappa)$ and any $r^*\in\S_n$, we have
\begin{align}
 &\p_{(\theta^*, r^*)}\left(\frac{\sum_{j \in S}A_{ji}(\bar{y}_{ij}-\psi(\theta_{r_i^*}^*-\theta_{r_j^*}^*))}{\sum_{j\in[n]\backslash\{i\}}A_{ji}\psi^\prime(\theta_{r_j^*}^*-\theta_{r_i^*}^*)} \leq -(1+\tilde{\delta}) \tilde{\Delta}_i\Bigg| A\right)   \nonumber\\
& \geq C_1\exp\left(-\frac{1+\delta_1}{2}\frac{\tilde{\Delta}_i^2npL}{V_{r_i^*}(\theta^*)}-C_2\sqrt{\frac{\tilde{\Delta}_i^2npL}{V_{r_i^*}(\theta^*)}}\right).\label{eqn:mle_lower_tail_prop_case}
\end{align}
Moreover, $\delta_1$ is able to be arbitrarily small if $\tilde{\delta}$ and $\rho$ are small enough. 
\end{lemma}

\subsection{Proof of Theorem \ref{thm:spectral-ranking-theta}}
We first give Lemma \ref{lem:spec-tail-case} to characterize entrywise tail behaviors of the spectral method (\ref{eq:spec-P}) which is crucial to the upper bound in Theorem \ref{thm:spectral-ranking-theta}.
\begin{lemma}\label{lem:spec-tail-case}
Assume $\frac{np}{\log n}\rightarrow\infty$ and $\kappa=O(1)$. Then, for the rank vector $\wh{r}$ that is induced by the stationary distribution of the Markov chain (\ref{eq:spec-P}), for any small constant $0.1>\delta>0$, there exists some constant $C>0$, such that for any $t\in\mathbb{R}$, any $\theta^*\in\Theta(k,0,\kappa)$, $r^*\in\S_n$, we have
\begin{equation}
\p_{(\theta^*, r^*)}\left(\wh{\pi}_i\leq \frac{e^t}{\sum_{j\in[n]}e^{\theta_j^*}}\right)\leq C\exp\left(-\frac{(1-\delta)(\theta_{r_i^*}^*-t)_{+}^2npL}{2\overline{V}_{r_i^*}(\theta^*)}\right)+Cn^{-4}, r_i^*\leq k;\label{eq:spec-tail-case-1}
\end{equation}
\begin{equation}
\p_{(\theta^*, r^*)}\left(\wh{\pi}_i\geq \frac{e^t}{\sum_{j\in[n]}e^{\theta_j^*}}\right)\leq C\exp\left(-\frac{(1-\delta)(t-\theta_{r_i^*}^*)_{+}^2npL}{2\overline{V}_{r_i^*}(\theta^*)}\right)+Cn^{-4}, r_i^*\geq k+1\label{eq:spec-tail-case-2}
\end{equation}
\end{lemma}
\begin{proof}
The proof follows the proof of Theorem \ref{thm:spectral-ranking} with slight modifications. Without loss of generality, we can assume $r_i^*=i$ for all $\in[n]$. Define $\bar \Delta_i$ as in (\ref{eq:delta-bar-i-case}). We only need to prove (\ref{eq:spec-tail-case-1}) since (\ref{eq:spec-tail-case-2}) can be proved similarly.

Consider any $m\in[k]$. When $(\theta^*_m-t)_{+}^2npL\leq c^\prime$ for some large enough constant to be specified later, we can directly bound the probability using the trivial bound $1$. Thus, we only need to  consider the regime when $(\theta^*_m-t)_{+}^2npL>c^\prime$. 

Following the proof of Theorem \ref{thm:spectral-ranking}, we have (\ref{eqn:pi_bar_def})-(\ref{eqn:spectral_equationtwo}) and (\ref{eq:easy-de-pi}) hold. Note that we now have $\bar{\Delta}_m^2Lnp>c'$ instead of $\bar{\Delta}_m^2Lnp \rightarrow\infty$ which is needed  in the proof of Theorem  \ref{thm:spectral-ranking}. As a consequence, we now have (\ref{eq:pi-hat-bar-approx}) hold with $\delta =4C_4e^\kappa/\sqrt{c'}$ instead of some $o(1)$ as in the proof of Theorem \ref{thm:spectral-ranking}. To sum up, with this $\delta$, we have
\begin{align}
\frac{|\wh{\pi}_m-\bar{\pi}_m|}{\pi_m^*} \leq \delta(1-e^{-\bar{\Delta}_m}), \label{eq:pi-hat-bar-approx-case} \\
\left|\frac{\sum_{j\in[n]\backslash\{m\}}A_{jm}\bar{y}_{jm}}{\sum_{j\in[n]\backslash\{m\}}A_{jm}\psi(\theta_j^*-\theta_m^*)}-1\right| \leq \delta, \label{eq:easy-de-pi-case}
\end{align}
hold with probability at least $1-O(n^{-4})-\exp\left(-\bar{\Delta}_m^2npL\frac{np}{\log n}\right)-\exp\left(-\bar{\Delta}_m^2npL\sqrt{\frac{npL}{\log n}}\right)$. We can make $\delta$ to be an arbitrarily small constant by setting $c'$ large as $\kappa=O(1)$. 

Then for any $i\leq k$, by the same argument as in the proof of Theorem \ref{thm:spectral-ranking}, we have
\begin{eqnarray*}
&& \mathbb{P}\left(\wh{\pi}_i \leq \frac{e^{t}}{\sum_{j=1}^ne^{\theta_j^*}}\right) \\
&=& \mathbb{P}\left(\frac{\wh{\pi}_i-\pi_i^*}{\pi_i^*} \leq e^{-(\theta_i^*-t)}-1\right) \\
&\leq& \mathbb{P}\left(\frac{\wh{\pi}_i-\pi_i^*}{\pi_i^*} \leq e^{-\bar{\Delta}_i}-1\right) \\
&\leq& \mathbb{P}\left(\frac{\sum_{j\in[n]\backslash\{i\}}A_{ji}(\bar{y}_{ij}-\psi(\theta_i^*-\theta_j^*))(1+e^{\theta_j^*-\theta_i^*})}{\sum_{j\in[n]\backslash\{i\}}A_{ji}\psi(\theta_j^*-\theta_i^*)} \leq -(1-\delta)^2(1-e^{-\bar{\Delta}_i})\right) \\
&& + O(n^{-4})+\exp\left(-\bar{\Delta}_i^2npL\frac{np}{\log n}\right)+\exp\left(-\bar{\Delta}_i^2npL\sqrt{\frac{npL}{\log n}}\right),
\end{eqnarray*}
which has the same upper bound as in (\ref{eqn:spectral_equationtwo}). We then have the same (\ref{eqn:spectral_equationthree}) as in the proof of Theorem \ref{thm:spectral-ranking} which leads to
\begin{eqnarray*}
&& \mathbb{P}\left(\frac{\sum_{j\in[n]\backslash\{i\}}A_{ji}(\bar{y}_{ij}-\psi(\theta_i^*-\theta_j^*))(1+e^{\theta_j^*-\theta_i^*})}{\sum_{j\in[n]\backslash\{i\}}A_{ji}\psi(\theta_j^*-\theta_i^*)} \leq -(1-\delta)^2(1-e^{-\bar{\Delta}_i})\right) \\
&\leq& \exp\left(-\frac{(1-o(1))Lp\bar{\Delta}_i^2\left(\sum_{j\in[n]\backslash\{i\}}\psi(\theta_j^*-\theta_i^*)\right)^2}{2\sum_{j\in[n]\backslash\{i\}}\psi'(\theta_i^*-\theta_j^*)\left(1+e^{\theta_j^*-\theta_i^*}\right)^2}\right) + O(n^{-4}) \\
&=&\exp\left(-\frac{(1-\delta_2)npL\bar{\Delta}_i^2}{2\overline{V}_i(\theta^*)}\right)+O(n^{-4})\\
&\leq&\exp\left(-\frac{(1-\delta_2)npL(\theta_i^*-t)^2_+}{2\overline{V}_i(\theta^*)}\right)+O(n^{-4})
\end{eqnarray*}
with $\delta_1, \delta_2>0$ being some constant that can be arbitrarily small. The last inequality holds because when $\min\left((\theta_i^*-t)_+^2,\sqrt{\frac{\log n}{np}}\right)=\sqrt{\frac{\log n}{np}}$, the first term becomes $\exp\left(-\frac{(1-\delta_2)L\sqrt{np\log n}}{2\overline{V}_i(\theta^*)}\right)$, which can be absorbed by $O(n^{-4})$. Since $\exp\left(-\bar{\Delta}_i^2npL\frac{np}{\log n}\right)+\exp\left(-\bar{\Delta}_i^2npL\sqrt{\frac{npL}{\log n}}\right)\leq \exp\left(-\frac{(1-\delta_2)(\theta_i^*-t)_+^2npL}{2\overline{V}_i(\theta^*)}\right) + O(n^{-4})$, we have
\begin{equation}
\mathbb{P}\left(\wh{\pi}_i \leq \frac{e^{t}}{\sum_{j=1}^ne^{\theta_j^*}}\right)\leq 2\exp\left(-\frac{(1-\delta_2)(\theta_i^*-t)_+^2npL}{2\overline{V}_i(\theta^*)}\right) + O(n^{-4}), \label{eq:prob-bound-1-spec}
\end{equation}
for all $i\leq k$. The proof is complete.
\end{proof}
%by the same argument as in the proof of Theorem \ref{thm:MLE-ranking}, we have
%
%We then have the same (\ref{eq:left}), (\ref{eq:apply-bern}), and the event $\mathcal{A}_i$ as in the proof of Theorem \ref{thm:MLE-ranking}. As a result,  

\begin{proof}[Proof of  (\ref{eq:spectral-ranking-theta-u})  of  Theorem \ref{thm:spectral-ranking-theta}]
The upper bound (\ref{eq:spectral-ranking-theta-u}) is a straightforward consequence of Lemma \ref{lem:spec-tail-case} in the same way as the proof of (\ref{eq:MLE-ranking-theta-u}) of Theorem \ref{thm:MLE-ranking-theta}, and hence is omitted here.
\end{proof}

The rest of the section focuses on the lower bound (\ref{eq:MLE-ranking-theta-l}). The proof follows the proof of Theorem \ref{thm:exact-lower} with some modification and is also very similar to the proof of (\ref{eq:MLE-ranking-theta-l}) of  Theorem \ref{thm:MLE-ranking-theta}. We include it below for completeness.
\begin{proof}[Proof of  (\ref{eq:spectral-ranking-theta-l}) of  Theorem \ref{thm:spectral-ranking-theta}]
To prove the lower bound (\ref{eq:spectral-ranking-theta-l}), we are going to show
\begin{equation}
\E_{(\theta^*, r^*)}\h_k(\wh{r},r^*)\gtrsim\frac{\overline{R}_1([k],\theta^*, t^*, -\delta)+\overline{R}_2([n]\backslash[k],\theta^*, t^*, -\delta)}{k}\label{eq:lower-bound-spec-case}
\end{equation}
where $t^*$ is the unique solution such that $\overline{R}_1([k],\theta^*, t^*, -\delta)=\overline{R}_2([n]\backslash[k],\theta^*, t^*, -\delta)$. The existence and uniqueness of $t^*$ follow the same argument as in the proof of (\ref{eq:MLE-ranking-theta-l}) of  Theorem \ref{thm:MLE-ranking-theta}. Recall the definition of $S_1(t)$ in (\ref{eqn:S_1_t}). Since we assume $\inf_t(\overline{R}_1([k],\theta^*, t, -\delta) + \overline{R}_2([n]\backslash[k],\theta^*, t, -\delta))\rightarrow\infty$, we have
\begin{align}\label{eq:spec-lower-condition-case}
\overline{R}_1([k],\theta^*, t^*, -\delta) \rightarrow\infty. 
\end{align}

The proof of (\ref{eq:lower-bound-spec-case}) follows the proof of Theorem \ref{thm:spectral_lower}. We will omit repeated details and only present the differences. Define 
\begin{equation}
\overline{S}_1^\prime(\rho, t^*)=\left\{i\in S_1(t^*):  \text{$\rho\abs{S_1(t^*)}$ indices in $S_1(t^*)$ with the smallest }\frac{(\theta_i^*-t^*)_+^2}{\overline{V}_i(\theta^*)}\right\}\label{eq:spec-def-S1-prime}
\end{equation}
for some small enough constant $\rho>0$ to be specified later. Following the same argument as in the proof of (\ref{eq:MLE-ranking-theta-l}) of  Theorem \ref{thm:MLE-ranking-theta}, we only need to show
\begin{equation}
\mathbb{P}_{(\theta^*,r^*)}\left(\sum_{i\in \overline{S}_1^\prime(\rho, t^*)}\indc{\wh{\pi}_i< t} \geq C^\prime\overline{R}_1(\overline{S}_1^\prime(\rho, t^*),\theta^*, t^*, -\delta) \right) \geq 3/4.\label{eq:H-lower-1-case-reduce}
\end{equation}
for some constant $C^\prime>0$. The remaining   proof is then devoted to proving (\ref{eq:H-lower-1-case-reduce}).

Recall the definition of $\bar \pi$  in (\ref{eqn:pi_bar_def}). Define $\tilde \Delta_i = (\theta^*_i - t^*)_+ \vee \alpha \sqrt{\frac{1}{npL}}$ where $\alpha$  is some large enough constant to be determined later.
Define the event $\overline{\mathcal{F}}_i$ as
\begin{align*}
\overline{\mathcal{F}}_i = \left\{\frac{|\wh{\pi}_i-\bar{\pi}_i|}{ \pi^*_i} \leq  \delta_0(1-e^{-\tilde \Delta_i}) \text{ and } \left|\frac{\sum_{j\in[n]\backslash\{i\}}A_{ji}\bar{y}_{ji}}{\sum_{j\in[n]\backslash\{i\}}A_{ji}\psi(\theta_j^*-\theta_i^*)}-1\right| \leq  \delta_0\right\}.
\end{align*}
When $(\theta^*_i -t^*)^2_+ npL >\alpha$, using a similar argument that leads to (\ref{eq:pi-hat-bar-approx-case})-(\ref{eq:easy-de-pi-case}), we can show that there exists some constant $\delta_0>0$, such that
\begin{equation}
\mathbb{P}_{(\theta^*,r^*)} (\overline{\mathcal{F}}_i) \geq 1-\left(  O(n^{-4})+\exp\left(-\tilde \Delta_i^2npL\frac{np}{\log n}\right)+\exp\left(-\tilde \Delta_i^{2}npL\sqrt{\frac{npL}{\log n}}\right)   \right). \label{eq:indi-F-i-spec}
\end{equation}
When $(\theta^*_i -t^*)^2_+ npL \leq  \alpha$, we can show
\begin{equation}
\mathbb{P}_{(\theta^*,r^*)} (\overline{\mathcal{F}}_i) \geq 1-\left(O(n^{-4})+e^{- (np/\log n)^{1/2}}+e^{-\sqrt{\log n}}\right). \label{eq:indi-F-i-2-spec}
\end{equation}
instead. To establish it, we can choose $x=(np/\log n)^{1/2}$ in (\ref{eq:spec-add-x-1}) and $x=\sqrt{\log n}$ in (\ref{eq:spec-add-x-2}) and then follow the same proof of (\ref{eq:pi-hat-bar-approx}) and (\ref{eq:easy-de-pi}) as in the proof of Theorem \ref{thm:MLE-ranking}. In both cases, this $\delta_0$ can be made arbitrarily small by setting $\alpha$ large. 

Assuming $\overline{\mathcal{F}}_i$ is true, we can use arguments similar to the establishment of (\ref{eqn:spectral_lower_L}) to have
\begin{align}
&\indc{\wh \pi_i < t  } \geq\indc{\frac{\sum_{j\in[n]\backslash\{i\}}A_{ji}(\bar{y}_{ij}-\psi(\theta_i^*-\theta_j^*))(1+e^{\theta_j^*-\theta_i^*})}{\sum_{j\in[n]\backslash\{i\}}A_{ji}\psi(\theta_j^*-\theta_i^*)} \leq -(1+ \delta_0)^2\alpha\sqrt{\frac{1}{npL}} }.\label{eqn:spectral_lower_L_2}
\end{align}
Define the RHS of the above display as $\overline{L}_i$. 
\begin{eqnarray}
\sum_{i\in \overline{S}_1^\prime(\rho, t^*)} \indc{\wh \pi_i < t  }&\geq& \sum_{i\in \overline{S}_1^\prime(\rho, t^*)}\overline{L}_i\mathbb{I}_{\overline{\mathcal{F}}_i}\geq \sum_{i\in \overline{S}_1^\prime(\rho, t^*)}\overline{L}_i - \sum_{i\in \overline{S}_1^\prime(\rho, t^*)}\mathbb{I}_{\overline{\mathcal{F}}_i^c}.\label{eq:spec-lower-pi}
\end{eqnarray}
By (\ref{eq:indi-F-i-spec}) and (\ref{eq:indi-F-i-2-spec}), we have
\begin{align*}
&\mathbb{E}\left(\sum_{i\in \overline{S}_1^\prime(\rho, t^*)}\mathbb{I}_{\overline{\mathcal{F}}_i^c}\right)\\
&\leq O(n^{-3})+ \sum_{i:i\in \overline{S}_1^\prime(\rho, t^*), (\theta_i^*-t^*)_+^2npL>\alpha}\exp\left(-\tilde \Delta_i^2npL\frac{np}{\log n}\right)+\exp\left(-\tilde \Delta_i^2npL\sqrt{\frac{npL}{\log n}}\right)\\
&\quad+\sum_{i:i\in \overline{S}_1^\prime(\rho, t^*), (\theta_i^*-t^*)_+^2npL\leq\alpha}\exp\left(-(np/\log n)^{1/2}\right)+\exp\left(-\sqrt{\log n}\right).\end{align*}
Since the above bound is of smaller order than 
$$n^{-2.9}+\sum_{i\in \overline{S}_1^\prime(\rho, t^*)}\exp\left[-\frac{\tilde{\Delta}_i^2npL}{2\overline{V}_i(\theta^*)}\left(\left(\frac{np}{\log n}\right)^{1/4}\wedge(\log n)^{1/4}\right)\right],$$
we can use Markov's inequality and obtain
\begin{equation}
\mathbb{P}_{(\theta^*,r^*)}\left(\sum_{i\in S_i^\prime(t^*)}\mathbb{I}_{\overline{\mathcal{F}}_i^c} \leq n^{-2.9}+\sum_{i\in \overline{S}_1^\prime(\rho, t^*)}\exp\left[-\frac{\tilde{\Delta}_i^2npL}{2\overline{V}_i(\theta^*)}\left(\left(\frac{np}{\log n}\right)^{1/4}\wedge(\log n)^{1/4}\right)\right]\right) \geq 1-o(1). \label{eq:etos}
\end{equation}

Now to lower bound $\sum_{i\in \overline{S}_1^\prime(\rho, t^*)}\overline{L}_i$, we define 
\begin{align}
\overline{\mathcal{A}} = \Bigg\{A: \forall i\in S_1(t^*), &\left|\frac{\sum_{j\in[n]\backslash\{i\}}A_{ij}\psi'(\theta_i^*-\theta_j^*)\left(1+e^{\theta_j^*-\theta_i^*}\right)^2}{p\sum_{j\in[n]\backslash\{i\}}\psi'(\theta_i^*-\theta_j^*)\left(1+e^{\theta_j^*-\theta_i^*}\right)^2}-1\right|\leq \delta_0,\label{eqn:A_1} \\
& \left|\frac{\sum_{j\in[n]\backslash\{i\}}A_{ji}\psi(\theta_j^*-\theta_i^*)}{p\sum_{j\in[n]\backslash\{i\}}\psi(\theta_j^*-\theta_i^*)}-1\right|\leq  \delta_0 ,\label{eqn:A_2}\\
& \abs{\sum_{j\in \overline{S}_1^\prime(\rho, t^*)}A_{ji}\psi'(\theta_i^*-\theta_j^*)(1+e^{\theta_j^*-\theta_i^*})^2 }\leq 2\rho kp + 10\log n \Bigg\}.\label{eqn:A_3}
\end{align}

By Bernstein's inequality and union bound, we have $\mathbb{P}(A\in \overline{\mathcal{A}})\geq 1- O(n^{-3})$. From now on, we use the notation $\mathbb{P}_A$ for the conditional probability $\mathbb{P}_{(\theta^*,r^*)}(\cdot|A)$ given $A$.
For any $s>0$,
\begin{align}
\mathbb{P}_{(\theta^*,r^*)}&\left(\sum_{i\in \overline{S}_1^\prime(\rho, t^*)}\overline{L}_i \geq s \right) \geq  \mathbb{P}(A\in\overline{\mathcal{A}}) \inf_{A\in \overline{\mathcal{A}}}\mathbb{P}_{A}\left(\sum_{i\in \overline{S}_1^\prime(\rho, t^*)}\overline{L}_i \geq s \right).  \label{eqn:spectral_final_1}
\end{align}
Now we study  $\mathbb{P}_{A}\left(\sum_{i\in \overline{S}_1^\prime(\rho, t^*)}L_i \geq s \right)$. Define $S = [n]\backslash \overline{S}_1^\prime(\rho, t^*)$. Note that for each $i\in \overline{S}_1^\prime(\rho, t^*)$, we have $L_i \geq L_{i,1}- L_{i,2}-L_{i,3}$, where
\begin{align*}
\overline{L}_{i,1} & =\indc{\frac{\sum_{j\in S}A_{ji}(\bar{y}_{ij}-\psi(\theta_i^*-\theta_j^*))(1+e^{\theta_j^*-\theta_i^*})}{\sum_{j\in[n]\backslash\{i\}}A_{ji}\psi(\theta_j^*-\theta_i^*)} \leq -(1+2\delta')(1+ \delta_0)^2\tilde{\Delta}_i},\\
\overline{L}_{i,2} & = \indc{\frac{\sum_{j\in \overline{S}_1^\prime(\rho, t^*):j <i}A_{ji}(\bar{y}_{ij}-\psi(\theta_i^*-\theta_j^*))(1+e^{\theta_j^*-\theta_i^*})}{\sum_{j\in[n]\backslash\{i\}}A_{ji}\psi(\theta_j^*-\theta_i^*)} \geq \delta'(1+\delta_0)^2\tilde{\Delta}_i},\\
\overline{L}_{i,3} & = \indc{\frac{\sum_{j\in \overline{S}_1^\prime(\rho, t^*):i <j}A_{ji}(\bar{y}_{ij}-\psi(\theta_i^*-\theta_j^*))(1+e^{\theta_j^*-\theta_i^*})}{\sum_{j\in[n]\backslash\{i\}}A_{ji}\psi(\theta_j^*-\theta_i^*)} \geq \delta'(1+\delta_0)^2\tilde{\Delta}_i}
\end{align*}
for some small constant $\delta'>0$ whose value will be determined later. We are going to control each term separately. 

\textbf{(1).}
Analysis of $\overline{L}_{i,1}$. Note that conditional on $A$, $\{\overline{L}_{i,1}\}_{i\in \overline{S}_1^\prime(\rho, t^*)}$ are all independent Bernoulli random variables. We have $\overline{L}_{i,1} \sim \text{Bernoulli}(p_i)$, where $p_i = \E_{(\theta^*,r^*)} (\overline{L}_{i,1}| A)$. By Chebyshev's inequality, we have 
\begin{align*}
\mathbb{P}_{A}\left(\sum_{i\in \overline{S}_1^\prime(\rho, t^*)}\overline{L}_{i,1} \geq \frac{1}{2} \sum_{i\in \overline{S}_1^\prime(\rho, t^*)}p_{i}\right) \geq 1-\frac{4}{\sum_{i\in \overline{S}_1^\prime(\rho, t^*)}p_{i}}.
\end{align*}
By Lemma \ref{lem:spectral_lower_tail_prop_case}, we can lower bound each $p_i$ by
\begin{align}
p_i &= \p_{A} \left(\frac{\sum_{j\in S}A_{ji}(\bar{y}_{ij}-\psi(\theta_i^*-\theta_j^*))(1+e^{\theta_j^*-\theta_i^*})}{\sum_{j\in[n]\backslash\{i\}}A_{ji}\psi(\theta_j^*-\theta_i^*)} \leq -(1+2\delta')(1+\delta_0)^2\tilde{\Delta}_i\right)   \nonumber\\
& \geq C_1\exp\left(-\frac{1+\delta_2}{2}\frac{\tilde{\Delta}_i^2npL}{\overline{V}_i(\theta^*)}-C_1^\prime\sqrt{\frac{\tilde{\Delta}_i^2npL}{\overline{V}_i(\theta^*)}}\right),\nonumber
\end{align}
for some constants $C_1,C_1'>0$ and some small constant $\delta_2>0$. Note that $\delta_2$ can be an arbitrarily small constant by making $\delta^\prime$ and $\rho$ small as well as making $\alpha$ large. Thus we can choose $\delta^\prime, \rho$ small  enough  and $\alpha$ large enough to let $\delta_2<\delta/2$. Then we have
\begin{align}
\nonumber&\sum_{i\in \overline{S}_1^\prime(\rho, t^*)} p_i \geq C_1\sum_{i\in \overline{S}_1^\prime(\rho, t^*)}\exp\left(-\frac{1+\delta_2}{2}\frac{\tilde{\Delta}_i^2npL}{\overline{V}_i(\theta^*)}-C_1^\prime\sqrt{\frac{\tilde{\Delta}_i^2npL}{\overline{V}_i(\theta^*)}}\right)\\ &\geq C_1\overline{R}_1(\overline{S}_1^\prime(\rho, t^*), \theta^*, t^*, -\delta)\label{eq:alpha-large}\\
&\geq C_1\rho\overline{R}_1(S_1(t^*), \theta^*, t^*, -\delta).\label{eqn:p_i_lower_bound}
\end{align}
by the same argument as in the proof of  (\ref{eq:MLE-ranking-theta-l}) of Theorem \ref{thm:MLE-ranking-theta}. As a result, under the condition (\ref{eq:spec-lower-condition-case}), we have $\sum_{i\in \overline{S}_1^\prime(\rho, t^*)}p_i\rightarrow\infty$.

Hence, we have proved
\begin{align*}
\inf_{A\in \overline{\mathcal{A}}}\mathbb{P}_{A}\left(\sum_{i\in \overline{S}_1^\prime(\rho, t^*)}\overline{L}_{i,1} \geq  \frac{1}{2}C_1\sum_{i\in \overline{S}_1^\prime(\rho, t^*)}\exp\left(-\frac{1+\delta_2}{2}\frac{\tilde{\Delta}_i^2npL}{\overline{V}_i(\theta^*)}-C_1^\prime\sqrt{\frac{\tilde{\Delta}_i^2npL}{\overline{V}_i(\theta^*)}}\right)\right) \geq  1- o(1).
\end{align*}

\textbf{(2).} Analysis of $\overline{L}_{i,2}$. By (\ref{eqn:A_1})-(\ref{eqn:A_3}) and Bernstein's inequality, we can bound $\E (\overline{L}_{i,2}|A)$ by
\begin{align*}
& \exp\left( - \frac{\left(\delta'(1+\delta_0)^2\tilde{\Delta}_iL \sum_{j\in[n]\backslash\{i\}}A_{ji}\psi(\theta_j^*-\theta_i^*)   \right)^2}{2 \left( L \sum_{j\in \overline{S}_1^\prime(\rho, t^*):j<i}A_{ji}\psi'(\theta_i^*-\theta_j^*)(1+e^{\theta_j^*-\theta_i^*})^2  + \frac{1}{3}\delta'(1+\delta_0)^2\tilde{\Delta}_i L  \sum_{j\in[n]\backslash\{i\}}A_{ji}\psi(\theta_j^*-\theta_i^*)   \right)}\right)\\
&\leq \exp\left( - \frac{\left(\delta'(1+\delta_0)^2\tilde{\Delta}_iL \sum_{j\in[n]\backslash\{i\}}p\psi(\theta_j^*-\theta_i^*)   \right)^2}{4 \left( 2L \rho kp +10\log n  + \frac{1}{3}\delta'(1+\delta_0)^2\tilde{\Delta}_i L  \sum_{j\in[n]\backslash\{i\}}p\psi(\theta_j^*-\theta_i^*)   \right)}\right).
\end{align*}
Now we set $\delta' = \rho^{1/8}$, and make $\rho$ small enough to ensure (\ref{eqn:p_i_lower_bound}). Then, there exists some constants $C_2,C_3>0$ such that
$$
\E (\overline{L}_{i,2}|A) \leq \exp\left(-C_2\rho^{-\frac{1}{2}}npL \tilde{\Delta}_i^2\right) \leq \exp\left(-C_3\rho^{-1/2}\frac{\tilde{\Delta}_i^2npL}{2\overline{V}_i(\theta^*)}\right).
$$
Then,
$$
\E \left(\sum_{i\in \overline{S}_1^\prime(\rho, t^*)}\overline{L}_{i,2}\Bigg|A\right) \leq \sum_{i\in \overline{S}_1^\prime(\rho, t^*)}\exp\left(-C_3\rho^{-1/2}\frac{\tilde{\Delta}_i^2npL}{2\overline{V}_i(\theta^*)}\right).
$$
By Markov inequality, we have
\begin{equation}
\inf_{A\in \overline{\mathcal{A}}}\mathbb{P}_{A}\left(\sum_{i\in \overline{S}_1^\prime(\rho, t^*)}\overline{L}_{i,2} \geq \sum_{i\in \overline{S}_1^\prime(\rho, t^*)}\exp\left(-\frac{1}{2}C_3\rho^{-1/2}\frac{\tilde{\Delta}_i^2npL}{2\overline{V}_i(\theta^*)}\right) \right) \leq \frac{\sum_{i\in \overline{S}_1^\prime(\rho,t^*)}\exp\left(-C_3\rho^{-1/2}\frac{\tilde{\Delta}_i^2npL}{2\overline{V}_i(\theta^*)}\right)}{\sum_{i\in \overline{S}_1^\prime(\rho, t^*)}\exp\left(-\frac{1}{2}C_3\rho^{-1/2}\frac{\tilde{\Delta}_i^2npL}{2\overline{V}_i(\theta^*)}\right)}. \label{eqn:L_i2_upper}
\end{equation}

\textbf{(3).} Analysis of $\overline{L}_{i,3}$. By a similar argument, we also have
\begin{equation}
\inf_{A\in \overline{\mathcal{A}}}\mathbb{P}_{A}\left(\sum_{i\in \overline{S}_1^\prime(\rho, t^*)}\overline{L}_{i,3} \geq \sum_{i\in \overline{S}_1^\prime(\rho, t^*)}\exp\left(-\frac{1}{2}C_3\rho^{-1/2}\frac{\tilde{\Delta}_i^2npL}{2\overline{V}_i(\theta^*)}\right) \right) \leq \frac{\sum_{i\in \overline{S}_1^\prime(\rho, t^*)}\exp\left(-C_3\rho^{-1/2}\frac{\tilde{\Delta}_i^2npL}{2\overline{V}_i(\theta^*)}\right)}{\sum_{i\in \overline{S}_1^\prime(\rho, t^*)}\exp\left(-\frac{1}{2}C_3\rho^{-1/2}\frac{\tilde{\Delta}_i^2npL}{2\overline{V}_i(\theta^*)}\right)}. \label{eqn:L_i3_upper}
\end{equation}

~\\
\indent Now we can combine the above analyses of $\overline{L}_{i,1}$, $\overline{L}_{i,2}$ and $\overline{L}_{i,3}$.
Since we are allowed to choose $\rho$ to be an arbitrarily small constant, we shall make 
$$\sum_{i\in \overline{S}_1^\prime(\rho, t^*)}\exp\left(-\frac{1}{2}C_3\rho^{-1/2}\frac{\tilde{\Delta}_i^2npL}{2\overline{V}_i(\theta^*)}\right)\leq\frac{1}{8}C_1\sum_{i\in \overline{S}_1^\prime(\rho, t^*)}\exp\left(-\frac{1+\delta_2}{2}\frac{\tilde{\Delta}_i^2npL}{\overline{V}_i(\theta^*)}-C_1^\prime\sqrt{\frac{\tilde{\Delta}_i^2npL}{\overline{V}_i(\theta^*)}}\right)$$
and
$$\frac{\sum_{i\in \overline{S}_1^\prime(\rho, t^*)}\exp\left(-C_3\rho^{-1/2}\frac{\tilde{\Delta}_i^2npL}{2\overline{V}_i(\theta^*)}\right)}{\sum_{i\in \overline{S}_1^\prime(\rho, t^*)}\exp\left(-\frac{1}{2}C_3\rho^{-1/2}\frac{\tilde{\Delta}_i^2npL}{2\overline{V}_i(\theta^*)}\right)}\leq\frac{1}{16}.$$
Thus, we have
\begin{align}
\inf_{A\in \overline{\mathcal{A}}}\mathbb{P}_{A}\left(\sum_{i\in \overline{S}_1^\prime(\rho, t^*)}\overline{L}_i \geq C_4\sum_{i\in \overline{S}_1^\prime(\rho, t^*)}\exp\left(-\frac{1+\delta_2}{2}\frac{\tilde{\Delta}_i^2npL}{\overline{V}_i(\theta^*)}-C_1^\prime\sqrt{\frac{\tilde{\Delta}_i^2npL}{\overline{V}_i(\theta^*)}}\right) \right)\geq \frac{7}{8}-o(1),\label{eqn:k_1}
\end{align}
for some constant $C_4 >0$. Then  (\ref{eq:spec-lower-pi}), (\ref{eq:etos}), (\ref{eqn:spectral_final_1}) together with (\ref{eq:spec-lower-condition-case}) lead to
\begin{align}
\mathbb{P}_{(\theta^*,r^*)}\left(\sum_{i\in \overline{S}_1^\prime(\rho, t^*)}\indc{\wh{\pi}_i <t}\geq \frac{C_4}{2}\sum_{i\in \overline{S}_1^\prime(\rho, t^*)}\exp\left(-\frac{1+\delta_2}{2}\frac{\tilde{\Delta}_i^2npL}{\overline{V}_i(\theta^*)}-C_1^\prime\sqrt{\frac{\tilde{\Delta}_i^2npL}{\overline{V}_i(\theta^*)}}\right)\right)\geq \frac{7}{8}-o(1).\label{eqn:k_1_1}
\end{align}
Finally, (\ref{eq:H-lower-1-case-reduce}) follows from (\ref{eqn:p_i_lower_bound}) which completes the proof.
\end{proof}

We state  Lemma \ref{lem:spectral_lower_tail_prop_case} to close this section. Its proof is essentially the same as the proof of Lemma \ref{lem:spectral_lower_tail_prop} and hence is omitted here.
\begin{lemma}\label{lem:spectral_lower_tail_prop_case}
Assume $\frac{np}{\log n}\rightarrow\infty$, $\kappa =O(1)$. Recall the definition of $\overline{S}_1^\prime(\rho, t^*)$ in (\ref{eq:spec-def-S1-prime}), $S=[n]\backslash \overline{S}_1^\prime(\rho,t^*)$ and $\tilde{\Delta}_i=(\theta_i^*-t^*)_+\vee\alpha\sqrt{\frac{1}{npL}}$. There exists some constants $C_1, C_2>0$ such that for any small constant $0.1>\tilde{\delta}>0$, there exists constant $\delta_1>0$ such that for any constant $\alpha>0$, $i\in S_1^\prime(t^*)$, any $A\in\overline{\mathcal{A}}$ where $\overline{\mathcal{A}}$ is defined in (\ref{eqn:A_1})-(\ref{eqn:A_3}), any $\theta^*\in\Theta(k,0,\kappa)$ and any $r^*\in\S_n$, we have
\begin{align}
 &\p_{(\theta^*, r^*)}\left(\frac{\sum_{j \in S}A_{ji}(\bar{y}_{ij}-\psi(\theta_{r_i^*}^*-\theta_{r_j^*}^*))(1+e^{\theta_{r_j^*}^*-\theta_{r_i^*}^*})}{\sum_{j\in[n]\backslash\{i\}}A_{ji}\psi(\theta_{r_j^*}^*-\theta_{r_i^*}^*)} \leq -(1+\tilde{\delta}) \tilde{\Delta}_i\Bigg| A\right)   \nonumber\\
& \geq C_1\exp\left(-\frac{1+\delta_1}{2}\frac{\tilde{\Delta}_i^2npL}{\overline{V}_{r_i^*}(\theta^*)}-C_2\sqrt{\frac{\tilde{\Delta}_i^2npL}{\overline{V}_{r_i^*}(\theta^*)}}\right).\label{eqn:spectral_lower_tail_prop_case}
\end{align}
Moreover, $\delta_1$ is able to be arbitrarily small if $\tilde{\delta}$ and $\rho$ are small enough. 
\end{lemma}

\section{Proofs of Technical Lemmas}\label{sec:pf-tech}

In this section, we prove Lemma \ref{lem:anderson}, Lemma \ref{lem:A-basic}, Lemma \ref{lem:A-bern}, Lemma \ref{lem:hessian-spec} and Lemma \ref{lem:concentration}.
We first list some additional technical results that will be needed in the proofs.

\begin{lemma}[Hoeffding's inequality]\label{lem:hoeffding}
For independent random variables $X_1,\cdots,X_n$ that satisfy $a_i\leq X_i\leq b_i$, we have
$$\mathbb{P}\left(\sum_{i=1}^n(X_i-\mathbb{E}X_i)\geq t\right)\leq \exp\left(-\frac{2t^2}{\sum_{i=1}^n(b_i-a_i)^2}\right),$$
for any $t>0$.
\end{lemma}

\begin{lemma}[Bernstein's inequality]\label{lem:bernstein}
For independent random variables $X_1,\cdots,X_n$ that satisfy $|X_i|\leq M$ and $\mathbb{E}X_i=0$, we have
$$\mathbb{P}\left(\sum_{i=1}^nX_i\geq t\right)\leq \exp\left(-\frac{\frac{1}{2}t^2}{\sum_{i=1}^n\mathbb{E}X_i^2+\frac{1}{3}Mt}\right),$$
for any $t>0$.
\end{lemma}

\begin{lemma}[Central limit theorem, Theorem 2.20 of \cite{ross2007second}]\label{lem:CLT-stein}
If $Z\sim N(0,1)$ and $W=\sum_{i=1}^nX_i$ where $X_i$ are independent mean $0$ and $\Var(W)=1$, then
$$\sup_t\left|\mathbb{P}(W\leq t)-\mathbb{P}(Z\leq t)\right| \leq 2\sqrt{3\sum_{i=1}^n\left(\mathbb{E}X_i^4\right)^{3/4}}.$$
\end{lemma}

\begin{proof}[Proof of Lemma \ref{lem:anderson}]
Without loss of generality, we consider $r_i^*=i$ so that $\theta_1^*\geq\cdots\geq \theta_n^*$. Then, we can write the loss as $2k\h_k(\wh{r},r^*)=\sum_{i=1}^k\indc{\wh{r}_i>k}+\sum_{i=k+1}^n\indc{\wh{r}_i\leq k}$. Since $\wh{r}\in\S_n$, we must have $\sum_{i=1}^k\indc{\wh{r}_i>k}=\sum_{i=k+1}^n\indc{\wh{r}_i\leq k}$. This implies
\begin{eqnarray}
\nonumber 2k\h_k(\wh{r},r^*) &=& 2\min\left(\sum_{i=1}^k\indc{\wh{r}_i>k},\sum_{i=k+1}^n\indc{\wh{r}_i\leq k}\right) \\
\nonumber &\leq& 2\min\left(\sum_{i=1}^k\indc{\wh{\theta}_i\leq\wh{\theta}_{(k+1)}},\sum_{i=k+1}^n\indc{\wh{\theta}_i\geq\wh{\theta}_{(k)}}\right) \\
\label{eq:some-t} &\leq& 2\max_t\min\left(\sum_{i=1}^k\indc{\wh{\theta}_i\leq t},\sum_{i=k+1}^n\indc{\wh{\theta}_i\geq t}\right) \\
\label{eq:any-t} &=& 2\min_t\max\left(\sum_{i=1}^k\indc{\wh{\theta}_i\leq t},\sum_{i=k+1}^n\indc{\wh{\theta}_i\geq t}\right) \\
\nonumber &\leq& 2\min_t\left(\sum_{i=1}^k\indc{\wh{\theta}_i\leq t}+\sum_{i=k+1}^n\indc{\wh{\theta}_i\geq t}\right).
\end{eqnarray}
The inequality (\ref{eq:some-t}) uses the fact that $\wh{\theta}_{(k)}\geq\wh{\theta}_{(k+1)}$ where $\{\theta_{(i)}\}_{i=1}^n$ are the order statistics with $\wh{\theta}_{(1)}$ being the largest and $\wh{\theta}_{(n)}$ being the smallest. The equality (\ref{eq:any-t}) holds since $\sum_{i=1}^k\indc{\wh{\theta}_i\leq t}$ is a nondecreasing function of $t$ and $\sum_{i=k+1}^n\indc{\wh{\theta}_i\geq t}$ is a nonincreasing function of $t$.
\end{proof}

\begin{proof}[Proof of Lemma \ref{lem:A-basic}]
The first conclusion is a direct consequence of Bernstein's inequality and a union bound argument. The second and third conclusion is a standard property of random graph Laplacian \citep{tropp2015introduction}.
\end{proof}

\begin{proof}[Proof of Lemma \ref{lem:A-bern}]
To see the first conclusion, we note that $\mathbb{E}(A_{ij}-p)^2\leq p$ and $\Var((A_{ij}-p)^2)\lesssim p$, and thus we can apply Bernstein's inequality followed by a union bound argument to obtain the desired result. The second conclusion is a direct consequence of Bernstein's inequality and a union bound argument.
\end{proof}

\begin{proof}[Proof of Lemma \ref{lem:hessian-spec}]
For any $u\in\mathbb{R}^n$ such that $\mathds{1}_{n}^Tu=0$,
$$u^TH(\theta)u=\sum_{1\leq i<j\leq n}A_{ij}\psi(\theta_i-\theta_j)\psi(\theta_j-\theta_i)(u_i-u_j)^2.$$
Since $\psi(\theta_i-\theta_j)\psi(\theta_j-\theta_i)\geq \frac{1}{4}e^{-M}$, we have $\lambda_{\min,\perp}(H(\theta))\geq \frac{1}{4}e^{-M}\lambda_{\min,\perp}(\mathcal{L}_A)$. By Lemma \ref{lem:A-basic}, we obtain the desired result.
\end{proof}

\begin{proof}[Proof of Lemma \ref{lem:concentration}]

Let $\mathcal{U}=\left\{u\in\mathbb{R}^n:\sum_{i\in[n]}u_i^2\leq1\right\}$ be the unit ball in $\mathbb{R}^n$. Then there exists a subset of $\mathcal{V}\subset\mathcal{U}$ such that for any $u\in\mathcal{U}$, there is a $v\in\mathcal{V}$ satisfying $\norm{u-v}\leq1/2$. Moreover, we also have $\log\abs{\mathcal{V}}\leq C^\prime n$ for some constant $C^\prime$. See Lemma 5.2 of \cite{vershynin2010introduction}. Then for any $u\in\mathcal{U}$, with the corresponding $v\in\mathcal{V}$, we have
\begin{align*}
&\sum_{i=1}^nu_i\left(\sum_{j\in[n]\backslash\{i\}}A_{ij}(\bar{y}_{ij}-\psi(\theta_i^*-\theta_j^*))\right)\\
&=\sum_{i=1}^nv_i\left(\sum_{j\in[n]\backslash\{i\}}A_{ij}(\bar{y}_{ij}-\psi(\theta_i^*-\theta_j^*))\right)+\sum_{i=1}^n(u_i-v_i)\left(\sum_{j\in[n]\backslash\{i\}}A_{ij}(\bar{y}_{ij}-\psi(\theta_i^*-\theta_j^*))\right)\\
&\leq\sum_{i=1}^nv_i\left(\sum_{j\in[n]\backslash\{i\}}A_{ij}(\bar{y}_{ij}-\psi(\theta_i^*-\theta_j^*))\right)+\frac{1}{2}\sqrt{\sum_{i=1}^n\left(\sum_{j\in[n]\backslash\{i\}}A_{ij}(\bar{y}_{ij}-\psi(\theta_i^*-\theta_j^*))\right)^2}.
\end{align*}
Maximize $u$ and $v$ on both sides of the inequality, after rearrangement, we have
\begin{align*}
&\sqrt{\sum_{i=1}^n\left(\sum_{j\in[n]\backslash\{i\}}A_{ij}(\bar{y}_{ij}-\psi(\theta_i^*-\theta_j^*))\right)^2}\\
&\leq2\max_{v\in\mathcal{V}}\sum_{i=1}^nv_i\left(\sum_{j\in[n]\backslash\{i\}}A_{ij}(\bar{y}_{ij}-\psi(\theta_i^*-\theta_j^*))\right)\\
&=2\max_{v\in\mathcal{V}}\sum_{i<j}A_{ij}(v_i-v_j)(\bar{y}_{ij}-\psi(\theta_i^*-\theta_j^*)).
\end{align*}
Conditional on $A$, applying Hoeffding's inequality and union bound on the last line, we have
\begin{align*}
\sum_{i=1}^n\left(\sum_{j\in[n]\backslash\{i\}}A_{ij}(\bar{y}_{ij}-\psi(\theta_i^*-\theta_j^*))\right)^2&\leq C^{\prime\prime}\frac{(\log n+n)\max_{v\in\mathcal{V}}\sum_{i<j}A_{ij}(v_i-v_j)^2}{L}\\
&\leq C^{\prime\prime}\frac{(\log n+n)\lambda_{\max}(\mathcal{L}_A)}{L}
\end{align*}
with probability at least $1-O(n^{-10})$. By Lemma \ref{lem:A-basic}, we obtain the desired bound for the first conclusion.

The second conclusion is a direct application of Hoeffding's inequality and a union bound argument.

The proof of the third conclusion is similar to that of the first one. Define $\mathcal{U}_i=\left\{u\in\mathbb{R}^{n-1}: \sum_{j\in[n]\backslash\{i\}}A_{ij}u_j^2\leq 1\right\}$. Conditioning on $A$, one can think of $\mathcal{U}_i$ as a unit ball with dimension $\sum_{j\in[n]\backslash\{i\}}A_{ij}-1$. Then, there exists a subset $\mathcal{V}_i\subset \mathcal{U}_i$ such that for any $u\in\mathcal{U}_i$, there is a $v\in\mathcal{V}_i$ that satisfies $\|u-v\|\leq \frac{1}{2}$. Moreover, we also have $\log|\mathcal{V}_i|\leq 2\sum_{j\in[n]\backslash\{i\}}A_{ij}$ by Lemma 5.2 of \cite{vershynin2010introduction}. For any $u\in\mathcal{U}_i$, with the corresponding $v\in\mathcal{V}_i$, following a similar argument of the proof of the first conclusion, we have
$$\sqrt{\sum_{j\in[n]\backslash\{i\}}A_{ij}(\bar{y}_{ij}-\psi(\theta_i^*-\theta_j^*))^2} \leq 2\max_{v\in\mathcal{V}_i}\sum_{j\in[n]\backslash\{i\}}A_{ij}v_{ij}(\bar{y}_{ij}-\psi(\theta_i^*-\theta_j^*)),$$
which implies
$$\sqrt{\max_{i\in[n]}\sum_{j\in[n]\backslash\{i\}}A_{ij}(\bar{y}_{ij}-\psi(\theta_i^*-\theta_j^*))^2} \leq 2\max_{i\in[n]}\max_{v\in\mathcal{V}_i}\sum_{j\in[n]\backslash\{i\}}A_{ij}v_{ij}(\bar{y}_{ij}-\psi(\theta_i^*-\theta_j^*)).$$
Applying Hoeffding's inequality and union bound, we have
$$\max_{i\in[n]}\sum_{j\in[n]\backslash\{i\}}A_{ij}(\bar{y}_{ij}-\psi(\theta_i^*-\theta_j^*))^2\leq C_1\frac{\log n+\max_{i\in[n]}\sum_{j\in[n]\backslash\{i\}}A_{ij}}{L},$$
with probability at least $1-O(n^{-10})$. Finally, applying Lemma \ref{lem:A-basic}, we obtain the desired bound for the third conclusion, which concludes the proof.
\end{proof}

%\end{raggedright}           % Comment this out if you don't want ragged edges.
\bibliographystyle{dcu}
\bibliography{reference}

\end{document}

%% file: anderson.tex
%%!TEX root = manuscript20200604.tex

\subsection{Proof of Theorem \ref{thm:spectral_lower}}\label{sec:pf-spec2}

To prove Theorem \ref{thm:spectral_lower}, we need two additional lemmas. The first lemma can be viewed as a reverse version of the inequality in Lemma \ref{lem:anderson}.

\begin{lemma}\label{lem:anderson_lower_bound}
Suppose $\wh{r}$ is a rank vector induced by $\wh{\theta}$, we then have
$$\h_k(\wh{r},r^*) \geq \frac{1}{k}\max_{t\in\mathbb{R}}\min\left(\sum_{i:r^*_i\leq k}\indc{\wh{\theta}_i< t} , \sum_{i:r^*_i>k}\indc{\wh{\theta}_i> t}\right).$$
The inequality holds for any $r^*\in\S_n$.
\end{lemma}
\begin{proof}
Following the proof of Lemma \ref{lem:anderson}, we have
\begin{eqnarray}
\nonumber 2k\h_k(\wh{r},r^*) &=& 2\max\left(\sum_{i=1}^k\indc{\wh{r}_i>k},\sum_{i=k+1}^n\indc{\wh{r}_i\leq k}\right) \\
\nonumber &\geq & 2\max\left(\sum_{i=1}^k\indc{\wh{\theta}_i<\wh{\theta}_{(k)}},\sum_{i=k+1}^n\indc{\wh{\theta}_i>\wh{\theta}_{(k+1)}}\right) \\
\label{eq:some-t_lower_spectral} &\geq& 2\min_t\max\left(\sum_{i=1}^k\indc{\wh{\theta}_i< t},\sum_{i=k+1}^n\indc{\wh{\theta}_i> t}\right)\\
\label{eq:any-t_lower_spectral} &=& 2\max_t\min\left(\sum_{i=1}^k\indc{\wh{\theta}_i< t},\sum_{i=k+1}^n\indc{\wh{\theta}_i> t}\right).
\end{eqnarray}
where (\ref{eq:some-t_lower_spectral}) and (\ref{eq:any-t_lower_spectral}) follow the same argument that leads to (\ref{eq:some-t}) and (\ref{eq:any-t}).
\end{proof}

%\begin{thm}\label{thm:spectral_lower}
%Assume $\frac{np}{\log n}\rightarrow\infty$ and $\kappa = O(1)$ \nb{and maybe $k\rightarrow\infty$} and\begin{equation}
%\frac{npL\Delta^2}{\overline{V}(\kappa)} < (1-\epsilon)2\left(\sqrt{\log k} + \sqrt{\log(n-k)}\right)^2, \label{eq:exact-threshold-lower}
%\end{equation}
%for some arbitrarily small constant $\epsilon>0$. Then, for the rank vector $\wh{r}$ that is induced by the stationary distribution of the Markov chain (\ref{eq:spec-P}), there exists some $\delta=o(1)$, such that
%\begin{align*}
%\sup_{\substack{r^*\in\S_n\\\theta^*\in\Theta(k,\Delta,\kappa)}}\mathbb{E}_{(\theta^*,r^*)}\h_k(\wh{r},r^*)\geq C\exp\left(-\frac{1}{2}\left(\frac{\sqrt{(1-\delta)\overline{\snr}}}{2}-\frac{1}{\sqrt{(1-\delta)\overline{\snr}}}\log\frac{n-k}{k}\right)^2\right),
%\end{align*}
%for some constant $C>0$ .
%\end{thm}

\begin{proof}[Proof of Theorem \ref{thm:spectral_lower}]
We first note that condition (\ref{eq:spec-lower-condition}) necessarily implies $\Delta =o(1)$. Throughout the proof, we assume 
%$k\rightarrow\infty$, 
$\kappa =\Omega(1)$ and there exists some $\delta_1 = o(1)$ such that  
\begin{align}
%\sqrt{\overline{\snr}} \geq 2 \log \frac{n-k}{k}.\label{eqn:overline_SNR_infty}
\frac{\sqrt{(1+\delta_1)\overline{\snr}}}{2} - \frac{1}{\sqrt{(1+ \delta_1)\overline{\snr}}} \log \frac{n-k}{k} \rightarrow\infty.\label{eqn:overline_SNR_infty}
\end{align}
The case with 
%$k=O(1)$, 
$\kappa =o(1)$ or $\overline{\snr}$ not satisfying (\ref{eqn:overline_SNR_infty}) will be addressed at the end of the proof.  

%Note that (\ref{eqn:overline_SNR_infty}) can be  equivalently stated as $\overline{\snr} = 2\log \frac{n-k}{k} + b\sqrt{\log\frac{n-k}{k}}$ for some $b\rightarrow\infty$.
%When $\overline{\snr}=O(1)$, we have $\sup_{\substack{r\in\S_n\theta\in\Theta(k,\Delta,\kappa)}}\mathbb{E}_{(\theta,r)}\h_k(\wh{r},r^*) \geq \inf_{\tilde r}\sup_{\substack{r\in\S_n\theta\in\Theta(k,\Delta,\kappa)}}\mathbb{E}_{(\theta,r)}\h_k(\tilde {r},r^*)\geq c$ for some constant $c>0$, by Theorem \ref{thm:minimax_lower}. Therefore, we only need to consider the case that $\overline{\snr}\rightarrow\infty$.

Choose $\kappa_1,\kappa_2\geq 0$ such that we have both $\kappa_1+\kappa_2\leq \kappa$ and
\begin{align*}
\frac{k\psi'(\kappa_1)(1+e^{\kappa_1})^2+(n-k)\psi'(\kappa_2)(1+e^{-\kappa_2})^2}{(k\psi(\kappa_1)+(n-k)\psi(-\kappa_2))^2/n} = \overline{V}(\kappa).
\end{align*}
Let $\rho=o(1)$ be a vanishing number that will be specified later. Since $k\rightarrow\infty$ and $\kappa=\Omega(1)$, one can easily check that $\kappa_2=\Omega(1)$. Define $\theta^*_i=\kappa_1$ for all $1\leq i\leq k-\rho k$, $\theta^*_i=0$ for $k-\rho k<i\leq k$, $\theta^*_i=-\Delta$ for $k<i\leq k+\rho(n-k)$ and $\theta^*_i=-\kappa_2$ for $k+\rho(n-k)< i\leq n$. For the simplicity of proof, we choose $\rho$ so that both $\rho k$ and $\rho(n-k)$ are integers. Define $r^*$ to be $r^*_i = i,\forall i\in[n]$. Then we have
$$\sup_{\substack{r\in\S_n\\\theta\in\Theta(k,\Delta,\kappa)}}\mathbb{E}_{(\theta,r)}\h_k(\wh{r},r) \geq \mathbb{E}_{(\theta^*,r^*)}\h_k(\wh{r},r^*).$$
We will utilize several results established in the proof of Theorem \ref{thm:spectral-ranking}. Define 
\begin{align}
\eta = \frac{1}{2} - \frac{\overline{V}(\kappa)}{(1+{\bar \delta})\Delta^2npL}\log\frac{n-k}{k},\label{eqn:spectral_lower_eta_def}
\end{align}
for $\bar \delta =o(1)$. 
The specific choice of $\bar \delta$ will be specified later in the proof.
Also define  $t=\frac{e^{(1-\eta)\theta_k^*+\eta\theta_{k+1}^*}}{\sum_{j=1}^ne^{\theta_j^*}} = \frac{e^{-\eta \Delta}}{\sum_{j=1}^ne^{\theta_j^*}}$. Then, by Lemma \ref{lem:anderson_lower_bound}, we have
\begin{align*}
%\sup_{\substack{r\in\S_n\\\theta\in\Theta(k,\Delta,\kappa)}}\mathbb{E}_{(\theta,r)}
\h_k(\wh{r},r^*) & \geq  \frac{1}{k}\min\left(\sum_{i=1}^k\indc{\wh{\pi}_i< t},\sum_{i=k+1}^n\indc{\wh{\pi}_i> t}\right) \\
& \geq  \frac{1}{k}\min\left(\sum_{k-\rho k  <i \leq k}\indc{\wh{\pi}_i< t},\sum_{k< i \leq k + \rho(n-k)}\indc{\wh{\pi}_i> t}  \right).
\end{align*}
For any $\delta>0$, define the function $\phi(\delta) = \frac{\sqrt{(1+\delta)\overline{\snr}}}{2}-\frac{1}{\sqrt{(1+\delta)\overline{\snr}}}\log\frac{n-k}{k}$.
It suffices to show there exists some 
%$\bar{\delta}=o(1)$and  some 
constant $C>0$ such that
\begin{align}
&\mathbb{P}_{(\theta^*,r^*)}\left(\sum_{k-\rho k  <i \leq k}\indc{\wh{\pi}_i< t} \geq Ck\exp\left(-\frac{\phi(\bar\delta)^2}{2} \right) \right) \geq 1-o(1) ,\label{eqn:H_lower_1}\\
\text{and }& \mathbb{P}_{(\theta^*,r^*)}\left(\sum_{k< i \leq k + \rho(n-k)}\indc{\wh{\pi}_i> t} \geq Ck\exp\left(-\frac{\phi(\bar\delta)^2}{2} \right) \right) \geq 1-o(1).\label{eqn:H_lower_2}
\end{align}
Suppose both inequalities hold, we have
$$\mathbb{P}_{(\theta^*,r^*)}\left(\h_k(\wh{r},r^*) > 0\right) \geq 1-o(1).$$
By Markov's inequality, we also have
\begin{align*}
 &\mathbb{E}_{(\theta^*,r^*)}\h_k(\wh{r},r^*) \geq  C\exp\left(-\frac{\phi(\bar\delta)^2}{2} \right) \mathbb{P}_{(\theta^*,r^*)}\left(\h_k(\wh{r},r^*)  \geq C\exp\left(-\frac{\phi(\bar\delta)^2}{2} \right) \right) 
 %&\geq C\exp\left(-\frac{\phi(\delta)^2}{2} \right) \times \left(1- \mathbb{P}_{(\theta^*,r^*)}\left(\sum_{k-\rho k  <i \leq k}\indc{\wh{\pi}_i< t} < C\exp\left(-\frac{\phi(\delta)^2}{2} \right) \right) - \mathbb{P}_{(\theta^*,r^*)}\left(\sum_{k< i \leq k + \rho(n-k)}\indc{\wh{\pi}_i> t} < C\exp\left(-\frac{\phi(\delta)^2}{2} \right) \right) \right)\\
 \geq  \frac{C}{2} \exp\left(-\frac{\phi(\bar\delta)^2}{2} \right).
\end{align*}
Therefore, we obtain the desired conclusions.
%We prove (\ref{eq:lower-bound-spec}).

\indent In the rest of the proof, we are going to establish (\ref{eqn:H_lower_1}). 
%The inequality (\ref{eqn:H_lower_2}) can be proved similarly and thus is omitted here.
 Recall the definition of $\bar \pi$  in (\ref{eqn:pi_bar_def}).
For any $k-\rho k<i\leq k$, define the event $\mathcal{F}$ as
\begin{align*}
\mathcal{F}_i = \left\{\frac{|\wh{\pi}_i-\bar{\pi}_i|}{ \pi^*_i} \leq  \delta_0(1-e^{-\eta\Delta}) \text{ and } \left|\frac{\sum_{j\in[n]\backslash\{i\}}A_{ji}\bar{y}_{ji}}{\sum_{j\in[n]\backslash\{i\}}A_{ji}\psi(\theta_j^*-\theta_i^*)}-1\right| \leq  \delta_0\right\}.
\end{align*}
Using a similar argument that leads to (\ref{eq:pi-hat-bar-approx}) and (\ref{eq:easy-de-pi}), we can show that there exists some $\delta_0=o(1)$ not dependent on $\bar \delta$, such that
\begin{equation}
\mathbb{P}_{(\theta^*,r^*)} (\mathcal{F}_i) \geq 1-\left(  O(n^{-4})+\exp\left(-\eta^2\Delta^2npL\frac{np}{\log n}\right)+\exp\left(-\eta^2\Delta^2npL\sqrt{\frac{npL}{\log n}}\right)   \right). \label{eq:indi-F-i}
\end{equation}
Suppose $\mathcal{F}_i$ holds, we then have
\begin{align}
\indc{\wh \pi_i < t  } & =  \indc{\wh \pi_i < \frac{e^{(1-\eta)\theta_k^*+\eta\theta_{k+1}^*}}{\sum_{j=1}^ne^{\theta_j^*}}  } \nonumber     \\
&  = \indc{\frac{\wh{\pi}_i-\pi_i^*}{\pi_i^*} \leq e^{(1-\eta)\theta_k^*+\eta\theta_{k+1}^*-\theta_i^*}-1    } \nonumber    \\
&  = \indc{\frac{\wh{\pi}_i-\pi_i^*}{\pi_i^*} \leq e^{-\eta\Delta}-1   } \nonumber    \\
& \geq \indc{\frac{\bar{\pi}_i-\pi_i^*}{\pi_i^*}\leq -(1+ \delta_0)(1-e^{-\eta\Delta})  } \nonumber    \\
& \geq  \indc{\frac{\sum_{j\in[n]\backslash\{i\}}A_{ji}(\bar{y}_{ij}-\psi(\theta_i^*-\theta_j^*))(1+e^{\theta_j^*-\theta_i^*})}{\sum_{j\in[n]\backslash\{i\}}A_{ji}\psi(\theta_j^*-\theta_i^*)} \leq -(1+ \delta_0)^2(1-e^{-\eta\Delta}) }  \nonumber  \\
& \geq \indc{\frac{\sum_{j\in[n]\backslash\{i\}}A_{ji}(\bar{y}_{ij}-\psi(\theta_i^*-\theta_j^*))(1+e^{\theta_j^*-\theta_i^*})}{\sum_{j\in[n]\backslash\{i\}}A_{ji}\psi(\theta_j^*-\theta_i^*)} \leq -(1+ \delta_0)^2\eta\Delta }.\label{eqn:spectral_lower_L} 
\end{align}
We use the notation $L_i$ for the indicator function on the right hand side of (\ref{eqn:spectral_lower_L}).
In other words, we have shown that
\begin{eqnarray*}
\sum_{k-\rho k<i\leq k}\indc{\wh \pi_i < t  } &\geq& \sum_{k-\rho k<i\leq k}L_i\mathbb{I}_{\mathcal{F}_i} \\
&\geq& \sum_{k-\rho k<i\leq k}L_i - \sum_{k-\rho k<i\leq k}\mathbb{I}_{\mathcal{F}_i^c}.
\end{eqnarray*}
By (\ref{eq:indi-F-i}), we have
$$\mathbb{E}\left(\sum_{k-\rho k<i\leq k}\mathbb{I}_{\mathcal{F}_i^c}\right)\leq O(n^{-3}) + \rho k\exp\left(-\eta^2\Delta^2npL\frac{np}{\log n}\right)+\rho k\exp\left(-\eta^2\Delta^2npL\sqrt{\frac{npL}{\log n}}\right).$$
Since the above bounds is of smaller order than $k\exp\left(-\frac{\eta^2\Delta^2npL}{2\overline{V}(\kappa)}\left(\frac{np}{\log n}\right)^{1/4}\right)$, we can use Markov's inequality and obtain
\begin{equation}
\mathbb{P}_{(\theta^*,r^*)}\left(\sum_{k-\rho k<i\leq k}\mathbb{I}_{\mathcal{F}_i^c} \leq k\exp\left(-\frac{\eta^2\Delta^2npL}{2\overline{V}(\kappa)}\left(\frac{np}{\log n}\right)^{1/4}\right)\right) \geq 1-o(1). \label{eq:etos}
\end{equation}
To lower bound $\sum_{k-\rho k<i\leq k}L_i$, we define 
\begin{align}
\mathcal{A} = \Bigg\{A: \forall k-\rho k < i \leq k, &\left|\frac{\sum_{j\in[n]\backslash\{i\}}A_{ij}\psi'(\theta_i^*-\theta_j^*)\left(1+e^{\theta_j^*-\theta_i^*}\right)^2}{p\sum_{j\in[n]\backslash\{i\}}\psi'(\theta_i^*-\theta_j^*)\left(1+e^{\theta_j^*-\theta_i^*}\right)^2}-1\right|\leq \delta_0,\label{eqn:A_1} \\
& \left|\frac{\sum_{j\in[n]\backslash\{i\}}A_{ji}\psi(\theta_j^*-\theta_i^*)}{p\sum_{j\in[n]\backslash\{i\}}\psi(\theta_j^*-\theta_i^*)}-1\right|\leq  \delta_0 ,\label{eqn:A_2}\\
& \abs{\sum_{k-\rho k< j <k}A_{ji}\psi'(\theta_i^*-\theta_j^*)(1+e^{\theta_j^*-\theta_i^*})^2 }\leq 2\rho kp + 10\log n \Bigg\}.\label{eqn:A_3}
\end{align}
%where $c_1$ is some constant. 
By Bernstein's inequality and union bound, we have $\mathbb{P}(A\in \mathcal{A})\geq 1- O(n^{-3})$. From now on, we use the notation $\mathbb{P}_A$ for the conditional probability $\mathbb{P}_{(\theta^*,r^*)}(\cdot|A)$ given $A$.
For any $s>0$,
\begin{align}
\mathbb{P}_{(\theta^*,r^*)}&\left(\sum_{k-\rho k  <i \leq k}L_i \geq s \right) \geq  \mathbb{P}(A\in\mathcal{A}) \inf_{A\in \mathcal{A}}\mathbb{P}_{A}\left(\sum_{k-\rho k  <i \leq k}L_i \geq s \right).  \label{eqn:spectral_final_1}
\end{align}
To study  $\mathbb{P}_{A}\left(\sum_{k-\rho k  <i \leq k}L_i \geq s \right)$, we define the set $S = \cbr{i\in[n]: i\leq k-\rho k \text{ or }i > k}$. 
%For each $k-\rho k  <i \leq k$, we have
%\begin{align*}
%L_i & = \indc{\frac{\left(\sum_{j\in S} + \sum_{k-\rho k<j<i} + \sum_{i<j\leq k +\rho(n-k)}\right)A_{ji}(\bar{y}_{ij}-\psi(\theta_i^*-\theta_j^*))(1+e^{\theta_j^*-\theta_i^*})}{\sum_{j\in[n]\backslash\{i\}}A_{ji}\psi(\theta_j^*-\theta_i^*)} \leq -(1+\delta)^2(1-e^{-\bar{\Delta}_i})} \\
%& \geq  \indc{\frac{\sum_{j\in S}A_{ji}(\bar{y}_{ij}-\psi(\theta_i^*-\theta_j^*))(1+e^{\theta_j^*-\theta_i^*})}{\sum_{j\in[n]\backslash\{i\}}A_{ji}\psi(\theta_j^*-\theta_i^*)} \leq -(1+\delta)^2(1-e^{-\bar{\Delta}_i})}.
%\end{align*}
Note that for each $k-\rho k  <i \leq k$, we have $L_i \geq L_{i,1}- L_{i,2}-L_{i,3}$, where
\begin{align*}
L_{i,1} & =\indc{\frac{\sum_{j\in S}A_{ji}(\bar{y}_{ij}-\psi(\theta_i^*-\theta_j^*))(1+e^{\theta_j^*-\theta_i^*})}{\sum_{j\in[n]\backslash\{i\}}A_{ji}\psi(\theta_j^*-\theta_i^*)} \leq -(1+2\delta')(1+ \delta_0)^2\eta\Delta} \\
L_{i,2} & = \indc{\frac{\sum_{k-\rho k< j <i}A_{ji}(\bar{y}_{ij}-\psi(\theta_i^*-\theta_j^*))(1+e^{\theta_j^*-\theta_i^*})}{\sum_{j\in[n]\backslash\{i\}}A_{ji}\psi(\theta_j^*-\theta_i^*)} \geq \delta'(1+\delta_0)^2\eta\Delta}\\
L_{i,3} & = \indc{\frac{\sum_{ i <j \leq k}A_{ji}(\bar{y}_{ij}-\psi(\theta_i^*-\theta_j^*))(1+e^{\theta_j^*-\theta_i^*})}{\sum_{j\in[n]\backslash\{i\}}A_{ji}\psi(\theta_j^*-\theta_i^*)} \geq \delta'(1+\delta_0)^2\eta\Delta},
\end{align*}
for some $\delta'=o(1)$ whose value will be determined later. We are going to control each term separately.

\textbf{(1).}
Analysis of $L_{i,1}$. Note that conditional on $A$, $\{L_{i,1}\}_{k-\rho k <i \leq k}$ are all independent Bernoulli random variables. We have $L_{i,1} \sim \text{Bernoulli}(p_i)$, where $p_i = \E_{(\theta^*,r^*)} (L_{i,1}| A)$. By Chebyshev's inequality, we have 
\begin{align*}
\mathbb{P}_{A}\left(\sum_{k-\rho k <i \leq k}L_{i,1} \geq \frac{1}{2} \sum_{k-\rho k <i \leq k}p_{i}\right) \geq 1-\frac{4}{\sum_{k-\rho k <i \leq k}p_{i}}.
\end{align*}
By Lemma \ref{lem:spectral_lower_tail_prop} stated and proved at the end of the section, we can lower bound each $p_i$ by
\begin{align}
p_i &= \p_{A} \left(\frac{\sum_{j\in S}A_{ji}(\bar{y}_{ij}-\psi(\theta_i^*-\theta_j^*))(1+e^{\theta_j^*-\theta_i^*})}{\sum_{j\in[n]\backslash\{i\}}A_{ji}\psi(\theta_j^*-\theta_i^*)} \leq -(1+2\delta')(1+\delta_0)^2\eta\Delta\right)   \nonumber\\
& \geq C_1\exp\left(-\frac{(1+\delta_2)\eta^2\Delta^2npL}{2\overline{V}(\kappa)} - C_1' \eta \sqrt{\frac{\Delta^2npL}{\overline{V}(\kappa)}}\right),\nonumber
\end{align}
for some constants $C_1,C_1'>0$ and some $\delta_2 = o(1)$ that are not dependent on $\eta$. By (\ref{eqn:overline_SNR_infty}), there exists some $\delta_3=o(1)$ such that
\begin{align}
\sum_{k-\rho k < i\leq k } p_i \geq  C_1k\exp\left(-\frac{(1+\delta_3)\eta^2\Delta^2npL}{2\overline{V}(\kappa)}\right).\label{eqn:p_i_lower_bound}
\end{align}
To obtain (\ref{eqn:p_i_lower_bound}), we need to set $\rho$ that tends to zero sufficiently slow so that it can be absorbed into the exponent.
Note that condition (\ref{eq:spec-lower-condition}) is equivalent to $\frac{(1+\epsilon)\overline{\snr}}{2} \left(\frac{1}{2} - \frac{1}{(1+\epsilon)\overline{\snr}} \log \frac{n-k}{k}\right)^2 < \log k$.
Since $\epsilon$ is a constant, it implies
\begin{align*}
\frac{\overline{\snr}}{2} \left(\frac{1}{2} - \frac{1}{(1+\bar \delta)\overline{\snr}} \log \frac{n-k}{k}\right)^2 < (1-\epsilon')^{-1}\log k,
\end{align*}
for some constant $\epsilon'>0$.
As a result, under the condition that $k \rightarrow\infty$, we have
\begin{align*}
\sum_{k-\rho k < i\leq k }p_i \geq \sum_{k-\rho k < i\leq k } C_1\exp\left( -(1+\delta_3)(1-\epsilon')\log k \right) \geq k^\frac{\epsilon'}{2}\rightarrow\infty.
\end{align*}
% and $\rho \gg k^{-\frac{\epsilon}{4}}$. 
Hence, we have proved
\begin{align*}
\inf_{A\in \mathcal{A}}\mathbb{P}_{A}\left(\sum_{k-\rho k  <i \leq k}L_{i,1} \geq  \frac{1}{2}C_1k\exp\left(-\frac{(1+\delta_2)\eta^2\Delta^2npL}{2\overline{V}(\kappa)} - C_1' \eta \sqrt{\frac{\Delta^2npL}{\overline{V}(\kappa)}}\right)  \right) \geq  1- o(1).
\end{align*}

\textbf{(2).} Analysis of $L_{i,2}$. By (\ref{eqn:A_1})-(\ref{eqn:A_3}) and Bernstein's inequality, we can bound $\E (L_{i,2}|A)$ by
\begin{align*}
& \exp\left( - \frac{\left(\delta'(1+\delta_0)^2\eta\Delta L \sum_{j\in[n]\backslash\{i\}}A_{ji}\psi(\theta_j^*-\theta_i^*)   \right)^2}{2 \left( L \sum_{k-\rho k< j <i}A_{ji}\psi'(\theta_i^*-\theta_j^*)(1+e^{\theta_j^*-\theta_i^*})^2  + \frac{1}{3}\delta'(1+\delta_0)^2\eta\Delta L  \sum_{j\in[n]\backslash\{i\}}A_{ji}\psi(\theta_j^*-\theta_i^*)   \right)}\right)\\
&\leq \exp\left( - \frac{\left(\delta'(1+\delta_0)^2\eta\Delta L \sum_{j\in[n]\backslash\{i\}}p\psi(\theta_j^*-\theta_i^*)   \right)^2}{4 \left( 2L \rho kp +10\log n  + \frac{1}{3}\delta'(1+\delta_0)^2\eta\Delta L  \sum_{j\in[n]\backslash\{i\}}p\psi(\theta_j^*-\theta_i^*)   \right)}\right).
\end{align*}
Now we set $\delta' = \max\{\rho^\frac{1}{2},\Delta^\frac{4}{3}, \left(\frac{\log n}{np}\right)^\frac{1}{2}\}$.
%$\delta'=\rho^{1/4}$. Recall that $\rho=o(1)$ is chosen to tend to zero sufficiently slow, and in particular we shall require $\rho\geq (\eta\Delta)^\frac{4}{3}$. 
Then, there exists some constant $C_2,C_3>0$ such that
$$
\E (L_{i,2}|A) \leq \exp\left(-C_2\rho^{-\frac{1}{2}}npL \eta^2\Delta^2\right) \leq \exp\left(-C_3\rho^{-1/2}\frac{\eta^2\Delta^2npL}{2\overline{V}(\kappa)}\right).
$$
Then,
$$
\E \left(\sum_{k-\rho k < i \leq k}L_{i,2}\Bigg|A\right) \leq \rho k\exp\left(-C_3\rho^{-1/2}\frac{\eta^2\Delta^2npL}{2\overline{V}(\kappa)}\right).
$$
By Markov inequality, we have
\begin{equation}
\inf_{A\in \mathcal{A}}\mathbb{P}_{A}\left(\sum_{k-\rho k  <i \leq k}L_{i,2} \geq \rho k  \exp\left(-\frac{1}{2}C_3\rho^{-1/2}\frac{\eta^2\Delta^2npL}{2\overline{V}(\kappa)}\right) \right) \leq \exp\left(-\frac{1}{2}C_3\rho^{-1/2}\frac{\eta^2\Delta^2npL}{2\overline{V}(\kappa)}\right). \label{eqn:L_i2_upper}
\end{equation}

\textbf{(3).} Analysis of $L_{i,3}$. By a similar argument, we also have
\begin{equation}
\inf_{A\in \mathcal{A}}\mathbb{P}_{A}\left(\sum_{k-\rho k  <i \leq k}L_{i,3} \geq \rho k  \exp\left(-\frac{1}{2}C_3\rho^{-1/2}\frac{\eta^2\Delta^2npL}{2\overline{V}(\kappa)}\right) \right) \leq \exp\left(-\frac{1}{2}C_3\rho^{-1/2}\frac{\eta^2\Delta^2npL}{2\overline{V}(\kappa)}\right). \label{eqn:L_i3_upper}
\end{equation}

~\\
\indent Now we can combine the above analyses of $L_{i,1}$, $L_{i,2}$ and $L_{i,3}$.
Since $\rho=o(1)$, the bounds (\ref{eqn:L_i2_upper}) and (\ref{eqn:L_i3_upper}) are of smaller order than (\ref{eqn:p_i_lower_bound}). We have
\begin{align}
\inf_{A\in \mathcal{A}}\mathbb{P}_{A}\left(\sum_{k-\rho k  <i \leq k}L_i \geq C_4k\exp\left(-\frac{(1+\delta_2)\eta^2\Delta^2npL}{2\overline{V}(\kappa)} - C_1' \eta \sqrt{\frac{\Delta^2npL}{\overline{V}(\kappa)}}\right) \right)\geq 1-o(1),\label{eqn:k_1}
\end{align}
for some constant $C_4 >0$. Then  (\ref{eq:etos}) and (\ref{eqn:spectral_final_1}) lead to
\begin{align}
\mathbb{P}_{(\theta^*,r^*)}\left(\sum_{k-\rho k  <i \leq k}\indc{\wh{\pi}_i <t}\geq C_4k\exp\left(-\frac{(1+\delta_2)\eta^2\Delta^2npL}{2\overline{V}(\kappa)} - C_1' \eta \sqrt{\frac{\Delta^2npL}{\overline{V}(\kappa)}}\right) \right)\geq 1-o(1).\label{eqn:k_1_1}
\end{align}

We are going to show it leads to (\ref{eqn:H_lower_1}) by selecting an appropriate $\bar \delta$ as follows. We write $\eta=\eta_{\bar\delta} = \frac{1}{2} - \frac{\overline{V}(\kappa)}{(1+{\bar\delta})\Delta^2npL}\log\frac{n-k}{k}$ to make the dependence on $\bar{\delta}$ explicit. Recall that $\delta_2$ and $C_1'$ are independent of the $\bar \delta$ in the definition of $\eta_{\bar \delta}$. First we can let $\bar\delta > \delta_1$, then we have
\begin{align*}
\frac{(1+\delta_2)\eta_{\bar \delta}^2\Delta^2npL}{2\overline{V}(\kappa)} + C_1' \eta_{\bar \delta} \sqrt{\frac{\Delta^2npL}{\overline{V}(\kappa)}} &\leq \left( 1+ \delta_2 + 2C'_1 \left(\eta_{\bar \delta}\frac{\Delta^2npL}{\overline{V}(\kappa)}\right)^{-\frac{1}{2}}\right)\frac{\eta_{\bar \delta}^2\Delta^2npL}{2\overline{V}(\kappa)} \\
& \leq \left( 1+ \delta_2 + 2C'_1 \left(\eta_{ \delta_1}\frac{\Delta^2npL}{\overline{V}(\kappa)}\right)^{-\frac{1}{2}}\right)\frac{\eta_{\bar \delta}^2\Delta^2npL}{2\overline{V}(\kappa)}\\
&\leq \left(1+\delta_4\right)\frac{\eta_{\bar \delta}^2\Delta^2npL}{2\overline{V}(\kappa)},
\end{align*}
for some $\delta_4=o(1)$ not dependent on $\bar\delta$. Here the second inequality is due to the fact that $\eta_\delta$ is in increasing function of $\delta$, and the last inequality is due to (\ref{eqn:overline_SNR_infty}). Then we can let $\bar \delta \geq \delta_4$ to have the above expression to be upper bounded by $\left(1+\bar\delta\right)\frac{\eta_{\bar \delta}^2\Delta^2npL}{2\overline{V}(\kappa)}$. Hence, (\ref{eqn:k_1_1}) leads to
\begin{align}
\mathbb{P}_{(\theta^*,r^*)}\left(\sum_{k-\rho k  <i \leq k}\indc{\wh{\pi}_i <t} \geq C_4k\exp\left(-\frac{(1+\bar\delta)\eta_{\bar \delta}^2\Delta^2npL}{2\overline{V}(\kappa)} \right)  \right)\geq 1-o(1),\label{eqn:k_2}
\end{align}
witch establishes (\ref{eqn:H_lower_1}). 

%Then we can chose $\bar \delta=o(1)$ such that $\frac{(1+\delta_2)\eta^2\Delta^2npL}{2\overline{V}(\kappa)} + C_1' \eta \sqrt{\frac{\Delta^2npL}{\overline{V}(\kappa)}}\leq \frac{(1+\bar \delta)\eta^2\Delta^2npL}{2\overline{V}(\kappa)} $. As a result, we have
%\begin{align*}
%\inf_{A\in \mathcal{A}}\mathbb{P}_{A}\left(\sum_{k-\rho k  <i \leq k}L_i \geq C_4k\exp\left(-\frac{(1+\bar\delta)\eta^2\Delta^2npL}{2\overline{V}(\kappa)} \right)  \right)\geq 1-o(1).
%\end{align*}

\indent Similar to (\ref{eqn:k_1_1}), we can establish
\begin{align*}
&\mathbb{P}_{(\theta^*,r^*)}\left(\sum_{k< i \leq k+\rho(n-k)}\indc{\wh{\pi}_i >t} \geq C_4(n-k)\exp\left(-\frac{(1+\delta_2)(1-\eta_{\bar \delta})^2\Delta^2npL}{2\overline{V}(\kappa)} - C_1' (1-\eta_{\bar \delta}) \sqrt{\frac{\Delta^2npL}{\overline{V}(\kappa)}}\right) \right)\\
&\geq 1-o(1).
\end{align*}
Due to (\ref{eqn:overline_SNR_infty}), we have $(1-\eta_{\bar \delta})\in[0,1]$, then
\begin{align*}
\frac{(1+\delta_2)(1-\eta_{\bar \delta})^2\Delta^2npL}{2\overline{V}(\kappa)} + C_1' (1-\eta_{\bar \delta}) \sqrt{\frac{\Delta^2npL}{\overline{V}(\kappa)}} & \leq \frac{(1+\delta_2)(1-\eta_{\bar \delta})^2\Delta^2npL}{2\overline{V}(\kappa)} + C_1' \sqrt{\frac{\Delta^2npL}{\overline{V}(\kappa)}} \\
& \leq   \left(1+\delta_5\right)\frac{(1-\eta_{\bar \delta})^2\Delta^2npL}{2\overline{V}(\kappa)},
\end{align*} 
for some $\delta_5=o(1)$ not dependent on $\bar \delta$.
Since $(1-\eta_{\bar \delta})^2\Delta^2npL/(2\overline{V}(\kappa)) = \eta_{\bar \delta}^2\Delta^2npL/(2\overline{V}(\kappa))  + 2\log\frac{n-k}{k}/(1+\bar\delta)$, we have
\begin{align*}
&(n-k)\exp\left(-\frac{(1+\delta_2)(1-\eta_{\bar \delta})^2\Delta^2npL}{2\overline{V}(\kappa)} - C_1' \eta_{\bar \delta} \sqrt{\frac{\Delta^2npL}{\overline{V}(\kappa)}}\right) \\
& \geq k \exp\left( \log \frac{n-k}{k}-\left(1+\delta_5\right)\frac{(1-\eta_{\bar \delta})^2\Delta^2npL}{2\overline{V}(\kappa)}\right) \\
& = k \exp\left( \frac{\bar \delta - \delta_5}{1+\bar \delta}\log \frac{n-k}{k}-\left(1+\delta_5\right)\frac{\eta_{\bar \delta}^2\Delta^2npL}{2\overline{V}(\kappa)}\right).
\end{align*}
By letting $\bar \delta \geq \delta_5$ and using the same argument as in obtaining (\ref{eqn:k_2}), we have
\begin{align}
\mathbb{P}_{(\theta^*,r^*)}\left(\sum_{k< i \leq k+\rho(n-k)}\indc{\wh{\pi}_i >t} \geq C_4k\exp\left(-\frac{(1+\bar\delta)\eta_{\bar \delta}^2\Delta^2npL}{2\overline{V}(\kappa)} \right)  \right)\geq 1-o(1),\label{eqn:k_3}
\end{align}
which establishes (\ref{eqn:H_lower_2}). 
To sum up, we can choose $\bar \delta = \max\{\delta_1,\delta_4,\delta_5\}$ to 
%obtain (\ref{eqn:k_2}) and (\ref{eqn:k_3}). Then together with (\ref{eq:etos}), we 
establish (\ref{eqn:H_lower_1}) and (\ref{eqn:H_lower_2}).

%Since $C_4$ and $\delta_4$ are not dependent on $\eta$, 
%We shall define $\eta$ as in (\ref{eqn:spectral_lower_eta_def}) with the choice of $\bar \delta =o(1)$ such that the above two displays hold. 
%We then obtain the desired conclusion by the inequalities (\ref{eq:etos}) and (\ref{eqn:spectral_final_1}).

\indent The above proof assumes that 
%$k\rightarrow\infty$, 
$\kappa=\Omega(1)$ and $\overline{\snr}$ satisfies (\ref{eqn:overline_SNR_infty}). When these two conditions do not hold, we need to slightly modify the argument.  When (\ref{eqn:overline_SNR_infty}) is not satisfied, there must exist some small constant $\bar\epsilon>0$ such that $\frac{\sqrt{(1+\bar \epsilon)\overline{\snr}}}{2} - \frac{1}{\sqrt{(1+\bar \epsilon)\overline{\snr}}} \log \frac{n-k}{k} =O(1)$. We can then take $\rho$ to be a sufficiently small constant, and the proof will go through with some slight modification.
When $\kappa=o(1)$, we can simply construct $\theta^*$ by $\theta_i^*=0$ for $1\leq i\leq k$ and $\theta_i^*=-\Delta$ for $k+1\leq i\leq n$.
%\begin{itemize}
%\item Assume $k\rightarrow\infty$. When (\ref{eqn:overline_SNR_infty}) is not satisfied, there must exist some small constant $\bar\epsilon>0$ such that $\frac{\sqrt{(1+\bar \epsilon)\overline{\snr}}}{2} - \frac{1}{\sqrt{(1+\bar \epsilon)\overline{\snr}}} \log \frac{n-k}{k} =O(1)$. We can then take $\rho$ to be a sufficiently small constant, and the proof will go through with some slight modification.
%When $\kappa=o(1)$, we can simply construct $\theta^*$ by $\theta_i^*=0$ for $1\leq i\leq k$ and $\theta_i^*=-\Delta$ for $k+1\leq i\leq n$.
%\item Assume $k=O(1)$. Due to (\ref{eq:spec-lower-condition}), it is sufficient to show $\mathbb{P}_{(\theta^*,r^*)}(\wh{r} \neq r^*)  \geq C'$ for some constant $C'$. Then we have $\mathbb{E}_{(\theta^*,r^*)}\h_k(\wh{r},r^*) \geq (2k^{-1})\mathbb{P}_{(\theta^*,r^*)}(\wh{r} \neq r^*)$ which is also a constant. To lower bound $\mathbb{P}_{(\theta^*,r^*)}(\wh{r} \neq r^*)$, note that $\mathbb{P}_{(\theta^*,r^*)}(\wh{r} \neq r^*) \geq \mathbb{P}_{(\theta^*,r^*)}(\wh{\pi}_{k+1} > \wh{\pi}_k)   \geq \mathbb{P}_{(\theta^*,r^*)}(\wh{\pi}_{k+1} > t,  \wh{\pi}_k < t)$, which is all about analyzing two entries of $\wh{\pi}$: $\wh{\pi}_{k}$ and $\wh{\pi}_{k+1}$. By going through a slightly modified argument, we can show $\mathbb{P}_{(\theta^*,r^*)}(\wh{\pi}_{k+1} > t,  \wh{\pi}_k < t) $ is lower bounded by a constant, and hence complete the proof.
%\end{itemize}
\end{proof}

Finally, we state and prove Lemma \ref{lem:spectral_lower_tail_prop} to close this section.

\begin{lemma}\label{lem:spectral_lower_tail_prop}
%Consider the same setup as in Lemma \ref{lem:two_point_testing_lower}. Additionally we assume $\theta_{k+1} =-\Delta$ and $\Delta=o(1)$.
Assume $\frac{np}{\log n}\rightarrow\infty$, $\kappa =O(1)$, $\rho=o(1)$, $k\to\infty$ and  (\ref{eq:spec-lower-condition}) holds for some arbitrarily small constant $\epsilon>0$. Choose $\kappa_1,\kappa_2\geq 0$ such that we have both $\kappa_1+\kappa_2\leq \kappa$ and
$$\frac{k\psi'(\kappa_1)(1+e^{\kappa_1})^2+(n-k)\psi'(\kappa_2)(1+e^{-\kappa_2})^2}{(k\psi(\kappa_1)+(n-k)\psi(-\kappa_2))^2/n}=\overline{V}(\kappa).$$
Define $\theta_i^*=\kappa_1$ for all $1\leq i\leq k-\rho k$, $\theta_i^*=0$ for $k-\rho k<i\leq k$, $\theta_i^*=-\Delta$ for $k+1\leq i\leq k+\rho(n-k)$ and $\theta_i^*=-\kappa_2$ for $k+\rho(n-k)< i\leq n$ and $S = \cbr{i\in[n]: i\leq k-\rho k \text{ or }i > k}$.
%We observe $\{A_i\}_{i=1}^{n}$ such that each $A_i\stackrel{ind}{\sim}\text{Bernoulli}(p)$. For each $i$ such that $i\neq k$ and $A_i = 1$, we observe $Y_{i,l} \stackrel{ind}{\sim}\text{Bernoulli}(\psi(\theta_{k} - \theta_i)),\forall l\in[L]$. Define $q_i = 1/(1+\exp(\theta_i^* - \theta_k^*)),\forall i\neq k$. 
There exists some constants $C_1>0$ such that for any $\tilde{\delta}=o(1)$, there exists $C_2>0$ and $\delta_1=o(1)$ such that for any $\eta<1/2$ and any $A\in\mathcal{A}$ where $\mathcal{A}$ is defined in (\ref{eqn:A_1})-(\ref{eqn:A_3}), we have
\begin{align}
 &\p \left(\frac{\sum_{j\in S}A_{ji}(\bar{y}_{ij}-\psi(\theta_i^*-\theta_j^*))(1+e^{\theta_j^*-\theta_i^*})}{\sum_{j\in[n]\backslash\{i\}}A_{ji}\psi(\theta_j^*-\theta_i^*)} \leq -(1+\tilde{\delta}) \eta\Delta\Bigg| A\right)   \nonumber\\
& \geq C_1\exp\left(-\frac{1+\delta_1}{2}\eta_+^2\overline{\snr}-C_2\eta_{+}\sqrt{\overline{\snr}}\right).\label{eqn:spectral_lower_tail_prop}
\end{align}
for any $k-\rho k<i\leq k$.
\end{lemma}
\begin{proof}
We suggest readers to go through the proof of Lemma \ref{lem:two_point_testing_lower} in Section \ref{sec:partial-lower-pf} first.
The proof of Lemma \ref{lem:spectral_lower_tail_prop} basically follows that of Lemma \ref{lem:two_point_testing_lower}. We will omit repeated details in the proof of Lemma \ref{lem:two_point_testing_lower} and only present key steps and calculations specific to this Lemma \ref{lem:spectral_lower_tail_prop}.

 We denote $q_j=\psi(\theta_i-\theta_j)$. Then $1+e^{\theta_j^*-\theta_i^*}=1/q_j$ and $ \psi(\theta_j-\theta_i)=1-q_j$. Then what we need to  lower bound can be written as
\begin{align*}
\p_A\left(\sum_{\ell\in[L]}\sum_{j\in S}A_{ji}\frac{q_j-y_{ij\ell}}{q_j}\geq Lt'\right),
\end{align*}
where $t'=(1+\delta^\prime)\eta\Delta\sum_{j\in[n]\backslash\{i\}}p(1-q_j)$ for some $\delta^\prime=o(1)$ due to (\ref{eqn:A_1})-(\ref{eqn:A_3}), and $\p_A$ is the conditional probability given $A$. Note that  $\delta^\prime$ can be chosen independent of $\eta$. We remark that 
$$\overline{\snr}=(1+\delta^{\prime\prime})\frac{L\Delta^2(\sum_{j\in[n]\backslash\{i\}}p(1-q_j))^2}{\sum_{j\in S}p\frac{1-q_j}{q_j}}$$
due to $\rho=o(1)$ for some $\delta^{\prime\prime}=o(1)$ independent of $\eta$.
We  still first consider the regime when
\begin{align}
\eta\sqrt{\overline{\snr}}\rightarrow\infty,\label{eqn:spectral_lower_tail_prop_1}
\end{align}
This implies $\eta\in(0,1/2)$. 

The conditional cumulant of $\sum_{j\in S}A_{ji}\frac{q_j-y_{ijl}}{q_j}$ for each $l\in[L]$ is
\begin{align*}
&\nu(u)=\sum_{j\in S}A_{ji}\log\left(q_j e^{\frac{u(q_j-1)}{q_j}}+(1-q_j)e^{u}\right)=\sum_{j\in S}A_{ji}\left[-u\frac{1-q_j}{q_j}+\log((1-q_j)e^{u/q_j}+q_j)\right].
\end{align*}
The function $\nu(u)$ acts as the same role as $K(u)$ in the proof of Lemma \ref{lem:two_point_testing_lower}.
Define
$$u^*=\arg\min_{u\geq0}\left(L\nu(u)-uLt^\prime\right).$$
Its first derivative is
\begin{align*}
\nu^\prime(u)&=\sum_{j\in S}A_{ji}\left[\frac{\frac{(1-q_j)}{q_j}e^{u/q_j}}{(1-q_j)e^{u/q_j}+q_j}-\frac{1-q_j}{q_j}\right].
\end{align*}
Following the same argument in the proof of Lemma \ref{lem:two_point_testing_lower}, we need to pin down a range for $u^*$. First due to (\ref{eqn:spectral_lower_tail_prop_1}) and $\nu^\prime(0)=0$, we have $t^\prime>0$ and thus $\nu^\prime(0)-t^\prime<0$.  
Now for $u=o(1)$, we can approximate $\nu^\prime(u)$ by Taylor expansion and obtain
\begin{align}
1-\delta_2\leq\frac{\nu^\prime(u)}{\overline{\nu^\prime}(u)}\leq1+\delta_2,\label{eqn:nu_prime_bar}
\end{align}
for some $0<\delta_2=o(1)$,
where $\overline{\nu^\prime}(u)=\sum_{j\in S}p\frac{1-q_i}{q_i}u$. Note that we can replace $A_{ji}$ by $p$ because of the condition $A\in\mathcal{A}$.  Then we consider $\tilde{u}=\frac{2t^\prime}{\sum_{j\in S}p\frac{1-q_i}{q_i}}$, which is $o(1)$ since $\Delta=o(1)$ and $\rho=o(1)$. Therefore,
\begin{align*}
&\nu^\prime(\tilde{u})-t^\prime\geq(1-\delta_2)\overline{\nu^\prime}(\tilde{u})-t^\prime=(1-\delta_2)t^\prime>0.
\end{align*}
This implies that $u^*\in\left(0, \frac{2t^\prime}{\sum_{j\in S}p\frac{1-q_i}{q_i}}\right)$. Thus $u^*=o(1)$. 

When $u=o(1)$, $\nu(u)$ also follows a second order Taylor expansion such that:
$$1-\delta_3\leq\frac{\nu(u)}{\bar{\nu}(u)}\leq1+\delta_3,$$
where $\bar{\nu}(u)=\frac{1}{2}\sum_{j\in S}p\frac{1-q_j}{q_j}u^{2}$ and $\delta_3=o(1)$ due to (\ref{eqn:A_1})-(\ref{eqn:A_3}). 

Following the change-of-measure argument in the proof of Lemma \ref{lem:two_point_testing_lower}, the probability of interest can be lower bounded by
$$\exp\left(-u^*T+L\nu(u^*)-Lu^*t^\prime\right)\mathbb{Q}_A\left(0\leq\sum_{l=1}^L\sum_{j\in S}Z_{jl}-Lt^\prime\leq T\right),$$
where $\mathbb{Q}_A$ is a measure under which $Z_{jl}$ are all independent given $A$ and follow
$$\mathbb{Q}_A(Z_{jl}=s)=e^{A_{ji}u^*s-A_{ji}\nu_j(u^*)}\p_A\left(A_{ji}\frac{q_j-y_{ijl}}{q_j}=s\right)$$
and $\nu_j(u)=-u\frac{1-q_j}{q_j}+\log((1-q_j)e^{u/q_j}+q_j)$. Then for each $Z_{jl}$ such that $A_{ij}=1$, its second and 4th moment under $\mathbb{Q}_A$ can be analyzed:
\begin{align}
\mathbb{Q}_A((Z_{jl}-\mathbb{Q}_A(Z_{jl}))^2)=\nu_j^{\prime\prime}(u^*)=\frac{1-q_j}{q_j}\frac{e^{u^*/q_j}}{[(1-q_j)e^{u^*/q_j}+q_j]^2}\in(C_1^\prime, C_2^\prime),\label{eq:spectral-second}
\end{align}
\begin{align}
\mathbb{Q}_A((Z_{jl}-\mathbb{Q}_A(Z_{jl}))^4)=\nu_j^{\prime\prime\prime\prime}(u^*)+3\nu_j^{\prime\prime}(u^*)\leq(3+C_4^\prime)\nu_j^{\prime\prime}(u^*)\leq C_3^\prime,\label{eq:spectral-fourth}
\end{align}
where (\ref{eq:spectral-fourth}) comes from
\begin{align*}
&\nu_j^{\prime\prime\prime\prime}(u)=\frac{1-q_j}{q_j^3}e^{u/q_j}\frac{(1-q_j)^3e^{3u/q_j}-3(1-q_j)^2q_je^{2u/q_j}-3(1-q_j)q_j^2e^{2u/q_j}+q_j^3}{[(1-q_j)e^{u/q_j}+q_j]^5}\\
&\leq\max_{j\in S}1/q_j^2\nu^{\prime\prime}(u)\leq C_4^\prime\nu^{\prime\prime}(u).
\end{align*}

Now, to lower bound $L\nu(u^*)-Lu^*t^\prime$:
\begin{align}
\nonumber&L\nu(u^*)-Lu^*t^\prime\geq L(1-\delta_3)\frac{1}{2}\sum_{j\in S}p\frac{1-q_j}{q_j}u^{*2}-Lu^*t^\prime\\
\label{eqn:spectral_minimization}&\geq L\min_{u\in(0,1)}\left[(1-\delta_3)\frac{1}{2}\sum_{j\in S}p\frac{1-q_j}{q_j}u^{*2}-u^*t^\prime\right]\\
\nonumber&\geq-\frac{1}{2}\frac{Lt^{\prime2}}{(1-\delta_3)\sum_{j\in S}p\frac{1-q_j}{q_j}}\\
\nonumber&\geq-\frac{1+\delta_4}{2}\eta^2\overline{\snr},
\end{align}
where (\ref{eqn:spectral_minimization}) is achieved at
$u=\frac{t^\prime}{(1-\delta_3)\sum_{j\in S}p\frac{1-q_j}{q_j}}$ and $\delta_4=o(1)$ since $\rho=o(1)$. This gives us the desired exponent. We remark that $\delta_4$ is independent of $\eta$. 

To choose $T$, observe that
$$\Var_{\mathbb{Q}_A}\left(\sum_{l\in[L]}\sum_{j\in S}Z_{jl}\right)\leq \tilde{C}_1npL,$$
for some constant $\tilde{C}_1>0$ using (\ref{eqn:A_1}) - (\ref{eqn:A_3}) , (\ref{eq:spectral-second}) and $\rho=o(1)$. Thus we choose $T=\sqrt{\tilde{C}_1npL}$, which leads to a term $C_2\eta\sqrt{\overline{\snr}}$ in the exponent for some $C_2>0$ independent of $\eta$. 

Finally, to lower bound the $\mathbb{Q}_A$ measure, we only need to verify the vanishing property of the 4th moment approximation bound in Lemma \ref{lem:CLT-stein}:
\begin{align}
\nonumber&\sqrt{L\sum_{j\in S}A_{ji}\left(\frac{\mathbb{Q}_A((Z_{j1}-\mathbb{Q}_A(Z_{j1})^4)}{(L\sum_{j\in S}A_{ji}\mathbb{Q}_A((Z_{j1}-\mathbb{Q}_A(Z_{j1})^2))^2}\right)^{3/4}}\\
&\leq \tilde{C}_2(npL)^{-1/4}\label{eq:spectral-clt}
\end{align}
where (\ref{eq:spectral-clt}) is by (\ref{eq:spectral-second}),   (\ref{eq:spectral-fourth}) and $\rho=o(1)$. To summarize, we have proved
$$\p_A\left(\sum_{l\in[L]}\sum_{j\in S}A_{ji}\frac{q_j-y_{ijl}}{q_j}\geq Lt'\right)\geq C_1\exp\left(-\frac{1+\delta_5}{2}\eta^2\overline{\snr}-C_2\eta\sqrt{\overline{\snr}}\right)$$
for some constant $C_1,C_2>0$ and $\delta_5=o(1)$, when (\ref{eqn:spectral_lower_tail_prop_1}) holds. This $\delta_5$ can be used as the $\delta_1$ in (\ref{eqn:spectral_lower_tail_prop}). We remark that $C_1,C_2, \delta_5$ are all independent of $\eta$.

Finally, when 
$$\eta\sqrt{\overline{\snr}}\leq C_3$$
for some constant $C_3>0$. This condition, together with (\ref{eqn:A_1})-(\ref{eqn:A_3}) and $\rho=o(1)$, implies that
$$Lt^\prime\leq C_4\sqrt{L\sum_{j\in S}A_{ji}\frac{1-q_i}{q_i}}.$$
Therefore,
\begin{align}
\nonumber\p_A\left(\sum_{l\in[L]}\sum_{j\in S}A_{ji}\frac{q_j-y_{ijl}}{q_j}\geq Lt'\right)&\geq\p_A\left(\sum_{l\in[L]}\sum_{j\in S}A_{ji}\frac{q_j-y_{ijl}}{q_j}\geq C_5\sqrt{L\sum_{j\in S}A_{ji}\frac{1-q_i}{q_i}}\right)\\
&\geq c_1-o(1)\label{spectral-constant}
\end{align}
where (\ref{spectral-constant}) comes from Lemma \ref{lem:CLT-stein}. The  4th moment approximation can be checked to be of order $(npL)^{-1/4}$ similarly as in (\ref{eq:spectral-clt}) using (\ref{eqn:A_1})-(\ref{eqn:A_3}) and $\rho=o(1)$ since the second and fourth moment of $\frac{q_j-y_{ijl}}{q_j}$ are at the constant order under measure $\p_A$, which completes the proof.

\end{proof}